\documentclass[12pt,a4paper]{my-amsart}
 
\usepackage{amssymb,amsmath,amsfonts,amsthm}

\usepackage{subfig,graphicx}
\usepackage{color}\usepackage{xspace}
\usepackage[colorlinks=true]{hyperref}

\usepackage{txfonts}
\usepackage{mathrsfs}

\DeclareMathSymbol{\intop}{\mathop}{symbols}{115}
\def\int{\intop}
\let \iint \relax
\DeclareMathOperator*{\iint}{\intop\!\!\intop}

\DeclareMathSymbol{\sumop}{\mathop}{largesymbols}{80}
\def\ssum{\sumop}
\let \sum \relax
\DeclareMathOperator*{\sum}{\textstyle{\ssum}}

\DeclareMathSymbol{\prodop}{\mathop}{largesymbols}{81}
\def\sprod{\prodop}
\let \prod \relax
\DeclareMathOperator*{\prod}{\textstyle{\sprod}}

\newtheorem{theorem}{Theorem}[section]
\newtheorem{proposition}[theorem]{Proposition}    
\newtheorem{lemma}[theorem]{Lemma}           
        
\newtheorem{corollary}[theorem]{Corollary}

\newcounter{thmbiss}
\newcounter{lemmabiss}

\theoremstyle{definition}
\newtheorem{definition}[theorem]{Definition}

\newtheorem{remark}[theorem]{Remark}

\renewcommand{\proofname}{{\bfseries Proof}}

\newcommand{\N}{\ensuremath{\mathbb{N}}}

\newcommand{\R}{\ensuremath{\mathbb{R}}}
\newcommand{\C}{\ensuremath{\mathbb{C}}}

\newcommand{\eps}{\varepsilon}

\newcommand{\B}{\ensuremath{\mathscr B}}
\newcommand{\D}{\ensuremath{\mathscr D}}
\newcommand{\E}{\ensuremath{\mathcal E}}
\renewcommand{\H}{\ensuremath{\mathscr H}}

\newcommand{\M}{\ensuremath{\mathcal M}}
\renewcommand{\O}{\ensuremath{\mathscr O}}

\newcommand{\U}{\ensuremath{\mathscr U}}

\newcommand{\W}{\ensuremath{\mathcal W}}

\newcommand{\Con}{\ensuremath{\mathscr C}}
\newcommand{\Cinf}{\ensuremath{\Con^\infty}}
\newcommand{\Cinfc}{\ensuremath{\Con^\infty}_{c}}
\renewcommand{\S}{\ensuremath{\mathscr S}}

\let \Re \relax
\DeclareMathOperator{\Re}{Re}
\let \Im \relax
\DeclareMathOperator{\Im}{Im}

\newcommand{\ovl}[1]{\overline{#1}}
\newcommand{\udl}[1]{\underline{#1}}

\DeclareMathOperator{\Span}{Span}

\DeclareMathOperator{\supp}{supp}

\DeclareMathOperator{\sgn}{sgn}

\DeclareMathOperator{\Char}{char}

\DeclareMathOperator{\dist}{dist}

\DeclareMathOperator{\id}{Id}

\newcommand{\inp}[2]{\langle #1, #2 \rangle} 
\newcommand{\scp}[2]{( #1 , #2 )}

\newcommand{\norm}[2]{{| #1 |}_{#2}}

\newcommand{\Norm}[2]{\| #1 \|_{#2}}

\newcommand{\imp}{\Rightarrow}

\DeclareMathOperator{\Op}{Op}
\DeclareMathOperator{\Opt}{Op_\T}

\newcommand{\brx}{\ensuremath{_{|x_d=0^+}}}
\newcommand{\brz}{\ensuremath{_{|z_N=0^+}}}
\newcommand{\brs}{\ensuremath{_{|s=0^+}}}

\newcommand{\transp}{\ensuremath{\phantom{}^{\displaystyle{t}}}}

\renewcommand{\d}{\ensuremath{\partial}}

\newcommand{\nhd}{neighborhood\xspace}
\newcommand{\nhds}{neighborhoods\xspace}

\newcommand{\cnhd}{conic neighborhood\xspace}

\renewcommand{\qedsymbol}{$\blacksquare$}

\newcommand{\bld}[1]{\mbox{\boldmath $#1$}}

\newcommand{\hf}{\frac{1}{2}}

\newcommand{\wrt}{w.r.t.\@\xspace}
\newcommand{\rhs}{r.h.s.\@\xspace}
\newcommand{\lhs}{l.h.s.\@\xspace}

\renewcommand{\ae}{a.e.\@\xspace}
\newcommand{\eg}{e.g.\@\xspace}
\newcommand{\resp}{resp.\@\xspace}
\newcommand{\viz}{viz.\@\xspace}

\newcommand{\suff}{sufficiently\xspace}
\newcommand{\st}{such that\xspace}

\newcommand{\Rp}{\R^N_+}
\newcommand{\Rpb}{\ovl{\R^N_+}}

\newcommand{\csp}{\gamma}
\newcommand{\ctp}{\eps}
\newcommand{\cpsi}{\psi_\ctp}
\newcommand{\cnorm}[1]{\norm{#1}{\ctp}}

\newcommand{\cB}{B_\ctp}

\newcommand{\y}{\varrho}

\newcommand{\ttau}{\tilde{\tau}}

\newcommand{\hm}{\hat{m}}
\newcommand{\hq}{\hat{q}}

\newcommand{\hlambdat}{\hat{\lambda}_{\mbox{\tiny ${\mathsf T}$}}}
\newcommand{\hmu}{\hat{\mu}}
\newcommand{\hrho}{\hat{\rho}}
\newcommand{\htau}{\hat{\tau}}

\newcommand{\hatt}{\hat{t}}

\newcommand{\tauast}{\tau_\ast}

\newcommand{\Ssc}{S_\tau}
\newcommand{\Psisc}{\Psi_\tau}

\newcommand{\T}{{\mbox{\tiny ${\mathsf T}$}}}
\newcommand{\Ssct}{S_{\T,\tau}}
\newcommand{\Psisct}{\Psi_{\T,\tau}}

\newcommand{\lambsc}{\lambda_{\ttau}}

\newcommand{\lambsct}{\lambda_{\T,\ttau}}
\newcommand{\Lambsct}{\Lambda_{\T,\ttau}}

\newcommand{\gt}{g_\T}

\newcommand{\Mt}{\ensuremath{\mathcal M}_{\T}}
\newcommand{\MtV}{\ensuremath{\mathcal M}_{\T, V}}
\newcommand{\MtVz}{\ensuremath{\mathcal M}_{\T, V_0}}
\newcommand{\MtVp}{\ensuremath{\mathcal M}_{\T, V'}}

\newcommand{\mi}{\alpha}
\newcommand{\smi}{\beta}
\newcommand{\tmi}{\delta}

\newcommand{\Pconj}{P_\varphi}
\newcommand{\pconj}{p_\varphi}

\newcommand{\tchi}{\tilde{\chi}}
\newcommand{\uchi}{\udl{\chi}}

\DeclareMathOperator{\trace}{tr}

\newcommand{\wlm}{w_{\ell,m}}

\newcommand{\tD}{\tilde{D}}

\newcommand{\lsct}{\lambda_{\T,\tau}}
\newcommand{\Lsct}{\Lambda_{\T,\tau}}

\newcommand{\Symbsc}{S_{\ttau}}
\newcommand{\PsiOpsc}{\Psi_{\ttau}}

\newcommand{\Cst}{\text{Cst}}

\numberwithin{equation}{section}

\begin{document}
\baselineskip 16pt

\date{\today}
% title
\title{Spectral inequality and resolvent estimate for the bi-Laplace operator}

\author{J\'{e}r\^{o}me Le Rousseau} \address{J\'{e}r\^{o}me Le
  Rousseau.  
  Laboratoire Analyse, G\'eom\'etrie et Applications, CNRS UMR 7539, 
  universit\'e Paris-Nord, 93430 Villetaneuse, France, Institut universitaire de
France.}  
\email{jlr@math.univ-paris13.fr}

\author{Luc Robbiano} \address{Luc Robbiano.  Laboratoire de Math\'ematiques de Versailles, UVSQ, CNRS, 
Universit\'e Paris-Saclay, 78035 Versailles, France
  }
\email{luc.robbiano@uvsq.fr}

\thanks{Part of this article was written when both authors were on a
  leave at Insitut Henri-Poincar\'e, Paris. They wish to thank the
  institute for its hospitality. Both authors acknowledge support from Agence Nationale de
  la Recherche (grant ANR-13-JS01-0006 - iproblems - Probl\`emes
  Inverses).
The authors also wish to thank two anonymous referees who have
contributed to improve the presentation of this article by their remarks.}

% abstract
\begin{abstract}  
  On a compact Riemannian manifold with boundary, we prove a spectral
  inequality for the bi-Laplace operator in the case of  so-called ``clamped''
  boundary conditions, that is, homogeneous Dirichlet and Neumann
  conditions simultaneously. We also prove a resolvent estimate for
  the generator of the damped plate semigroup associated with these
  boundary conditions. The spectral inequality allows one to
  observe finite sums of eigenfunctions for this fourth-order elliptic
  operator, from an arbitrary open subset of the manifold. Moreover,
  the constant that appears in the inequality grows as
  $\exp( C \mu^{1/4})$ where $\mu$ is the largest eigenvalue
  associated with the eigenfunctions appearing in the sum. This type
  of inequality is known for the Laplace operator. As an application,
  we obtain a null-controllability result for a higher-order parabolic
  equation. The resolvent estimate provides the spectral behavior of
  the plate semigroup generator on the imaginary axis. This type of
  estimate is known in the case of the damped wave semigroup. As an
  application, we deduce a
  stabilization result for the damped plate equation, with a
  $\log$-type decay.

  The proofs of both the spectral inequality and the resolvent
  estimate are based on the derivation of different types of Carleman
  estimates for an elliptic operator related to the bi-Laplace
  operator: in the interior and at some boundaries. One of these
  estimates exhibits a loss of one full derivative.  Its proof
  requires the introduction of an appropriate semi-classical calculus
  and a delicate microlocal argument.

  \medskip

  
\end{abstract}

\maketitle

\vspace*{-0.5cm}
\small
\setcounter{tocdepth}{2}
\tableofcontents

%%%%%%%%%%%%%%%
%% Section             %
%%%%%%%%%%%%%%%
\section{Introduction}
Let $A$ be the positive Laplace operator on a compact Riemannian manifold 
$(\Omega, \mathfrak{g})$, of dimension $d\geq 1$, 
with nonempty boundary $\d \Omega$. In
local coordinates, it reads
$$A = -\Delta =  |\mathfrak{g}|^{-1/2} \sum_{1\leq i,j\leq d} D_i
\big(|\mathfrak{g}|^{1/2} \mathfrak{g}^{i j} D_j \big),$$
where $D = -i \d$.
For boundary conditions, say of homogeneous Dirichlet type\footnote{What
  we describe is yet valid for more general
  boundary conditions of Lopatinskii type for the Laplace operator.}, we denote by
$0 < \omega_1 \leq \cdots \leq \omega_j\leq \cdots$, the eigenvalues of the
operator $A$, associated with a family $(\phi_j)_{j \in \N}$ of
eigenfunctions that form a Hilbert basis for $L^2(\Omega)$. We refer to
this selfadjoint operator as the Dirichlet Laplace operator.  The
following spectral inequality originates from
\cite{LR:95,LZ:98,JL:99}.
%%%%%%%%%%%%%%%%%%%%%%%%
% theorem              %
%%%%%%%%%%%%%%%%%%%%%%%%
\begin{theorem}
  \label{theorem: spectral inequality Laplace operator}
  Let $\O$ be an open subset of $\Omega$. There exists $C>0$ \st
  \begin{align}
     \label{eq: spectral inequality Laplace operator}
    \Norm{u}{L^2(\Omega)}\leq C e^{C\omega^{1/2}} \Norm{u}{L^2(\O)}, 
    \quad \omega >0, \ \ u \in \Span\{\phi_j; \ \omega_j \leq \omega\}. 
    \end{align}
\end{theorem}
It provides an observation estimate of finite sums of
eigenfunctions. The constant $C e^{C\omega^{1/2}}$ in the inequality
is in fact optimal if $\O \Subset \Omega$ \cite{JL:99,LL:12}, and
can be seen as a measure of the loss of orthogonality of the
eigenfunctions $\phi_j$ when restricted to $\O$. This inequality
has various applications.  It can be used to prove the null-controllability of the heat
equation \cite{LR:95} (see also the review article \cite{LL:12}), the
null-controllability of the thermoelasticity system \cite{LZ:98}, the
null-controllability of the thermoelastic plate system
\cite{BN:02,Miller:07}, and the null-controllability of some systems of
parabolic PDEs \cite{Leautaud:09}. 
It can also be used to estimate the $(d-1)$-Hausdorff measure of the nodal
set of finite sums of eigenfunctions of $A$, in the case of an analytic
Riemannian manifold \cite{JL:99}, recovering the result of
\cite{Lin:91}, that generalizes a result of \cite{DF:88} for eigenfunctions.

\medskip
Consider now the unbounded operator acting on $H^1_0(\Omega) \times L^2(\Omega)$
\begin{equation*}
  \mathcal A = \begin{pmatrix}
    0 & -1 \\
    A & \alpha
    \end{pmatrix},
\end{equation*}
with domain $D(\mathcal A) = (H^2(\Omega) \cap H^1_0(\Omega)) \times
H^1_0(\Omega)$, where $\alpha(x)$ is a nonnegative function. One can
prove the following 
resolvent estimate  \cite{Lebeau:96}.
\begin{theorem}
  \label{theorem: resolvent estimate A}
  Let $\O$ be an open subset of $\Omega$ and $\alpha$ be such that
  $\alpha(x) \geq \delta >0$ on $\O$. Then, the unbounded operator $i\sigma
  \id -\mathcal A$ is invertible on $\H = H^1_0(\Omega) \times L^2(\Omega)$ for
  all $\sigma \in \R$
and  there exist $K>0$ and $\sigma_0>0$ such that 
  \begin{equation}
    \label{eq: resolvent estimate laplace}
    \Norm{  (i\sigma \id -\mathcal A)^{-1}}{\mathscr L(\H,\H)}
    \le Ke^{K|\sigma|}, \qquad \sigma\in\R, \ |\sigma|\geq \sigma_0.
  \end{equation}
\end{theorem}
This resolvent estimate allows one to deduce a logarithmic type
stabilization result for the damped wave equation
\begin{equation*}
  \d_t^2 y  + A  y + \alpha \d_t  y =0, \qquad y_{|t=0} = y_0, \ \d_t
  y_{|t=0} = y_1,\qquad y_{| [0,+\infty) \times \d \Omega}=0,
\end{equation*}
for $y_0$ and $y_1$  chosen \suff regular, \eg $(y_0, y_1) \in
D(\mathcal A)$ \cite{Lebeau:96,Burq:98,Ba-Du:08}. 

\bigskip It is quite natural to wish to obtain similar 
inequalities for higher-order elliptic operators on $\Omega$, along with
appropriate boundary conditions. 
The bi-Laplace operator, that can be
encountered in models originating from elasticity for example,
appears as a natural candidate for such a study. To understand some of
the issues associated with the boundary conditions one may wish to
impose let us consider the case of a spectral inequality of the form
of \eqref{eq: spectral inequality Laplace operator}.
 If the boundary
conditions used for the bi-Laplace operator precisely make it the
square of the Laplace operator $A$ (with its boundary conditions) then
the spectral inequality is obvious as the eigenfunctions are the same
for the two operators and $\lambda_j\geq 0$ is an eigenvalue of the
bi-Laplace operator if and only if $\sqrt{\lambda_j}$ is one for the
Laplace operator. To be clearer, let us consider the positive
Dirichlet Laplace operator $A$. If $A^2$ is the bi-Laplace operator on
$\Omega$ along with the boundary conditions $u_{|\d \Omega}=0$ and
$\Delta u_{|\d \Omega}=0$, then the family $(\phi_j)_{j \in \N}$
introduced above, is in fact  composed of eigenfunctions for $A^2$ associated with
the eigenvalues $\lambda_j = \omega_j^2$.  This set of boundary conditions
is known as the ``hinged'' boundary conditions. We refer to this
operator as the ``hinged'' bi-Laplace operator, and, for 
this operator, with Theorem~\ref{theorem: spectral inequality Laplace operator}, we directly
have  the following spectral inequality, for $\O \subset \Omega$,
\begin{align}
     \label{eq: spectral inequality bi-Laplace operator}
    \Norm{u}{L^2(\Omega)}\leq C e^{C\lambda^{1/4}} \Norm{u}{L^2(\O)}, 
    \quad \lambda >0, \ \ u \in \Span\{\phi_j; \ \lambda_j \leq \lambda\}. 
    \end{align}
    One is naturally inclined to consider another set of boundary
    conditions, the so-called ``clamped'' boundary conditions,
    $u_{|\d \Omega}=0$ and $\d_\nu u_{|\d \Omega}=0$, where $\nu$ is
    the outward normal to $\d \Omega$. We refer to this operator as
    the ``clamped'' bi-Laplace operator. It is sometimes referred to
    as the Dirichlet-Neumann bi-Laplace operator.  Eigenfunctions of
    the ``clamped'' bi-Laplace operator are not related to
    eigenfunctions of the Dirichlet Laplace operator. In fact, observe
    that an eigenfunction of the ``clamped'' bi-Laplace operator
    cannot be an eigenfunction for the Laplace operator $A$,
    independently of the boundary conditions used for $A$. Indeed,
    from unique continuation arguments, if a $H^2$-function $\phi$ is
    such that $A \phi = \lambda \phi$ on $\Omega$ and
    $\phi _{|\d \Omega}=\d_\nu \phi _{|\d \Omega}= 0$, then $\phi$
    vanishes identically.  Thus, a spectral inequality for the
    ``clamped'' bi-Laplace cannot be deduced from a similar inequality
    for the Laplace operator $A$ with some well chosen boundary
    conditions. Yet, such an inequality is valuable to have at hand,
    in particular as the ``clamped'' bi-Laplace operator appears
    naturally in models. It is however often disregarded in the
    mathematical literature and replaced by the ``hinged'' bi-Laplace
    operator for which analysis can be more direct, in particular for
    the reasons we put forward above. A resolvent estimate of the form
    of \eqref{eq: resolvent estimate laplace} is also of interest
    towards stabilization results.

    The main purpose of the present article is to show that a spectral
    inequality of the form of \eqref{eq: spectral inequality Laplace operator} and
    a resolvent estimate of the form \eqref{eq: resolvent estimate
      laplace} hold for the ``clamped'' bi-Laplace operator and, more
    generally, to provide some analysis tools to carefully study
    fourth-order operators that have a product structure. Carleman
    estimates will be central in the analysis here and we shall show
    how their derivation is feasible when the so-called
    sub-ellipticity condition does not hold, which is typical for
    product operators. If $B$ is the ``clamped'' bi-Laplace operator,
    that is, the unbounded operator $B = \Delta^2$ on $L^2(\Omega)$,
    with domain $D(B) =H^4(\Omega) \cap H^2_0(\Omega)$, which turn $B$
    into a selfadjoint operator, let $(\varphi_j)_{j\in \N}$ be a
    family of eigenfunctions of $B$ that form a Hilbert basis for
    $L^2(\Omega)$, associated with the eigenvalues
    $0 < \mu_1 \leq \cdots \leq \mu_j\leq \cdots$ (the selfadjointness
    of $B$ and the existence of such a family are recalled in
    Section~\ref{sec: properties bi-laplace operator} below).  We
    shall prove the following spectral inequality.
    %%%%%%%%%%%%%%%%%%%%%%%%
    % theorem              %
    %%%%%%%%%%%%%%%%%%%%%%%% 
    \begin{theorem}[Spectral inequality for the ``clamped'' bi-Laplace
      operator] 
      \label{theorem: spectral inequality} 
      Let $\O$ be an open subset of $\Omega$.
      There exists $C>0$ \st 
      \begin{align*}
        \Norm{u}{L^2(\Omega)} \leq C e^{C \mu^{1/4}}
        \Norm{u}{L^2(\O)}, \qquad \mu >0, \quad   u \in \Span \{\varphi_j; \ \mu_j \leq \mu\}.
      \end{align*}
\end{theorem}
Note that the spectral inequality of Theorem~\ref{theorem: spectral
  inequality} was recently proven in \cite{AE:13} and  \cite{Gao:15}.
In \cite{AE:13} the coefficients and the domain are assumed to be
analytic (the techniques used for the proof are then very different and exploit
the analytic properties of the eigenfunctions). In  \cite{Gao:15},
the result is obtained  in one space dimension; yet , therein,
the factor $e^{C \mu^{1/4}}$ is replaced by $e^{C \mu^{1/2}}$,
yielding a weaker form of the spectral inequality.

 We shall present a null controllability result for the parabolic
 equation associated with $B$ which is a consequence of this spectral
 inequality. Such a result can be found in \cite{AE:13,EMZ:15} in the case
 of analytic coefficients and domain. Here, coefficients are only assumed
 smooth. We conjecture that regularity could be lowered as low as
 $W^{1,\infty}$ for the coefficients in the principal part of the
 operator. This would require further developments in the line of what
 is done in \cite{DCFLV:17} for instance. This  would however
 significantly increase the size of the present article. Note that the analytic
 setting allows the authors of \cite{AE:13,EMZ:15,EMZ:15b} to obtain
 control  properties by only requiring  the control domain to be of
 positive measure.

\medskip
We shall also prove a resolvent estimate for 
 the unbounded operator acting on $H^2_0(\Omega) \times L^2(\Omega)$,
\begin{equation}
  \label{eq: definition mathcal B}
  \mathcal B = \begin{pmatrix}
    0 & -1 \\
    B & \alpha
    \end{pmatrix},
\end{equation}
with domain $D(\mathcal B) = (H^4(\Omega) \cap H^2_0(\Omega)) \times
H^2_0(\Omega)$, where $\alpha(x)$ is a nonnegative function.
\begin{theorem}
  \label{theorem: resolvent estimate bilaplace clamped}
  Let $\O$ be an open subset of $\Omega$ and $\alpha$ be such that
  $\alpha (x)\geq \delta >0$ on $\O$. Then, the unbounded operator
  $i\sigma \id -\mathcal B$ is invertible on
  $\H = H^2_0(\Omega) \times L^2(\Omega)$ for all $\sigma \in \R$ and
  there exist $K>0$ and $\sigma_0>0$ such that
  \begin{equation*}
    %\label{eq: resolvent estimate bilaplace}
    \Norm{  (i\sigma \id -\mathcal B)^{-1}}{\mathscr L(\H,\H)}
    \le Ke^{K|\sigma|^{1/2}}, \qquad \sigma\in\R, \ |\sigma|\geq \sigma_0.
  \end{equation*}
\end{theorem}
We shall present a log type stabilization result that is a consequence
of Theorem~\ref{theorem: resolvent estimate bilaplace clamped}  for the following
damped plate equation 
\begin{equation*}
  \d_t^2 y  + \Delta^2  y + \alpha \d_t  y =0, \qquad y_{|t=0} = y_0, \ \d_t
  y_{|t=0} = y_1,
  \qquad y_{| [0,+\infty) \times \d \Omega}=\d_\nu y_{| [0,+\infty) \times \d \Omega}=0.
\end{equation*}

\medskip Both the proofs of the spectral inequality and the resolvent
estimate are based on Carleman estimates for the fourth-order operator
$P = D_s^4 + B$.

The subject of the present article is connected to that of unique
continuation, in particular through the use of Carleman
estimates. Moreover, the spectral inequality of Theorem~\ref{theorem:
  spectral inequality} is a quantified version of the unique
continuation property for finite sums of eigenfunctions.  There is an
extensive literature on unique continuation for differential operators; yet,
positive results require assumptions on the operator or on the hypersurface across
which unique continuation is pursued. For instance, a  simple-root
assumption is often made following the work of  A.~Calder\'on
\cite{Calderon:58} or the celebrated strong pseudo-convexity condition
is assumed following the work of L.~H\"ormander \cite{Hoermander:58,Hoermander:63}.
For second-order elliptic operators (with smooth complex coefficients) these
assumptions are fulfilled. However, for higher-order operators they may not be satisfied. Counterexamples for the non
uniqueness of  fourth-order and higher-order operators with smooth coefficients can be found in \cite{Plis:61}
and \cite{Hoermander:75}. See also  the monograph \cite{Zuily:83} for
manifold positive and negative results. The question of strong unique
continuation is also of interest for higher-order operators; see for
instance \cite{AB:80} for a positive result and \cite{Alinhac:80} for
a large class of negative results. 
Note that the above literature concerns unique continuation properties
{\em away from boundaries}. For the results of  Theorems~\ref{theorem:
  spectral inequality} and \ref{theorem: resolvent estimate bilaplace clamped}
the analysis we use mainly focuses on the \nhd of
the boundary of the open set $\Omega$. There are few results on
unique continuation near a boundary. Under the strong
pseudo-convexity condition the unique continuation property can be
obtained, even for higher-order operators; see \cite{Tataru:96} and
\cite{BLR:13}.  For the operator  $P = D_s^4 + B$ that   we
consider here, the strong pseudo-convexity property fails to hold near the
boundary and also away from it. General approaches as developed in
\cite{Tataru:96,BLR:13} cannot be used. This is one of the interests of
the present article.

\subsection{On Carleman estimates}
Carleman estimates are weighted {\em a priori} inequalities for the solutions of
a partial differential equation (PDE), where the weight is of exponential type.  For a partial
differential operator $Q$ {\em away from boundaries}, it takes the form:
\begin{equation*}
  %\label{eq: intro Carleman}
  \Norm{e^{\tau \varphi} w}{L^2}
  \lesssim \Norm{e^{\tau \varphi} Q w}{L^2}, \qquad 
  w \in \Cinfc(\Omega), \ \tau \geq \tau_0.
\end{equation*}
The exponential weight involves a parameter $\tau$ that can be taken
as large as desired.  Additional terms in the \lhs, involving
derivatives of $u$, can be obtained depending on the order of $Q$ and
on the {\em joint} properties of $Q$ and $\varphi$.  For instance for
a second-order operator $Q$, such an estimate can take the form
\begin{align}
  \label{eq: example Carleman}
  \tau^{3/2} \Norm{e^{\tau \varphi} u }{L^2}
  + \tau^{1/2} \Norm{e^{\tau \varphi} D_x u }{L^2}
  \lesssim
  \Norm{e^{\tau \varphi} Q u }{L^2}, \qquad 
  \tau \geq \tau_0,\  u \in \Cinfc(\Omega).
\end{align}
One says that this estimate is characterized by the {\em loss of a
  half derivative}. This terminology originates from the underlying
semi-classical calculus where one gives the same strengths to the
parameter $\tau$ and to $D$. Whereas $Q$ is a second-order operator,
the \lhs only exhibits derivatives or powers of $\tau$ of order
$3/2$. For most operators, this cannot be improved
\cite{Hoermander:63,Hoermander:V4}. In the proof of a Carleman
estimate one introduces the so-called conjugated operator
$Q_\varphi = e^{\tau \varphi} Q e^{-\tau \varphi}$, and estimate
\eqref{eq: example Carleman} reads 
\begin{align*}
  \tau^{3/2} \Norm{v }{L^2}
  + \tau^{1/2} \Norm{D_x v }{L^2}
  \lesssim
  \Norm{Q_\varphi v }{L^2}, \qquad 
  \tau \geq \tau_0,\  v= e^{\tau \varphi}u, \ \ u \in \Cinfc(\Omega).
\end{align*}

This type of estimate was used for the first time by T.~Carleman
\cite{Carleman:39} to achieve uniqueness properties for the Cauchy
problem of an elliptic operator. Later, A.-P.~Calder\'on and
L.~H\"ormander further developed Carleman's method
\cite{Calderon:58,Hoermander:58}.  To this day, the method based on Carleman estimates
remains essential to prove unique continuation properties;
see for instance \cite{Zuily:83} for an overview.  On such questions,
more recent advances have been concerned with differential operators
with singular potentials, starting with the contribution of D.~Jerison
and C.~Kenig \cite{JK:85}. There, Carleman estimates rely on
$L^p$-norms rather than $L^2$-norms as in the estimates
above. The proof of such $L^p$ Carleman estimates is very
delicate. The reader is also referred to \cite{Sogge:89,KT:01,KT:02,DDSF:05,KT:05}.
In more recent years, the field of applications of Carleman estimates
has gone beyond the original domain; they are also used in the study
of:
\begin{itemize}
\item Inverse problems, where Carleman estimates are used to obtain
  stability estimates for the unknown sought quantity (\eg
  coefficient, source term) with respect to norms on measurements
  performed on the solution of the PDE, see \eg
  \cite{BK:81,Isakov:98,Kubo:00,IIY:03}; Carleman estimates are also
  fundamental in the construction of complex geometrical optic solutions
  that lead to the resolution of inverse problems such as the Calder\'on
  problem with partial data \cite{KSU:07,DKSU:09}.
  
\item Control theory for PDEs;  Carleman estimates yield the null controllability of linear parabolic equations
\cite{LR:95} and the null controllability of classes of semi-linear
parabolic equations \cite{FI:96,Barbu:00,FZ:00}.
They can also be used to prove unique continuation properties, that in
turn are crucial for the treatment of low frequencies for  exact
controllability results for  hyperbolic equations as in \cite{BLR:92}.
\end{itemize}

\bigskip
To indicate how the spectral inequality of Theorem~\ref{theorem:
  spectral inequality} for the bi-Laplace operator $B$ can be proven
by means of Carleman estimates, 
we first review a method, that yields the spectral inequality of
Theorem~\ref{theorem: spectral inequality Laplace operator} for the
Laplace operator $A$. In this introductory section, we have
chosen to mainly focus on the method of proof of the spectral
inequality; a comprehensive presentation including a presentation of
the proof of the resolvent estimates of Theorems~\ref{theorem:
  resolvent estimate A} and \ref{theorem: resolvent
  estimate bilaplace clamped} would not bring any further insight to
the reader as the line
of arguments is quite similar.

\subsection{A method to prove the spectral inequality for the Laplace
  operator}
The method we describe here originates from \cite{LR:95}.  We consider
the elliptic  operator $P_A = D_s^2 + A$ on $Z= (0,S_0)\times \Omega$, for some
$S_0>0$ meant to remain fixed.  We also pick $0< \alpha< S_0/2$.
Three different types of Carleman estimates are proven for the
operator $P_A$: (i) in the interior of $(0,S_0)\times \Omega$; (ii) at
the boundary $\{ s=0\} \times \Omega$; (iii) at the boundary
$(\alpha,S_0-\alpha)\times \d\Omega$. The three regions where these
Carleman estimates are derived are illustrated in Figure~\ref{fig:
  Carleman method}.
%% figure
\begin{figure}
\begin{center}
\begin{picture}(0,0)%
\includegraphics{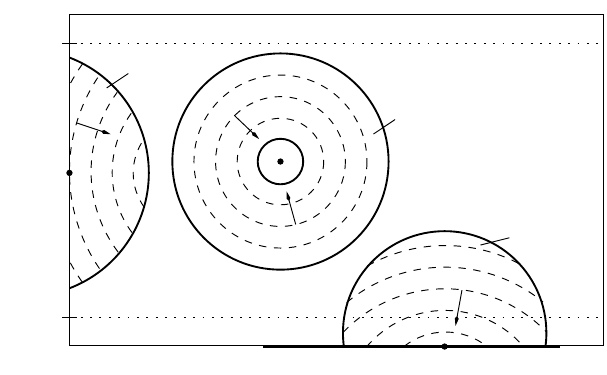}%
\end{picture}%
\setlength{\unitlength}{3947sp}%
\begingroup\makeatletter\ifx\SetFigFont\undefined%
\gdef\SetFigFont#1#2#3#4#5{%
  \reset@font\fontsize{#1}{#2pt}%
  \fontfamily{#3}\fontseries{#4}\fontshape{#5}%
  \selectfont}%
\fi\endgroup%
\begin{picture}(4836,2987)(340,-3063)
\put(355,-445){\makebox(0,0)[lb]{\smash{{\SetFigFont{10}{12.0}{\rmdefault}{\mddefault}{\updefault}{\color[rgb]{0,0,0}$S_0 - \alpha$}%
}}}}
\put(1414,-3003){\makebox(0,0)[lb]{\smash{{\SetFigFont{10}{12.0}{\rmdefault}{\mddefault}{\updefault}{\color[rgb]{0,0,0}$\Omega$}%
}}}}
\put(1374,-696){\makebox(0,0)[lb]{\smash{{\SetFigFont{10}{12.0}{\rmdefault}{\mddefault}{\updefault}{\color[rgb]{0,0,0}$V_3$}%
}}}}
\put(4428,-2010){\makebox(0,0)[lb]{\smash{{\SetFigFont{10}{12.0}{\rmdefault}{\mddefault}{\updefault}{\color[rgb]{0,0,0}$V_2$}%
}}}}
\put(658,-223){\makebox(0,0)[lb]{\smash{{\SetFigFont{10}{12.0}{\rmdefault}{\mddefault}{\updefault}{\color[rgb]{0,0,0}$S_0$}%
}}}}
\put(662,-2636){\makebox(0,0)[lb]{\smash{{\SetFigFont{10}{12.0}{\rmdefault}{\mddefault}{\updefault}{\color[rgb]{0,0,0}$\alpha$}%
}}}}
\put(727,-2877){\makebox(0,0)[lb]{\smash{{\SetFigFont{10}{12.0}{\rmdefault}{\mddefault}{\updefault}{\color[rgb]{0,0,0}$0$}%
}}}}
\put(664,-1523){\makebox(0,0)[lb]{\smash{{\SetFigFont{8}{9.6}{\rmdefault}{\mddefault}{\updefault}{\color[rgb]{0,0,0}$z^{(3)}$}%
}}}}
\put(2619,-2977){\makebox(0,0)[lb]{\smash{{\SetFigFont{10}{12.0}{\rmdefault}{\mddefault}{\updefault}{\color[rgb]{0,0,0}$\O$}%
}}}}
\put(3850,-2979){\makebox(0,0)[lb]{\smash{{\SetFigFont{8}{9.6}{\rmdefault}{\mddefault}{\updefault}{\color[rgb]{0,0,0}$z^{(2)}$}%
}}}}
\put(3516,-1061){\makebox(0,0)[lb]{\smash{{\SetFigFont{10}{12.0}{\rmdefault}{\mddefault}{\updefault}{\color[rgb]{0,0,0}$V_1$}%
}}}}
\put(2559,-1490){\makebox(0,0)[lb]{\smash{{\SetFigFont{8}{9.6}{\rmdefault}{\mddefault}{\updefault}{\color[rgb]{0,0,0}$z^{(1)}$}%
}}}}
\end{picture}%

\caption{Location and geometry of the three types of estimates. Dashed
  are level sets for the weight functions $\varphi$ used in the Carleman
  estimates. Arrows
represent the directions of the (non vanishing) gradient of $\varphi$.}
\label{fig: Carleman method}
\end{center}
\end{figure}
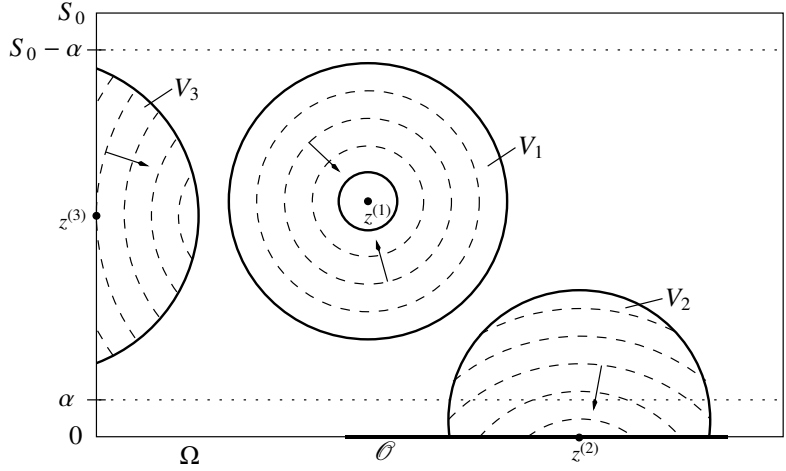
It is simpler to first describe Case (i), that is, the estimate in
the interior. In Figure~\ref{fig:
  Carleman method}, this corresponds to the \nhd $V_1$ of some point
$z^{(1)} \in Z$. There, the Carleman estimate for this operator $P_A$
is of the form described above, that is, 
\begin{align}
  \label{2.carleman}
  \tau^{3/2}\Norm{e^{\tau \varphi} w }{L^2(Z)}
  + \tau^{1/2} \Norm{e^{\tau \varphi} D_z w }{L^2(Z)} 
  \lesssim
  \Norm{e^{\tau \varphi} P_A  w }{L^2(Z)}, 
\end{align}
where the weight function $\varphi = \varphi(z)$ is real-valued with a
non-vanishing gradient, $\tau$ is a large positive parameter, and $w$
is any smooth function compactly supported in $V_1$. In fact, this
estimate holds if the so-called sub-ellipticity condition is fulfilled
by $P_A$ and $\varphi$.  If $p_A(z, \zeta)$ is the principal symbol of
$P_A$, the sub-ellipticity condition in $V_1$ reads
\begin{align}
  \label{3. carleman}
  p_A(z,\zeta + i \tau d\varphi(z)) =0  
  \ \ \imp  \ \ \frac{1}{2 i} \{ \ovl{p_A (z,\zeta + i \tau
  d\varphi(z))}, p_A(z,\zeta + i \tau d\varphi(z))  \} >0, 
\end{align}
for $z \in V_1$, $\zeta \in \R^{d+1}$, and $\tau \geq 0$. It is in
fact equivalent to a Carleman estimate of the form~\eqref{2.carleman}
for $P_A$
(see \cite{Hoermander:63} or \cite{LL:12}). Observe that $p_A(z,
\zeta + i \tau d\varphi(z))$ is the semi-classical principal symbol of
the conjugated operator $P_{A,\varphi} = e^{\tau \varphi} P_A e^{-\tau
  \varphi}$. 

The function $\varphi$ is chosen of the form
$\varphi(z) = \exp(- \csp |z-z^{(1)}|^2)$ and $V_1$ is an annulus around
$z^{(1)}$, thus avoiding where the gradient of $\varphi$ vanishes (see
Figure~\ref{fig: Carleman method}). For
$\csp>0$ chosen \suff large, one can prove that the sub-ellipticity
condition~\eqref{3. carleman} holds and thus estimate
\eqref{2.carleman} is achieved (see \eg \cite{LR:95} or \cite{LL:12}).

From estimate~\eqref{2.carleman}, one can deduce the following local
interpolation inequality, for all $r>0$ chosen \suff small, for some
$\delta \in (0,1)$ (see \eg \cite{LR:95}),
\begin{align}
  \label{4.carleman}
    \Norm{v}{H^1(B(z^{(1)},3 r))}
    \lesssim  \Norm{v}{H^1(Z)}^{1-\delta}
    \Big(\Norm{P_A v}{L^2(Z)} + \Norm{v}{H^1(B(z^{(1)},r))}
  \Big)^{\delta},
  \qquad v \in H^2(Z).
\end{align}

\bigskip
We now consider Case (ii). 
In a \nhd $V_2$  of a point $z^{(2)} \in \{0\} \times \O$, one can derive an
estimate of the same form as \eqref{2.carleman},  yet,  with two trace terms in
the \rhs, that is, 
\begin{align}
  \label{5.carleman}
  \sum_{|\smi|\leq  1} \tau^{3/2-|\smi|} \Norm{e^{\tau
  \varphi} D^\smi w }{L^2(Z)}
  \lesssim
  \Norm{e^{\tau \varphi} P_A  w }{L^2(Z)} + 
  \tau^{1/2} \Big(\norm{e^{\tau \varphi} w\brs}{H^1(\O)}
  + \norm{e^{\tau \varphi}\d_s w\brs}{L^2(\O)}\Big), 
\end{align}
for $\tau \geq \tau_0\geq 1$ and $w$ smooth up to the boundary
$\{s=0\}$,  with $\supp(w) \cap Z \subset V_2$, with $V_2$ as  represented in
Figure~\ref{fig: Carleman method}. This can be obtained by locally choosing
a weight function of the form $\varphi (z) = \exp(\csp \psi(z))$ with
$\psi(z)$ such that $\d_s \psi(z) \leq -C <0$ in $V_2$ and choosing the
parameter $\csp>0$ \suff large (see \eg \cite{LZ:98}). We
use the notation $\Norm{.}{}$ for functions in the interior of the
domain and $\norm{.}{}$ for functions on the boundaries. 

From estimate~\eqref{5.carleman} one deduces the following local
interpolation inequality: 
 there exist $V\subset V_2 $
  and $\delta\in(0,1)$ such that  
    \begin{align}
    \label{6.carleman}
    \Norm{v}{H^1(V\cap Z)} \lesssim \Norm{v}{H^1(Z)}^{1-\delta}
    \Big(
    \Norm{P_A v}{L^2(Z)} +   \norm{v\brs}{H^1(\O)} 
    + \norm{\d_s v\brs}{L^2(\O)}   \Big)^\delta,  \qquad v \in H^2(Z).
  \end{align}  

\bigskip 
We finally consider Case (iii). 
In a \nhd of a point $z^{(3)} \in (\alpha, S_0- \alpha) \times \d\Omega$, one can derive an
estimate of the same form as \eqref{2.carleman},  yet,  with a single trace term in
the \rhs, that is, 
\begin{multline}
  \label{7.carleman}
  \sum_{|\smi| \leq 1} \tau^{3/2-|\smi|} 
  \Norm{e^{\tau \varphi} D^\smi w }{L^2(Z)} 
  + \tau^{1/2} \norm{e^{\tau \varphi} \d_\nu w_{|\d Z}}{L^2((\alpha, S_0 -
  \alpha) \times \d \Omega)}\\
  \lesssim
  \Norm{e^{\tau \varphi} P_A  w }{L^2(Z)} + 
  \sum_{|\smi'| \leq 1} \tau^{3/2 - |\smi'|} 
  \norm{e^{\tau \varphi} D_{\T}^{\smi'} w_{|\d Z}}{L^2((\alpha, S_0 -
  \alpha) \times \d \Omega)}, 
\end{multline}
for $\tau \geq \tau_0\geq 1$ and $w$ smooth up to the boundary
$(\alpha, S_0 - \alpha) \times \d \Omega$, with
$\supp(w) \cap Z \subset V_3$, with $V_3$ as represented in
Figure~\ref{fig: Carleman method}. This can be obtained by locally
choosing a weight function of the form
$\varphi (z) = \exp(\csp \psi(z))$ with $\psi(z)$ such that
$\d_\nu \psi(z) \leq -C <0$ in $V_3$, where $\nu$ is the outward
normal to $\d \Omega$, and choosing the parameter $\csp>0$ \suff large
(see \eg \cite{LR:95}). Here, for $|\smi'|\geq 1$,  $D_{\T}^{\smi'}$ stand as
differentiations in the tangential directions only, along vector
fields that form a local frame.

From estimate~\eqref{7.carleman} one deduces the following local
interpolation inequality: 
 there exist $V\subset V_3$, with $\ovl{V}$ \nhd of $z^{(3)}$ in
 $\ovl{Z}$, 
 some open subset $\mathscr Q \subset V_3$ with positive distance
  to the boundary, 
  and $\delta\in(0,1)$ such that  
 \begin{align}
    \label{8.carleman}
    \Norm{v}{H^1(V\cap Z)} \lesssim  \Norm{v}{H^1(Z)}^{1-\delta}
    \left(\Norm{P_A v}{L^2(Z)} + \Norm{v}{H^1(\mathscr Q)}
   \right)^\delta,
   \qquad v \in H^2(Z), \ \  v_{|(0,S_0) \times \d \Omega}=0. 
  \end{align}  

  \bigskip
  The three interpolation inequalities \eqref{4.carleman},
  \eqref{6.carleman}, and \eqref{8.carleman} can be used to form a global
  interpolation inequality, by means of compactness arguments. In particular,
  the interior inequality~\eqref{4.carleman} permits the ``propagation'' of
  the estimate. Then,  there exists
  $\delta \in (0,1)$, such that 
  \begin{align}
    \label{9.carleman}
    \Norm{v}{H^1((\alpha, S_0 - \alpha)\times \Omega)} 
    \lesssim \Norm{v}{H^1(Z)}^{1-\delta}
    \left(\Norm{P_A v}{L^2(Z)} +  \norm{v\brs}{H^1(\O)} 
    + \norm{\d_s v\brs}{L^2(\O)}  \right)^\delta,
  \end{align}  
  for $v \in H^2(Z)$ satisfying $v_{|(0,S_0) \times \d \Omega}=0$.
  This inequality then implies the spectral property for the Laplace
  operator for
  $u = \sum_{\omega_j \leq \omega} u_j \phi_j \in \Span\{\phi_j; \ \omega_j \leq \omega\}$, if applied to a well
  chosen function $v(s,x)$, namely, 
  \begin{align*}
    v(s,x) = \sum_{\omega_j \leq \omega} u_j  \omega_j^{-1/2}\sinh(\omega_j^{1/2} s)
    \phi_j(x). 
  \end{align*}
  Details can for instance be found in \cite{LL:12}. In the present paper, we shall apply
  this approach for the bi-Laplace operator, the argument is provided
  in details in Section~\ref{sec: from interpolation to spectral inequality}.

\subsection{Outline of the proof of the spectral inequality for the
 bi-Laplace operator}
\label{sec: proof outline bi-laplace}
Above we described how Carleman estimates can be used to prove a
spectral inequality of the form given in Theorem~\ref{theorem:
  spectral inequality Laplace operator} for the Laplace operator. To prove the spectral
inequality of Theorem~\ref{theorem: spectral inequality} for the
``clamped'' bi-Laplace operator, we shall prove several Carleman
estimates for the following fourth-order elliptic operator
\begin{equation*}
  P = D_s^4 + \Delta^2\qquad  \text{on}\   Z= 0,S_0) \times \Omega.
\end{equation*}  
As for $P_A$
above, we shall prove such estimates at three different locations: (i)
in the interior of $(0,S_0)\times \Omega$, in Section~\ref{sec:
  estimate away from boundaries}; (ii) at the boundary
$\{ s=0\} \times \Omega$, in Section~\ref{sec: Carleman boundary s=0};
(iii) at the boundary $(\alpha,S_0-\alpha)\times \d\Omega$, in
Section~\ref{sec: Carleman boundary x}. In Section~\ref{sec: spectral
  inequality}, these three types of estimations are then used to
achieve local interpolation inequalities that can be used to prove,
first, a global interpolation inequality and, second, the spectral
inequality of Theorem~\ref{theorem: spectral inequality}.
Note that for the proof of the resolvent estimate of Theorem~\ref{theorem:
  resolvent estimate bilaplace clamped} only steps (ii) and (iii) are needed.

\medskip
\paragraph{\bfseries Cases (i) and (ii)}
The weight functions that we shall
use will be the same as that used for the operator $P_A$ for Cases~(i)
and (ii). In Case (ii), the estimate we obtain for $P$ takes the form
\begin{multline*}
    \tau^{-1/2}
    \sum_{|\mi| \leq 4} \Norm{\tau^{4-|\mi|} e^{\tau \varphi} D_{s,x}^\mi u}{L^2(Z)}
    \lesssim
    \Norm{e^{\tau \varphi} P u}{L^2(Z)}\\
    + \sum_{j=0}^3\Big( \tau^{7/2 - j}\norm{e^{\tau \varphi} D_{s}^j
  u\brs}{L^2(\Omega)} +  \norm{e^{\tau \varphi} D_{s}^j
  u\brs}{H^{7/2-j}(\Omega)} \Big), 
  \end{multline*}
  for functions localized near a point $z^{(2)} \in \{0\}\times \O$,
  with $\O \subset \Omega$.  We have observation terms at the boundary
  $\{s=0\}$.  We
use the notation $\Norm{.}{}$ for functions in the interior of the
domain and $\norm{.}{}$ for functions on the boundaries.

  Note that this estimate is characterized by the loss
  of half-derivative, similarly to the estimate one can derive for
  $P_A$. In fact, the sub-ellipticity condition holds in $V_2$ despite
  the fact that $\Pconj = e^{\tau \varphi} P e^{-\tau \varphi}$ can be written as a product of two
  operators, $\Pconj = Q_1Q_2$, as, here, $\Char(Q_1) \cap \Char(Q_2)
  = \emptyset$. 

In Case (i), however, the estimate we obtain is characterized
by the loss of {\em one full} derivative, taking the form
\begin{align*}
    \sum_{|\mi| \leq 4} \tau^{3-|\mi|} 
    \Norm{e^{\tau \varphi} D^{\mi} u}{L^2(Z)} 
    \lesssim \Norm{e^{\tau \varphi} P u}{L^2(Z)}, 
  \end{align*}
  for functions compactly supported away from boundaries.  In fact,
  this loss cannot be improved as explained in Section~\ref{sec:
    Carleman higher-order operators}.  Here also,  the operator $\Pconj$
  can be written as a product of two operators, $\Pconj = Q_1Q_2$, and here, as
  opposed to Case (ii), we have
  $\Char(Q_1) \cap \Char(Q_2) \neq \emptyset$.

  We provide fairly short proofs of the
  Carleman estimates in Cases (i) and (ii) in Sections~\ref{sec: estimate away from boundaries} and
  \ref{sec: Carleman boundary s=0}.
   Note, however, that the loss of a full
  derivative in Case (i) does not create any obstruction to the
  derivation of a local interpolation inequality in Section~\ref{sec: spectral inequality}.

  \begin{remark}
    Sub-ellipticity does not hold in $V_1$. The reader should note
    that the failure of the sub-ellipticity property does not
    automatically imply a loss of one full derivative. The phenomena
    that can occur require a fine analysis to be understood. This is
    carried out in \cite{lerner:88}. Roughly speaking, if
    sub-ellipticity does not hold, and if some iterated Poisson brackets
    vanish up to order $k$ and an iterated Poisson bracket of order
    $k+1$ is positive, then an estimate can be obtained with a loss of
    $k/(k+1)$ derivative.  In the present case, as we can prove that
    the loss of one full derivative cannot be improved, we then know
    that all the iterated Poisson brackets used in \cite{lerner:88} 
    vanish. The essential problem is that the conjugated operator
    $\Pconj$ can be written as a product of two operators $Q_1Q_2$, and
    in the case  $\Char(Q_1) \cap \Char(Q_2) \neq \emptyset$, not only
    does sub-ellipticity not hold, but we see that the iterated Poisson
    brackets also vanish.
  \end{remark}

  \medskip 
  \paragraph{\bfseries Case (iii).} This case is delicate and the
  derivation of the Carleman estimate at the boundary
  $(\alpha,S_0-\alpha)\times \d\Omega$ is one of the main results of
  the present article. This case is also precisely where we have to
  take into account the boundary conditions for the bi-Laplace
  operator $B$.  The estimate we obtain in Case (iii) in
  Section~\ref{sec: Carleman boundary x} is characterized by the loss
  of {\em one full} derivative and, as for case (i), this cannot be
  improved as explained in Section~\ref{sec: Carleman higher-order
    operators}. This is a source of major complications for the proof
  of the Carleman estimate itself. As in Case (i) this, however, does
  not create any obstruction in the derivation of the local
  interpolation inequality in Section~\ref{sec: spectral inequality}.
  In fact, the proof of the local Carleman estimate in $V_3$, a \nhd
  of a point of  the boundary
  $(\alpha,S_0-\alpha)\times \d\Omega$, requires microlocal arguments.
  This implies the introduction of microlocalization operators that
  realize some partition of unity in phase space over $V_3$. For each induced
  microlocal region, a Carleman estimate is derived. One region is
  less favorable: there, the fourth-order conjugated operator $\Pconj$
  can we written as a product of four first-order factors, and two of
  them fail to be elliptic. Moreover, their characteristic sets
  intersect; sub-ellipticity does not hold there and, in fact, this
  generates a loss of a full derivative in the estimation. There, the
  a priori estimate one derives permits to only estimate the
  semi-classical $H^3$-norm, \viz,
  $\Norm{w}{3,\tau} \asymp \tau^3 \Norm{w}{L^2} + \Norm{w}{H^3}$.  In
  other microlocal regions  over $V_3$, the conjugated operator $\Pconj$ exhibits
  at most a non elliptic first-order factor only yielding a half
  derivative loss as sub-ellipticity holds. If one does not proceed
  carefully, the derivation in the least favorable region yields error
  terms that can be of the same strength as the norm
  $\Norm{w}{3,\tau}$, preventing to conclude positively to the
  Carleman estimate.

  We define the weight function in the form
  $\varphi(z) = e^{\csp \psi(z)}$ and keep track of the parameter
  $\csp$ that is meant to be large. The function $\psi$ is chosen such that
  $\d_\nu \psi \leq -C < 0$ in a \nhd of a point of the boundary where
  we try to derive the Carleman estimate.  The use of an exponential form
  for the weight function can already be found in the seminal work of
  L.~H\"ormander (\cite[Section 8.6]{Hoermander:63} and \cite[Section
  28.3]{Hoermander:V4}), in connexion with the celebrated notions of
  pseudo-convexity and strong pseudo-convexity. This introduces a
  second large parameter.  Several authors have derived Carleman
  estimates for some operators in which the dependence upon the second
  large parameters is explicit. See for instance \cite{FI:96}. Such
  result can be very useful to address applications such as inverse
  problems. On such questions see for instance
  \cite{Eller:00,EI:00,IK:08,BY:10}.  In \cite{LeRousseau:12}, an
  analysis framework is introduced, based on the Weyl-H\"ormander
  calculus (\cite{Hoermander:79}, \cite[Sections
  18.4--18.6]{Hoermander:V3}), that allows one to describe the explicit
  dependence of Carleman upon the second large parameter $\csp$ for
  general classes of operators. That analysis is carried out
  away from boundaries. Here, we use that approach by means of a
  tangential Weyl-H\"ormander calculus.   The introduction of the
  second large parameter $\csp$ allows us to handle some error terms
  in the derivation of the Carleman estimate in $V_3$. This is however not
  sufficient to have control over all the error terms that appear in the
  microlocal region within $V_3$ where sub-ellipticity does not hold,
  since the operator under study is a product of two second-order
  operators (see above). 

  Yet, when one attempts to derive the estimate, one realizes that the
  derivation is possible in the case $\varphi$, and thus $\psi$, only
  depend on the normal variable to the boundary.  Yet, for the
  interpolation inequality we wish to derive at the boundary
  $(\alpha,S_0-\alpha)\times \d\Omega$, some convexity of the level
  sets of the weight function $\varphi$ is needed: $\varphi$ cannot be
  constant along the boundary. This is illustrated in Figure~\ref{fig:
    Carleman method} (in the \nhd $V_3$).  We thus introduce the
  function $\cpsi(z) =\psi( \ctp z', z_N)$, where $z'$ denotes the
  tangential variables and $z_N$ denotes the normal variable (in local
  coordinates where the boundary is given by $\{z_N=0\}$), and we set
  $\varphi(z) = e^{\csp \cpsi(z)}$. Here, $\ctp$ is a small parameter,
  $\ctp \in (0,1)$. Keeping track of the dependence of the microlocal
  estimates in this third parameter too then allows one to obtain a
  Carleman estimate, at the boundary, with a weight function with some
  convexity of its level sets with respect to the boundary. This is
  precisely done by extending some of the work of \cite{LeRousseau:12}
  and introducing a Weyl-H\"ormander calculus, with three parameters: the large
  semi-classical parameter $\tau$, the second large parameter $\csp$,
  and this new parameter $\ctp \in (0,1)$ that controls the convexity
  of the level sets of the weight function. Note that even in the case
  $\psi = \psi(z_n)$, the proof of the Carleman estimate relies on
  taking the second parameter $\csp$ \suff large (see the end of
  Proposition~\ref{prop: microlocal estimate E0} below). The
  introduction of the parameter $\ctp$ alone would not be
  sufficient. Only the joint introduction of the two parameters allows
  us to conclude positively to the Carleman estimate in the microlocal
  region where a full derivative is lost.

  All the different
  microlocal estimates need to be derived within the refined
  semi-classical calculus with three parameters. Arguments are based on the ellipticity or
  sub-ellipticity of the different factors building the fourth-order
  operator $\Pconj$, and the position of theirs roots in the complex
  plane. This analysis follows in part from the different works
  \cite{Bellassoued:03,LRR:10,LRR:11,LR-L:13,CR:14}.

  Eventually, the various microlocal estimates we obtain need to be
  patched together.  This procedure generates commutators of the fourth operator
  $\Pconj$ and the microlocal cut-offs, leading to  some third-order
  error terms that can be handled thanks to the better microlocal
  estimates obtained away from the least favorable region.

  Near a point of the boundary $\d Z = (0,S_0) \times \d\Omega$
  locally written in the form $\{x_d =0\}$ with
  $Z = (0,S_0) \times \Omega=\{ x_d>0\}$, the estimate we obtain, for
  $\tau$ and $\csp$ large and $\eps$ small,  is of
  the form
  \begin{multline*}
    \csp 
    \sum_{|\mi| \leq 4}\Norm{\ttau^{3-|\mi|}e^{\tau \varphi}
    D_{s,x}^\mi u}{L^2(Z)} 
    + 
    \sum_{0\leq j \leq 3} \norm{e^{\tau \varphi} D_{x_d}^r u_{|\d Z}}{7/2-j,\ttau}
    \lesssim \Norm{e^{\tau \varphi} P u}{L^2(Z)}
    + 
    \sum_{j=0,1} \norm{e^{\tau \varphi} D_{x_d}^j u_{|\d Z}}{7/2-j,\ttau}. 
  \end{multline*}
  On the \lhs we find norms of all traces; on the \rhs we only have
  observation with the traces $u_{|\d Z}$ and $D_{x_d} u_{|\d Z}$
  associated with the clamped boundary conditions. Here $\ttau = \tau \csp \varphi$.

\subsection{On Carleman estimates for higher-order elliptic operators}
\label{sec: Carleman higher-order operators}

If $Q$ is an elliptic operator of even order $m$, and $\varphi$ is
a weight function \st the couple $(P, \varphi)$ satisfies the
sub-ellipticity condition (as stated above),  then a Carleman estimate 
can be obtained, even at a boundary, for instance with the results of
\cite{BLR:13}. We use those results in Section~\ref{sec: Carleman
  boundary s=0} for the proof of the Carleman estimate at the boundary
$\{s=0\}$.

If $m\geq 4$, it is however quite natural to not have the
sub-ellipticity condition, in particular if the operator $Q$ is in the
form of a product of two operators, say $Q= Q_1 Q_2$. Denote by $q$,
$q_1$, and $q_2$ the principal symbols of $Q$, $Q_1$, and $Q_2$
respectively.  The conjugated operator
$Q_\varphi = e^{\tau \varphi} Q e^{-\tau \varphi}$ reads
$Q_\varphi = Q_{1,\varphi} Q_{2,\varphi}$, with
$Q_{k,\varphi} = e^{\tau \varphi} Q_k e^{-\tau \varphi}$, $k=1,2$. If
we have
$\Char(Q_{1,\varphi}) \cap \Char(Q_{2,\varphi}) \neq \emptyset$ then
the sub-ellipticity condition fails to hold.  In fact, if $q_\varphi$,
$q_{1,\varphi}$, and $q_{2,\varphi}$ are the semi-classical principal
symbols of $Q_\varphi$, $Q_{1,\varphi}$, and $Q_{2,\varphi}$, that is,  
$q_\varphi = q(z, \zeta + i \tau d \varphi(z))$ and $q_{k,\varphi} =
q_k(z, \zeta + i \tau d \varphi(z))$, $k=1,2$,   we can write
\begin{align*}
  \{ \ovl{q}_\varphi, q_\varphi\}  = 
  |q_{1,\varphi}|^2 \{\ovl{q}_{2,\varphi}, q_{2,\varphi}\} 
  + |q_{2,\varphi}|^2 \{\ovl{q}_{1,\varphi}, q_{1,\varphi}\} 
  + f |q_{1,\varphi}|  \, |q_{2,\varphi}|, 
\end{align*}
for some function $f$.
Thus $\{ \ovl{q}_\varphi, q_\varphi\}$ vanishes if $q_{1,\varphi} =
q_{2,\varphi}=0$. Then, the sub-ellipticity property of \eqref{3. carleman}
cannot hold for $Q$. 

Observe that in the above example we have
$d_{z,\zeta} q (z,\zeta + i \tau d \varphi(z)) =0$ if
$q_{2} (z,\zeta + i \tau d \varphi(z)) = q_{1} (z,\zeta + i \tau d
\varphi(z)) =0$.
The following proposition (that applies to operators that need not be
elliptic) shows that in such case of symbol ``flatness'', the Carleman
estimate we can derive for $Q$ exhibits at least a loss of one full
derivative.
%%%%%%%%%%%%%%%%%%%%%%%%
% proposition          %
%%%%%%%%%%%%%%%%%%%%%%%%
\begin{proposition}
  \label{prop: Carleman P2 - optim}
  Let $Q=Q(z,D_z)$ be a smooth operator of order $m\geq 1$ in $Z$, an open
  subset of $\R^N$. Assume further that 
 there exist a smooth weight function $\varphi(z)$, $C>0$, $\tau_1>0$, a
 multi-index $\mi$ with $0\leq |\mi| \leq m$, and $\delta \geq 0$ \st
 \begin{align}
   \label{eq: Carleman P2 - optim}
   \tau^{m -1-|\mi| + \delta}
   \Norm{e^{\tau \varphi} D_z^\mi u}{L^2}
   \leq C \Norm{e^{\tau \varphi}  Q u}{L^2},
 \end{align}
 for $\tau \geq \tau_1$ and 
 for $u \in \Cinf(\R^N)$ with $\supp(u) \subset Z$.  Let
 $q(z,\zeta)$ be the principal symbol of $Q$. If there exist $z_0 \in
 Z$, $\zeta_0 \in \R^N$ and $\tau_0>0$ \st $\theta_0^\mi
 \neq 0$, with $\theta_0= \zeta_0 + i \tau_0 d \varphi(z_0)$, and 
 \begin{align*}
   q(z_0,\theta_0 )= q_\varphi (z_0,\zeta_0, \tau_0)=0, 
   \quad  d_{z,\zeta} q (z_0,\theta_0 )=0, 
 \end{align*}
then $\delta =0$. 
\end{proposition}
In other words,  if there is a point  $(x_0,\xi_0,\tau_0)$  where the
symbol $q_\varphi$ vanishes at second
order, then if a Carleman estimate holds it exhibits at least the loss
of a full derivative.

We refer to Section~\ref{sec: prop: Carleman P2 - optim} for a proof.

\begin{remark}$\phantom{-}$
  %\begin{enumerate}
  %\item In Proposition~\ref{prop: Carleman P2 - optim} we restrict
  %  ourselves to estimates away from boundaries but the same
  %  would apply for estimates at a boundary.
  %\item 
    This loss of at least one full derivative shows that the
    analysis of \cite{lerner:88} cannot be applied here, as it
    concerns Carleman estimate with losses of less that one
    derivative. In particular, one can check that iterated Poisson
    brackets used in \cite{lerner:88} all vanish at points where 
    $q_\varphi$ vanishes at second order.
  %\end{enumerate}
\end{remark}

In dimension greater than $1$, this proposition applies to the
bi-Laplace operator  $B$ introduced above on the manifold $\Omega$. If
$a(x,\xi)$ is the principal symbol of the Laplace operator in a local
chart $V$, for all
$x_0 \in V$, there exists $\xi_0$ and $\tau_0>0$ such that
$a(x_0, \xi_0 + i \tau_0 d_x \varphi(x_0))=0$. Then, the symbol
$b= a^2$ vanishes at second order at
$(x_0, \xi_0 + i \tau_0 d_x \varphi(x_0)$. Hence, we cannot hope for a
Carleman estimate for $B$ with a loss of less than one full
derivative. In fact, such an estimate can be obtained by using twice in
cascade the Carleman estimate for the Laplace operator. This is
consistent, as the estimate for the Laplace operator exhibits a loss a
half derivative in dimension greater than $1$ (if $\varphi$ is chosen
\st sub-ellipticity
holds -- see \cite{LL:12}).

 In dimension one, however,
$B= D_x^4$ and the conjugated operator $(D_x + i \tau d \varphi(x))^4$ is elliptic (in the
sense of semi-classical operators) if $d \varphi (x) \neq 0$ in
$\Omega$. Then, the resulting Carleman estimate is characterized by no
derivative loss.

Concerning the operator $P = D_s^4 + B$ in $Z=(0,S_0)\times \Omega$,
that is central in the present article, we write $P = P_1P_2$ with
$P_k = (-1)^k i D_s^2 + A$. Setting $P_{k,\varphi}  =e^{\tau \varphi}
P_k e^{-\tau \varphi}$, with semi-classical principal symbols given by 
\begin{align*}
  p_{k,\varphi}(z,\zeta,\tau) =  (-1)^k i (\sigma + i \tau \d_s \varphi(z))^2
  + a(x,\xi+ i \tau d_x\varphi(z)),\quad k=1,2,
\end{align*}
where $z=(s,x) \in Z $ and $\zeta = (\sigma, \xi) \in \R^{1+d} =
\R^N$. 
Let $d\geq 2$.  If, for some $z_0 \in Z$, we have
$\d_s \varphi(z_0) =0$, if we choose $\xi_0\in \R^d$ and $\tau_0>0$ such that
$a(x_0,\xi_0+ i \tau_0 d_x\varphi(z_0))=0$, then for $\sigma_0=0$, we
have $\zeta_0 = (0, \xi_0) \in \R^N$ and $\theta_0 = (0,\xi_0) + i \tau_0 (0, d_x \varphi(z_0))$ and
$p_{k,\varphi}(z_0,\zeta_0,\tau_0) = p_k (z_0, \theta_0) =0$
and $d_{z,\zeta} p  (z_0, \theta_0) =0$, where $p$ and $p_k$ are the
principal symbols of $P$ and $P_k$, $k=1,2$.
Hence, in a \nhd of $z_0$, Proposition~\ref{prop: Carleman P2 - optim}
applies. 

This situation occurs in Cases (i) and (iii) described in
Section~\ref{sec: proof outline bi-laplace} and Figure~\ref{fig:
  Carleman method}. In the \nhds $V_1$ and $V_3$ introduced there, we
have points where $\d_s \varphi$ vanishes (as can observed by the
shapes of the
level sets of $\varphi$ in Figure~\ref{fig: Carleman method}). 
This explains why we can only obtain estimates with a
loss of one full derivative for those cases. 
In case (ii), however, this does not occur, and there we obtain an
estimation with only a loss of a half derivative.

\subsection{Some perspectives}

The present article deals with the natural ``clamped'' boundary
conditions, that is, homogeneous Dirichlet and Neumann conditions
simultaneously. In the light of the results obtained here and those
that can be obtained for very general boundary conditions of
Lopatinskii type in \cite{Tataru:96,BLR:13}, for instance for unique
continuation through the derivation of Carleman estimates at the
boundary for general elliptic operators with complex coefficient in
cases where the sub-ellipticity property hold, one is inclined to
attempt to prove estimates similar to those proven in the present
article, in the case of an operator, such as the operator
$P = D_s^4 + B$ studied here, for which the sub-ellipticity condition cannot
hold everywhere and for general boundary conditions of Lopatinskii type. 

Here, we considerer the bi-Laplace operator $B=\Delta^2$. It would be of
interest to consider more general polyharmonic operators such as
$\Delta^k$, $k \in \N$, on $\Omega$ along with natural boundary conditions, \eg, 
\begin{align*}
  u_{|\d \Omega} =0,   \dots , \d_\nu^{k-1} u_{|\d \Omega} =0,
\end{align*}
or more general Lopatinskii type conditions.

\subsection{Notation}
We shall use some spaces of smooth functions in the closed half space.
We set
\begin{equation*}
 \S(\Rpb) = \big\{ u_{|\Rpb}; \ u \in \S(\R^N)\big\}.
\end{equation*}

The reader needs to be warned that in some sections $z\in \R^N$ will
denote $(x,s)$, with $x \in \R^d = \R^{N-1}$ and $s\in \R$, and thus,
there, $z_N =s$. This is the case in Section~\ref{sec: Carleman
  boundary s=0}.  In other sections,
$z$ will denote $(s,x)$, and thus there $z_N = x_d$. This is the case
of Section~\ref{sec: Carleman boundary x} and Appendices~\ref{sec: proofs
  semi-classical calculus} and \ref{sec: elliptic
  sub-elliptic estimates boundary x}.

\medskip
Some specific notation for semi-classical tangential operators will be introduced
in Section~\ref{sec: tangential semi-classical
  calculus}, and they allow us to derive the
Carleman estimate for $D_s^4 + B$ at the
boundary $\{0\} \times \Omega$ (Cases (i) above). Semi-classical calculus is
characterized by the presence of a large parameter denoted by $\tau$
here, that is precisely the large parameter that appears in the
Carleman estimates (for readers familiar with semi-classical analysis
this is done by taking $\tau = 1/h$ where $h$ is the Planck constant.)

A special class of semi-classical calculus is also introduced in
Section~\ref{sec: pseudo -3p} and is characterized by three
parameters. This calculus is essential in the proof of the Carleman
estimate for $D_s^4 + B$ at the boundary $(0,S_0)\times \d\Omega$
(Case (iii) above).

\medskip
In this article, when the constant $C$ is used, it refers to a
constant that is independent of the semi-classical parameters, \eg $\tau$, $\csp$,
$\ctp$.  Its value may however change from one line to another. If we
want to keep track of the value of a constant we shall use another
letter.

For concision, we use the notation
$\lesssim$ for $\leq C$, with a constant $C>0$.
We also write $a \asymp b$ to denote $a \lesssim b \lesssim a$.
As done above, we shall
use the notation $\Norm{.}{}$ for functions in the interior of the
domain and $\norm{.}{}$ for functions on the boundaries.

\bigskip
We finish this introductory section by stating some basic properties
of the ``clamped'' bi-Laplace operator that will be used at places in
this article (some were implicitly used above). 

\subsection{Some basic properties  of the bi-Laplace operator}
\label{sec: properties bi-laplace operator}
We recall here some facts on the ``clamped'' bi-Laplace operator. 
We define the operator $B = \Delta^2$ on $L^2(\Omega)$ with domain
$D(B) = H^4(\Omega) \cap H^2_0(\Omega)$. 
%%%%%%%%%%%%%%%%%%%%%%%%
% proposition          %
%%%%%%%%%%%%%%%%%%%%%%%%
\begin{proposition}
  \label{prop: bi-laplace selfadjoint max monotone}
  The operator $(B, D(B))$ is selfadjoint on  $L^2(\Omega)$ and
  maximal monotone.
\end{proposition}
In particular, if $\mu \geq 0$, there exists $C>0$ \st, for any $f \in L^2(\Omega)$, there
exists a unique $u \in D(B)$ \st
\begin{align}
  \label{eq: bilaplace elliptic estimate}
  \Delta^2 u + \mu u = f, \ \ \text{and} \ \ \Norm{u}{H^4(\Omega)} \leq C \Norm{f}{L^2(\Omega)}.
\end{align}
This can be proven by first finding a unique solution in
$H^2_0(\Omega)$ with the Lax-Milgram theorem and then applying Theorem 20.1.2 in \cite[Section 20.1]{Hoermander:V3}.
Note in particular that $\Norm{\Delta^2 u}{L^2}$ is a equivalent
  norm on $H^4(\Omega) \cap H^2_0(\Omega)$ by \eqref{eq: bilaplace elliptic estimate}.

As a consequence of Proposition~\ref{prop: bi-laplace selfadjoint max
  monotone} we have the existence of a Hilbert basis for $L^2(\Omega)$ made of eigenfunctions. 
%%%%%%%%%%%%%%%%%%%%%%%%
% Corollary   %
%%%%%%%%%%%%%%%%%%%%%%%%
\begin{corollary}
  \label{cor: bi-laplace spectral family}
  There exist $(\mu_j)_{j \in \N} \subset \R$, and $(\varphi_j )_{j \in \N}
  \subset D(B)$ such that 
  \begin{align*}
    0 < \mu_1 \leq \mu_2\leq \cdots \leq \mu_j \leq \cdots, \qquad
    \lim_{j \to \infty} \mu_j = + \infty, 
    \qquad B \varphi_j = \mu_j    \varphi_j, 
  \end{align*}
  and the family $(\varphi_j )_j$ forms a Hilbert basis for $L^2(\Omega)$. 
\end{corollary}

%%%%%%%%%%%%%%%%%%%%%%%%
% Corollary   %
%%%%%%%%%%%%%%%%%%%%%%%%
\begin{corollary}
  \label{cor: semigroup}
  The operator $(B, D(B))$ generates an analytic $C_0$-semigroup $S(t)$ on $L^2(\Omega)$. 
  
  For $T>0$, $y_0 \in L^2(\Omega)$, and $f \in  L^2(0,T; H^{-2}(\Omega))$,  there exists a unique 
  \begin{align*}
    y \in L^2([0,T]; H^2_0(\Omega)) \cap \Con([0,T]; L^2(\Omega)) \cap
    H^1 (0,T; H^{-2}(\Omega)), 
  \end{align*}
 given by $y(t) = S(t) y_0 + \int_0^t S(t-s) f(s) ds$, such that 
\begin{align*}
    \d_t y + \Delta^2 y =f \ \ \text{for}\  t \in (0,T)\  \text{\ae}, \quad
  y_{|t=0} =y_0  
\end{align*}
\end{corollary}
For semigroup theory we refer the reader to \cite{Pazy:83}.

For the operator $\mathcal B$ defined in \eqref{eq: definition mathcal
  B} we have the following property.
\begin{proposition}
  The spectrum of $\mathcal B$ is contained in $\{ z \in \C; \Re (z)
  >0\}$. Moreover, for  $z \in \C$ such that $\Re z
  <0$, we have 
  \begin{equation*}
    \Norm{(z \id_\H - \mathcal B) U}{\H} \geq |\Re z| \, \Norm{U}{\H},
    \quad U \in D(\mathcal B), 
  \end{equation*}
  with $\H= H^2_0(\Omega) \times L^2(\Omega)$.
\end{proposition}
With the Hille-Yoshida theorem \cite[Theorem 3.1, Chapter
1]{Pazy:83} we then have the following results.
\begin{corollary}
  \label{theorem: existence semigroup damped plate}
  The unbounded operator $(\mathcal B, D(\mathcal B))$ generates a $C_0$-semigroup of
  contraction $\Sigma(t)$ on
  $\H$. 
\end{corollary}

\begin{corollary}
  \label{theorem: existence semigroup damped plate 2}
For $(y_0,y_1)  \in D(\mathcal B)$ there exists a unique 
  \begin{align*}
    y \in \Con^2([0,+\infty); L^2(\Omega)) \cap \Con^1([0,+\infty); H^{2}_0(\Omega)) \cap \Con^0([0,+\infty); D(B)),
  \end{align*}
  such that 
\begin{align*}
 \d_t^2 y + \Delta^2 y + \alpha \d_t y =0\  \text{in} \
  L^\infty((0,+\infty); L^2(\Omega)), \qquad y_{|t=0} =y_0, \  \d_t
  y_{|t=0} =y_1.
\end{align*}
The solution is given by the first component of
$\Sigma(t) Y_0$ with $Y_0 =(y_0, y_1)$. The energy $t \mapsto  E(y)(t)$ with 
\begin{align}
  \label{eq: energy plate equation}
  E(y)(t) = \frac12 \Norm{\d_t y(t)}{L^2(\Omega)}^2 + \frac12
  \Norm{\Delta y(t)}{L^2(\Omega)}^2,
\end{align}
is nonincreasing:
for $0\leq t_1 \leq t_2$ we have
  $E(y)(t_2) - E(y)(t_1) = - \int_{t_1}^{t_2} \Norm{\alpha^{1/2} \d_t y(t)}{L^2(\Omega)}^2 \, d t.$
\end{corollary}

%%%%%%%%%%%%%%%
%% Section             %
%%%%%%%%%%%%%%%
\section{Estimate away from boundaries}
\label{sec: estimate away from boundaries}

For operators exhibiting at most double (complex) roots, estimates can
be found in the proof of Theorem~28.1.8 in \cite{Hoermander:V4}. Here,
the structure of the operator $P$ is explicit which allows one to
expose the argumentation in a self contained yet short presentation.

\subsection{Simple-characteristic property of 
  second-order factors}
\label{sec: simple char Qk}
We consider the augmented operator $P = D_s^4 + B$ in
$Z = (0,S_0) \times \Omega$, remaining away from boundaries here.  We
write
\begin{align}
  \label{eq: operators Pk}
  P = P_1 P_2, \quad \text{with}\  P_k= (-1)^k i D_s^2  + A.
\end{align}
Here, we show that $P_1$ and $P_2$ both satisfy the so-called simple
characteristic property in the case of a weight function whose
differential does not vanish.

Let $\ell(z,\zeta)$, with $(z, \zeta) \in \R^N\times \R^N$, be polynomial
of degree $m$
in $\zeta$, with smooth coefficient in $z$. For $z \mapsto M(z) \in \R^N\setminus \{0\}$,
we introduce the map 
\begin{align}
  \label{eq: simp char map}
  \begin{array}{rl}
  \rho_{z,\zeta,M}: \R_+     &\to \C, \\
  \theta &\mapsto \ell(z,\zeta + i \theta M(z)).
  \end{array}
\end{align}
%%%%%%%%%%%%%%%%%%%%%%%%
% definition           %
%%%%%%%%%%%%%%%%%%%%%%%%
\begin{definition}
  \label{def: simple characteristics}
  Let $W$ be an open set of $\R^N$. 
We say that $\ell$ satisfies the  simple-characteristic property in 
direction $M$ in $\ovl{W}$ if,  for
  all $z \in \ovl{W}$, we have $\zeta = 0$ and $\theta=0$ when  the map
  $\rho_{z,\zeta,M}$ has a double root. 
\end{definition}
We can formulate this condition as follows
\begin{align}
    \label{eq: simple root property}
    \ell(z,\zeta + i \theta M(z)) = d_\zeta \ell (z,\zeta + i \theta M(z))(M(z)) =0 \quad \Rightarrow \quad \zeta= 0, \ \theta =0.
  \end{align}

\bigskip
%%%%%%%%%%%%%%%%%%%%%%%%
% lemma                %
%%%%%%%%%%%%%%%%%%%%%%%%
\begin{lemma}
  \label{lemma: dim geq 3}
  Let $W$ be an open set of $\R^N$. 
  If $N\geq 3$ and $\ell(z,\zeta)$ is of order two (with complex coefficients) and elliptic  for
  $z\in \ovl{W}$, then for any map $z \mapsto M(z) \in \R^N\setminus \{0\}$,  $\ell$
  satisfies the simple-characteristic property in direction $M$ in $\ovl{W}$.
\end{lemma}
\begin{proof}
  The proof can be adapted from classical ideas (see \cite[proof of
  Proposition 1.1, Chapter 2]{LM:68} or \cite{Hoermander:83b}).  We
  consider the polynomial $f_{z, \zeta, M} (t) = \ell(z,\zeta + t M(z))$
  where $t$ is a complex variable, for $z \in \ovl{W}$, $\zeta \in \R^N$.  

  If $\zeta$ is colinear to $M(z)$, \eg
  $\zeta = \alpha M(z)$ then
  $f_{z, \zeta, M}(t) = (\alpha + t)^2 \ell(z,M(z))$. Because of the
  ellipticity of $\ell$, $\ell(z,M(z)) \neq 0$, and we only  have $t =
  -\alpha$ as a double {\em real} root for $f$. 

  We set $ J = \R^{N} \setminus \Span (M(z))$. Note that $z$ is fixed
  here and $\Span (M(z))$ is a vector line. The set $J$ is connected as
  $N\geq 3$.  Let now $\zeta \in J$, that is, $\zeta$ is not colinear
  to $M(z)$.  As $\ell$ is elliptic, the roots of $f_{z, \zeta, M}$ cannot
  be real numbers. We denote by $m^+(\zeta)$ and $m^-(\zeta)$ the
  number of roots with positive and negative imaginary parts,
  respectively.  We have $2 = m^+(\zeta) + m^-(\zeta)$. Since roots
  are continuous \wrt $\zeta$ and cannot be real, they remain in the
  upper- or lower-half complex plane as $\zeta$ varies in $J$, as $J$
  is connected, meaning
  that $m^+$ and $m^-$ are then invariant.  In particular,
  $m^+(\zeta) = m^+(-\zeta)$ and $m^-(\zeta) = m^-(-\zeta)$.
  Observing however, that if $t_0$ is a root of
  $t \mapsto \ell(z,\zeta + t M(z))$ then $-t_0 $ is a root of
  $t \mapsto \ell(z,-\zeta + t M(z))$, we find that
  $m^+(\zeta) = m^-(-\zeta)$.  This gives
  $m^+(\zeta) = m^-(\zeta) =1$.  Hence, complex roots are simple.
  
  In any case, we see that if the map $\theta 
\mapsto \rho_{z,\zeta,M} = f_{z, \zeta, M}(i\theta)$ has a double real
  root  $\theta_0$  then $\theta_0=0$  and $\zeta=0$. The simple-characteristic
  property is thus fulfilled.
\end{proof}

If we consider a weight function $\psi = \psi(s,x)$, for the operators
$P_k$, $k=1,2$, introduced in \eqref{eq: operators Pk}, we have the
following proposition.
%%%%%%%%%%%%%%%%%%%%%%%%
% proposition          %
%%%%%%%%%%%%%%%%%%%%%%%%
\begin{proposition}
  \label{prop: simple-char}
  Let $k=1$ or $2$. 
  Assume that
  $d \psi \neq 0$ in $\ovl{(0,S_0)\times \Omega}$. Then, $P_k$
  satisfies the simple-characteristic property in direction $d
  \psi$ in $\ovl{(0,S_0)\times \Omega}$.  
\end{proposition}
%%%% proof of proposition
\begin{proof}
  Here, the dimension is $N=d+1$.  The case $d \geq 2$ is treated in
  Lemma~\ref{lemma: dim geq 3}.  It only remains to treat the case of
  dimension $d=1$. Then, the principal symbol of $A$ reads $a(x,\xi) =
  \alpha(x) \xi^2$, with $\alpha(x) \geq C>0$.   
  We set $M(z) = (M_\sigma(z), M_\xi(z))= d \psi(z) \in \R^N\setminus \{0\}$. We write $\rho$ in place of
  $\rho_{(z_0, \zeta_0,M)}$ for concision.

  With $\zeta = (\sigma, \xi)$, we have
  \begin{align*}
    \rho(\theta) &= p_k\big(z_0,
    \zeta + i \theta M\big)
    =  (-1)^k i\big( \sigma + i \theta M_\sigma\big)^2 
    +  \alpha(x_0) \big(\xi + i  \theta M_\xi \big)^2\\
    &= \alpha(x_0)\xi^2 - \alpha(x_0) (\theta  M_\xi)^2
      - 2 (-1)^k \theta\sigma M_\sigma 
      + i \big( 
      (-1)^k \sigma^2 -  (-1)^k (\theta M_\sigma)^2
      + 2 \theta \alpha(x_0)\xi M_\xi
      \big).
  \end{align*}
  We thus have 
    $\hf \d_\theta \rho (\theta) =
    - \alpha(x_0) \theta  M_\xi^2
   -  (-1)^k\sigma  M_\sigma 
   +
    i \big(\alpha(x_0)\xi M_\xi
    -  (-1)^k \theta M_\sigma^2    \big)
   $.
  Assuming that  $ M_\sigma  \neq 0$, if $\d_\theta \rho =0$ we find
  \begin{align*}
    \theta
    =  (-1)^k \frac{\alpha(x_0)\xi M_\xi}
    { M_\sigma^2 }, \quad \text{and}\ \ 
   \sigma
    = -  \alpha(x_0)^2 \xi \frac{ M_\xi^3}
    {M_\sigma^3 }.
  \end{align*}
  This yields
    $\rho = \alpha(x_0)  \xi^2 \Big(
    1 +(-1)^k i\alpha(x_0)M_\xi^2/ M_\sigma^2
   \Big)
    \Big(1 +
    \alpha(x_0)^2 M_\xi^4/ M_\sigma^4 
    \Big)$.
  In this case, we thus have $\rho = \d_\theta\rho =0$ if and only if
  $\theta =0$ and $\zeta =(\sigma,\xi)=(0,0)$.

  We assume now that $M_\sigma= 0$. Since $M
  \neq 0$, we find that $\d_\theta \rho=0$ implies $\theta=0$ and
  $\xi=0$. 
Then $\rho=0$ gives $\sigma=0$. Hence, in any case, the simple
characteristic property is  fulfilled. 
\end{proof}

\subsection{Local Carleman estimates away from boundaries}

Let $V$ be an open subset of $Z= (0,S_0) \times \Omega$.  We set
$z = (s,x)$. Let $L= L(z, D_z)$ be a differential operator of order $m$,
with smooth principal symbol, $\ell(z,\zeta)$.
%%%%%%%%%%%%%%%%%%%%%%%%
% definition           %
%%%%%%%%%%%%%%%%%%%%%%%%
\begin{definition}
  \label{def: sub-ellipticity}
  Let $\varphi(z)$ be defined and smooth in $\ovl{V}$ and \st
  $|d\varphi|\geq C>0$.
  We say that
  the couple $(L, \varphi)$ satisfies the sub-ellipticity
   condition in $\ovl{V}$ if we have
   \begin{multline*}
     \ell(z, \zeta + i \tau \varphi(z)) =0 \ \ \imp \\
     \frac{1}{2i}\{ \ovl{\ell(z, \zeta + i \tau d\varphi(z))}, \ell(z,
     \zeta + i \tau d\varphi(z))\}
     = \{ \Re \ell(z, \zeta + i \tau d\varphi(z)), \Im \ell(z, \zeta +
     i \tau d\varphi(z))\} >0, 
   \end{multline*}
   for all $z\in \ovl{V}$ and $\zeta\in \R^N$ and $\tau>0$. 
\end{definition}

Let $\psi(z)$ be smooth in $\ovl{V}$ and  \st $|d \psi| \geq C
   >0$ in $\ovl{V}$. We define $\varphi(z) = \exp(\csp
   \psi(z))$. Sub-ellipticity for the couple  $(P_k, \varphi)$ can
   be easily achieved by the following lemma.
%%%%%%%%%%%%%%%%%%%%%%%%
% lemma          %
%%%%%%%%%%%%%%%%%%%%%%%%
 \begin{lemma}
   \label{prop: pseudo-convexity}
   The couple $(P_k, \varphi)$ satisfies the sub-ellipticity
   condition in $\ovl{V}$ for $\csp>0$ chosen \suff large.
 \end{lemma}
 %%%% proof of proposition
 \begin{proof}
   By Proposition~\ref{prop: simple-char} we see that 
  $P_k$ satisfies the simple-characteristic property in direction $d
  \psi$ in $\ovl{V}$. This implies that $\psi$ is strongly
  pseudo-convex with respect to $P_k$ in the sense given in
  \cite[Section 28.3]{Hoermander:V4} at every point in $V$. We then
  obtain that  the  couple $(P_k, \varphi)$ satisfies the sub-ellipticity
   condition in $\ovl{V}$ for $\csp>0$ chosen \suff large by Proposition
   28.3.3 in \cite{Hoermander:V4}.
\end{proof}

A consequence of the sub-ellipticity property is  the following
Carleman estimate for $P_k$ in $V$, that is, away from boundaries. 
%%%%%%%%%%%%%%%%%%%%%%%%
% proposition          %
%%%%%%%%%%%%%%%%%%%%%%%%
\begin{proposition}
  \label{prop: Carleman Pk interior}
  Let $k=1$ or $2$.
  Let $\varphi = \exp(\csp \psi)$ with $|d \psi|\geq C >0$ in
  $\ovl{V}$. For $\csp>0$ chosen \suff large, there exist $C>0$ and
  $\tau_0$ \st 
  \begin{align*}
    \sum_{|\mi| \leq 2} \tau^{3/2-|\mi|} 
    \Norm{e^{\tau \varphi} D_z^{\mi} u}{L^2(Z)} 
    \leq C \Norm{e^{\tau \varphi} P_k u}{L^2(Z)},  
  \end{align*}
  for $\tau \geq \tau_0$ and $u \in \Cinfc(V)$. 
\end{proposition}
We refer to \cite[Theorem 28.2.3]{Hoermander:V4} for a proof.  In fact,
to incorporate the term associated with $|\mi|=2$ see
\cite{Hoermander:63}. This estimate is characterized by the loss of a
half derivative.

From this estimate for $P_k$, $k=1,2$, we deduce the following
estimate for the operator $P = P_1 P_2$.
%%%%%%%%%%%%%%%%%%%%%%%%
% proposition          %
%%%%%%%%%%%%%%%%%%%%%%%%
\begin{proposition}
  \label{prop: Carleman P interior}
  Let $\varphi = \exp(\csp \psi)$ with $|d \psi|\geq C >0$ in
  $\ovl{V}$. For $\csp>0$ chosen \suff large, there exist $C>0$ and
  $\tau_0$ \st 
  \begin{align*}
    \sum_{|\mi| \leq 4} \tau^{3-|\mi|} 
    \Norm{e^{\tau \varphi} D_z^{\mi} u}{L^2(Z)} 
    \leq C \Norm{e^{\tau \varphi} P u}{L^2(Z)},  
  \end{align*}
  for $\tau \geq \tau_0$ and $u \in \Cinfc(V)$. 
\end{proposition}
This estimate is characterized by the loss of a full derivative.

\begin{proof}
  With the estimate of Proposition~\ref{prop: Carleman Pk interior}
  for the operator $P_1$ applied to $P_2 u \in \Cinf(V)$ we have  
  \begin{align*}
    \sum_{|\mi| \leq 2} \tau^{3/2-|\mi|} 
    \Norm{e^{\tau \varphi} D_z^{\mi} P_2 u}{L^2(Z)} 
    \lesssim \Norm{e^{\tau \varphi} P u}{L^2(Z)}.  
  \end{align*}
  Observing that $[D_z^{\mi},  P_2]$ is a differential operator of
  order $1 + |\mi|$ we obtain
  \begin{align}
    \label{eq: Caleman P1 away from boundaries}
    \sum_{|\mi| \leq 2} \tau^{3/2-|\mi|} 
    \Norm{e^{\tau \varphi} P_2 D_z^{\mi}  u}{L^2(Z)} 
    \lesssim \Norm{e^{\tau \varphi} P u}{L^2(Z)}
    + \sum_{|\smi| \leq 3} \tau^{5/2-|\smi|} 
    \Norm{e^{\tau \varphi} D_z^{\smi}  u}{L^2(Z)}.
  \end{align}
  Applying now the estimate of
  Proposition~\ref{prop: Carleman Pk interior} for the operator $P_2$
  to $D_z^{\mi} u \in \Cinf(V)$ we obtain 
  \begin{align*}
    \sum_{|\tmi| \leq 2} \tau^{3/2-|\tmi|} 
    \Norm{e^{\tau \varphi} D_z^{\tmi + \mi} u}{L^2(Z)} 
    \lesssim \Norm{e^{\tau \varphi} P_2 D_z^{\mi}  u}{L^2(Z)}.  
  \end{align*}
With~\eqref{eq: Caleman P1 away from boundaries}
  we then obtain
  \begin{align*}
    \sum_{|\mi| \leq 4} 
  \tau^{3-|\mi|} \Norm{e^{\tau \varphi}
    D_z^{\mi}  u}{L^2(Z)} 
    &\asymp
    \sum_{|\tmi| \leq 2} 
    \sum_{|\mi|\leq 2} \tau^{3-|\tmi|-|\mi|} \Norm{e^{\tau \varphi}
    D_z^{\tmi+ \mi}  u}{L^2(Z)} \\
    &\lesssim \Norm{e^{\tau \varphi} P u}{L^2(Z)}
    + \sum_{|\smi| \leq 3} \tau^{5/2-|\smi|} 
    \Norm{e^{\tau \varphi} D_z^{\smi}  u}{L^2(Z)}.
  \end{align*}
  We then conclude by choosing $\tau>0$ \suff large.
\end{proof}

%%%%%%%%%%%%%%%
%% Section             %
%%%%%%%%%%%%%%%
\section{Estimate at the  boundary $\{s=0\}$}
\label{sec: Carleman boundary s=0}
\subsection{Tangential semi-classical calculus and associated
Sobolev norms}
\label{sec: tangential semi-classical calculus}

Considering   boundary problems,  
we shall locally
use coordinates so that the geometry is that of the  half space
\begin{equation*}
  \Rpb=\{ z\in\R^N,\ z_N>0\}, 
  \quad z = (z',z_N)\ \text{with} \ z' \in \R^{N-1}, \ z_N\in \R.
\end{equation*}

We shall use the notation $\y = (z,\zeta,\tau)$ and $\y'=(z,\zeta',\tau)$ in this section. (This notation is not to be confused
with that introduced and used in Section~\ref{sec: Carleman boundary
  x} and Appendix~\ref{sec: elliptic sub-elliptic estimates boundary x}.)

Let $a(\y') \in \Cinf(\Rpb\times\R^{N-1})$, with $\tau$ as a parameter
  in $ [1,+\infty)$ and $m \in \R$, be \st, for all multi-indices $\mi, \smi$, we have
  \begin{equation*}
    |\d_z^\mi \d_{\zeta'}^\smi a(\y') | \leq C_{\mi,\smi} \lsct ^{m-|\smi|}, 
    \quad z\in \Rpb,\ \zeta'\in\R^{N-1},\ \tau \in  [1,+\infty),
  \end{equation*}
  where $\lsct^2 =|\zeta'|^2 + \tau^2$. We write $a \in
  \Ssct^m$. 
  We also define $\Ssct^{-\infty} = \cap_{r \in \R} \Ssct^{r}$.
  For $a \in \Ssct^m$ we call principal symbol, $\sigma(a)$,   the equivalence class 
of $a$ in  $\Ssct^m/ \Ssct^{m-1}$. Note that we have $\lsct^m \in \Ssct^m$.

If $a(\y') \in \Ssct^m$, we set 
\begin{align*}
  \Opt(a) u(z) &:= (2 \pi)^{-(N-1)}  \int_{\R^{N-1}} e^{i\scp{z'}{\zeta'}} 
  a(\y')\ \hat{u}(\zeta',z_N)  \ d \zeta',
\end{align*}
for $u \in \S(\Rpb)$, where $\hat{u}$ is the partial Fourier transform
of $u$ with respect to the tangential variables $z'$. We denote by $\Psisct^m$ the set of these pseudo-differential operators. For $A \in \Psisct^m$,
  $\sigma(A) = \sigma (a)$  will be its principal
  symbol in $\Ssct^m /\Ssct^{m-1}$.
We also set $\Lsct^m = \Opt(\lsct^m)$, $m \in \R$.

Let $m\in \N$ and $m' \in \R$. 
If we
    consider $a$ of the form
    \begin{align*}
      a(\y) = \sum_{j=0}^m a_j (\y') \zeta_N^j, \qquad
      a_j \in \Ssct^{m+ m'-j},
    \end{align*}
 we define $\Op(a) := \sum_{j=0}^m \Opt(a_j) D_{z_N}^j$. We write $a
 \in \Ssc^{m,m'}$ and $\Op(a) \in \Psisc^{m,m'}$.

We define the following norm, for $m\in \N$ and $m' \in \R$, 
\begin{align*}
  &\Norm{u}{m,m',\tau} 
    \asymp \sum_{j=0}^m \Norm{\Lsct^{m+m'-j}D_{z_N}^j u}{+}
    \\
  &\Norm{u}{m,\tau}=\Norm{u}{m,0,\tau}
  \asymp \sum_{j=0}^m \Norm{\Lsct^{m-j }D_{z_N}^j u}{+},
  \quad u \in \S(\Rpb), 
\end{align*}
where $\Norm{.}{+} := \Norm{.}{L^2(\Rp)}$. 
We have 
\begin{align*}
  \Norm{u}{m,\tau} \asymp \sum_{{|\mi| \leq m}  \atop \mi \in \N^N}
  \tau^{m- |\mi|} \Norm{ D^\mi u}{+},
\end{align*}
and in the case $m'\in \N$ we have
\begin{align*}
  \Norm{u}{m,m',\tau} 
  \asymp \sum_{\mi_N \leq m}  
  \sum_{{|\mi|\leq m+m'}  \atop {\mi = (\mi',\mi_N) \in \N^N}}
    \tau^{m+m' -|\mi|} \Norm{ D^{\mi} u}{+}.
\end{align*}

If $m, m' \in \N$ and $m'', m''' \in \R$, and if 
 $a \in \Ssct^{m''}$, then we have 
\begin{align*}
  \Norm{\Opt(a) u }{m',m''',\tau} \leq C \Norm{u }{m',m''+ m''',\tau}, 
  \qquad u \in \S(\Rpb). 
\end{align*}
If $a \in \Ssct^{m,m''}$, then we have 
\begin{align*}
  \Norm{\Opt(a) u }{m',m''',\tau} \leq C \Norm{u }{m+m',m''+ m''',\tau}, 
  \qquad u \in \S(\Rpb). 
\end{align*}
The following argument will be used on many occasions in what follows, 
for $m \in \N$, $m ', \ell\in \R$, with  $\ell>0$, 
\begin{align}
  \label{eq: usual semi-classical argument 2}
  \Norm{w}{m,m', \tau} \ll
  \Norm{w}{m,m'+ \ell,\tau}.
\end{align}

At the boundary $\{z_N=0\}$ we define the following norms, for $m \in
\N$ and $m' \in \R$, 
\begin{equation*}
  \norm{\trace(u)}{m,m',\tau}^2
  =\sum_{j=0}^m \norm{ \Lsct^{m-j+m'} D_{z_N}^j u_{|z_N=0^+}}{L^2(\R^{N-1})}^2,
  \quad u \in \S(\Rpb).
\end{equation*}

\subsection{Statement of the Carleman estimate}
\label{sec: statement  Carleman boundary s=0}

In this section, we consider $z = (x, s)\in \R^N$ with $x\in \R^d$ and
$s\in \R$. We also set $Z = \Omega \times (0,S_0)$. 
We write $x=z'$ and $s = z_N$, in connexion with the notation
introduced for the tangential calculus in Section~\ref{sec: tangential semi-classical calculus}. 

Let $z_0=(x_0, 0)$ with $x_0 \in \Omega$.  We consider a function
$\psi \in \Cinf(\R^N)$ \st $\d_{s} \psi(z) \leq -C<0$ in a bounded
open \nhd $V$ of $z_0$ in $\R \times \Omega$. We then set $\varphi(z) = e^{\csp \psi(z)}$. 

Using the notation introduced in Section~\ref{sec: tangential
  semi-classical calculus} for semi-classical norms, we have the
following Carleman estimate at the boundary $ \Omega \times \{ 0\}$ for
functions defined in $\{ s \geq 0\} \cap V$. 
%%%%%%%%%%%%%%%%%%%%%%%%
% theorem              %
%%%%%%%%%%%%%%%%%%%%%%%%
\begin{theorem}
  \label{theorem: sharp Carleman boundary s=0}
  Let $P = D_s^4 + B =  D_s^4 + \Delta^2$ on $Z = \Omega \times (0,S_0)$. 
  Let $W$ be an open set of $\R^N$ \st $W \Subset V$.
  For $\csp>0$ chosen \suff large, there exist $\tau_0\geq 1$ and
  $C>0$ \st 
  \begin{align*}
    \sum_{|\mi| \leq 4} \tau^{7/2-|\mi|}\Norm{ e^{\tau \varphi} D_{s,x}^\mi u}{L^2(Z)}
    \leq C \Big( 
    \Norm{e^{\tau \varphi} P u}{L^2(Z)}
    + \sum_{j=0}^3 \norm{\trace(e^{\tau \varphi} D_{s}^j u)}{0,7/2-j,\tau}
    \Big),
  \end{align*}
  for $\tau \geq \tau_0$  and for $u = w_{|Z}$, with
  $ w \in\Cinfc(\R^d \times \R)$ and $\supp(w) \subset W$.
\end{theorem}
This Carleman estimate is characterized by the loss of a half
derivative. 
%%%%%%%%%%%%%%%%%%%%%%%%
% Corollary      %
%%%%%%%%%%%%%%%%%%%%%%%%
\begin{corollary}
  \label{cor: Carleman boundary s=0}
  Let $P = D_s^4 + B =  D_s^4 + \Delta^2$ on $Z = \Omega \times (0,S_0)$. 
  Let $W$ be an open set of $\R^N$ \st $W \Subset V$.
  For $\csp>0$ chosen \suff large, there exist $\tau_0\geq 1$ and
  $C>0$ \st 
  \begin{align*}
    \sum_{|\mi| \leq 3} \tau^{7/2-|\mi|}\Norm{ e^{\tau \varphi} D_{s,x}^\mi u}{L^2(Z)}
    \leq C \Big( 
    \Norm{e^{\tau \varphi} P u}{L^2(Z)}
    + \tau^{1/2} \sum_{j=0}^3 \norm{\trace(e^{\tau \varphi} D_{s}^j u)}{0,3-j,\tau}
    \Big),
  \end{align*}
  for $\tau \geq \tau_0$  and for $u = w_{|Z}$, with
  $ w \in\Cinfc(\R^d \times \R)$ and $\supp(w) \subset W$.
\end{corollary}
Proofs are given below.

\subsection{Sub-ellipticity property}
As in Section~\ref{sec: simple char Qk}, we write $P = P_1 P_2$ with
$P_k = (-1)^k i D_s^2 + A$, and
$\Pconj = e^{\tau \varphi} P e^{-\tau \varphi} = Q_1 Q_2$ with
$Q_k = e^{\tau \varphi} P_k e^{-\tau \varphi}$.
The principal symbol of $q_k$, in the sense of semi-classical operators,
is given by 
\begin{align*}
  q_k (z,\zeta, \tau) = (-1)^k i( \sigma + i \htau_\sigma)^2 + a(x,\xi
  + i \htau_\xi), \quad \htau(z,\tau)  = (\htau_\xi, \htau_\sigma) = \tau d \varphi \in \R^N,
\end{align*}
where $a(x,\xi)$ denotes the principal symbol of the Laplace operator
$A$.

Recalling the definition of the semi-classical characteristic set of a
(pseudo-)differential operator $A$, with principal symbol $a(\y)$,
$$\Char (A) = \{ \y= (z,\zeta, \tau) \in \ovl{V} \times \R^N \times
\R_+; \ (\zeta, \tau) \neq (0,0), \ \text{and} \
a(\y)=0\},$$ 
we have the
following results for the characteristic sets of $Q_k$, $k=1,2$.
%%%%%%%%%%%%%%%%%%%%%%%%
% lemma                %
%%%%%%%%%%%%%%%%%%%%%%%%
\begin{lemma}
  \label{lemma: disjoint characteristic sets}
  In $\ovl{V}$, we have $\Char(Q_1) \cap \Char(Q_2) = \emptyset$. 
\end{lemma}
%%%% proof of lemma
\begin{proof}
  Let $\y = (z,\zeta, \tau) \in \ovl{V} \times \R^N \times \R_+$, with
  $(\zeta, \tau)\neq (0,0)$, be \st $q_1 (\y) = q_2 (\y) =0$, which
  reads $(-1)^k i( \sigma + i \htau_\sigma)^2 
    + a(x,\xi+ i \htau_\xi) =0$, for both $k=1$ and $k=2$, 
  meaning that we have 
  \begin{align*}
    ( \sigma + i \htau_\sigma)^2 =0, \quad 
    a(x,\xi+ i \htau_\xi) =0.
  \end{align*}
  In particular this implies $\sigma= 0$ and $\htau_\sigma = \tau \d_s
  \varphi =0$. As here $\d_s
  \varphi \neq 0$ we thus have $\sigma=\tau=0$. With $\tau=0$,  we
  have $\htau_\xi=0$, and we thus obtain 
  $a(x,\xi) =0$, implying $\xi=0$ because of the ellipticity of
  $a(x,\xi)$. 
\end{proof}
%%%%%%%%%%%%%%%%%%%%%%%%
% lemma                %
%%%%%%%%%%%%%%%%%%%%%%%%
\begin{lemma}
  \label{lemma: product sub-ellipticity}
  Let $L_1$ and $L_2$ be differential operators in $V$.  Let
  $\varphi \in \Cinf(Z)$ and set
  $L_{k,\varphi} = e^{\tau \varphi} L_k e^{-\tau \varphi}$,
  $k=1,2$. Assume that
  $\Char(L_{1,\varphi}) \cap \Char(L_{2,\varphi}) = \emptyset$.  Then
  the couple $(L_1L_2, \varphi)$ satisfies the sub-ellipticity
  condition of Definition~\ref{def: sub-ellipticity} in $\ovl{V}$ {\em
    if and only if} both $(L_k, \varphi)$, $k=1,2$, satisfy this
  property.
\end{lemma}
%%%% proof of lemma
\begin{proof}
  We denote by $\ell_{k}$, the principal symbols of
  $L_{k,\varphi}$, $k=1,2$, and $\ell = \ell_1 \ell_2$ the principal
  symbol of  $e^{\tau \varphi} L_1 L_2 e^{-\tau \varphi}$. 
  We observe that 
  \begin{align*}
    \{ \ovl{\ell}, \ell\} = 
    |\ell_1|^2 \{ \ovl{\ell_2}, \ell_2\}  +  |\ell_2|^2\{ \ovl{\ell_1}, \ell_1\} 
    + f |\ell_1|\, |\ell_2|, 
  \end{align*}
  for some function $f$. 
  If $(\ell, \varphi)$ satisfies the sub-ellipticity condition and if
  $\ell_1(\y)=0$,  with $\y = (z,\zeta, \tau) \in \ovl{V}\times \R^N
  \times \R_+$, then $\ell_2(\y) \neq 0$ and
  $0< \{\ovl{\ell} , \ell\} (\y)/i = |\ell_2|^2 \{
  \ovl{\ell}_1,\ell_1\}/i$,
  thus yielding the sub-ellipticity condition at $\y$ for
  $\ell_1$. The same argument applies for $\ell_2$.
  
  Let us now assume that $\ell_1$ and $\ell_2$ both satisfy the
  sub-ellipticity condition. If $\ell(\y)=0$ then either  $\ell_1(\y)=0$ or
  $\ell_2(\y)=0$. Let us assume that
  $\ell_1(\y)=0$. Then $\ell_2(\y)\neq
  0$ and $\{ \ovl{\ell}_1,\ell_1\}(\y) /i>0$ . We then have 
  $\{\ovl{\ell} , \ell\}(\y)/i = |\ell_2(\y)|^2 \{ \ovl{\ell}_1,
  \ell_1\} (\y)/i >0$. 
\end{proof}
By Lemma~\ref{prop: pseudo-convexity}, the couples $(P_k, \varphi)$
satisfy the sub-ellipticity condition in $\ovl{V}$.  From
Lemmata~\ref{lemma: disjoint characteristic sets} and \ref{lemma:
  product sub-ellipticity} we deduce the following result.
%%%%%%%%%%%%%%%%%%%%%%%%
% corollary            %
%%%%%%%%%%%%%%%%%%%%%%%%
\begin{corollary}
  \label{cor: sub-ellipticity P}
  The couple $(P, \varphi)$ satisfies the sub-ellipticity
  condition of Definition~\ref{def: sub-ellipticity} in $\ovl{V}$.
\end{corollary}

\subsection{Proof of the estimate at $\{s=0\}$}
\label{sec: estimates  s=0}

The proof of Theorem~\ref{theorem: sharp Carleman boundary s=0} uses
Lemma~4.3 in \cite{BLR:13}.
\begin{proof}[\bfseries Proof of Theorem~\ref{theorem: sharp Carleman boundary s=0}]
  We denote by $a(\y)$ the principal symbol of $(\Pconj + \Pconj^*)/2$
  and by $b(\y)$ that  of $(\Pconj -
  \Pconj^*)/(2i)$. We have $a \in \Ssc^{4,0}$ and $b \in \Ssc^{3,1}$. We
  set $A = \Op(a)$ and $B= \Op(b)$
  and
  \begin{align*}
    Q_{a,b} (w) = 2 \Re \scp{A w}{B w}_+. 
  \end{align*}
  The sub-ellipticity of $(P,\varphi)$ given by Corollary~\ref{cor: sub-ellipticity P} reads
  \begin{align*}
    a(\y) = b(\y) =0 \ \ \imp \ \ \{ a, b\}>0, \qquad 
    \y \in \ovl{V}\times
    \R^N \times \R_+.
  \end{align*}
  With Lemma~4.3 in \cite{BLR:13}, we obtain, for some $C>0$ and
  $C'>0$, for $\tau\geq 1$ chosen \suff large, 
  \begin{equation*}
    C \Norm{v}{4,\tau}^2 \leq
    C' \big( \Norm{A v}{+}^2 + \Norm{B v}{+}^2 +
    \norm{\trace(v)}{3,1/2,\tau}^2 \big)
    + \tau \big(Q_{a,b}(v)-\Re \B_{a,b}(v)\big), 
  \end{equation*}
  where $|\B_{a,b}(v)| \lesssim \norm{\trace(v)}{3,1/2,\tau}^2$, 
  for $v = w_{|Z}$, with
  $ w \in\Cinfc(\R^d \times \R)$ and $\supp(w) \subset W$.
  We thus obtain 
  \begin{equation*}
    \tau^{-1} \Norm{v}{4,\tau}^2 \lesssim
    \Norm{(A + i B) v}{+}^2 
    +  \norm{\trace(v)}{3,1/2,\tau}^2.
  \end{equation*}
  As we have $\Pconj = A + i B \mod \Psisc^{3,0}$, by taking $\tau$ \suff
  large, with the  usual semi-classical argument \eqref{eq: usual
    semi-classical argument 2} we obtain
  \begin{equation}
    \label{eq: equiv form Carleman s=0}
    \tau^{-1/2} \Norm{v}{4,\tau} \lesssim
    \Norm{\Pconj v}{+}
    +  \norm{\trace(v)}{3,1/2,\tau}.
  \end{equation}
  The conclusion of the proof is then classical.
\end{proof}
\begin{proof}[\bfseries Proof of Corollary~\ref{cor: Carleman boundary s=0}]
  Let $W'$ be an open set of $\R^N$ \st $W \Subset W' \Subset V$ and
  let $\chi, \tchi \in \Cinfc(W')$ be \st $\chi\equiv 1$ in a \nhd of
  $W$ and $\tchi \equiv 1$ in a \nhd of $\supp(\chi)$. 

  We may apply
  estimate~\eqref{eq: equiv form Carleman s=0}, an equivalent form of the
  estimate of 
  Theorem~\ref{theorem: sharp Carleman boundary s=0},  to
  the function $\tau^{1/2} \chi(z) \Lsct^{-1/2} v$, for 
  $v = w_{|Z}$, with
  $ w \in\Cinfc(\R^d \times \R)$ and $\supp(w) \subset W$.
  Observe that we have
  \begin{align*}
   \chi(z) \Lsct^{-1/2} v
    = \Lsct^{-1/2} v + R_{0,-M} v,
  \qquad 
    P_\varphi \chi(z) \Lsct^{-1/2} v
    =  \tchi(z) P_\varphi \Lsct^{-1/2} v + R_{4,-M} v,
  \end{align*}
  because of the support of $v$, 
  with $R_{0,-M} \in \Psisc^{0, -M}$, and $R_{4,-M} \in \Psisc^{4, -M}$, for any $M\in \N$. 
  
  Setting $\tilde{v} = \tau^{1/2} \Lsct^{-1/2} v \in \S(\Rpb)$, we thus obtain,
  with \eqref{eq: equiv form Carleman s=0},
  \begin{equation}
    \label{eq: equiv form Carleman s=0 - 2}
    \tau^{-1/2} \Norm{\tilde{v}}{4,\tau} \lesssim
    \Norm{\tchi\Pconj \tilde{v}}{+}
    +  \norm{\trace(\tilde{v})}{3,1/2,\tau}
    + \Norm{v}{4,-M,\tau}.
  \end{equation}
  We then observe that we have
  \begin{align*}
    \tau^{-1/2}\Norm{\tilde{v}}{4,\tau}
    =\Norm{\Lsct^{-1/2} v}{4,\tau} 
    =\Norm{v}{4,-1/2,\tau}.
  \end{align*}
  We also have
  $\norm{\trace(\tilde{v})}{3,1/2,\tau} = \tau^{1/2}
  \norm{\trace(v)}{3,0,\tau}$,
  as $[D_s, \Lsct^r] =0$, $r \in \R$.  Next, as
  $[\tchi \Pconj, \Lsct^{-1/2}] \in \Psisc^{4,-3/2}$, we have
  \begin{align*}
    \Norm{\tchi \Pconj \tilde{v}}{+}
    \lesssim \tau^{1/2}\Norm{\Lsct^{-1/2} \tchi \Pconj v }{+}
    + \tau^{1/2} \Norm{v}{4,-3/2,\tau}
    \lesssim \Norm{\Pconj  v }{+}
    + \tau^{1/2} \Norm{v}{4,-3/2,\tau}.
  \end{align*}
  From~\eqref{eq: equiv form Carleman s=0 - 2}, we thus obtain 
  \begin{equation*}
    \Norm{v}{4,-1/2,\tau} 
    \lesssim
    \Norm{\Pconj v}{+}
    +  \tau^{1/2}\norm{\trace(\tilde{v})}{3,0,\tau}
    + \tau^{1/2} \Norm{v}{4,-3/2,\tau}.
  \end{equation*}
  With the usual
  semi-classical argument \eqref{eq: usual semi-classical argument 2}
  we conclude the proof, as $\Norm{v}{4,-1/2,\tau} \gtrsim
  \tau^{1/2}\Norm{v}{3,\tau}$. 
\end{proof}

%%%%%%%%%%%%%%%
%% Section             %
%%%%%%%%%%%%%%%
\section{Estimate at the
  boundary  $(0,S_0) \times \d\Omega$}
\label{sec: Carleman boundary x}

\subsection{A semi-classical calculus with three parameters}
\label{sec: pseudo -3p}
\setcounter{equation}{0}

We set $\W = \R^N \times \R^N$, $N=d+1$, often referred to as phase-space. A
typical element of $\W$ will be $X = (s,x,\sigma,\xi)$, with $s \in
\R$,  $x \in \R^d$, $\sigma \in \R$, and
$\xi \in \R^d$. We also write $x = (x', x_d)$, $x' \in \R^{d-1}$, $x_d
\in \R$,  and accordingly $\xi = (\xi', \xi_d)$. 

With $s$ and $x$ playing very similar r\^ole in the definition of the
calculus, we set $z=(s,x) \in \R^{N}$, $z'=(s,x') \in \R^{N-1}$, and
$z_N = x_d$. 
We also set $\zeta = (\sigma, \xi)\in \R^N$, $\zeta' = (\sigma, \xi')
\in \R^{N-1}$,  and $\zeta_N = \xi_d$.

In this section, we shall consider a weight
function of the form 
\begin{equation}
  \label{eq: weight function - au bord en x}
  \varphi_{\csp, \ctp} (z) 
  = e^{\csp \cpsi (z)},\qquad \cpsi (z) = \psi(\ctp z '\! ,z_N),
\end{equation}
with $\csp$ and $\ctp$ as parameters, satisfying $\csp\geq 1$, 
$\ctp\in [0,1]$, and $\psi\in \Cinf(\R^N)$.  To define a
proper pseudo-differential calculus, we assume the following properties
of $\psi$:
\begin{equation}
  \label{eq: cond psi}
    \psi \geq C>0, \quad  \Norm{\psi^{(k)}}{L^\infty} < \infty,  \ k
    \in \N.
\end{equation}
In particular, there exists $k>0$ \st 
\begin{equation}
  \label{eq: cond inf-sup}
  \sup_{\R^N} \psi \leq (k+1) \inf_{\R^N} \psi.
\end{equation}

\subsubsection{A class of semi-classical symbols}

We introduce the following class of tangential symbols depending on the variables 
$z\in \R^N$, $\zeta'\in \R^{N-1}$ and $\hatt \in \R^N$.
We set $\hlambdat^2 = |\zeta'|^2 + |\hatt|^2$. 
%%%%%%%%%%%%%%%%%%%%%%%%
% definition           %
%%%%%%%%%%%%%%%%%%%%%%%%
\begin{definition}
  \label{def: semi-classical symbol}
  Let $m \in \R$. We say that $a(z,\zeta', \hatt) \in \Cinf(\Rpb\times
  \R^{N-1}\times \R^N)$ belong to the class $S^m_{\T,\hatt}$ if, for all multi-indices $\mi\in \N^N$, $\smi\in \N^{N-1}$,
  $\tmi\in \N^N$, there exists $C_{\mi, \smi, \tmi}>0$ \st 
  \begin{align*}
    |\d_z^\mi \d_{\zeta'}^\smi \d_{\hatt}^\tmi a (z,\zeta', \hatt)|
    \leq C_{\mi, \smi, \tmi} \hlambdat^{m - |\smi| - |\tmi|}, \quad
    (z,\zeta', \hatt) \in \Rpb\times
  \R^{N-1}\times \R^N, \ |\hatt|\geq 1. 
  \end{align*}
  If $\Gamma$ is a conic  open set of $\Rpb\times
  \R^{N-1}\times \R^N$, we say that $a \in S^m_{\T,\hatt}$ in $\Gamma$
  if the above property holds for $(z,\zeta', \hatt) \in \Gamma$. 
\end{definition}
Note that, as opposed to usual semi-classical symbols, we ask for some
regularity with respect to the semi-classical parameter that is a
vector of $\R^N$ here.

This class of symbols will not be used as such to define a class of
pseudo-differential operators but rather to generate other classes of
symbols and associated operators in a more refined semi-classical
calculus that we present now.

\subsubsection{Metrics}
For $\tauast\geq 2$, 
we set 
\begin{align*}
  &\M = \R^N \times \R^{N}\times   [\tauast,+\infty)
      \times [1,+\infty)\times [0,1], \\
  &\Mt = \Rpb \times \R^{N-1}\times   [\tauast,+\infty)
      \times [1,+\infty)\times [0,1]. 
\end{align*}
We denote by $\y = (z,\zeta,\tau,\csp,\ctp)$ a point in $\M$ and by $\y' =
(z,\zeta',\tau,\csp,\ctp)$  a point in $\Mt$.

We set $\ttau = \tau \csp \varphi_{\csp, \ctp}(z) \in \R_+$. For simplicity, even though $\ttau$ is independent of
$\zeta'$, we shall write $\ttau = \ttau(\y')$, when we wish to keep in mind
that $\ttau$ is not a  simple parameter but rather a function.
As $\psi>0$, $\tau\geq \tauast$, and $\csp \geq 1$,we note that we have $\ttau\geq
\tauast$. 
We then set 
\begin{align*}
  \lambsc^2 =  \lambsc^2(\y)
  =|\zeta|^2 +\ttau(\y')^2, 
  \qquad 
  \lambsct^2 =  \lambsct^2(\y')
  =|\zeta'|^2 +\ttau(\y')^2.
\end{align*}
The explicit dependences of $\lambsc$ and $\lambsct$ upon $\y$ and
$\y'$  are now dropped to ease notation in this section.
Similarly,  we shall write $\varphi(z)$, or simply $\varphi$, 
in place of $\varphi_{\csp, \ctp}(z)$.

We consider the following metric on phase-space $\W = \R^N \times \R^N$
\begin{align}
  \label{eq: def g}
  g = (1+\csp \ctp)^2 |d z'|^2
  + \csp^2 |d z_N|^2
  + \lambsc^{-2}|d \zeta|^2,
 \end{align}
for $\tau \geq \tauast$, $\csp \geq 1$, and $\ctp \in [0,1]$. 
(Note that this metric is not to be confused with the Riemannian
metric $\mathfrak{g}$
on $\Omega$.)

On the phase-space $\W' = \R^N \times \R^{N-1}$ adapted to a
tangential calculus, we consider
the following metric:
\begin{align*}
  \gt = (1+ \csp \ctp)^2 |d z'|^2
  + \csp^2 |d z_N|^2
  + \lambsct^{-2} |d \zeta'|^2,
\end{align*}
for $\tau \geq \tauast$, $\csp \geq 1$, and $\ctp \in [0,1]$.

\medskip
The first result of this section shows that the metric $g$ on $\W$ defines a
Weyl-H\"ormander pseudo-differential calculus, and that both $\varphi$
and $\lambsc$ have the
properties to be used as proper order functions. For a presentation of the 
Weyl-H\"ormander calculus we refer to
\cite{Lerner:10}, \cite[Sections 18.4--6]{Hoermander:V3} and
\cite{Hoermander:79}.

%%%%%%%%%%%%%%%%%%%%%%%%
% proposition          %
%%%%%%%%%%%%%%%%%%%%%%%%
\begin{proposition}
  \label{prop: admissible metric and order function}
  The metric $g$ and the order functions $\varphi_{\csp, \ctp}$, $\lambsc$ are admissible, 
  in the sense that, the following properties hold (uniformly with respect to
  the parameters $\tau$, $\csp$, and $\ctp$):
  \begin{enumerate}
  \item $g$ satisfies the uncertainty principle, that is $ h_g^{-1} = 
    \csp^{-1} \lambsc \geq 1$.
    \item $\varphi_{\csp, \ctp}$, $\lambsc$ and $g$ are slowly varying;
    \item $\varphi_{\csp, \ctp}$, $\lambsc$ and $g$ are temperate.
  \end{enumerate}
\end{proposition}
We refer to Appendix~\ref{proof prop: admissible metric and order function} for a proof.
Similarly, we have the following proposition.
%%%%%%%%%%%%%%%%%%%%%%%%
% proposition          %
%%%%%%%%%%%%%%%%%%%%%%%%
\begin{proposition}
  \label{prop: admissible metric and order function tangential}
  The metric $\gt$ and the order functions $\varphi_{\csp, \ctp}$,
  $\lambsct$ are admissible.
  For the tangential calculus we have $h^{-1}_{\gt} =
(1+ \ctp \csp)^{-1} \lambsct \geq 1$.
\end{proposition}
Note that the proof of the uncertainty principle uses that $\tauast
\geq 2$. The condition $\tauast\geq 1$ would suffice if we chose $\psi
\geq \ln (2)$. We preferred not to add this technical condition on the
weight function $\psi$. 

Consequently, $\ttau(\y')$ is also an admissible order function for
both calculi.
 
\subsubsection{Symbols} 
Let $a(\y) \in \Cinf(\R^N\times\R^N)$, with $\tau$, $\csp$, and $\ctp$
acting as parameters, and $m, r\in \R$, be \st for all
multi-indices $\mi, \smi \in \N^N$, with  $\mi = (\mi',\mi_N)$, we have
  \begin{equation}
      \label{eq: semi-classical symbols new}
    |\d_{z}^{\mi}\d_\zeta^\smi a(\y)|
    \leq C_{\mi,\smi}\  \csp^{|\mi_N|}(1+ \ctp \csp)^{|\mi'|} \ttau^r \lambsc^{m-|\smi|},
      \quad \  \y \in \M.
  \end{equation}
With the notation of \cite[Sections 18.4-18.6]{Hoermander:V3} we then
have $a (\y) \in S(\ttau^r \lambsc^m, g)$.

\medskip
Similarly, let $a(\y') \in \Cinf(\Rpb\times\R^{N-1})$, with $\tau$, $\csp$, and $\ctp$
acting as parameters, and $m \in
\R$. If  for all multi-indices $\mi=(\mi',\mi_N) \in \N^N , \smi' \in
\N^{N-1}$,  we have
  \begin{equation}
      \label{eq: semi-classical symbols tangential new}
      |\d_z^\mi \d_{\zeta'}^{\smi'} a(\y')|
    \leq C_{\mi,\smi'}\  \csp^{|\mi_N|}(1+ \ctp \csp)^{|\mi'|} \ttau^r \lambsct^{m-|\smi'|},
      \quad \y' \in \Mt, 
  \end{equation}
we then
write  $a (\y') \in S(\ttau^r \lambsct^m, \gt)$.
Observe that $S(\ttau^r \lambsct^m, \gt) \subset S(\lambsct^{r+m},
\gt)$.

The principal symbol associated with $a(\y') \in S(\ttau^r \lambsct^m, \gt)$
is given by its equivalence class in $S(\ttau^r \lambsct^m, \gt) / S((
1 + \ctp\csp) \ttau^r
\lambsct^{m-1}, \gt)$. We denote this principal part by $\sigma(a)$. 
Often, an homogeneous representative can be selected and the
principal part is then identified with this particular representative
of the equivalence class. (Conic sets and homogeneous symbols are
precisely defined in Section~\ref{sec: conic homog} below.)

\medskip
We define the following class of symbols, that are polynomial
with respect to $\xi_N$, 
\begin{align*}
  \Symbsc^{m,m'}= \sum_{j=0}^m S(\lambsct^{m+m'-j}, \gt) \zeta_N^j.
\end{align*}
For $a (\y) \in \Symbsc^{m,m'}$, with $a (\y) = \sum_{j=0}^m a_j (\y')
\zeta_N^j$, with $a_j(\y') \in S(\lambsct^{m+m'-j}, \gt)$, we denote
its principal part by $\sigma(a) (\y) = \sum_{j=0}^m \sigma(a_j) (\y')
\zeta_N^j$.

\medskip
For this calculus with parameters to make sense, it is important to
check that $\lambsc \in S(\lambsc, g)$ and
$\lambsct \in S(\lambsct, \gt)$ and $\ttau \in S(\ttau, g) \cap
S(\ttau, \gt)$. In fact,  the latter property implies the first two.
%%%%%%%%%%%%%%%%%%%%%%%%
% lemma                %
%%%%%%%%%%%%%%%%%%%%%%%%
\begin{lemma}
  \label{lemma: ttau in the calculus}
  We have $\ttau = \tau \csp \varphi_{\csp, \ctp} \in S(\ttau, g) \cap
S(\ttau, \gt)$.
\end{lemma}
We refer to  Section~\ref{proof: lemma: ttau in the calculus} for a proof.

\subsubsection{A semi-classical cotangent vector}
\label{sec: semiclassical cotangent vector htau}
We set $\htau = \tau d_z \varphi_{\csp, \ctp}(z) = \tau \csp
\varphi_{\csp, \ctp}(z) d_z \cpsi(z) = \ttau(\y') d_z \cpsi(z)
\in \R^N$. As for $\ttau$, we shall write $\htau = \htau(\y')$, when we wish to keep in mind
that $\htau$ is not a constant cotangent vector. 
Note that  $\htau = (\htau', \htau_N)$
with 
\begin{align*}
   \htau'(\y') = \ttau(y') d_{z'} \cpsi (z)
  = \ctp \ttau(y')  d_{z'}\psi(\ctp z '\!
  ,z_N), 
  \quad 
  \htau_N(\y') = \ttau(y')  \d_{z_N}\psi(\ctp z '\! ,z_N).
\end{align*}

As $d_{z'} \cpsi \in S(\ctp, \gt)$ and $\d_{z_N} \cpsi \in S(1,
\gt)$, we have the following result.
%%%%%%%%%%%%%%%%%%%%%%%%
% lemma                %
%%%%%%%%%%%%%%%%%%%%%%%%
\begin{lemma}
  \label{lemma: htau in the calculus}
  We have $\htau' \in S(\ctp \ttau, g)^{N-1} \cap S(\ctp \ttau, \gt)^{N-1}$
  and $\htau_N \in S(\ttau, g) \cap S(\ttau, \gt)$.
\end{lemma}

For later use, we also introduce the following notation:
\begin{align}
  \label{eq: decomposition ttau}
   &\htau_{\sigma} = \htau_{\sigma}(\y') =\tau \d_s \varphi_{\csp,
     \ctp}(z) \in \R, 
  \quad   \htau_\xi =\htau_\xi (\y') = \tau d_x \varphi_{\csp,
     \ctp}(z)\in\R^{N-1}= \R^d, \\
  &\htau_{\xi_d} =\htau_{\xi_d} (\y')= \tau \d_{x_d}
    \varphi_{\csp, \ctp}(z) \in \R, 
    \quad \htau_{\xi'} = \htau_{\xi'} (\y')
    = \tau d_{x'} \varphi_{\csp,\ctp}(z) \in\R^{N-2}= \R^{d-1}.
    \notag
\end{align}
We then have 
\begin{align}
  \label{eq: decomposition ttau2}
  \htau = (\htau_{\sigma}, \htau_\xi) = (\htau_{\sigma}, \htau_{\xi'},
  \htau_{\xi_d}),
  \qquad 
  \htau' = (\htau_{\sigma}, \htau_{\xi'}), \quad 
  \htau_N = \htau_{\xi_d}.
\end{align}

Even thought the following lemma is very elementary, we state it for
futur reference.
%%%%%%%%%%%%%%%%%%%%%%%%
% lemma                %
%%%%%%%%%%%%%%%%%%%%%%%%
\begin{lemma}
  \label{lemma: equiv ttau |htau|} 
  Let $V$ be an open set of $\R^N$ \st $\d_{x_d} \psi(z)\geq
  C>0$ for $z \in V$. 
  Then, we have  \begin{align}
  |\htau| \asymp \htau_{\xi_d} \asymp \ttau,
  \qquad z \in V.
\end{align}
\end{lemma}
\begin{proof}
  As $\Norm{\psi'}\infty\leq C$,  if 
$\d_{x_d} \psi\geq C>0$ for $z \in V \subset \R^N$, then we have
$|\htau| \lesssim \ttau \lesssim
\htau_{\xi_d}$ and thus the result.
\end{proof}

\subsubsection{Conic sets and homogeneity}
\label{sec: conic homog}
We recall that a set $\Gamma\subset \Rpb \times \R^{N-1}\times \R^N$
is said to be conic if $(z,\zeta',\hatt) \in \Gamma$ implies that $(z,
\nu \zeta',\nu \hatt) \in \Gamma$ for all $\nu>0$.

We introduce the map 
\begin{align*}
\kappa: \Mt &\to \Rpb \times \R^{N-1} \times \R^N, \\
  \y'= (z,\zeta', \tau, \csp, \ctp) &\mapsto  (z,\zeta', \htau(\y')).
\end{align*} 
Throughout Section~\ref{sec: Carleman boundary x} and Appendix~\ref{sec: elliptic sub-elliptic estimates boundary x}, we shall use the following terminology.
%%%%%%%%%%%%%%%%%%%%%%%%
% definition           %
%%%%%%%%%%%%%%%%%%%%%%%%
\begin{definition}
  \label{def: conic set.}
   An open subset $\U$ of $\Mt$ is said to be conic if $\Gamma = \kappa(\U)$ is conic in $\Rpb\times \R^{N-1} \times \R^N$.
   
   A function $f: \U \to E$, $E$ a vector space, is said to be
   homogeneous of degree $m$ if $f$ takes the form $f= g \circ \kappa$
   with $g: \Rpb\times \R^{N-1} \times \R^N \to E$ \st 
   $g (z,\nu\zeta,\nu\hatt) = \nu^m g (z,\zeta,\hatt)$, for $\nu>0$.
\end{definition}
In other words,  conic sets and homogeneity are to be
understood with respect to the variables $(z,\zeta, \htau)$ instead of
the variables $(z,\zeta, \tau, \csp, \ctp)$, where, as above, $\htau = \tau d_z \varphi_{\csp,
  \ctp}(z)= \tau  \csp \varphi_{\csp,\ctp}(z) d_z \cpsi(z)$.

\medskip
If $\U$ is a conic  open subset of $\Mt$ we shall say
that $a \in S(\ttau^r \lambsct^m, \gt)$ in $\U$ if property \eqref{eq:
  semi-classical symbols tangential new}
holds in $\U$, with a similar terminology for symbols that satisfy the
defining property of $\Symbsc^{m,m'}$ in
$\U$.

\medskip In what follows, the following lemma will be used for
instance, to generate cutoff functions. It will also be used to obtain
symbols with the adapted homogeneity with respect to $\zeta'$ and
$\htau$. We refer to Section~\ref{proof lemma: from on calculus to the other} for a proof.
%%%%%%%%%%%%%%%%%%%%%%%%
% lemma                %
%%%%%%%%%%%%%%%%%%%%%%%%
\begin{lemma}
  \label{lemma: from on calculus to the other}
  Let $\U$ be a conic  open subset of $\Mt$ and set $\Gamma =
  \kappa(\U)$. Assume also that $|\htau| \asymp \ttau$ in $\U$. 
  Let $m\in \R$ and  $\hat{a} (z,\zeta', \hatt)\in S^m_{\T,\hatt}$ in $\Gamma$  (as
  given in Definition~\ref{def: semi-classical symbol}).  We then have  
  $a (\y')  = \hat{a} \circ \kappa(\y') \in  S(\lambsct^m, \gt)$ in
  $\U$. 
  In fact, if $\hat{a}$ is polynomial in $(\zeta', \hatt)$ the
  assumption $|\htau| \asymp \ttau$ in $\U$ is not needed.
\end{lemma}

The following lemma is elementary.
%%%%%%%%%%%%%%%%%%%%%%%%
% lemma                %
%%%%%%%%%%%%%%%%%%%%%%%%
\begin{lemma}
  \label{eq: local symbol global symbol}
  Let $\U$ be a conic open subset of $\Mt$ and let
  $a \in S(\ttau^r \lambsct^m, \gt)$ in $\U$. Let $\chi \in S(1, \gt)$ in
  $\Mt$, with $\supp(\chi) \subset \U$.  Then,  $\chi a \in S(\ttau^r \lambsct^m, \gt)$ in $\Mt$. 
\end{lemma}

\subsubsection{Operators and Sobolev bounds}
\label{sec: operator sobolev}
For $a \in S(\ttau^r \lambsc^m, g)$  we define the following pseudo-differential
operator in $\R^N$:
\begin{align}
  \label{eq: pseudo}
  \displaystyle
  \Op(a) u(z)
  &= (2 \pi)^{-N} \int_{\R^N} e^{i z\cdot\zeta} 
    a(z, \zeta, \tau,\csp, \ctp)
  \hat{u}(\zeta) \ d \zeta, \qquad u \in \S(\R^N),
\end{align}
where $\hat{u}$ is the Fourier transform of $u$.
In the sense of oscillatory integrals, we have
\begin{align*}
 \Op(a) u(z) &= (2 \pi)^{-N} \iint_{\R^{2N}} e^{i (z-y)\cdot \zeta} a(z, \zeta, \tau,\csp, \ctp)
  u(y) \ d \zeta \ d y.
\end{align*}
\medskip
The associated class of pseudo-differential operators is denoted by
$\Psi(\ttau^r\lambsc^m, g)$. 
If $a$ is polynomial in the variables $\zeta$ and  $\htau(\y') = \ttau d_z
\cpsi(z)$, we then 
write $\Op(a) \in \D(\ttau^r\lambsc^m, g)$.

\medskip
Tangential operators are defined similarly. For $a \in
S(\ttau^r\lambsct^m, \gt)$ we set
\begin{align}
  \label{eq: tangential pseudo}
 \Opt(a) u(z) &= (2 \pi)^{-(N-1)} \iint_{\R^{2N-2}}
 e^{i (z'-y')\cdot \zeta'}  a(z,\zeta', \tau, \csp, \ctp)
  \ u(y',z_N) \ d \zeta' \ d y', 
\end{align}
for $u \in \S(\Rpb)$, where $z \in \Rpb$.
We write $A=\Opt(a) \in \Psi(\ttau^r\lambsct^m,
\gt)$.
We set $\Lambsct^m = \Opt(\lambsct^m)$.

\bigskip
We also introduce the following class of operators that act as
differential operators in the $z_N$ variable and as tangential
pseudo-differential operators in the $z'$ variables:
\begin{align}
  \label{eq: differential-psido}
  \PsiOpsc^{m,r} = \sum_{j=0}^m \Psi(\lambsct^{m+r-j}, \gt) D_{z_N}^j,
  \qquad m \in \N, \ r \in \R, 
\end{align}
that is, $\Op(a) \in \PsiOpsc^{m,r}$ if $a \in \Symbsc^{m,r}$. 
Operators of this class can be applied to functions that are only defined on the
half-space $\{ z_N \geq 0\}$.

At places, it will be handy to use the Weyl quantization for
tangential operators, namely with
$a \in S(\ttau^r \lambsct^m, \gt)$  we define
\begin{align}
  \label{eq: tangential pseudo Weyl}
  \Opt^w(a) u(z) &= (2 \pi)^{-(N-1)} \iint_{\R^{2N-2}}
 e^{i (z'-y')\cdot \zeta'}  a\big((z'+y')/2,z_N,\zeta', \tau, \csp, \ctp)
  u(y',z_N) \ d \zeta' \ d y'.
\end{align}
This quantification is often advantageous as $\Opt^w(a)^\ast =
\Opt^w(\ovl{a})$, and thus, for the symbol $a$ real, the
operator $\Opt^w(a)$ is 
(formally) selfadjoint. Note that $\Opt(a) - \Opt^w(a) \in (1 + \ctp\csp)
\Psi(\ttau^r \lambsct^{m-1}, \gt)$.

\bigskip
We now present some Sobolev-bound type result that we shall use in
what follows.
We use the following notation
$$\Norm{.}{+} =  \Norm{.}{L^2(\Rpb)}, 
\qquad \scp{.}{.}_{+} = \scp{.}{.}_{L^2(\Rpb)},
$$
for the $L^2$-norm on the half space $\Rpb$ and the associated scalar
product. 

\bigskip We have the following lemma whose proof is similar to that of
Lemma~2.7 in \cite{LeRousseau:12}.
%%%%%%%%%%%%%%
% lemma             %
%%%%%%%%%%%%%%
\begin{lemma}
  \label{lem: regularity 2p}
  Let $r, m \in \R$ and $a \in   S(\ttau^r \lambsct^m, \gt)$. There exists
  $C>0$ \st, for $\tau$ \suff large,
  \begin{align*}
    |\scp{\Opt(a) u}{v}_+| \leq C
    \Norm{\Opt(\ttau^{r'}\lambsct^{m'}) u}{+}
    \Norm{\Opt(\ttau^{r''}\lambsct^{m''}) v}{+},
    \qquad u, v\in \S(\Rpb).
  \end{align*}
  for $r=r'+r''$, $m=m'+m''$, with $r', r'' \in \R$, $m', m'' \in \R$. 
\end{lemma}
This contains the estimate
\begin{align}
  \label{eq: boundedness operators}
    \Norm{\Opt(\ttau^{r'}\lambsct^{m'}) \Opt(a) u}{+}
    \leq C
    \Norm{\Opt(\ttau^{r+r'}\lambsct^{m+m'}) u}{+}, 
    \qquad u\in \S(\Rpb), 
  \end{align}
for $r,m' \in \R$. The proof of Lemma~\ref{lem: regularity 2p} relies
in the fact that, for $r , m \in \R$, 
$$
\Opt(\ttau^{r} \lambsct^m) \Opt(\ttau^{-r} \lambsct^{-m}) = \id + R_1, 
$$ 
with $R_1 \in (1 + \ctp \csp) \Psi(\lambsct^{-1}, \gt)$ and $\Norm{R_1}{L^2 \to
  L^2} \ll 1$ for $\tau$ large.

Note also that we have the following result (see Section~\ref{proof lemma: equivalence norms} for a proof).
%%%%%%%%%%%%%%%%%%%%%%%%
% lemma                %
%%%%%%%%%%%%%%%%%%%%%%%%
\begin{lemma}
  \label{lemma: equivalence norms}
  We have 
  \begin{align}
    \label{eq: equivalence norms}
    \Norm{\Opt(\ttau^{r}\lambsct^{m}) u}{+} \asymp
    \Norm{\Opt(\lambsct^{m}) \ttau^{r} u}{+} ,  \qquad u \in \S(\Rpb),
\end{align}
and 
\begin{align}
    \label{eq: equivalence norms 2}
    \norm{\Opt(\ttau^{r}\lambsct^{m}) u\brz}{L^2(\R^{N-1})} \asymp
    \norm{\Opt(\lambsct^{m}) \ttau^{r} u\brz}{L^2(\R^{N-1})} ,  
  \qquad u \in \S(\R^{N-1}),
\end{align}
for $\tau$  chosen \suff large.
\end{lemma}

\bigskip
We define the following semi-classical Sobolev norms
\begin{equation*}
  \norm{u}{m,\ttau} = \norm{\Lambsct^{m} u\brz}{L^2(\R^{N-1})},
  \quad m\in \R, \ \ u \in \S(\R^{N-1}), 
\end{equation*}
\begin{equation*}
  \Norm{u}{m,\ttau}
  \asymp \sum_{j=0}^m \Norm{\Lambsct^{m-j} D_{z_N}^j u}{+},
  \quad m \in \N, \ \  u \in \S(\Rpb).
\end{equation*}
We also set, for $m \in \N$
 and  $m'\in \R$, 
\begin{equation*}
  \Norm{u}{m,m',\ttau} 
  \asymp \sum_{j=0}^m\Norm{\Lambsct^{m-j+m'} D_{z_N}^j u}{+},
  \quad u \in \S(\Rpb).
\end{equation*}

At the boundary $\{z_N=0\}$ we define the following norms, for $m \in
\N$ and $m' \in \R$, 
\begin{equation*}
  \norm{\trace(u)}{m,m',\ttau}
  \asymp \sum_{j=0}^m \norm{ \Lambsct^{m-j+m'} D_{z_N}^j u\brz)}{L^2(\R^{N-1})},
  \quad u \in \S(\Rpb).
\end{equation*}

The following argument will be used on many occasions in what follows, 
for $r, r', m \in \R$, and $\ell>0$, 
\begin{align}
  \label{eq: usual semi-classical argument}
  \csp^r \Norm{\ttau^{r'} w}{m,\ttau} \ll 
  \Norm{\ttau^{r'+\ell} w}{m,\ttau} \lesssim
  \Norm{\ttau^{r'} w}{m+\ell,\ttau}, 
\end{align}
for $\tau$ chosen \suff large, as $\csp^r \lesssim
\varphi_{\csp,\ctp} = \exp(\csp \cpsi)$ since
$\cpsi \geq C>0$. We have similar such inequalities for the other norms
introduced above.

With the above results we deduce the following two propositions.
%%%%%%%%%%%%%%
% proposition      %
%%%%%%%%%%%%%%
\begin{proposition}
  \label{prop: Sobolev regularity pseudo tangential}
  Let $r,m \in \R$, and $a \in 
  S(\ttau^r\lambsct^m, \gt)$. Then, for $r', m'\in \R$, 
  there exists $C>0$ \st 
  \begin{align*}
    \norm{\ttau^{r'} \Opt(a) u\brz}{m',\ttau} 
    \leq C \norm{\ttau^{r+r'} u\brz}{m+m',\ttau}, \quad u
    \in \S(\Rpb),
  \end{align*}
  for $\tau$ \suff large. 
\end{proposition}

\begin{proposition}
  \label{prop: Sobolev regularity pseudo}
  Let $r,m' \in \R$, $m \in \N$, and $a \in \ttau^r
  \Symbsc^{m,m'}$. Then, for $r',m''' \in \R$ and $m'' \in \N$, 
  there exists $C>0$ \st 
  \begin{align*}
    \Norm{\ttau^{r'} \Op(a) u}{m'',m''',\ttau} 
    \leq C \Norm{\ttau^{r+r'} u}{m+m'',m'+ m''',\ttau}, \quad u
    \in \S(\Rpb),
  \end{align*}
   for $\tau$ \suff large. 
\end{proposition}

Similarly to Lemma~\ref{lemma: equivalence norms}, we have the
following equivalences for norms.
%%%%%%%%%%%%%%%%%%%%%%%%
% lemma                %
%%%%%%%%%%%%%%%%%%%%%%%%
\begin{lemma}
  \label{lemma: different form norm}
  Let $m \in \N$ and $r, m' \in \R$.
  We have, for $\tau$ chosen \suff large, 
  \begin{align*}
    \Norm{\ttau^r w}{m,m',\ttau} 
    &\asymp \sum_{j=0}^m \Norm{D_{z_N}^j ( \ttau^r w)}{0, m +m' - j,\ttau}
   \asymp \sum_{j=0}^m \Norm{ \ttau^{r'_j} \Lambsct^{m''_j} D_{z_N}^j
      (\ttau^{r''_j} \Lambsct^{m'''_j}  w) }{+},
  \end{align*}
 where  $r= r'_j+ r''_j$, and  $m+ m' -j = m_j''+ m_j'''$, with $r_j', r_j'' \in \R$ and
 $m_j'', m_j''' \in \R$, $j=1, \dots, m$. Similarly, we have
  \begin{align*}
    \norm{\trace(\ttau^r w)}{m,m',\ttau} 
    \asymp \sum_{j=0}^m \norm{D_{z_N}^j (\ttau^r w)\brz }{m +m' - j,
    \ttau}
    \asymp \sum_{j=0}^m 
    \norm{\ttau^{r'_j} \Lambsct^{m''_j} \big(D_{z_N}^j (\ttau^{r''_j} \Lambsct^{m'''_j}  w)\big)\brz }{L^2(\R^{n-1})}.
  \end{align*}
\end{lemma}
See Section~\ref{proof lemma: different form norm} for a proof.

%%%%%%%%%%%%%%%%%%%%%%%%
% proposition           %
%%%%%%%%%%%%%%%%%%%%%%%%
\begin{proposition}[local tangential G{\aa}rding inequality]
  \label{prop: local Garding}
  Let $W_0, W_1$ be two open sets of $\R^N$, with $W_0 \Subset W_1$.
  Let $a(\y') \in S(\ttau^r\lambsct^m, \gt)$, with principal part
  $a_{r,m}$. If there exist $C>0$ and $R>0$ \st
  \begin{equation*}
    \Re a_{r,m}(\y') \geq C \ttau^r\lambsct^m, \quad z \in W_1, \quad \zeta' \in \R^{N-1}, 
    \quad \tau \geq \tauast, \quad\lambsct \geq R,
    \end{equation*}
  then
  for any $0<C'<C$ there exists  $\tau_1\geq \tauast$ \st 
  \begin{equation*}
    \Re \scp{\Opt(a)u}{ u}_{+} \geq C' \Norm{\ttau^{r/2}   u}{0,m/2,\ttau}^2, 
\qquad \tau \geq \tau_1.
  \end{equation*}
  for $u = w_{|Z}$, with
  $ w \in\Cinfc((0,S_0) \times \R^d)$ and $\supp(w) \subset W_0$.
\end{proposition}

In many occurrences, we shall use the following microlocal version of
the G{\aa}rding inequality.
%%%%%%%%%%%%%%%%%%%%%%%%
% proposition              %
%%%%%%%%%%%%%%%%%%%%%%%%
\begin{proposition}[microlocal tangential G{\aa}rding inequality]
  \label{prop: microlocal tangential Garding}
  Let $\U \subset \Mt$ be a conic open set.  Let also
  $\chi (\y') \in S(1, \gt)$ be homogeneous of degree zero and \st
  $\supp(\chi) \subset \U$. Let $r, m \in \R$ and
  $a(\y') \in S(\ttau^r\lambsct^m, \gt)$, with principal part
  $a_{r,m}$. If there exist $C>0$ and $R>0$ \st
  \begin{equation*}
    \Re a_{r,m}(\y') \geq C \ttau^r\lambsct^m, \quad \y' \in \U, 
    \quad \tau \geq \tauast, \quad\lambsct \geq R,
    \end{equation*}
  then
  for any $0<C'<C$, $M \in \N$, there exist $C_M$ and $\tau_0\geq \tauast$
  \st 
  \begin{equation*}
    \Re \scp{\Opt(a) \Opt(\chi) u}{ \Opt(\chi) u}_{+} 
    \geq C' \Norm{\ttau^{r/2} \Opt(\chi)
      u}{0,m/2,\ttau}^2 -C_M \Norm{u}{0,-M,\ttau}^2 , 
  \end{equation*}
  for $u \in \S(\Rpb)$ and $\tau \geq \tau_0$.
\end{proposition}

\subsection{Local setting and statement of the Carleman estimate}
\label{sec: local setting}
To explain the construction of the phase function, it is useful to use
a particular  set of coordinates.
We set $Z = (0,S_0) \times \Omega$ and $\d Z = (0,S_0) \times \d\Omega$.

Let $z_0 = (s_0,x_0) \in \d Z$.
In a \nhd $V$ of $z_0$ in $\R^N$, using normal geodesic coordinates for
the $x$ variable, we can express
the principal part of the Laplace operator $A$ in the following form
\begin{align}
  \label{eq: laplace normal geodesic coordinates}
  A = D_{x_d}^2 + R(x,D_{x'}), 
\end{align}
where $R(x,D_{x'})$ is a tangential differential operator of
order $2$ with principal symbol $r(x,\xi')$, 
\begin{align}
  \label{eq: quadratic form r}
  r(x,\xi') \geq C |\xi'|^2, 
\end{align}
where $C>0$.  We denote by $\tilde{r}(x,\xi',\eta')$ the associated
real symmetric bilinear form. The boundary $(0,S_0)\times \d \Omega$
is locally given by $\{ z_N=0\} = \{ x_d =0\}$. 

Without any loss of generality we shall assume that $V$ is a bounded
open set.

We then let $\psi(z)$ be defined in $\R^N$ and fulfilling the
properties listed in \eqref{eq: cond psi} with moreover, 
\begin{align}
  \label{eq: cond psi2}
  \d_{x_d} \psi (z)= \d_{z_N} \psi (z) \geq C>0, \qquad z \in V, 
\end{align}
and we set $\varphi_{\csp, \ctp} (z) = \exp(\csp \cpsi(z))$ with
$\cpsi(z) = \psi(\ctp s, \ctp x', x_d)$, for $\csp\geq 1$ and
$\ctp \in [0,1]$.  As mentioned above, we shall often write $\varphi$
in place if $\varphi_{\csp, \ctp}$ for the sake of concision.

\medskip
The main result of this section is the following Carleman estimate. 
%%%%%%%%%%%%%%%%%%%%%%%%
% theorem              %
%%%%%%%%%%%%%%%%%%%%%%%%
\begin{theorem}
  \label{theorem: Carleman boundary x}
  Let $P = D_s^4 + A^2$. Let $z_0= (s_0,x_0) \in (0,S_0) \times \d\Omega$. Let
  $\varphi(z) = \varphi_{\csp, \ctp} (z) $ be defined as above. There
  exists an open \nhd $W$ of $z_0$ in $(0,S_0) \times \R^d$, $W \subset
  V$, 
  and there exist $\tau_0\geq \tauast$, $\csp_0\geq 1$, $\ctp_0 \in (0,1]$, and $C>0$ \st
  \begin{multline}
    \label{eq: Carleman boundary x}
    \csp 
    \sum_{|\mi| \leq 4}\Norm{\ttau^{3-|\mi|}e^{\tau \varphi}
    D_{s,x}^\mi u}{+} 
    + 
    \sum_{0\leq j \leq 3} \norm{e^{\tau \varphi} D_{x_d}^r u_{|\d Z}}{7/2-j,\ttau}\\
    \leq C \Big( \Norm{e^{\tau \varphi} P u}{+}
    + 
    \sum_{j=0,1} \norm{e^{\tau \varphi} D_{x_d}^j u_{|\d Z}}{7/2-j,\ttau}
    \Big), 
  \end{multline}
 for $\tau \geq \tau_0$, $\csp \geq \csp_0$, $\ctp \in [0,\ctp_0]$,
  and for $u = w_{|Z}$, with
  $ w \in\Cinfc((0,S_0) \times \R^d)$ and $\supp(w) \subset W$.
\end{theorem}

As written in Case (iii) of  Section~\ref{sec: proof outline
  bi-laplace}, the proof we provide of this theorem is based on a
decomposition of phase-space in three microlocal regions
and the derivation of an adapted estimate in each one of these
regions. The  definition of these three regions is based on the properties 
of the roots of the principal symbol of $P$ viewed as a polynomial function
of degree four in the variable $\xi_d$. We start with the analysis of
those properties in the next section and define the microlocal regions
in Section~\ref{sec: def microlocal regions} below. In section~\ref{sec: Proof strategies in the three microlocal regions}
we provide  a proof scheme for a microlocal Carleman estimate in each
of the three regions. Then, in Sections~\ref{sec: Microlocal estimate
  in the region E-}--\ref{sec: Microlocal estimate in the region F} we
precisely state and prove the microlocal estimate associated with each region. Finally, in Section~\ref{sec: Proof of the Carleman estimate boudnary x} the various microlocal estimates are
  patched together, to yield the estimate of Theorem~\ref{theorem: Carleman boundary x}. 

\subsection{Root properties}
\label{sec: root properties}
Here, $z$ will
be assumed to be an element of $V$ so that all the symbols are well  defined. 
We write, as in Sections~\ref{sec: estimate away from boundaries} and
\ref{sec: Carleman boundary s=0},  
\begin{align*}
  P = P_2 P_1,  \quad \text{with}\  P_k= (-1)^k i D_s^2  + A.
\end{align*}

Setting $\Pconj = e^{\tau \varphi} P e^{-\tau \varphi}$ we have 
\begin{align}
  \label{eq: factorization Q2Q1}
  \Pconj=Q_2 Q_1,  \quad \text{with}\
  Q_k=  e^{\tau \varphi} P_k e^{-\tau \varphi}  = (-1)^k i (D_s+ i \tau \d_s \varphi(z))^2+ A_\varphi,
\end{align}
with, in the selected normal geodesic coordinates, 
\begin{align*}
  A_\varphi = e^{\tau \varphi} A e^{-\tau \varphi} =  (D_{x_d}+ i \tau \d_{x_d} \varphi(z))^2 
  + R(x,D_{x'} + i \tau d_{x'} \varphi(z)), \quad z=(s,x).
\end{align*}
In fact, we shall write $Q_k$ in the following form
\begin{align}
  \label{eq: def Mk}
  Q_k= (D_{x_d}+ i \tau \d_{x_d} \varphi(z))^2  + M_k,
  \qquad  M_k = (-1)^k i (D_s+ i \tau \d_s \varphi(z))^2+ R(x,D_{x'} + i \tau d_{x'} \varphi(z)).
\end{align}
This form will allow us, when a smooth square root $H_k$ of $M_k$ is
available in the tangential calculus associated with $\gt$, to write,
up to lower order terms,
$$
Q_k= (D_{x_d}+ i \tau \d_{x_d} \varphi + iH_k)
(D_{x_d}+ i \tau\d_{x_d} \varphi - iH_k),
$$
and, then, we shall base our derivation of a Carleman estimate for $P$
on estimates for first-order factors. This approach was introduced in
the seminal work of A.-P.~Calder\'on \cite{Calderon:58}. It has been
used recently to address boundary and interface problems in the
derivation of Carleman estimates \cite{LR-L:13,CR:14}. Of course, the
two smooth square roots, $H_1$ and $H_2$, may not always be
available. Still, on the occurrence of such a case, we shall find that
the operators $Q_1$ and $Q_2$ will be characterized by perfectly
elliptic estimates at the boundary, that is, one can estimate the
semi-classical Sobolev norm of the solution in $\Omega$ as well as the
counterpart norms for the traces of normal derivatives of the solution on
$\d \Omega$ (with the natural $1/2$ derivative shift for the
traces) --see Section~\ref{sec: perfect elliptic estimate}. As a preliminary to this analysis, we shall study the
properties of the principal symbols of $Q_1$ and $Q_2$ and the
properties of their roots.

\medskip
We denote the principal parts of $Q_k$ and $M_k$ by $q_k$ and
$m_k$, which gives, with $\y = (z, \zeta, \tau,\csp, \ctp)$ and $\zeta = (\sigma, \xi)$,  
\begin{align}
  \label{eq: def symbol qk}
  q_k(\y)=  \big(\xi_d+ i \tau \d_{x_d} \varphi(z)\big)^2   
    + m_k (\y') = 
  \big(\xi_d+ i \htau_{\xi_d} (\y')\big)^2   
    + m_k (\y') ,
\end{align}
with 
\begin{align}
  \label{eq: def symbol k}
    m_k(\y') =(-1)^k i \big(\sigma+ i \htau_\sigma(\y')\big)^2  
    + r\big(x,\xi' + i \htau_{\xi'}(\y')\big),
\end{align}
recalling the definition of $\htau(\y')$ introduced in
Section~\ref{sec: semiclassical cotangent vector htau} and using the
notation \eqref{eq: decomposition ttau}--\eqref{eq: decomposition
  ttau2}.

For $\hatt= (\hatt_\sigma, \hatt_\xi) \in \R \times \R^{d}$, with
$\hatt_\xi = (\hatt_{\xi'}, \hatt_{\xi_d}) \in \R^{d-1} \times \R$, we set 
\begin{align}
  \label{eq: def hm k}
  \hq_k(z,\zeta,\hatt)= \big(\xi_d+ i \hatt_{\xi_d} \big)^2 
  + \hm_k (z,\zeta',\hatt), 
  \quad 
\hm_k (z,\zeta',\hatt) := 
  (-1)^k i (\sigma+ i \hatt_\sigma)^2 
  +  r(x,\xi' + i \hatt_{\xi'}).
\end{align}
We have $q_k(\y)  = \hq_k (z,\zeta,\htau)$ and $m_k(\y')  = \hm_k (z,\zeta',\htau)$.

We now study the roots of $\hq_k(z,\zeta',\xi_d,\hatt)$, with $\zeta' =
(\sigma, \xi')$, when
viewed as a polynomial in the variable $\xi_d$, with the other  variables, $z$,
$\zeta'$, and $\hatt$  acting as parameters.
To that purpose, we introduce the following quantity
\begin{align}
  \label{eq: def hmu k}
  \hmu_k (z,\zeta',\hatt) := 4 \hatt_{\xi_d}^2 \Re \hm_k
  (z,\zeta',\hatt)
  - 4 \hatt_{\xi_d}^4 + \big(\Im \hm_k  (z,\zeta',\hatt) \big)^2.
\end{align}
We choose $\hat{h}_k (z,\zeta',\hatt) \in \C$ \st 
\begin{align}
  \label{eq: def ht k}
  \Re \hat{h}_k (z,\zeta',\hatt) \geq 0 \quad \text{and} \ \ \hat{h}_k
  (z,\zeta',\hatt)^2 = \hm_k (z,\zeta',\hatt). 
\end{align}
We may then write 
\begin{equation}
  \label{eq: hqk factorizatrion}
    \hq_k (z,\zeta,\hatt) = (\xi_d + i \hatt_{\xi_d})^2 + \hat{h}_k^2 (z,\zeta',\hatt)
    = \big(\xi_d - \hrho_{k,+}(z,\zeta',\hatt)\big)  \big(\xi_d - \hrho_{k,-}(z,\zeta',\hatt)\big),
\end{equation}
with
\begin{equation}
  \label{eq: def hrho}
  \hrho_{k,\pm}(z,\zeta',\hatt) = -  i \hatt_{\xi_d} \pm i \hat{h}_k (z,\zeta',\hatt).  
\end{equation}
The choice of  $\hat{h}_k$ is unique if $\hm_k\notin \R_-$. The results of this section are yet
valid in the case $\hm_k\in \R_-$; however, in the following sections, those
results based on the factorization~\eqref{eq: hqk factorizatrion} will
only be used in settings where $\hm_k\in \R_-$ does not occur. 

  We give some properties of the roots $\hrho_{k,\pm}(z,\zeta',\hatt)$. 
%%%%%%%%%%%%%%%%%%%%%%%%
% lemma                %
%%%%%%%%%%%%%%%%%%%%%%%%
\begin{lemma}
  \label{lemma: root behaviors1} 
  We assume that $\hatt_{\xi_d}\geq 0$. Let $k=1$ or $2$. The roots $\hrho_{k,+}(z,\zeta',\hatt)$ and
  $\hrho_{k,-}(z,\zeta',\hatt)$ are both homogeneous of degree one in $(\zeta',\hatt)$,
  and \st 
  \begin{equation}
    \label{eq: root behavior1}
    \Im \hrho_{k,-}\leq - \hatt_{\xi_d} \leq \Im \hrho_{k,+}.
  \end{equation}
  We also have 
  \begin{equation}
     \label{eq: root behavior2}
    \hrho_{k,-} = \hrho_{k,+} \quad  \Leftrightarrow  \quad   \hrho_{k,-} =
     \hrho_{k,+} = - i \hatt_{\xi_d}\quad \Leftrightarrow \quad
     \hm_k =0.
  \end{equation}
  Moreover, if $\hatt_{\xi_d}>0$, we have 
  \begin{equation}
     \label{eq: root behavior3}
   \Im \hrho_{k,+} \lesseqqgtr 0 \quad \Leftrightarrow \quad \hmu_k  \lesseqqgtr 0 .
  \end{equation}
\end{lemma}
In particular, if $\hatt_{\xi_d}>0$, observe that the root $\hrho_{k,-}$ remains in the
lower half complex plane, independently of the values of $z,\zeta'$,
and $\hatt$, while the root $\hrho_{k,+}$ may cross the real line.
%% proof of lemma
\begin{proof}
 The roots can be chosen continuous with respect to $\zeta$ and $\hatt$ and
  homogeneity comes  naturally. Observe that $\Im \hrho_{k,\pm} =-  
  \hatt_{\xi_d} \pm \Re \hat{h}_k$.  As $\Re \hat{h}_k \geq 0$ then \eqref{eq:
    root behavior1} is clear. The form of $\hrho_{k,\pm}$ above
  yields the equivalences in \eqref{eq:
    root behavior2}. 

  Finally, as $\Im \hrho_{k,+} \lesseqqgtr 0$ is equivalent to
  $\Re \hat{h}_k \lesseqqgtr \hatt_{\xi_d}$, Lemma~\ref{lemma: prep prop
    root behaviors1} below implies \eqref{eq: root behavior3}, since 
  $\Re \hat{h}_k \geq 0$ and $\hatt_{\xi_d} >0$. 
\end{proof} 

%%%%%%%%%%%%%%%%%%%%%%%%
% lemma                %
%%%%%%%%%%%%%%%%%%%%%%%%
\begin{lemma}
 \label{lemma: prep prop root behaviors1}
 Let $t \in \C$ and $m = t^2$. We then have, for $x_0 \in \R$ \st $x_0 \neq 0$, 
 \begin{align*}
   |\Re t| \lesseqqgtr |x_0| \quad \Leftrightarrow \quad 4 x_0^2 \Re m -  4x_0^4 +
   (\Im m)^2 \lesseqqgtr 0 .
 \end{align*}
\end{lemma}
%% proof of lemma
\begin{proof}
  Let $t = x + i y$. 
  We have $\Re m = x^2 - y^2$ and $\Im m = 2 xy$ and we observe that 
  \begin{align*}
    4 x_0^2 \Re m -  4x_0^4 + (\Im m)^2 = 4 (x_0^2 + y^2)(x^2 -
    x_0^2), 
 \end{align*}
  which gives the result.
\end{proof}
%%%%%%%%%%%%%%%%%%%%%%%%
% corollary            %
%%%%%%%%%%%%%%%%%%%%%%%%
\begin{corollary}
  \label{cor: positivity mu}
  We assume that $\hatt_{\xi_d}>0$. Let $k=1$ or $2$. 
  If $C>0$, there exists $C'>0$ \st 
  \begin{align*}
    \hmu_k (z, \zeta', \hatt) \geq C (|\hatt|^2 + |\zeta'|^2)^2 \quad \Rightarrow
    \Im \hrho_{k,+}(z, \zeta', \hatt) \geq C'\hlambdat, \qquad \hlambdat
    = (|\hatt|^2 + |\zeta'|^2)^{1/2}, 
  \end{align*}
  for $(z, \zeta', \hatt) \in \ovl{V \cap \Rp} \times \R^{N-1} \times \R^N$.
\end{corollary}
%%%% proof of corollary
\begin{proof}
  We consider the compact set (recall that $V$ is bounded)
  \begin{align*}
    \Con = \{ (z, \zeta', \hatt) \in \ovl{V \cap \Rp} \times \R^{N-1}
    \times \R^N ;\ \hlambdat=1\}.
  \end{align*}
  The inequality $\hmu_k \geq C$ yields a compact set $K$ of $\Con$.
  By \eqref{eq:
    root behavior3} in  Lemma~\ref{lemma: root behaviors1}, we have
  $\Im \hrho_{k,+}\geq C'>0$ on $K$, and we conclude by homogeneity.
\end{proof}

%%%%%%%%%%%%%%%%%%%%%%%%
% proposition                %
%%%%%%%%%%%%%%%%%%%%%%%%
\begin{proposition}
  \label{prop root behaviors2} 
   We assume that $\hatt_{\xi_d}\geq 0$. Let $k=1$ or $2$, we have the following properties:
  \begin{enumerate}
  \item There exist $\theta_0 \in (0,1)$ and $C>0$ 
    \st if 
    \begin{align*}
      z \in V \ \ \text{and} \ \ |\hatt| \leq \theta_0\ \hlambdat, 
    \end{align*}
    then the roots $\hrho_{k,\pm}$ are simple and non real, and  moreover
    \begin{align}
      \label{eq: position roots htau=0}
      \Im \hrho_{k,+} \geq C\hlambdat, 
      \quad  \Im \hrho_{k,-} \leq -C \hlambdat \quad 
      (z,\zeta', \hatt) \in V \times \R^{N-1}\times \R^N,
   \end{align} %% 
   with $\hlambdat = (|\hatt|^2+ |\zeta'|^2)^{1/2}$. 
  \item There exists $C>0$ \st 
    \begin{align*}
     0\leq \hatt_{\xi_d} \leq C \big(|\hatt'| +  |\zeta'|\big), \quad \text{and} \ \ 
      |\zeta'| \leq C |\hatt|, 
    \end{align*}
    if $\hrho_{k,+}\in \R$, where $\hatt' = (\hatt_{\sigma},
    \hatt_{\xi'})$.  In such case,  the
    value of the imaginary part of the second root is prescribed and
    nonpositive: $\Im \hrho_{k,-}=-2\hatt_{\xi_d}$.
    \item There exists $C>0$ \st $|\hatt'|/C \leq  |\zeta'| \leq C
      |\hatt'|$,  if
   $\hq_k$ has a double root.  
\end{enumerate}

\medskip Finally, if $\hatt_{\xi_d}>0$ and if $|\hatt'| /
\hatt_{\xi_d}$ is \suff small, and if the polynomial $\hq_k$,
$k=1$ or $2$, has a double root, then both roots of the second symbol,
$\hq_{k'}$ with $k'\neq k$, are in the lower half complex plane. More
precisely, there exist $C_0,C_1>0$ \st if $|\hatt'| \le C_0
\hatt_{\xi_d}$ then
  \begin{equation}
    \label{eq: double -> Im < 0}
    \hrho_{k,+}=\hrho_{k,-} \quad \Rightarrow \quad  \Im
 \hrho_{k',\pm}\leq -C_1\hatt_{\xi_d}.
  \end{equation}

\end{proposition}
%%%% proof of proposition
\begin{proof}
  
 \noindent {\bfseries Proof of point (1).}
  Because of homogeneity it is sufficient to  assume that
  $(\zeta',\hatt)$ is on
  the sphere $\mathbb S = \big\{ \hlambdat =1\big\}.$
  If $\hatt=0$, then we have $\hm_k = \hat{h}_k^2 =
  r(x,\xi')+(-1)^ki\sigma^2$. 
  Observe that $\hm_k\neq 0$ here. Otherwise
  $\sigma=0$ and $\xi'=0$, which  cannot hold as $|\zeta'|=1$. Moreover $\Re \hm_k\geq 0$. Hence, we have
  $\Re\hat{h}_k >0$. 
  Then we write 
  \begin{align*}
    \hq_k = \xi_d^2 + \hat{h}_k^2 = (\xi_d + i \hat{h}_k)(\xi_d - i
    \hat{h}_k), 
  \end{align*}
  yielding $\hrho_{k,-} = - i \hat{h}_k$ and $\hrho_{k,+} =  i
  \hat{h}_k$, which gives $\Im \hrho_{k,-}<0$ and $\Im \hrho_{k,+}
  >0$. As $\mathbb S \cap \{ \hatt=0\}$ is compact we find that 
  $\Im \hrho_{k,-}\leq - C< 0$ and $\Im \hrho_{k,+}
  \geq C>0$, for some $C>0$, for $|\zeta'|=1$ and $\hatt=0$. Then,
  using a  compactness argument once more,  using the continuity of
  the roots, 
   there exist $\theta_0\in (0,1)$ \st 
  \begin{align*}
    \Im \hrho_{k,-} (z, \zeta', \hatt) \leq - C'< 0, \quad \Im \hrho_{k,+}
  (z, \zeta', \hatt)\geq C'>0, 
  \end{align*}
  if $z \in \ovl{V}$ and $|\hatt| \leq \theta_0$, recalling that $V$
  is bounded.
  We then obtain \eqref{eq: position roots htau=0} in $V$ by
  homogeneity. 
In particular this excludes having double roots and real roots.

  \bigskip
  \noindent{\bfseries Proof of point (2).}
  Observe that the inequality $ |\zeta'| \leq C |\hatt|$, in the case
  of a real root is simply another formulation of part of point (1). 
  Next, we observe that $|\hm_k| \lesssim |\hatt'|^2 + |\zeta'|^2$ implies
  $|\Re \hat{h}_k| \lesssim  |\hatt'| + |\zeta'|$. Since having
  $\hrho_{k,+}\in \R$ is equivalent to $\Re \hat{h}_k = \hatt_{\xi_d}$
  by \eqref{eq: def hrho},
  we thus obtain $\hatt_{\xi_d} \lesssim  |\hatt'| + |\zeta'|$.
 As $\Im \hrho_{k,-} = -  \hatt_{\xi_d} - \Re \hat{h}_k$, we then have $\Im \hrho_{k,-}=-2\hatt_{\xi_d}$.
 
 \bigskip
 \noindent{\bfseries Proof of point (3)}
  The equation $\hm_k=0$, which is equivalent to having a double root,
  reads 
  \begin{align}
    \label{eq: double root 1}
    r(x,\xi') - r(x, \hatt_{\xi'}) - (-1)^k 2 \sigma \hatt_\sigma=0,
      \qquad
    \sigma^2 - \hatt_\sigma^2 + (-1)^k 2 \tilde{r}(x, \xi',
    \hatt_{\xi'}) =0,
  \end{align}
with $\tilde{r}(x,\xi',\eta')$ defined below~\eqref{eq: quadratic
    form r}. 
  From \eqref{eq: double root 1}, using that $r(x,.)$ is uniformly
  positive definite, we obtain  
  \begin{align*}
    |\xi'|^2 \lesssim |\hatt_{\xi'}|^2 + |\sigma| |\hatt_\sigma|, 
    \qquad 
    |\sigma|^2 \lesssim |\hatt_{\sigma}|^2 + |\xi'| |\hatt_{\xi'}|.
  \end{align*}
  The sum of the two estimates gives $|\zeta'|^2 \lesssim |\hatt'|^2 +
  |\sigma| |\hatt_\sigma| + |\xi'| |\hatt_{\xi'}|$, 
  and with the Young inequality we obtain 
  $|\zeta'| \lesssim |\hatt'|$. Similarly,   from \eqref{eq: double
    root 1} we obtain  
  \begin{align*}
   |\hatt_{\xi'}|^2 \lesssim  |\xi'|^2 + |\sigma| |\hatt_\sigma|, 
    \qquad 
    |\hatt_{\sigma}|^2  \lesssim |\sigma|^2 + |\xi'| |\hatt_{\xi'}|, 
  \end{align*}
  and the  sum of the two estimates gives 
   $|\hatt'|^2 \lesssim  |\zeta'|^2 + |\sigma| |\hatt_\sigma| + |\xi'|
   |\hatt_{\xi'}|$, and with the Young inequality we obtain 
  $|\hatt'| \lesssim |\zeta'|$.

  \medskip
  Note that we could deduce that $|\zeta'| \lesssim
  |\hatt|$ from point (1). Here, we have obtained a sharper estimate. 
  
  \bigskip
  \noindent{\bfseries Proof of (\ref{eq: double -> Im < 0}).}
  If $\hq_k$ has a double root, then  $|\hatt'|\asymp |\zeta'|$ by
  point (3). Let $\delta \in (0,1)$, and set $C_1 = 1-\delta$. 
  To have $\Im \hrho_{k',\pm} \leq - C_1 \hatt_{\xi_d}$ it suffices
  to have $\Im \hrho_{k',+} \leq - C_1 \hatt_{\xi_d}$ by
  Lemma~\ref{lemma: root behaviors1}. With the notation of the proof
  of that lemma, this reads $-\hatt_{\xi_d} + \Re \hat{h}_{k'} \leq - C_1
  \hatt_{\xi_d}$, that is $0 \leq \Re \hat{h}_{k'} \leq \delta
  \hatt_{\xi_d}$. Now as we have $|\Re \hat{h}_{k'}|\leq |\hat{h}_{k'}| \leq |\hm_{k'}|^{1/2}
  \lesssim |\hatt'| +  |\zeta'|$, we find that $0\leq \Re \hat{h}_{k'} \lesssim
  |\hatt'|$ here.
  The result thus follows if we assume that $|\hatt'| / \hatt_{\xi_d}$
  is chosen \suff small. 
\end{proof}
%%%%%%%%%%%%%%%%%%%%%%%%
% lemma                %
%%%%%%%%%%%%%%%%%%%%%%%%
\begin{lemma}
  \label{lemma: mu > -C}
  Assume that $|\hatt'| \leq C_0\hatt_{\xi_d}$ for some $C_0>0$. 
  There exists $\delta_0>0$ \st if $\delta \in (0, \delta_0)$
  and  $\hmu_k (z,\zeta', \hatt)
  \geq - \delta \hlambdat^4$, with $ \hlambdat^2 = |\hatt|^2 + |\zeta'|^2$, then the roots of $\hq_k$
  are simple. 
\end{lemma}
%%%% proof of lemma
\begin{proof}
  Because of homogeneity it is sufficient to work on
  the sphere $\mathbb S = \big\{ \hlambdat=1\big\}.$
  Writing $\hm_k = \hat{h}_k^2$ with $\Re \hat{h}_k\geq 0$ as in 
  \eqref{eq: def ht k}, 
  we observe that $\hmu_k \geq -\delta$ reads
  \begin{align*}
    4 \big(\hatt_{\xi_d}^2 + (\Im \hat{h}_k)^2\big) \big((\Re
    \hat{h}_k)^2 - \hatt_{\xi_d}^2 \big) \geq -\delta,
  \end{align*}
  using the computation of the proof of Lemma~\ref{lemma: prep prop
    root behaviors1} with $x_0= \hatt_{\xi_d}$.  Assume that we have a
  double root.  In such case $\hm_k=0$ by Lemma~\ref{lemma: root
    behaviors1} and $|\hatt'| \asymp |\zeta'|$ by
  point (3) of Proposition~\ref{prop root behaviors2}.  We then have $\hat{h}_k=0$,
  yielding
  $4 \hatt_{\xi_d}^4 \leq \delta = \delta \hlambdat^4
    \lesssim \delta \hatt_{\xi_d}^4$,
  using that $|\hatt'| \leq C_0\hatt_{\xi_d}$.
  Thus, for $\delta$ chosen \suff small we reach a contradiction.
\end{proof}

%%%%%%%%%%%%%%%%%%%%%%%%
% lemma                %
%%%%%%%%%%%%%%%%%%%%%%%%
\begin{lemma}
  \label{lemma: equivalence tau zeta'}
  Let $k=1$ or $2$. 
 If both $\delta>0$ and $|\hatt'|/ \hatt_{\xi_d}$ are
  \suff small, there exists $C>0$ \st for
  $(z, \zeta', \hatt) \in  \ovl{V\cap \Rp}\times \R^{N-1}\times
  \R^N$
  $$\hmu_k(z,\zeta', \hatt) 
      \geq - \delta \hlambdat^4 
      \ \ \Rightarrow \ \ |\hatt|\leq C|\zeta'|,$$
 with $ \hlambdat^2 = |\hatt|^2 + |\zeta'|^2$.
\end{lemma}
%%%% proof of lemma
\begin{proof}
  Because of homogeneity it is sufficient to work on
  the sphere $\mathbb S = \big\{ \hlambdat=1\big\}.$

Let us now assume that the implication does not hold. Then there
  exists
  $(z^{(n)},\zeta^{\prime(n)}, \hatt^{(n)}) \in \ovl{V\cap \Rp} \times
  \mathbb S$,
  \st
  $\hmu_k (z^{(n)},\zeta^{\prime(n)}, \hatt^{(n)}) \geq - \delta$ and
  $|\hatt^{(n)}| > n |\zeta^{\prime(n)}|$.
   As
  $(z^{(n)},\zeta^{\prime(n)}, \hatt^{(n)})$ lays in a compact set
  (recall that $V$ is bounded), it
  converges, up to a subsequence, to
  $(z^{(\infty)},\zeta^{\prime(\infty)}, \hatt^{(\infty)}) \in \ovl{V\cap \Rp} \times
  \mathbb S$.
  We find that $\zeta^{\prime(\infty)} =0$ and 
    $\hm_{k} (z^{(\infty)},0,\hatt^{(\infty)}  )  
    =  (-1)^{k-1}i (\hatt^{(\infty)}_\sigma)^2 -
    r(x,\hatt^{(\infty)}_{\xi'})$ , yielding
  \begin{align*}
    \hmu_k (z^{(\infty)},0,\hatt^{(\infty)}  )  
    = - 4 (\hatt^{(\infty)}_{\xi_d})^2 r(x,\hatt^{(\infty)}_{\xi'}) 
    -4 (\hatt^{(\infty)}_{\xi_d})^4 
    + (\hatt^{(\infty)}_\sigma)^4 \leq -3, 
  \end{align*}
    for $|\hatt^{(\infty)\prime}|/ \hatt^{(\infty)}_{\xi_d}$ \suff
    small, as we have $|\hatt^{(\infty)}|=1$.
    For $\delta$ \suff small, we hence reach a contradiction.
\end{proof}

\subsection{Microlocal regions}
\label{sec: def microlocal regions}
With the functions $\hmu_k$, $k=1,2$,  introduced in \eqref{eq: def hmu k} we shall define several
microlocal regions. Observe first that $\hmu_k$ is an homogeneous
polynomial function of
degree four in $(\zeta', \hatt)$. We thus have $\hmu_k \in S^4_{\T,\hatt}$ in the
sense given by Definition~\ref{def: semi-classical symbol}. 
From Lemma~\ref{lemma: from on calculus to the other}, we find that we
have $\hmu_k (z, \zeta', \htau(\y')) \in S(\lambsct^4,
\gt)$. We thus define
\begin{align}
  \label{eq: def muk}
  \mu_k (\y') := \lambsct^{-4}(\y')\  \hmu_k (z,
  \zeta', \htau(\y')) \in S(1, \gt), \quad k=1,2, \
  \y'= (z, \zeta', \tau, \csp, \ctp).
\end{align}
We recall that $\htau = \tau d_z \varphi_{\csp, \ctp}(z) = \ttau(\y')
d_z \cpsi(z)$ and $\cpsi(z) =\psi( \ctp z', z_N)$ with $0< \eps< 1$. 
Observe that we have
$|\htau(\y')| = \ttau(\y') \Norm{d_z \cpsi}{L^\infty} \leq
\ttau(\y') \Norm{d_z \psi}{L^\infty}$.  Thus, having
$0\leq \ttau \leq \delta \theta_0 \lambsct(\y') / \Norm{d_z \psi}{L^\infty}$, for
$\delta \in (0,1]$, implies $|\htau(\y')| \leq
\theta_0\lambsct(\y')$. The value $\theta_0$ is as introduced
in Proposition~\ref{prop root behaviors2}. We set $\theta_1=
\frac{1}{32} \theta_0 / \Norm{d_z \psi}{L^\infty}$. 

Let $\delta\in(0,1]$ and let $V$ be the bounded
open \nhd in $\R^N$  of
$z_0 \in \d Z$, introduced in Section~\ref{sec: local setting}.
We set $\MtV = V \times \R^{N-1} \times   [\tauast,+\infty)
      \times [1,+\infty)\times [0,1]$. 
We define the following microlocal regions, for $k=1,2$, 
\begin{align*}
  &F (V,\delta ) = \{ \y' \in \MtV; \ z \in V, \
    \ttau (\y') \leq \delta \theta_1 \lambsct(\y') \} ,
  \\
  &E^{(k)}_-(V, \delta) = \{ \y' \in \MtV; \ z \in V, \
    \mu_k (\y') \leq - \delta\},\\
  &E_0^{(k)}(V, \delta) = \{ \y' \in \MtV; \ z \in V, \
    \mu_k (\y')  \geq -\delta\}.
\end{align*}
Evidently, we have $\MtV =E^{(k)}_-(V, \delta)   \cup E^{(k)}_0(V, \delta)$.
We now set 
\begin{align*}
  &\E_-(V, \delta) = E^{(1)}_-(V, \delta) \cup E^{(2)}_-(V, \delta),
    \qquad \E_0(V, \delta) = E^{(1)}_0(V, \delta) \cap E^{(2)}_0(V, \delta),
\end{align*}
and we have $\MtV = \E_-(V, \delta)  \cup \E_0(V, \delta)$.
Below, in the text, when no precision is needed,  we shall use the ``vague'' terminology
$F$, $\E_-$, or $\E_0$, to refer to microlocal regions that take the
forms of $F(V,\delta)$, $\E_-(V, \delta)$, $\E_0(V, \delta)$.

We let $\chi_-, \chi_0 \in \Cinf(\R)$, with values in $[0,1]$, be \st 
\begin{align*}
  & \chi_-\equiv 1 \ \text{on}\ (-\infty, -1], \quad\text{and} \ \
    \supp(\chi_-) \subset (-\infty, -1/2], \\
  & \chi_0\equiv 1 \ \text{on}\  [-2, +\infty), \quad\text{and} \ \
    \supp(\chi_0) \subset [-3, \infty).
\end{align*}
Let $V_0 \Subset V$ be an open \nhd of $z_0$ in $\R^N$ and let
$\chi_{V_0} \in \Cinf(\R^N)$ be \st
$\supp(\chi_{V_0}) \subset V$ and $\chi_{V_0}\equiv 1$ in an open \nhd
of $V_0$.  With $\eta \in \Cinfc(-\theta_1,\theta_1)$, with values in
$[0,1]$ \st $\eta\equiv 1$ in $[-\theta_1/2,\theta_1/2]$, we set
\begin{align*}
  \chi_{\delta,F}(\y') = \eta \big(\ttau(\y')/(\delta \lambsct(\y'))\big) \in S(1, \gt).
\end{align*}
and 
\begin{align*}
  \chi_{F}(\y') = \chi_{V_0}(z)\  \chi_{1,F}(\y') \in S(1, \gt).
\end{align*}
We set 
\begin{align*}
  &\chi^{(k)}_{\delta,-} (\y') = \chi_{V_0}(z)\ (1-\chi_{1/4,F}(\y'))\
    \chi_-(\mu_k(\y')/\delta) \in S(1,\gt).
\end{align*}
Observe that we have 
\begin{align*}
 \chi^{(k)}_{\delta,-} \equiv 1 \ \ \text{on} \ \
  E^{(k)}_-(V_0,\delta) \setminus F(V,1/4), 
  \quad \supp(\chi^{(k)}_{\delta,-}) \subset  E^{(k)}_-(V,\delta/2) \setminus F(V,1/8), 
\end{align*}
and thus 
\begin{align*}
\chi^{(1)}_{\delta,-} + \chi^{(2)}_{\delta,-} \geq 1 \ \ \text{on} \ \
  \E_-(V_0,\delta) \setminus F(V,1/4), 
 \quad \supp(\chi^{(1)}_{\delta,-} + \chi^{(2)}_{\delta,-}) 
  \subset  \E_-(V,\delta/2) \setminus F(V,1/8).
\end{align*}
We finally set 
\begin{align*}
  \chi_{\delta,0} (\y') = \chi_{V_0}(z)\ (1-\chi_{1/4,F}(\y'))\ 
  \chi_0(\mu_1(\y') /\delta) \ \chi_0(\mu_2(\y') /\delta)  \in S(1,\gt).
\end{align*}

%% figure
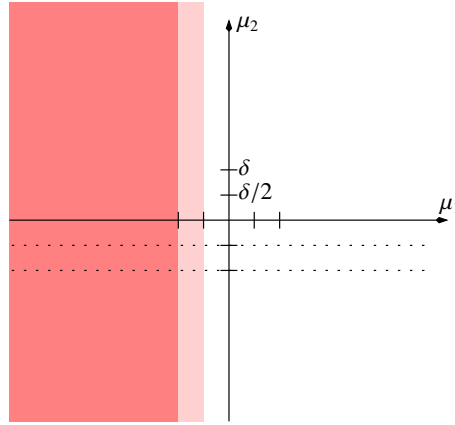
\begin{figure}
\begin{center}
\begin{picture}(0,0)%
\includegraphics{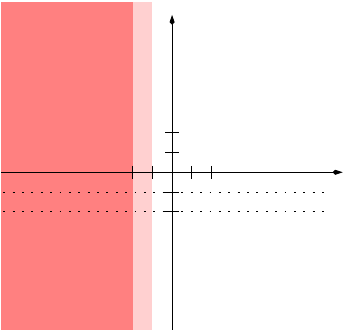}%
\end{picture}%
\setlength{\unitlength}{3947sp}%
\begingroup\makeatletter\ifx\SetFigFont\undefined%
\gdef\SetFigFont#1#2#3#4#5{%
  \reset@font\fontsize{#1}{#2pt}%
  \fontfamily{#3}\fontseries{#4}\fontshape{#5}%
  \selectfont}%
\fi\endgroup%
\begin{picture}(2754,2638)(1489,-2848)
\put(2927,-1442){\makebox(0,0)[lb]{\smash{{\SetFigFont{9}{10.8}{\rmdefault}{\mddefault}{\updefault}{\color[rgb]{0,0,0}$\delta/2$}%
}}}}
\put(2927,-1295){\makebox(0,0)[lb]{\smash{{\SetFigFont{9}{10.8}{\rmdefault}{\mddefault}{\updefault}{\color[rgb]{0,0,0}$\delta$}%
}}}}
\put(4178,-1507){\makebox(0,0)[lb]{\smash{{\SetFigFont{9}{10.8}{\rmdefault}{\mddefault}{\updefault}{\color[rgb]{0,0,0}$\mu_1$}%
}}}}
\put(2918,-369){\makebox(0,0)[lb]{\smash{{\SetFigFont{9}{10.8}{\rmdefault}{\mddefault}{\updefault}{\color[rgb]{0,0,0}$\mu_2$}%
}}}}
\end{picture}%

\caption{Microlocal region $\E_-$. In dark color is the region where
  $\chi_{\delta,-}^{(1)} \equiv 1$. In light color is the support of
  $\chi_{\delta,-}^{(1)}$.  The boundaries of the associated regions
  for $\chi_{\delta,-}^{(2)}$ are marked dashed.}
\label{fig: region E-}
\end{center}
\end{figure}

%% figure
\begin{figure}
\begin{center}
\begin{picture}(0,0)%
\includegraphics{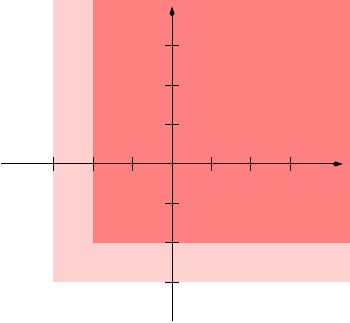}%
\end{picture}%
\setlength{\unitlength}{3947sp}%
\begingroup\makeatletter\ifx\SetFigFont\undefined%
\gdef\SetFigFont#1#2#3#4#5{%
  \reset@font\fontsize{#1}{#2pt}%
  \fontfamily{#3}\fontseries{#4}\fontshape{#5}%
  \selectfont}%
\fi\endgroup%
\begin{picture}(2799,2583)(1489,-2848)
\put(4072,-1493){\makebox(0,0)[lb]{\smash{{\SetFigFont{9}{10.8}{\rmdefault}{\mddefault}{\updefault}{\color[rgb]{0,0,0}$\mu_1$}%
}}}}
\put(2640,-453){\makebox(0,0)[lb]{\smash{{\SetFigFont{9}{10.8}{\rmdefault}{\mddefault}{\updefault}{\color[rgb]{0,0,0}$\mu_2$}%
}}}}
\put(1876,-1786){\makebox(0,0)[lb]{\smash{{\SetFigFont{9}{10.8}{\rmdefault}{\mddefault}{\updefault}{\color[rgb]{0,0,0}$-3\delta$}%
}}}}
\put(2176,-1486){\makebox(0,0)[lb]{\smash{{\SetFigFont{9}{10.8}{\rmdefault}{\mddefault}{\updefault}{\color[rgb]{0,0,0}$-2\delta$}%
}}}}
\end{picture}%

\caption{Microlocal region $\E_0$. In dark color is the region where
  $\chi_{\delta,0} \equiv 1$. In light color is the support of
  $\chi_{\delta,0}$.}
\label{fig: region Z}
\end{center}
\end{figure}
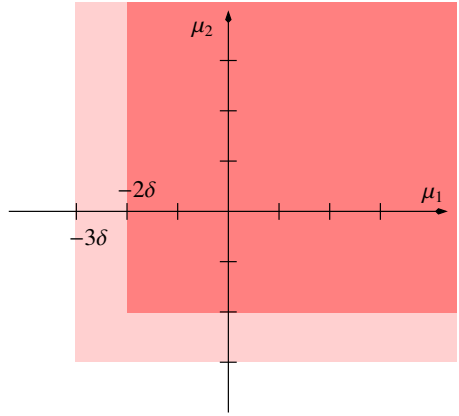

Observe that we have 
\begin{align*}
\chi_{\delta,0} \equiv 1 \ \ \text{on} \ \ \E_0(V_0,2\delta) \setminus F(V,1/4)
\quad \supp(\chi_{\delta,0}) \subset \E_0(V,3 \delta) \setminus F(V,1/8),
\end{align*}
and 
\begin{align}
  \label{eq: microlocal covering property}
\chi_{F}+ \chi^{(1)}_{\delta,-} + \chi^{(2)}_{\delta,-}
  + \chi_{\delta,0} 
  \geq  1\ \ \text{on a \cnhd of} \ \MtVz.
\end{align}
With the microlocal cutoff functions we have just introduced we
associate tangential pseudo-differential operators, all in $\Psi(1, \gt)$,
\begin{align}
  \label{eq: def microlocalization operators}
  \Xi_{F} = \Opt \big(\chi_{F}\big),
  \qquad 
  \Xi^{(k)}_{\delta, -} = \Opt \big(\chi^{(k)}_{\delta,-}\big), \ \
  k=1,2, \qquad \text{and} \ \
  \Xi_{\delta,0} = \Opt \big(\chi_{\delta,0}\big).
\end{align}

\subsection{Proof strategies in the three microlocal regions}
\label{sec: Proof strategies in the three microlocal regions}

Derivations in all three microlocal regions require first the proof of estimates
for various factors and second the  concatenation of those
estimates. For this second part, to avoid redundancies, we 
describe in Appendix~\ref{sec: Estimate concatenations}, along with
proofs, how various type of estimates can be  concatenated. 

The estimate associated with region
$\E_-$ is proven in Section~\ref{sec: Microlocal estimate in the region E-}. 
In region $\E_-$, we have $\Pconj = Q_1 Q_2$ where at least one of the factors
is characterized by a principal symbol with two roots in the lower
half complex plane. This yields for this factor, say $Q_1$, a perfectly
elliptic estimate at the boundary $\{x_d=0\}$, as given by
Lemma~\ref{lemma: microlocal elliptic estimate} (see
Appendix~\ref{sec: perfect elliptic estimate}). For the second
operator $Q_2$, one can derive an estimate whose form is classical and
exhibits a loss of a half derivative, as given in
Proposition~\ref{prop: estimate Qk}. A proof is provided in
Appendix~\ref{sec: An Estimate for Qk}, in particular since the
estimate needs to hold uniformly with respect to all parameters
introduced.  Finally, the two estimates are concatenated
to obtain an estimate for $\Pconj$ in $\E_-$.

The estimate associated with region
$\E_0$ is proven in Section~\ref{sec: Microlocal estimate in the region E0}. 
The treatment this region  requires the most delicate argument and
justifies the development of the Weyl-H\"ormander calculus of
Section~\ref{sec: pseudo -3p}. 
Microlocally, in this region we write $\Pconj = Q_1 Q_2$ 
and we manage to write each $Q_k$, $k=1,2$, in the form
\begin{align*}
  Q_k = Q_{k,-} Q_{k,+} + (1 + \csp \ctp) R_1,
\end{align*}
where $Q_{k,-}$, $Q_{k,+}$ and $R_1$ are all first-order operators.
The operator $Q_{k,-}$ is characterized by a principal symbol with a
root in the lower half complex plane. Setting $Q^- = Q_{1,-} Q_{2,-}$
and $Q^+ = Q_{1,+} Q_{2,+}$, with delicate commutator arguments we
obtain $\Pconj = Q^- Q^+ + (1 + \csp \ctp) R_3$ where
$R_3 \in \PsiOpsc^{2,1}$. We thus manage gather togethers factors with
similar root locations without generating a remainder in
$\csp \PsiOpsc^{2,1}$. Observe that this latter class for the remainder
is obtained if operator commutations within the Weyl-H\"ormander
calculus are carried invoking usual arguments. Here, to obtain the
sharper class $(1 + \csp \ctp) \PsiOpsc^{2,1}$, we use the
precise forms of the involved operators and symbols. 

 For $Q^-$ we have a
perfectly elliptic estimate at the boundary $\{x_d=0\}$, as given by
Lemma~\ref{lemma: microlocal elliptic estimate} (see
Appendix~\ref{sec: perfect elliptic estimate}). For each operator
$Q_{k,+}$ a sub-elliptic estimate can be obtained with a trace term used as
an observation as given by Lemma~\ref{lemma: sub-ellitptic estimate
  Qk+} in Appendix~\ref{sec: A root with a vanishing imaginary part}.
Concatenated together, two such estimates yield an estimate for
$Q_{+}$ with a loss of a full derivative and  observation
terms that involve both the Dirichlet trace and the Neumann trace of
the solution. Concatenating now the estimates for $Q^-$ and $Q^+$ one
obtains microlocally an estimate of the form 
\begin{align*}
    \csp \Norm{\ttau^{-1} v}{4,0,\ttau}
    + \norm{\trace(v)}{3, 1/2,\ttau} 
    \lesssim
      \Norm{Q^- Q^{+}   v}{+}
    +\norm{\trace(v)}{1,5/2,\ttau}.
  \end{align*}
With the form of the remainder term $(1 + \csp \ctp) R_3$ that appeared above in
the decomposition of $\Pconj$ one then sees that a similar estimate
can be obtained for $\Pconj$ in place of $Q^- Q^{+}$ by choosing
$\csp>0$ \suff  large and $\ctp>0$ \suff small.  
Observe that if the remainder term had been in  $\csp \PsiOpsc^{2,1}$ we would not
have been able to transform the estimate obtained for $Q^- Q^{+}$
into an estimate for $\Pconj$.

Some technical aspects of the proof in the region $\E_0$ described above
require to have $\ttau(\y')$ of the same order as
$\lambsct(\y')$. This is however not true in that region. One thus
rather considers a region of the form $\E_0 \setminus F$, since 
region $F$ is characterized by $\ttau(\y') \leq C \lambsct(\y')$ for a
well chosen constant (see above). A last microlocal region, namely
$F$, thus needs to be considered. 

The treatment of region $F$ is given in Section~\ref{sec: Microlocal
  estimate in the region F} and has some similarities with what is
done in the region $\E_0 \setminus F$. Yet, the treatment of remainder
terms needs not be as refined. The operator $\Pconj$ is
written in the form $\Pconj = Q^- Q^+ + \csp R_3$ with $R_3\in  \PsiOpsc^{2,1}$
and again $Q^- = Q_{1,-} Q_{2,-}$ and $Q^+ = Q_{1,+} Q_{2,+}$.
Here also, for $Q^-$ we have a perfectly elliptic estimate at the
boundary $\{x_d=0\}$. For $Q^+$, the estimate we obtain is very
different from what is done in $\E_0 \setminus F$. The region $F$ is
designed so that the roots associated with the factors $ Q_{1,+}$ and
$Q_{2,+}$ are both located
in the upper half complex plane. For each of these operators one can
thus obtain  a microlocal elliptic estimate at the boundary $\{x_d=0\}$ with
one trace used as an observation term yet without any loss of
derivative as given
in Lemma~\ref{lemma: elliptic estimate Qk+} in Appendix~\ref{sec: A
  root with a positive imaginary part}.  Put together, with a
concatenation argument, an estimate for $Q^+$ is obtained
with observation terms that involve both the Dirichlet trace and the
Neumann trace of the solution. This estimate for $Q^+$ does not exhibit
any loss of
derivative: it is an {\em elliptic} estimate. 
Concatenated together the estimates for
$Q^+$ and $Q^-$ yield also an elliptic estimate for $Q^- Q^+$ with the
above  two traces as observation terms. 
The elliptic strength of this
estimate then allows one to handle the remainder term in $\csp
\PsiOpsc^{2,1}$ yielding  a similar result for $\Pconj$ in the microlocal region $F$.

As a final step of the proof of  Theorem~\ref{theorem: Carleman
  boundary x}, we patch together  the estimates obtained in the above three microlocal regions.
This is done in Section~\ref{sec: Proof of the Carleman estimate boudnary x}.

\subsection{Microlocal estimate in the region $\E_-$}
\label{sec: Microlocal estimate in the region E-}
We prove the following estimate.
%%%%%%%%%%%%%%%%%%%%%%%%
% proposition          %
%%%%%%%%%%%%%%%%%%%%%%%%
\begin{proposition}
  \label{prop: microlocal estimate E-}
  Let $M \in \N$.
  Let $k=1$ or $2$. For $\delta\in (0,1)$, there exist $\tau_0\geq \tauast$,
  $\csp_0\geq 1$, and $C>0$ \st
  \begin{align}
    \label{eq: microlocal estimate E-}
    \csp^\hf
    \Norm{\ttau^{-1/2} \Xi^{(k)}_{\delta,-} v}{4, 0, \ttau} 
    +\norm{\trace(\Xi^{(k)}_{\delta,-} v)}{3,1/2,\ttau} 
    \leq C \Big( \Norm{\Pconj \Xi^{(k)}_{\delta,-}v}{+}
    +\norm{\trace(\Xi^{(k)}_{\delta,-}v)}{0,7/2,\ttau} 
    + \Norm{v}{4,-M, \ttau} 
    \Big), 
  \end{align}
 for $\tau \geq \tau_0$, $\csp \geq \csp_0$, $\ctp \in [0,1]$,
  and for $v  \in \S(\Rpb)$.
\end{proposition}
The term $ \Norm{v}{4,-M, \ttau}$ in the \rhs stands as a remainder that
will be 'absorbed' once the estimations in the different microlocal
regions are patched together. In fact, observe that this term
is much weaker than that in the \lhs in the Carleman estimate \eqref{eq:
  Carleman boundary x} of Theorem~\ref{theorem: Carleman boundary
  x}. The meaningful observation term in the \rhs of \eqref{eq:
  microlocal estimate E-} is $\norm{\trace(v)}{0,7/2,\ttau}$, which is
of the strength as the terms in the \lhs of \eqref{eq:
  Carleman boundary x}, and can be found in the \rhs of that latter estimate.

\begin{proof}
  We have $\Pconj = Q_1 Q_2$. 
  We consider the case $k=1$. The same proof can be written in the
  case $k=2$. To ease notation we write $\chi$ in place of
  $\chi_{\delta,-}$ and $\Xi$ in place of $\Xi_{\delta,-}$.

  In a \cnhd of $\supp(\chi) \subset \MtV$, with $V$ introduced in
  Section~\ref{sec: local setting}, we have $\mu_1 \leq -C
  \delta$. As \eqref{eq: cond psi2} holds in $V$ we have 
  $\htau_{\xi_d} \geq C \ttau$ and thus $|\htau_\xi| \asymp  \ttau$.
  By Lemma~\ref{lemma: root behaviors1}, both roots of the
  symbol $q_1$ of the operator $Q_1$ are in the lower half complex
  plane. Thus,
  \begin{equation}
    \label{eq: E- Q1 perfect elliptic}
    \text{the operator $Q_1$ fulfills the requirements of  Lemma~\ref{lemma: microlocal elliptic estimate}.}
  \end{equation}

Also, for the operator $Q_2$, without any assumption on the position
of the roots in the complex plane, we have the following estimate,
characterized by the loss of a half derivative and a boundary
observation term,  by
Proposition~\ref{prop: estimate Qk}, for $\ell \in \R$, 
\begin{align}
  \label{eq: E- Q2 1/2 derivative loss}
    \csp^{1/2}\Norm{\ttau^{-1/2} \Xi v}{2,\ell,\ttau}
    + \norm{\trace(\Xi v)}{1,\ell+ 1/2,\ttau} 
      \lesssim
      \Norm{Q_2 \Xi v}{0,\ell,\ttau}
    + \norm{\trace(\Xi v)}{0,\ell+ 3/2,\ttau},
   \end{align}
   for $v\in \S(\Rpb)$, for $\tau\geq \tauast$ and $\csp\geq 1$ chosen \suff large,
   and $\ctp \in [0,1]$ (recall that $\supp(\chi)\subset \MtV$
   which gives $\supp(\Xi v) \subset V' \Subset V$, for some open set
   $V'$, thus permitting the application of Proposition~\ref{prop: estimate Qk}).

   With \eqref{eq: E- Q1 perfect elliptic}, \eqref{eq: E- Q2 1/2
     derivative loss}, and Proposition~\ref{prop: estimate order 4
     operator}, applied with $Q^-= Q_1$ and $Q^+ = Q_2$ here, and with
   $\alpha_1=0$ and $\alpha_2=1$ and $\delta_1=1$ and $\delta_2=0$, we
   obtain the result of the proposition, by choosing
   $\tau\geq \tauast$ and $\csp\geq 1$ \suff large.
\end{proof}

\subsection{Microlocal estimate in the region $\E_0 \setminus F$}
\label{sec: Microlocal estimate in the region E0}

We prove the following estimate.
%%%%%%%%%%%%%%%%%%%%%%%%
% proposition          %
%%%%%%%%%%%%%%%%%%%%%%%%
\begin{proposition}
  \label{prop: microlocal estimate E0}
 % Let $k=1$ or $2$.
  Let $M \in \N$.  For $\delta_0\in (0,1)$ chosen \suff small and
  $\delta \in (0, \delta_0]$, there exist $\tau_0\geq \tauast$, $\csp_0\geq 1$,
  $\ctp_0 \in (0,1]$, and $C>0$ \st
  \begin{align*}
    \csp
    \Norm{\ttau^{-1}\Xi_{\delta,0} v}{4, 0, \ttau} 
    +\norm{\trace(\Xi_{\delta,0} v)}{3,1/2,\ttau} 
    \leq C \Big( \Norm{\Pconj \Xi_{\delta,0}  v}{+}
    +\norm{\trace(\Xi_{\delta,0} v)}{1,5/2,\ttau} 
    + \Norm{v}{4,-M, \ttau} 
    \Big), 
  \end{align*}
 for $\tau \geq \tau_0$, $\csp \geq \csp_0$, $\ctp \in [0,\ctp_0]$,
  and for  $v  \in \S(\Rpb)$. 
\end{proposition}

Before giving the proof of this microlocal estimate we need to provide
some additional properties of the symbols $m_k$ introduced in
Section~\ref{sec: root properties} and its square root, $h_k$. Note
that the region $F$ is introduced to isolate the case where
$\ttau \leq C|\zeta'|$ and this permits to exploit the relation
$|\zeta'|\leq C \ttau $ in the region $\E_0 \setminus F$. This is used
to
obtain some symbol properties of $h_k$.

We recall the form of the tangential differential operator $M_k$, as
introduced in \eqref{eq: def Mk}, 
\begin{align*}
M_k := (-1)^k i (D_s+ i \tau \d_s \varphi(z))^2 +  R(x,D_{x'} + i \tau
d_{x'} \varphi(z)), 
\end{align*}
whose principal symbol is given by 
$m_k(\y') := (-1)^k i (\sigma+ i \tau \d_s \varphi(z))^2 +  r(x,\xi' + i \tau
d_{x'} \varphi(z)) \in S(\lambsct^2, \gt)$.
Observe that we have the following symbol  estimation.
%%%%%%%%%%%%%%%%%%%%%%%%
% lemma                %
%%%%%%%%%%%%%%%%%%%%%%%%
\begin{lemma}
  \label{lemma: prop d xd mk}
  We have  $\d_{x_d} m_k \in S((1+ \ctp \csp)\lambsct^2, \gt)$.
\end{lemma}
%%%% proof of lemma
\begin{proof}
  We write $m_k(\y')  =  (-1)^k i \big(\sigma+ i \htau_\sigma (\y')\big)^2 
  +  r\big(x,\xi' + i \htau_{\xi'}(\y')\big)$, with the notation of
  \eqref{eq: decomposition ttau}.
  We then have
  \begin{align*}
    \d_{x_d} m_k 
    = -2 (-1)^k (\d_{x_d} \htau_\sigma) 
    (\sigma+ i\htau_\sigma)
    +2 i \tilde{r}(x,\xi' + i \htau_{\xi'}, \d_{x_d} \htau_{\xi'} )
    + \d_{x_d} r(x,\xi' + i \htau_{\xi'}), 
  \end{align*}
  with $\tilde{r}(x,\xi',\eta')$ defined below~\eqref{eq: quadratic
    form r}.  By Lemma~\ref{lemma: htau in the calculus},
  we have $\htau' = ( \htau_\sigma, \htau_{\xi'}) \in S(\ctp \ttau ,
  \gt)^{N-1}$ yielding $\d_{x_d}\htau' \in S(\ctp \csp \ttau ,
  \gt)^{N-1}$, and as $\d_{x_d} r (x,\xi' + i \htau_{\xi' }) \in S(\lambsct^2,
  \gt)$,
  the result follows.
\end{proof}

Let $k=1,2$.  If $\mu_k (\y') \geq - C\delta$, and for $\delta >0$
\suff small then $m_k \neq 0$ and, equivalently, the roots of
$q_k(\y)$ are simple, by Lemma~\ref{lemma: mu > -C} since
$|\htau'|\lesssim \htau_{\xi_d}$ for $z\in V$; recall the
definition of the operator $Q_k$,
$Q_k = (D_{x_d} + i \tau \d_{x_{d}} \varphi(x))^2 + M_k$, and its
principal symbol $q_k$ in \eqref{eq: factorization Q2Q1}--\eqref{eq:
  def symbol qk}.
%%%%%%%%%%%%%%%%%%%%%%%%
% lemma                %
%%%%%%%%%%%%%%%%%%%%%%%%
\begin{lemma}
  \label{lemma: squareroot mk}
  Let $C, C'>0$ and let $\U_\delta$ be a conic open set of
  $\MtV$
  \st $\mu_k (\y') \geq - C \delta$ and
  $\lambsct \leq C' |\htau(x)|$ in $\U_\delta$.  For $\delta_0\in (0,1)$ and
  $\ctp_0>0$ chosen
  \suff small, if $0< \delta \leq \delta_0$ and $0 \leq \ctp \leq
  \ctp_0$, the symbol $m_k$ is elliptic and
  there exists $h_k \in S(\lambsct,\gt)$ in $\U_\delta$ that is
  elliptic and that satisfies
  \begin{align*}
    h_k^2 = m_k \ \ \text{and}\ \ 
    \Re h_k \geq 0.
  \end{align*}
  Moreover, we have $\d_{x_d} h_k \in S((1+ \ctp \csp)\lambsct, \gt)$ in $\U_\delta$.
\end{lemma}
The second part of Lemma~\ref{lemma: squareroot mk} improves, for $h_k$,  upon the natural behavior of an
arbitrary element of $f \in S(\lambsct, \gt)$ for which we have
$\d_{x_d} f \in  S(\csp \lambsct, \gt)$. This is a key aspect of our
proof strategy of the Carleman estimate. In fact, if one chooses
$\ctp=0$, that is, a weight function $\psi = \psi(x_d)$, then one finds
directly that $\d_{x_d} m_k \in S(\lambsct^2, \gt)$ and $\d_{x_d}  h_k
\in S(\lambsct, \gt)$, as confirmed by 
Lemmata~\ref{lemma: prop d xd mk} and \ref{lemma: squareroot
  mk}. However, such a weight function is not convex with respect to
the boundary $\{ x_d =0\}$, which turns out to be an obstruction for the
applications of the Carleman estimate we consider here. If we simply let $\psi$ be of
the form $\psi(s,x',x_d)$ we then obtain $\d_{x_d}  h_k
\in S(\csp \lambsct, \gt)$ and the proof scheme for the Carleman
estimate collapses: the
parameter $\csp$ needs to be set large, which yields uncontrolled terms in
the derivation. The introduction of the
parameter $\ctp$, writing $\psi_\ctp(z) = \psi(\ctp s,\ctp x',x_d)$ is thus designed to control this behavior and to bring
the analysis as ``close'' as we wish to the case $\eps=0$ for the
derivation of the estimate and yet preserving some convexity   with respect to
the boundary $\{ x_d =0\}$.

%%%% proof of lemma
\begin{proof}
  In $V$, we have $\d_{x_d} \psi \geq C>0$ yielding $|\htau| \asymp
  \htau_{\xi_d}\asymp \ttau$ by Lemma~\ref{lemma: equiv ttau |htau|}. Next, $|\htau'|/\htau_{\xi_d}$ can be
  made as small as needed by choosing $\ctp>0$ small. 
 Thus, if we choose $\delta\in (0,\delta_0]$ and $\ctp>0$ \suff small, by 
  Lemma~\ref{lemma: equivalence tau zeta'} we have
  $|\htau(x)| \lesssim |\zeta'|$ and with the additional assumption made
  here we obtain 
  \begin{equation}
    \label{eq: equivalence tau zeta'}
  |\zeta'| \asymp \ttau \asymp \ttau_{\xi_d} \qquad \text{in} \ \ovl{\U_\delta}.
  \end{equation}

  If  $m_k(\y')$ remains away from a \nhd of the negative real axis in
  the complex plane
  for $\y' \in \ovl{\U_\delta}$, we can then define $h_k(\y')$ as the
  principal square root of $m_k(\y')$. Then, it is straightforward to
  obtain $h_k \in  S(\lambsct,\gt)$ in $\U_\delta$. In fact, if we assume
  $|\Im m_k(\y')| \leq \alpha \lambsct^2$, as we have, recalling the
  definition of $\mu_k$ in \eqref{eq: def muk}, 
  \begin{align*}
    \mu_k (\y') \lambsct^4 (\y')
    = 4 \htau_{\xi_d}^2 \Re m_k (\y')
  - 4 \htau_{\xi_d}^4 + \big(\Im m_k  (\y')\big)^2
  \end{align*}
  it yields, using \eqref{eq: equivalence tau zeta'}, 
     $\Re m_k (\y') \geq  \htau_{\xi_d}^2 (1 + \O(\delta + \alpha))$. 
   By choosing $\alpha$ and  $\delta$ \suff small, we obtain
   $\Re m_k (\y') \gtrsim  \htau_{\xi_d}^2$ in $\ovl{\U_\delta}$.
   
   As $m_k(\y')$ is homogeneous of degree two, we find that $h_k$ is
   homogeneous of degree one in $\ovl{\U_\delta}$.  Recalling that $z =(x,s)$
   remains in a compact domain here, we thus find
   \begin{equation}
     \label{eq: ellipticity tk}
     |h_k (\y')| \gtrsim  \lambsct \qquad \text{in} \ \ovl{\U_\delta}.
   \end{equation}

   Next, as $h_k^2 = m_k \neq 0$ in $\ovl{\U_\delta}$ we may write, with
   Lemma~\ref{lemma: prop d xd mk}, 
   \begin{align*}
     2 h_k \d_{x_d} h_k = \d_{x_d} m_k \in  S((1+ \ctp \csp)\lambsct^2, \gt).
   \end{align*}
   which yields the result using the ellipticity estimate~\eqref{eq: ellipticity tk}. 
\end{proof}

We let $\uchi_{\delta}, \chi_{\delta, 1} \in S(1, \gt)$ be supported in $\MtV$, homogeneous
of degree zero,  and  be \st 
$\mu_k \geq -C \delta$ for  both $k=1,2$ on their supports and 
$\chi_{\delta,1} \equiv 1$ in a \cnhd of
$\supp(\chi_{\delta,0})$ and $\uchi_{\delta} \equiv 1$ in a \cnhd of
$\supp(\chi_{\delta,1})$. Recalling the notation of Section~\ref{sec:
  def microlocal regions} and the microlocalization symbols constructed
there, this can be done as follows, for instance
for the construction of $\chi_{\delta,1}$.
Let $\hat{\chi_1}  \in \Cinf(\R)$ be \st
\begin{equation*}
  \supp(\hat{\chi_1})\subset [-4,+\infty), \quad
  \hat{\chi_1} \equiv 1 \ \text{on a \nhd of}\ [-3,+\infty).
\end{equation*}
%% figure
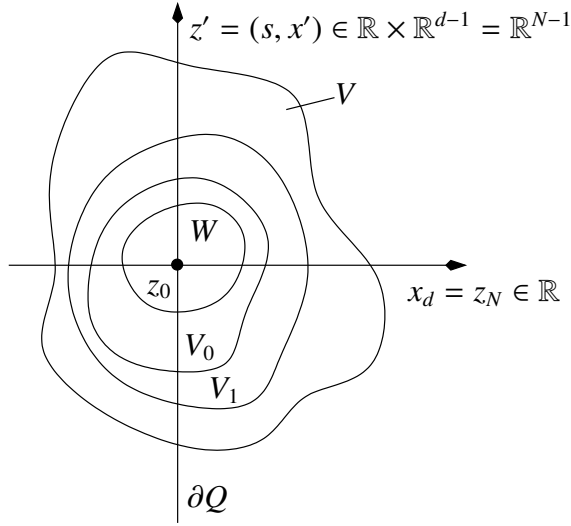
\begin{figure}
\begin{center}
\begin{picture}(0,0)%
\includegraphics{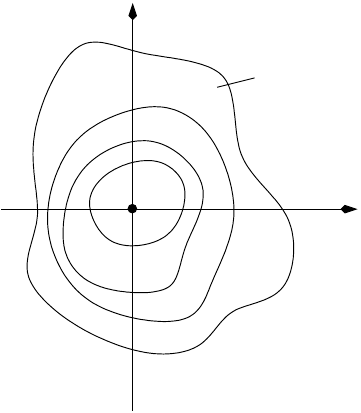}%
\end{picture}%
\setlength{\unitlength}{3947sp}%
\begingroup\makeatletter\ifx\SetFigFont\undefined%
\gdef\SetFigFont#1#2#3#4#5{%
  \reset@font\fontsize{#1}{#2pt}%
  \fontfamily{#3}\fontseries{#4}\fontshape{#5}%
  \selectfont}%
\fi\endgroup%
\begin{picture}(2874,3288)(739,-3637)
\put(1872,-1858){\makebox(0,0)[lb]{\smash{{\SetFigFont{12}{14.4}{\rmdefault}{\mddefault}{\updefault}{\color[rgb]{0,0,0}$W$}%
}}}}
\put(1861,-3523){\makebox(0,0)[lb]{\smash{{\SetFigFont{12}{14.4}{\rmdefault}{\mddefault}{\updefault}{\color[rgb]{0,0,0}$\d Q$}%
}}}}
\put(3226,-2236){\makebox(0,0)[lb]{\smash{{\SetFigFont{12}{14.4}{\rmdefault}{\mddefault}{\updefault}{\color[rgb]{0,0,0}$x_d = z_N\in \R$}%
}}}}
\put(1876,-586){\makebox(0,0)[lb]{\smash{{\SetFigFont{12}{14.4}{\rmdefault}{\mddefault}{\updefault}{\color[rgb]{0,0,0}$z'= (s,x')\in \R\times\R^{d-1} =\R^{N-1}$}%
}}}}
\put(1840,-2572){\makebox(0,0)[lb]{\smash{{\SetFigFont{12}{14.4}{\rmdefault}{\mddefault}{\updefault}{\color[rgb]{0,0,0}$V_0$}%
}}}}
\put(1993,-2842){\makebox(0,0)[lb]{\smash{{\SetFigFont{12}{14.4}{\rmdefault}{\mddefault}{\updefault}{\color[rgb]{0,0,0}$V_1$}%
}}}}
\put(1612,-2188){\makebox(0,0)[lb]{\smash{{\SetFigFont{12}{14.4}{\rmdefault}{\mddefault}{\updefault}{\color[rgb]{0,0,0}$z_0$}%
}}}}
\put(2776,-1003){\makebox(0,0)[lb]{\smash{{\SetFigFont{12}{14.4}{\rmdefault}{\mddefault}{\updefault}{\color[rgb]{0,0,0}$V$}%
}}}}
\end{picture}%

\caption{Open \nhds of $z_0 \in \d Z$
  introduced in the course of the proof of Theorem~\ref{theorem:
    Carleman boundary x}.}
\label{fig: nhd space}
\end{center}
\end{figure}
We also introduce $V_1\subset V$ an open  \nhd of $\supp(\chi_{V_0})$
in $\Rpb$, in particular $V_0 \Subset V_1$ (the local geometry is
illustrated in Figure~\ref{fig: nhd space})
and we choose $\chi_{V_1}
\in \Cinf(\Rpb)$ \st 
\begin{align*}
  &\chi_{V_1} \equiv 1 \ \text{on a \nhd of}\ V_1, \qquad
  \supp(\chi_{V_1}) \subset V.
\end{align*}
We set
\begin{align*}
  \chi_{\delta,1}(\y')  
  =\chi_{V_1}(z) \ (1 -\chi_{1/16,F}(\y'))\ 
  \hat{\chi}_1(\mu_1(\y')/\delta)
  \hat{\chi}_1(\mu_2(\y')/\delta)\in S(1, \gt).
\end{align*}
we have $\chi_{\delta,1} \equiv 1$ in a \cnhd of
$\supp(\chi_{\delta,0})$.

\medskip
We choose $\delta_0>0$ \suff small so that the results of
Lemmata~\ref{lemma: mu > -C} and \ref{lemma: equivalence tau zeta'}
apply, that is,  on $\supp(\uchi_\delta)$ the roots of $q_k$ are simple and
$|\htau(\y')| \lesssim |\zeta'|$, and also the result of Lemma~\ref{lemma: squareroot mk}
holds for $\U_\delta$ a \cnhd of $\supp(\uchi_{\delta})$, for $\delta
\in (0,\delta_0)$ and for $\ctp>0$ chosen \suff small. 
  With the value of $\delta$ fixed
now, to ease notation we now write $\uchi, \chi_0, \chi_1$ in place of
$\uchi_\delta, \chi_{\delta,0},  \chi_{\delta,1}$ and $\Xi_0, \Xi_1$
in place of $\Opt(\chi_{\delta,0}), \Opt(\chi_{\delta,1})$.
%%%%%%%%%%%%%%%%%%%%%%%%
% lemma                %
%%%%%%%%%%%%%%%%%%%%%%%%
\begin{lemma}
  \label{lemma: factorization Qk}
  Let $\chi = \chi_0$ or $\chi_1$ and, accordingly, $\Xi= \Xi_0$ or $\Xi_1$. We have 
  \begin{align*}
    Q_k \Xi
    &= Q_{k,+}Q_{k,-} \Xi
    + (1+\csp \ctp) R_1 \Xi + R_{-M}\\
    &=Q_{k,-}Q_{k,+} \Xi
    + (1+\csp \ctp) R_1' \Xi + R_{-M}',
\end{align*}
where $Q_{k,a} =\big(D_{x_d}+ i \htau_{\xi_d} - i \, a\Opt^w(h_k \uchi)
\big)$, $a \in \{ +,-\}$, and $R_1,R_1' \in \Psi(\lambsct, \gt)$ and $R_{-M}, R_{-M}'\in
\Psi(\lambsct^{-M}, \gt)$, for arbitrary large $M \in \N$.
\end{lemma}
%%%% proof of lemma
\begin{proof}

  In the proof we shall denote by $R_j$ a generic operator in
  $\Psi(\lambsct^j,\gt)$, $j \in \mathbb \R$, whose expression may
  change from one line to the other.

Observe that we have, for any $M \in \N$, 
\begin{align*}
  M_k \Xi &= M_k \Opt(\uchi)^2\Xi  + R_{-M}
            = \Opt(m_k \uchi^2) \Xi + (1+\csp \ctp) R_1 \Xi +   R_{-M}.
\end{align*}

With Lemma~\ref{lemma: squareroot mk} applied with $\U_\delta$, a \cnhd
of $\supp(\uchi)$, we have 
\begin{align*}
  \Opt(m_k \uchi^2) =\Opt (h_k \uchi)^2 \mod \Psi((1+\csp \ctp) \lambsct,\gt),
\end{align*}
using the properties of the tangential calculus (see
Proposition~\ref{prop: admissible metric and order function
  tangential}). This yields
\begin{align*}
  M_k \Xi= \Opt (h_k \uchi)^2 \Xi +(1+\csp \ctp)R_1 \Xi +   R_{-M}\\
  = \Opt^w(h_k \uchi)^2 \Xi +(1+\csp \ctp)R_1 \Xi +   R_{-M}.
\end{align*}
We then find
\begin{align*}
  Q_k \Xi
  &= (D_{x_d}+ i \htau_{\xi_d})^2 \Xi +M_k \Xi
  \\
  &=\big(D_{x_d}+ i  \htau_{\xi_d} + i \Opt^w(h_k \uchi) \big)
  \big(D_{x_d}+ i \htau_{\xi_d}  - i \Opt^w(h_k \uchi) \big)\Xi  \\
  & \quad + i  [D_{x_d}+ i  \htau_{\xi_d},\Opt^w(h_k \uchi)] \Xi
    +(1+\csp \ctp)R_1 \Xi +   R_{-M},
\end{align*}
In fact, the order of the operators can be changed and we find
\begin{align*}
  Q_k \Xi
  &=\big(D_{x_d}+ i  \htau_{\xi_d} - i \Opt^w(h_k \uchi) \big)
  \big(D_{x_d}+ i \htau_{\xi_d}  + i \Opt^w(h_k \uchi) \big)\Xi  \\
  & \quad - i  [D_{x_d}+ i  \htau_{\xi_d},\Opt^w(h_k \uchi)] \Xi
    +(1+\csp \ctp)R_1 \Xi +   R_{-M}.
\end{align*}
The following lemma then yields the result.
\end{proof}
%%%%%%%%%%%%%%%%%%%%%%%%
% lemma                %
%%%%%%%%%%%%%%%%%%%%%%%%
\begin{lemma}
  \label{lemma: property tk}
  Let $\chi = \chi_0$ or $\chi_1$ and, accordingly, $\Xi= \Xi_0$ or $\Xi_1$.
  We have, for $a \in \{+,-\}$,  $[D_{x_d}+ i \htau_{\xi_d}, Q_{k,a}] \Xi =  -i\, a [D_{x_d}+ i \htau_{\xi_d},\Opt^w(h_k \uchi)] \Xi =
  (1+\csp \ctp) R_1 \Xi + R_{-M}$ with $R_1 \in  \Psi(\lambsct,\gt)$
  and $R_{-M} \in \Psi(\lambsct^{-M}, \gt)$. 
\end{lemma}
%%%% proof of lemma
\begin{proof}
  We have
  $[\htau_{\xi_d}, \Opt^w(h_k \uchi)] \in \Psi((1+\csp \ctp)
  \lambsct,\gt)$
  as a consequence of the tangential calculus we have introduced.  We
  have $[D_{x_d},\Opt^w(h_k \uchi)] = \Opt^w\big(D_{x_d} (h_k \uchi)\big)$.
  We then write
  \begin{align*}
  D_{x_d} (h_k \uchi) =  D_{x_d} (h_k) \uchi + h_k D_{x_d} (\uchi).
  \end{align*}
  Because of the definition of $\uchi$ we have $D_{x_d} \uchi(\y')
  \equiv 0$ in $\supp (\chi(\y'))$. Thus $\Opt^w\big((D_{x_d} \uchi) h_k\big) \Xi
  \in \Psi(\lambsct^{-M}, \gt)$, for any $M \in \N$. Next, by
  Lemma~\ref{lemma: squareroot mk} we  have 
  $\uchi D_{x_d} h_k\in S((1+ \ctp \csp)\lambsct, \gt)$, which
  concludes the proof.
\end{proof}
%%%%%%%%%%%%%%%%%%%%%%%%
% lemma                %
%%%%%%%%%%%%%%%%%%%%%%%%
\begin{lemma}
  \label{lemma: commutation Qk+ Ql-}
  Let $\chi = \chi_0$ or $\chi_1$  and, accordingly, $\Xi= \Xi_0$ or $\Xi_1$. 
  Let $k, \ell \in \{1,2\}$ and $a,b \in \{+,- \}$.  We have, for any
  $M\in \N$, 
  \begin{align*}
    [Q_{k,a}, Q_{\ell, b}] \Xi = (1+\csp \ctp) R_1 \Xi + R_{-M}, 
   \quad [D_{x_d}+ i \htau_{\xi_d}, Q_{k,a}Q_{\ell, b}] \Xi =
    (1+\csp \ctp) R_{1,1} \Xi + R_{1, -M},
  \end{align*}
  with $R_1 \in  \Psi(\lambsct,\gt)$, $R_{1,1} \in \PsiOpsc^{1,1}$,
  $R_{-M} \in \PsiOpsc^{0,-M}$, and $R_{1,-M} \in \PsiOpsc^{1,-M}$.
\end{lemma}
%%%% proof of lemma
\begin{proof}
Since $[\Opt^w(h_k \uchi), \Opt^w(h_l \uchi)] \in  \Psi((1+\csp \ctp)
\lambsct,\gt)$, using the properties of the tangential calculus (see
Proposition~\ref{prop: admissible metric and order function
  tangential}),  the result follows from Lemma~\ref{lemma: property tk}.
\end{proof}

We may now provide a proof of the microlocal estimate for the region
$\E_0$. 
\begin{proof}[\bfseries Proof of Proposition~\ref{prop: microlocal
  estimate E0}.]

In the proof, we shall denote by $R_{j,k}$ a generic operator in
  $\PsiOpsc^{j,k}$, $j \in \mathbb N$, $k \in \R$,  whose expression may
  change from one line to the other. We denote by $M$ an arbitrarily
  large integer whose value may change from one line to the other.

With the previous lemmata we write, using that
$\chi_1 \equiv 1$ on $\supp(\chi_0)$, 
\begin{align}
  \label{eq: decomposition Pvarphi}
  \Pconj \Xi_0  &= Q_1 Q_2 \Xi_0  
                            = Q_1 \Xi_1Q_2 \Xi_0  + R_{4,-M}\\
                            &= Q_{1,-} Q_{1,+} \Xi_1Q_{2,-} Q_{2,+}\Xi_0 
                     + (1+\csp \ctp)R_{2,1} \Xi_0 
                     + R_{4,-M} \notag\\
    &= Q_{1,-} Q_{1,+} 
      Q_{2,-} Q_{2,+}\Xi_0  
      + (1+\csp \ctp)R_{2,1} \Xi_0 + R_{4,-M}\notag\\
      &=Q_{1,-} Q_{1,+} Q_{2,-} \Xi_1 Q_{2,+}\Xi_0  + (1+\csp \ctp)R_{2,1} \Xi_0 + R_{4,-M}\notag\\ 
      &=
      Q_{1,-}  Q_{2,-} Q_{1,+}\Xi_1 Q_{2,+}\Xi_0   + (1+\csp \ctp)R_{2,1} \Xi_0 + R_{4,-M}\notag\\
       &= Q^- Q^+ \Xi_0  + (1+\csp \ctp)R_{2,1} \Xi_0 + R_{4,-M},\notag
\end{align}
with $ Q^- = Q_{1,-}Q_{2,-}  $ and  $ Q^+ = Q_{1,+}Q_{2,+}$.

\medskip
The principal symbol of $Q^-$ is
  $q^- = q_{1-} q_{2,-} \in \Symbsc^{2,0}$ in a \cnhd of
  $\supp(\chi_0)$, where all the roots of $q^-$
  have negative imaginary parts. Thus, 
  \begin{equation}
    \label{eq: E0 Q- perfect elliptic}
    \text{the operator $Q^-$ fulfills the requirements of  Lemma~\ref{lemma: microlocal elliptic estimate}.}
  \end{equation}

For both $Q_{1,+}$ and $Q_{2,+}$ we have the following estimate,
characterized by the loss of a half derivative and a trace
observation, as given
by Lemma~\ref{lemma: sub-ellitptic estimate Qk+}, for $\ell, m \in
\R$, 
\begin{align*}
      \csp^{1/2}\Norm{\ttau^{m-1/2} \Xi_0 v}{1,\ell, \ttau} 
      \lesssim \Norm{\ttau^{m} Q_{k,+} \Xi_0 v}{0,\ell,\ttau}
  +  \norm{\trace(\ttau^{m}\Xi_0 v)}{0,\ell+ 1/2,\ttau} 
  + \Norm{v}{0,-M, \ttau}
  , \qquad k=1,2,
  \end{align*}
  for $v \in \S(\Rpb)$, and for $\tau\geq \tauast$ and $\csp\geq 1$ chosen \suff large,
  and $\ctp \in [0,1]$.  Then, according to Proposition~\ref{prop:
    estimate concatenation}, applied with $\alpha_1 = \alpha_2=1$, we
  have the following estimate for the operator $Q^+$, for
  $M >0$ and $\ell \in \R$,
  \begin{align}
    \label{eq: E0 est Q+}
    &\csp\Norm{\ttau^{-1}\Xi_0 v}{2,\ell,\ttau}
    + \norm{\trace(\Xi_0 v)}{1,\ell+ 1/2,\ttau} 
      \lesssim
      \Norm{Q^{+} \Xi_0 v}{0,\ell,\ttau}
    +\norm{\trace(\Xi_0 v)}{1,\ell+ 1/2,\ttau} 
      + \Norm{v}{2,-M, \ttau},
  \end{align}
for $v \in \S(\Rpb)$, and for $\tau$ and $\csp$ chosen \suff large.

With~\eqref{eq: E0 Q- perfect elliptic} and \eqref{eq: E0 est Q+},  applying now Proposition~\ref{prop: estimate order 4 operator},
and  using that, for any $M \in \N$, 
$[D_{x_d}+ i \htau_{\xi_d}, Q^+] \Xi_1 = (1+\csp \ctp) R_{1,1} \Xi_1 +
R_{1,-M}$ by Lemma~\ref{lemma: commutation Qk+
  Ql-}, we obtain
\begin{align*}
    &\csp \Norm{\ttau^{-1} \Xi_0 v}{4,0,\ttau}
    + \norm{\trace(\Xi_0  v)}{3, 1/2,\ttau} 
    \lesssim
      \Norm{Q^- Q^{+} \Xi_0   v}{+}
    +\norm{\trace(\Xi_0 v)}{1,5/2,\ttau} 
     + \Norm{v}{4,-M, \ttau},
      \notag
  \end{align*}
  for $v \in \S(\Rpb)$, and for $\tau\geq \tauast$ and $\csp\geq 1$ chosen \suff large,
  for $\ctp \in [0,\ctp_1]$ with $\ctp_1>0$ chosen \suff small. Finally,
  with \eqref{eq: decomposition Pvarphi}, we conclude the proof of
  Proposition~\ref{prop: microlocal estimate E0} by choosing $\csp$ large and $\ctp \in [0,\ctp_2]$
  with $\ctp_2>0$ chosen \suff small. 
 \end{proof}
 \begin{remark}
   Note that the end of the proof of Proposition~\ref{prop: microlocal
     estimate E0} is a point where the introduction of the second
   large parameter $\csp$ is crucial. Even in the case $\ctp=0$, that
   is for a weight function that only depend on the variable $z_N$,
   taking $\csp$ large is needed to conclude.
 \end{remark}

\subsection{Microlocal estimate in the region $F$}
\label{sec: Microlocal estimate in the region F}

In the region $F$ we have $\ttau \lesssim \lambsct$ 
and the symbols of the operators $Q_k$ are characterized by two simple
roots that are separated (see the first item of Proposition~\ref{prop root behaviors2}).
We prove the following estimate.
%%%%%%%%%%%%%%%%%%%%%%%%
% proposition          %
%%%%%%%%%%%%%%%%%%%%%%%%
\begin{proposition}
  \label{prop: microlocal estimate F}
  Let $M \in \N$. There exist $\tau_0\geq  \tauast$, $\csp_0\geq 1$, and $C>0$ \st
  \begin{align*}
    %\label{eq: microlocal estimate F}
    \Norm{ \Xi_F v}{4, 0, \ttau} 
    +\norm{\trace(\Xi_F v)}{3,1/2,\ttau} 
    \leq C \Big( \Norm{\Pconj \Xi_F v}{+}
    +\norm{\trace(\Xi_F v)}{1,5/2,\ttau} 
    + \Norm{v}{4,-M, \ttau} 
    \Big), 
  \end{align*}
 for $\tau \geq \tau_0$, $\csp \geq \csp_0$, $\ctp \in [0,1]$,
  and for $v  \in \S(\Rpb)$.
\end{proposition}
\begin{proof}
  We write $\chi_0 = \chi_{F}$ and $\Xi_0 =
  \Xi_F$, to ease the reading of the proof.

  We also let $\chi_1, \uchi \in S(1, \gt)$ be supported in $\MtV$,
  homogeneous of degree zero, and be \st
  $|\htau(\y')| \leq \frac12 \theta_0\lambsct(\y')$ in their support
  (using the notation of Section~\ref{sec: def microlocal regions})
  and \st $\chi_1 \equiv 1$ on a \cnhd of $\supp (\chi_0)$ and
  $\uchi \equiv 1$ on a \cnhd of $\supp (\chi_1)$.  This can be done as
  follows, for instance for the construction of $\chi_{1}$. We
  introduce $V_1\subset V$ an open set of $\R^N$ that is a \nhd of $\supp(\chi_{V_0})$ in
  $\Rpb$, in particular $V_0 \Subset V_1$ (the local geometry is
  illustrated in Figure~\ref{fig: nhd space}) and we choose $\chi_{V_1}
\in \Cinf(\Rpb)$ \st 
\begin{align*}
  &\chi_{V_1} \equiv 1 \ \text{on a \nhd of}\ V_1, \qquad
  \supp(\chi_{V_1}) \subset V.
\end{align*}
We set
  \begin{align*}
    \chi_1(\y')= \chi_{V_1}(z) \chi_{4,F}\in S(1, \gt),
  \end{align*}
   with the function $\chi_{\delta,F}$ as introduced in Section~\ref{sec: def
     microlocal regions}. 
   We have $|\htau(\y')| \leq \frac18 \theta_0\lambsct(\y')$, which
   leaves ``enough room'' for a similar construction for $\uchi$. 
   We set $\Xi_1 = \Opt(\chi_1)$.

 With Proposition~\ref{prop root behaviors2}, 
  in a \cnhd of $\supp(\uchi)$ the roots of $q_k$, $k=1,2$, are simple, and we
  may write  
  \begin{align*}
    q_k (\y) =q_{k,+}(\y) q_{k,-}(\y) , \qquad  q_{k,\pm}(\y) = \xi_d - \rho_{k,\pm}(\y'),  
  \end{align*}
  where $\rho_{k,\pm}  \in S(\lambsct, \gt)$ in a \cnhd of $\supp(\uchi)$
  and there we have 
  \begin{align*}
      \Im \rho_{k,+} \geq C\lambsct, 
      \quad  \Im \rho_{k,-} \leq -C \lambsct .
   \end{align*}
   We set $Q_{k,\pm} = D_{x_d} - \Opt^w(\uchi \rho_{k,\pm})$. 

   In the proof we shall denote by $R_{j,k}$ as a generic operator in
  $\PsiOpsc^{j,k}$, $j \in \mathbb N$, $k \in \R$,  whose expression may
  change from one line to the other. 
   
   %%%%%%%%%%%%%%%%%%%%%%%%
% lemma                %
%%%%%%%%%%%%%%%%%%%%%%%%
\begin{lemma}
  \label{lemma: E0 factorization Qk}
  Let $\Xi = \Xi_0$ or $\Xi_1$.  We have,  for arbitrary large $M \in
  \N$, 
  \begin{align*}
    Q_k \Xi
    &= Q_{k,+}Q_{k,-} \Xi
    + \csp R_{1,0} \Xi + R_{2,-M}\\
    &=Q_{k,-}Q_{k,+} \Xi
    + \csp R_{1,0} \Xi + R_{2,-M}.
\end{align*}
\end{lemma}
%%%% proof of lemma
\begin{proof}
  We have 
  \begin{align*}
    Q_{k,+}Q_{k,-} &= D_{x_d}^2  - \big( \Opt^w(\uchi \rho_{k,+}) 
    + \Opt^w(\uchi \rho_{k,-})\big) D_{x_d} 
    + \Opt^w(\uchi \rho_{k,+}) \Opt^w(\uchi \rho_{k,-}) 
    + \csp R_{0,1}. 
  \end{align*}
  We thus find, for any $M\in \N$, 
  \begin{align*}
    Q_{k,+}Q_{k,-} \Xi 
    &= \Big(\Opt^w(\uchi) D_{x_d}^2  - \big( \Opt^w(\uchi \rho_{k,+}) 
    + \Opt^w(\uchi \rho_{k,-})\big) D_{x_d} 
    + \Opt^w(\uchi \rho_{k,+}) \Opt^w(\uchi \rho_{k,-}) \Big)\Xi\\
    &\quad + \csp R_{1,0}\Xi
     + R_{2,-M} \\
    &=\Op^w(\uchi q_k) \Xi + \csp R_{1,0} \Xi + R_{2,-M}\\
    &= Q_k \Xi + \csp R_{1,0} \Xi + R_{2,-M}.
  \end{align*}
\end{proof}
This result yields, for any $M \in \N$, 
   \begin{align}
     \label{eq: F decomposition Pvarphi}
     \Pconj\Xi_0 &= Q_1 Q_2 \Xi_0 =    Q_1 \Xi_1 Q_2 \Xi_0  + R_{4,-M}\\
     &= Q_{1,-}Q_{1,+}  \Xi_1 Q_{2,-}Q_{2,+} \Xi_0 + \csp R_{3,0}\Xi_0 +
     R_{4,-M}\notag\\
     &= Q_{1,-}Q_{1,+}  Q_{2,-}Q_{2,+} \Xi_0 + \csp R_{3,0}\Xi_0 +
     R_{4,-M}\notag\\
     &= Q^- Q^++ \csp R_{3,0}\Xi_0 +
     R_{4,-M}\notag,
   \end{align}
   where $Q^- = Q_{1,-} Q_{2,-}$ and $Q^+ = Q_{1,+}  Q_{2,+} $.

  Both roots of the
  symbol $q^-$ of the operator $Q^-$ are in the lower half complex
  plane in a \cnhd of $\supp(\chi_0)$. Then, with Lemma~\ref{lemma: microlocal elliptic estimate}
  we have the following perfect elliptic estimate, for any $M >0$, 
  \begin{equation} 
    \label{eq: F est Q-}
    \Norm{\Xi_0 v}{2,0,\ttau}
    +\norm{\trace(\Xi_0 v)}{1,1/2,\ttau}
    \lesssim
    \Norm{Q^- \Xi_0 v}{+}
    + \Norm{v}{2,-M,\ttau},
\end{equation}
for $v\in \S(\Rpb)$, for $\tau\geq \tauast$ and $\csp\geq 1$ chosen \suff large, and
$\ctp \in [0,1]$.

The roots of the first-order factor $Q_{k,+}$, $k=1$ or $2$, 
are in upper half complex plane. Then, with Lemma~\ref{lemma: elliptic
  estimate Qk+}, 
 we have the following elliptic estimate, yet  with a trace
 observation term in the \rhs, for $M >0$ and $\ell \in \R$, 
\begin{align*}
  \Norm{|\Xi_0 v}{1,\ell,\ttau}
  \leq C \Big( \Norm{Q_{k,+} \Xi_0 v}{0,\ell,\ttau}
  + \norm{\trace(\Xi_0 v)}{0,\ell+ 1/2,\ttau} 
  + \Norm{v}{0,-M, \ttau}
  \Big),
  \end{align*}
for $v\in \S(\Rpb)$, for $\tau\geq \tauast$ and $\csp\geq 1$ chosen \suff large, and
$\ctp \in [0,1]$. Then, according to Proposition~\ref{prop: estimate concatenation},
  applied with $\alpha_1 = \alpha_2=0$ and $\delta_1 = \delta_2 =0$, we have the following
  estimates
  for the operator $Q^+$, for $M >0$ and $\ell \in \R$, 
  \begin{align}
    \label{eq: F est Q+}
    &\Norm{\Xi_0 v}{2,\ell,\ttau}
    + \norm{\trace(\Xi_0 v)}{1,\ell+ 1/2,\ttau} 
      \lesssim
      \Norm{Q^{+} \Xi_0 v}{0,\ell,\ttau}
    +\norm{\trace(\Xi_0 v)}{1,\ell+ 1/2,\ttau} 
      + \Norm{v}{2,-M, \ttau},
  \end{align}
for $v \in \S(\Rpb)$, and for $\tau\geq \tauast$ and $\csp\geq 1$ chosen \suff large.

Applying now Proposition~\ref{prop:  estimate order 4 operator},
with~\eqref{eq: F est Q-} and \eqref{eq: F est Q+}, we obtain 
\begin{align*}
    &\Norm{\Xi_0 v}{4,0,\ttau}
    + \norm{\trace(\Xi_0  v)}{3, 1/2,\ttau} 
    \lesssim
      \Norm{Q^- Q^{+} \Xi_0   v}{+}
    +\norm{\trace(\Xi_0 v)}{1,5/2,\ttau} 
     + \Norm{v}{4,-M, \ttau},
      \notag
  \end{align*}
  for $v \in \S(\Rpb)$, and for $\tau\geq \tauast$ and $\csp\geq 1$ chosen \suff large,
  for $\ctp \in [0,1]$. Finally,
  with \eqref{eq: F decomposition Pvarphi}, we conclude the proof of
  Proposition~\ref{prop: microlocal estimate F} by choosing $\tau$ and $\csp$ large. 
\end{proof}

\subsection{Proof of the Carleman estimate of Theorem~\ref{theorem:
    Carleman boundary x}}
\label{sec: Proof of the Carleman estimate boudnary x}
We choose $W$ an open \nhd of $z_0$ in $\R^N$ \st $W \Subset
V_0$ (see Figure~\ref{fig: nhd space}).
Let $u = w_{|Z}$, with
$ w \in\Cinfc((0,S_0) \times \R^d)$ and $\supp(w) \subset W$. We set
$v = e^{\tau \varphi} u$.

We collect the different estimations that we have obtained in
Propositions~\ref{prop: microlocal estimate E-}, \ref{prop: microlocal
  estimate E0}, and \ref{prop: microlocal estimate F}.  
For some $\delta =\delta_0\in (0,1)$ to be kept fixed, for
$\tau_0\geq \tauast$, $\csp_0\geq 1$,
  and $\ctp_0 \in (0,1]$ we have
\begin{align}
    \label{eq: collected estimate E-}
    \csp^\hf
    \Norm{\ttau^{-1/2} \Xi^{(k)}_{\delta,-} v}{4, 0, \ttau} 
    +\norm{\trace(\Xi^{(k)}_{\delta,-} v)}{3,1/2,\ttau} 
    \lesssim \Norm{\Pconj \Xi^{(k)}_{\delta,-} v}{+}
    +\norm{\trace(\Xi^{(k)}_{\delta,-} v)}{0,7/2,\ttau} 
    + \Norm{v}{4,-M, \ttau} , 
  \end{align}
for $k=1,2$, and 
  \begin{align}
    \label{eq: collected estimate E0}
    \csp
    \Norm{\ttau^{-1}\Xi_{\delta,0} v}{4, 0, \ttau} 
    +\norm{\trace(\Xi_{\delta,0} v)}{3,1/2,\ttau} 
    \lesssim \Norm{\Pconj \Xi_{\delta,0}  v}{+}
    +\norm{\trace(\Xi_{\delta,0} v)}{1,5/2,\ttau} 
    + \Norm{v}{4,-M, \ttau} , 
  \end{align}
and
\begin{align}
    \label{eq: collected estimate F}
    \Norm{ \Xi_F v}{4, 0, \ttau} 
    +\norm{\trace(\Xi_F v)}{3,1/2,\ttau} 
    \lesssim\Norm{\Pconj \Xi_F v}{+}
    +\norm{\trace(\Xi_F v)}{1,5/2,\ttau} 
    + \Norm{v}{4,-M, \ttau} , 
  \end{align}
for $\tau \geq \tau_0$, $\csp \geq \csp_0$, $\ctp \in [0,\ctp_0]$. 

We then pick $\alpha>0$ meant to be chosen small in what follows, 
and we shall consider $\alpha \big( \eqref{eq: collected estimate E-} +
\eqref{eq: collected estimate F}\big) + \eqref{eq: collected estimate
  E0}$.
We will choose $\tau$ \suff large so that $\alpha \tau^{1/2} \geq 1$. 

We first note that we have the following lemma whose proof is provided below.
%%%%%%%%%%%%%%%%%%%%%%%%
% lemma                %
%%%%%%%%%%%%%%%%%%%%%%%%
\begin{lemma}
  \label{lemma: gather 1}
  There exists $C>0$ such that
  \begin{align*}
    \alpha \csp^\hf\sum_{k=1,2}
    \Norm{\ttau^{-1/2} \Xi^{(k)}_{\delta,-} v}{4, 0, \ttau} 
    + \csp
    \Norm{\ttau^{-1}\Xi_{\delta,0} v}{4, 0, \ttau} 
    + \alpha \Norm{ \Xi_F v}{4, 0, \ttau} 
    \geq C \csp
    \Norm{\ttau^{-1} v}{4, 0, \ttau},   
  \end{align*}
  for $\tau$ chosen \suff large.
\end{lemma}
With a similar, yet simpler, proof, we have the following lemma.
\begin{lemma}
  \label{lemma: gather 2}
  We have 
  \begin{align*}
    \alpha \sum_{k=1,2}
    \norm{\trace(\Xi^{(k)}_{\delta,-} v)}{3,1/2,\ttau} 
    + \norm{\trace(\Xi_{\delta,0} v)}{3,1/2,\ttau} 
    + \alpha \norm{\trace(\Xi_F v)}{3,1/2,\ttau} 
    \gtrsim \alpha  \norm{\trace(v)}{3,1/2,\ttau}, 
  \end{align*}
  for $\tau$ chosen \suff large.
\end{lemma}

With these two lemmata we obtain 
\begin{multline}
  \label{eq: gather 1}
  \csp \Norm{\ttau^{-1} v}{4, 0, \ttau} + \alpha \Big(\Norm{ \Xi_F v}{4, 0,
    \ttau} + \csp^\hf
    \sum_{k=1,2}\Norm{\ttau^{-1/2} \Xi^{(k)}_{\delta,-} v}{4, 0,
      \ttau} \Big)\\
    + \alpha  \norm{\trace(v)}{3,1/2,\ttau}
  \lesssim \alpha \Big( \text{\rhs}\eqref{eq: collected estimate E-} +
\text{\rhs}\eqref{eq: collected estimate F}\Big) + \text{\rhs}\eqref{eq: collected estimate
  E0}.
\end{multline}

The next lemma is crucial in the computation of the commutator
$[\Pconj, \Xi_{\delta, 0}]$. A proof is given below.
%%%%%%%%%%%%%%%%%%%%%%%%
% lemma                %
%%%%%%%%%%%%%%%%%%%%%%%%
\begin{lemma}
  \label{lemma: gather 3}
  We have $[\Pconj, \Xi_{\delta, 0}] = \Op(g) +  \Op(h) + \csp^2 R_{3,-1}$, where
  $g,h \in \csp \PsiOpsc^{3,0}$ and $R_{3,-1} \in \PsiOpsc^{3,-1}$, with
  \begin{itemize}
    \item $g(\y)= 0$ for $z$ in a \nhd of  $V_0$
    \item $h(\y) = \sum_{j=0}^3 h_j(\y') \xi_d^j$,
      $h_j \in \csp \Psi(\lambsct^{3-j},\gt)$, homogeneous of degree
      $3-j$, and
      $\chi_{\delta,-}^{(1)} +\chi_{\delta,-}^{(2)} + \chi_F \geq 1$
      in a \cnhd of $\supp(h_j)$ in the variables
      $(\zeta, \tau, \csp, \ctp)$, for $z \in V_0$, $j=0, \dots, 3$.
  \end{itemize}
 \end{lemma}

We have  $[\Pconj, \Xi] \in  \csp R_{3,0}$, for $\Xi = \Xi_{\delta,-}^{(1)}$,
$\Xi_{\delta,-}^{(2)}$ or $\Xi_F$. Lemma~\ref{lemma:
  gather 3} gives, for any $M\in \N$,   $\Norm{\Op(g) v}{+} \lesssim
\Norm{v}{3,-M,\ttau}$, and we obtain 
\begin{align*}
  \alpha \sum_{k=1,2}\Norm{\Pconj \Xi^{(k)}_{\delta,-} v}{+}
  + \Norm{\Pconj \Xi_{\delta,0}  v}{+}
  + \alpha\Norm{\Pconj \Xi_F v}{+}
  \lesssim \Norm{\Pconj v}{+}
  + \Norm{\Op(h) v}{+}
  + \alpha \csp \Norm{v}{3,0,\ttau}
  + \csp^2 \Norm{v}{3,-1,\ttau}.
\end{align*}
From \eqref{eq: gather 1} and \eqref{eq: collected estimate
  E-}--\eqref{eq: collected estimate F} we thus obtain, for $\alpha$
chosen \suff small (and kept fixed for the remainder of the proof) 
and $\tau$ chosen \suff large
\begin{multline}
\label{eq: gather 2}
  \csp \Norm{\ttau^{-1} v}{4, 0, \ttau} + \Norm{ \Xi_F v}{4, 0,
    \ttau} + \csp^\hf
    \sum_{k=1,2}\Norm{\ttau^{-1/2} \Xi^{(k)}_{\delta,-} v}{4, 0,
      \ttau} \\
    +  \norm{\trace(v)}{3,1/2,\ttau}
  \lesssim 
  \Norm{\Pconj v}{+}
  + \norm{\trace(v)}{1,5/2,\ttau} 
  + \Norm{\Op(h) v}{+}.
\end{multline}

We set $\chi = \chi_{\delta,-}^{(1)} +\chi_{\delta,-}^{(2)} + \chi_F$.
We have the following lemma whose proof is given below.
%%%%%%%%%%%%%%%%%%%%%%%%
% lemma                %
%%%%%%%%%%%%%%%%%%%%%%%%
\begin{lemma}
  \label{lemma: gather 4}
  Let $W$ be an open set of $\R^N$ with $W \Subset V_0$.
  There exist $C>0$ and $\tau_1\geq \tauast$ \st 
  $\Norm{\Opt(h_j) w}{+} \leq C \csp \big( \Norm{\Opt(\chi)w}{0,3-j, \ttau}
  + \csp (1 + \ctp\csp)\Norm{w}{0, 2-j,\ttau}\big)$, for $w \in \S(\Rpb)$,
  $\supp(w) \subset W$ and $\tau
\geq \tau_1$. 
\end{lemma}
Thus, we obtain 
\begin{align*}
  \Norm{\Op(h) v}{+} &\leq \sum_{j=0}^3 \Norm{\Op(h_j) D_{x_d}^j v}{+} 
  \lesssim \sum_{j=0}^3 \csp \Norm{\Opt(\chi)  D_{x_d}^j v}{0,3-j,\ttau} 
  +  \csp  (1 + \ctp\csp) \Norm{v}{3,-1,\ttau}.
\end{align*}
As $[\Opt(\chi),  D_{x_d}^j] \in \csp \PsiOpsc^{j-1,0}$ we obtain
\begin{align*}
  \Norm{\Op(h) v}{+}
  &\lesssim \csp \Norm{\Opt(\chi) v}{3,0,\ttau} 
    + \csp^2 \Norm{ v}{3,-1,\ttau} \\
  &\lesssim
    \csp \Big( 
    \sum_{k=1,2}\Norm{\Xi^{(k)}_{\delta,-} v}{3, 0,
    \ttau}
    + \Norm{ \Xi_F v}{3, 0,
    \ttau} 
    \Big) 
    + \csp^2 \Norm{ v}{3,-1,\ttau} .
  \end{align*}
  Using this estimate in \eqref{eq: gather 2},  for $\tau$ chosen
  \suff large,  we thus obtain 
  \begin{align*}
  \csp \Norm{\ttau^{-1} v}{4, 0, \ttau} + \Norm{ \Xi_F v}{4, 0,
    \ttau} + \csp^\hf
    \sum_{k=1,2}\Norm{\ttau^{-1/2} \Xi^{(k)}_{\delta,-} v}{4, 0,
      \ttau} 
    +  \norm{\trace(v)}{3,1/2,\ttau}
  \lesssim 
  \Norm{\Pconj v}{+}
  + \norm{\trace(v)}{1,5/2,\ttau}.
\end{align*}
The end of the proof of Theorem~\ref{theorem:
    Carleman boundary x} is then classical.
\hfill \qedsymbol \endproof

\bigskip
%%%% proof of lemma
\begin{proof}[\bfseries Proof of Lemma~\ref{lemma: gather 1}]
  With Lemma~\ref{lemma: different form norm} we may write 
  \begin{align*}
    X&= \alpha \csp^\hf
    \sum_{k=1,2}\Norm{\ttau^{-1/2} \Xi^{(k)}_{\delta,-} v}{4, 0, \ttau} 
    + \csp
    \Norm{\ttau^{-1}\Xi_{\delta,0} v}{4, 0, \ttau} 
    + \alpha \Norm{ \Xi_F v}{4, 0, \ttau} \\
    &\gtrsim
      \sum_{j=0}^4 \Big( 
      \alpha \csp^\hf
    \sum_{k=1,2}\Norm{\ttau^{-1/2}\Lambsct^{4-j} D_{x_d}^j \Xi^{(k)}_{\delta,-} v}{+} 
    + \csp
    \Norm{\ttau^{-1}\Lambsct^{4-j}  D_{x_d}^j \Xi_{\delta,0} v}{+} 
    + \alpha \Norm{\Lambsct^{4-j}   D_{x_d}^j  \Xi_F v}{+}
      \Big)
  \end{align*}
  yielding 
  \begin{align*}
    X \gtrsim
      \csp \sum_{j=0}^4 \Big( 
    \sum_{k=1,2}\Norm{\ttau^{-1}\Lambsct^{4-j} D_{x_d}^j \Xi^{(k)}_{\delta,-} v}{+} 
    + 
    \Norm{\ttau^{-1}\Lambsct^{4-j}  D_{x_d}^j \Xi_{\delta,0} v}{+} 
    + \Norm{\ttau^{-1} \Lambsct^{4-j}   D_{x_d}^j  \Xi_F v}{+}
      \Big),
  \end{align*}
  as $\alpha \geq \alpha  \csp^\hf \ttau^{-1/2} \geq \csp \ttau^{-1}$
  using, on the one hand, that $(\tau \varphi)^{-1/2} = \csp^\hf
  \ttau^{-1/2} \leq 1$ since $\tau\geq \tauast\geq 1$ and $\varphi\geq 1$, and, on
  the other hand, that $\alpha \tau^{1/2}\geq 1$ implies $\alpha
  \ttau^{1/2} = \alpha
  (\tau\csp \varphi)^{1/2}\geq \csp^{1/2}$ since $\varphi\geq 1$. 
  We then find, with $h
  = \chi^{(1)}_{\delta,-} + \chi^{(2)}_{\delta,-} +\chi_{\delta,0} +
  \chi_F \in S(1, \gt)$, 
    $X \gtrsim
      \csp \sum_{j=0}^4 \Norm{\ttau^{-1}\Lambsct^{4-j} D_{x_d}^j
    \Opt(h) v}{+}$. 
  As $[D_{x_d}^j, \Opt(h) ] \in \csp \PsiOpsc^{j-1,0}$, we obtain 
  \begin{align*}
    X  + \csp^2 \Norm{\ttau^{-1} v}{3,0,\ttau}
    \gtrsim
      \csp \sum_{j=0}^4 \Norm{\ttau^{-1}\Lambsct^{4-j} \Opt(h)
      D_{x_d}^j  v}{+} .
  \end{align*}
  By the (local) G{\aa}rding inequality of Proposition~\ref{prop: local Garding}, as $h
  (\y') \geq 1$ in a \nhd of $V_0 \cap \Rpb$ that contains $\supp(v)$, 
  we obtain 
  \begin{align*}
    &X  + \csp^2 \Norm{\ttau^{-1} v}{3,0,\ttau}
      \gtrsim \csp  \sum_{j=0}^4 \Norm{\ttau^{-1}
      D_{x_d}^j  v}{0,4-j,\ttau}
      \asymp \csp \Norm{\ttau^{-1} v}{4, 0, \ttau},   
  \end{align*}
  We conclude by taking $\tau$ \suff large with  the usual
  semi-classical inequality \eqref{eq: usual semi-classical argument}.
\end{proof}

\bigskip
{\em \bfseries Proof of Lemma~\ref{lemma: gather 3}.}
  Up to $\csp^2 \Symbsc^{3,-1}$, the principal symbol of $[\Pconj,
  \Xi_{\delta, 0}]$ is given by $-i \{\pconj, \chi_{\delta,0} \}$,
  and thus involves derivatives of $\chi_{\delta,0}$. We recall the
  form of $\chi_{\delta,0}$, as introduced in Section~\ref{sec: def microlocal regions}, 
  \begin{align*}
    \chi_{\delta,0} (\y') = \chi_{V_0}(z)\ (1-\chi_{1/4,F}(\y'))\ 
    \chi_0(\mu_1(\y') /\delta) \ \chi_0(\mu_2(\y') /\delta).
  \end{align*}
  Computing $-i \{\pconj, \chi_{\delta,0} \}$, we obtain the following
  list of terms.
  \begin{description}
    \item[Terms involving derivatives of $\bld{\chi_{V_0}(z)}$]
      Those terms contribute to the symbol $g$ that vanishes in a
      \nhd of $V_0$.
    \item[Terms involving derivatives of $\bld{\chi_{1/4,F}(\y')}$]
      Those terms are supported in 
      $\{\theta_1 \lambsct /8 \leq \ttau \leq \theta_1 \lambsct /4\}$,
      using the notation of Section~\ref{sec: def microlocal
        regions}. As $\chi_{1, F}= 1$ for $\ttau\leq \theta_1 \lambsct
      /2$, we see that $\chi_F(\y') = \chi_{V_0}(z) \chi_{1, F}
      (\y')=1$ in a \nhd of the support of those terms for $z \in
      V_0$. Those terms contribute to the symbol $h$.
      \item[Terms involving derivatives of $\bld{\chi_0(\mu_k(\y')
          /\delta)}$, $\bld{k=1,2}$] 
        From the definition of $\chi_0$ we see that those terms are
        supported in $\{ -3\leq \mu_k(\y')/\delta \leq -2\}$.
        We have $\chi_-(\mu_k(\y')/\delta)=1$ in a \cnhd of this set. As 
        $\chi_{1,F}(\y') +  (1-\chi_{1/4,F}(\y')) \geq 1$ we find that 
        $\chi_F (\y') + \chi^{(k)}_{\delta,-}(\y') \geq 1$ in the
        support of those terms if $z \in
        V_0$. Those terms contribute to the symbol $h$.
        \hfill \qedsymbol \endproof
  \end{description}

\begin{proof}[\bfseries Proof of Lemma~\ref{lemma: gather 4}]
  Let $\chi_W(z) \in \Cinfc(V_0)$ be \st $\chi_W \equiv 1$ in a
  \nhd of $W$. 
  The microlocal version of the G{\aa}rding inequality of
  Proposition~\ref{prop: microlocal tangential Garding} gives, 
  by Lemma~\ref{lemma: gather 3}, 
  \begin{align*}
    \Re \scp{\Opt(\chi) \Opt(\chi_Wh_j) w}{ \Opt(\chi_W h_j) w}_{+} 
    + \Norm{w}{0,-M,\ttau}^2 
    \gtrsim\Norm{\Opt(\chi_W h_j) w}{+}^2.
  \end{align*}
  Then, with the Young inequality, we obtain 
  \begin{align*}
    \Norm{\Opt(\chi) \Opt(\chi_W h_j) w}{+} + \Norm{w}{0,-M,\ttau} , 
    \gtrsim \Norm{\Opt(\chi_W h_j) w}{+}.
  \end{align*}
  Since $\Opt(\chi_W h_j) w  = \Opt(h_j) w +
  R_{0,-M} w$, with $R_{0,-M}  \in \PsiOpsc^{0,-M}$, for any $M \in \N$, we
  obtain 
  \begin{align*}
    \Norm{\Opt(\chi) \Opt(h_j) w}{+} + \Norm{w}{0,-M,\ttau} , 
    \gtrsim \Norm{\Opt(h_j) w}{+}.
  \end{align*}
  As $[\Opt(\chi),  \Opt(h_j) ] \in \csp (1 + \ctp \csp)
  \Psi(\lambsct^{2-j}, \gt)$, we obtain the sought estimate.
\end{proof}

%%%%%%%%%%%%%%%
%% Section             %
%%%%%%%%%%%%%%%
\section{Spectral inequality and application}
\label{sec: spectral inequality}

We start this section by stating and proving an interpolation type
inequality.  Next, we prove the spectral inequality of
Theorem~\ref{theorem: spectral inequality}. Finally, as an
application, we state a null-controllability result that follows from
it.

\subsection{An interpolation inequality}
Let $S_0>0$ and $\alpha \in (0,S_0/2)$. We recall the notation
$Z=(0,S_0)\times \Omega$ and we introduce
$Y=(\alpha,S_0-\alpha)\times \Omega$ for some $\alpha>0$.  As is done
in other sections, we denote by $z=(s,x) \in Z$, with $s \in (0,S_0)$ and
$x \in \Omega$.  We recall that $P$ denotes the augmented elliptic operator
$P:=D_s^4 + B$, where $B = \Delta_x^2$. 
%%%%%%%%%%%%%%%%%%%%%%%%
% theorem              %
%%%%%%%%%%%%%%%%%%%%%%%%
\begin{theorem}[Interpolation inequality]
  \label{theorem: interpolation inequality}
  Let $\O$ be a nonempty open subset of $\Omega$. There exist $C>0$ and
  $\delta \in (0,1)$ \st for  $u \in H^4(Z)$ that satisfies
  \begin{align*}
    u(s,x)|_{x\in \d \Omega} =0, \ \  \d_\nu u(s,x)|_{x\in \d \Omega}
    =0, \qquad s \in (0,S_0),
  \end{align*}
    we have
  \begin{align}
    %\label{eq: interpolation inequality}
     \label{eq: interpolation boundary inequality}
    \Norm{u}{H^3(Y)} \leq C \Norm{u}{H^3(Z)}^{1-\delta}
    \Big(\Norm{P u}{L^2(Z)} 
      + \sum_{0\leq j \leq 3}\Norm{\d_s^j u_{|s=0}}{H^{3-j}(\O)}
      \Big)^\delta.
  \end{align}  
\end{theorem}

First, we provide a local interpolation estimate in a \nhd of a point of
$\{ 0\} \times \O$.
%%%%%%%%%%%%%%%%%%%%%%%%
% lemma                %
%%%%%%%%%%%%%%%%%%%%%%%%
\begin{lemma}[local interpolation near ${s=0}$]
  \label{lemma: local interpolation boundary s=0}
  Let $x_0\in\O$, there exist $V$ a \nhd of $(0,x_0)$ in $\R\times\R^d$, $C>0$, 
  and $\delta\in(0,1)$ such that  for  $u \in H^4(Z)$
  we have
    \begin{align}
    \label{eq: boundary interpolation inequality s=0}
    \Norm{u}{H^3(V\cap Z)} \leq C \Norm{u}{H^3(Z)}^{1-\delta}
    \Big(
    \Norm{P u}{L^2(Z)} + \sum_{0\leq j \leq 3}\Norm{\d_s^j u_{|s=0}}{H^{3-j}(\O)} \Big)^\delta.
  \end{align}  
\end{lemma}

Second, we provide an interpolation estimate with an interior
observation, that is, we have an estimate away from the boundary
${0}\times \Omega$.
%%%%%%%%%%%%%%%%%%%%%%%%
% proposition       %
%%%%%%%%%%%%%%%%%%%%%%%%
\begin{proposition}[Interpolation with an interior
observation]
  \label{prop: interpolation interior observation}
  Let $\mathscr Z$ be a nonempty open set in  $Z$.  
  There exist $C>0$ and $\delta \in (0,1)$ \st for  $u \in H^4(Z)$ that satisfies
  \begin{align*}
    u(s,x)|_{x\in \d \Omega} =0, \ \  \d_\nu u(s,x)|_{x\in \d \Omega}
    =0, \qquad s \in (0,S_0),
  \end{align*}
    we have
    \begin{align}
    \label{eq: interpolation interior inequality}
    \Norm{u}{H^3(Y)} \leq C \Norm{u}{H^3(Z)}^{1-\delta}
    \Big(\Norm{P u}{L^2(Z)} 
      + \Norm{u}{L^2(\mathscr Z)}  
      \Big)^\delta.
  \end{align}  
\end{proposition}
With these two local interpolation results, whose proofs are given below, we can then write a proof of Theorem~\ref{theorem: interpolation inequality}. 
%%%% proof of theorem
\begin{proof}[\bfseries Proof of Theorem~\ref{theorem: interpolation
    inequality}]
  Introducing $V$ as given in Lemma~\ref{lemma: local
    interpolation boundary s=0}, we 
  let $\mathscr Z$ be an open subset of $V\cap Z$. With  Lemma~\ref{lemma: local
    interpolation boundary s=0} we then have 
 \begin{align}
    \label{eq: boundary interpolation inequality neighborhood s=0}
   \Norm{P u}{L^2(Z)} + \Norm{u}{H^3(\mathscr Z)} 
   \leq C \Norm{u}{H^3(Z)}^{1-\delta}
    \left(\Norm{P u}{L^2(Z)} 
   + \sum_{0\leq j \leq 3}\Norm{\d_s^j u_{|s=0}}{H^{3-j}(\O)} \right)^\delta,
  \end{align}  
  as we can assume that $\Norm{P u}{L^2(Z)} \leq  \Norm{u}{H^3(Z)}$ 
  otherwise estimate~\eqref{eq: interpolation boundary inequality} is trivial.
  Applying Proposition~\ref{prop: interpolation interior observation} we
  have, 
  for some $\delta' \in (0,1)$,
  $$
   \Norm{u}{H^3(Y)} \leq C \Norm{u}{H^3(Z)}^{1-\delta'}
    \left(\Norm{P u}{L^2(Z)} + \Norm{u}{L^2(\mathscr Z)} \right)^{\delta'}.
  $$
  This, with \eqref{eq: boundary interpolation inequality neighborhood
    s=0}, gives \eqref{eq: interpolation boundary inequality} with
  $\delta' \delta$ in place of $\delta$.
\end{proof}

For the proofs of Lemma~\ref{lemma: local interpolation boundary
  s=0}  and Proposition~\ref{prop: interpolation interior
  observation}.  We shall need the following lemma whose proof can be found in
\cite{Robbiano:95}.
%%%%%%%%%%%%%%%%%%%%%%%%
% lemma                %
%%%%%%%%%%%%%%%%%%%%%%%%
\begin{lemma}
  \label{lemma: optimization interpolation}
  Let $A\ge0$, $B\ge0$, and $C\ge0$. We assume that $A\le B$ and that  there exist $\tau_0>0$, $ \mu>0$ and $\nu>0$ such that 
\begin{align}\label{eq: optimization interpolation}
  A\le e^{-\nu\tau}B+e^{\mu\tau}C, \quad \text{for} \  \tau \geq \tau_0.
\end{align}
  Then $A\le KB^{1-\delta}C^{\delta}$, where $K=\max (2, e^{\mu \tau_0})$ and
  $\delta=\nu/(\nu+\mu) \in (0,1)$.
\end{lemma}

%%%% proof of lemma
\begin{proof}[\bfseries Proof of Lemma~\ref{lemma: local interpolation boundary s=0}]
  Let $r>0$ and $z_0=(-r,x_0)$, where $r$ is chosen \suff small to
  have $B \cap \{s=0\}\subset \O$ with $B= B(z_0,4r)$. Let
  $\psi = -|z-z_0|^2$, with $z = (s,x)$. We have
  $\d_{s} \psi(z) \leq -C<0$ in $B$.  We set
  $\varphi(z)=e^{\csp \psi(z)}$. Let $\chi\in\Con_0^\infty(\R^{d+1})$
  be \st $\chi(z)=1$ if $|z-z_0|\le 7r/2$ and $\chi(z)=0$ if
  $ |z-z_0|\ge 15r/4$. We apply the local Carleman estimate of
  Corollary~\ref{cor: Carleman boundary s=0} to $v=\chi u$, and we
  obtain, for $\csp\geq 1$ chosen \suff large (to be kept fixed in what
  follows),
  \begin{align}
    \label{eq: Carleman estimate s=0}
    \sum_{|\mi|\leq 3}\tau^{7/2-|\mi|} \Norm{e^{\tau\varphi}
      D_z^{\mi} v }{L^2(B\cap Z)}
    \lesssim
    \Norm{e^{\tau\varphi} P v}{L^2(Z)}
    + \tau^{1/2} \sum_{j=0}^3 \norm{\trace(e^{\tau \varphi} D_{s}^j
    v\brs)}{0,3-j,\tau}.
  \end{align}
Note that if $\csp$ is fixed we have $\tau \asymp \ttau$. 
In $\{0\}\times\O$, we have $\varphi\le e^{-\csp r^2}$
then
\begin{align}
  \label{eq: estimate boundary terms s=0}
   \tau^{1/2} \sum_{j=0}^3 \norm{\trace(e^{\tau \varphi} D_{s}^j
    v\brs)}{0,3-j,\tau}
  \lesssim 
  e^{C_3\tau}\sum_{j=0}^3 \norm{D_s^j u\brs}{H^{3-j}(\O)},
  \qquad C_3= (1+a) e^{-\csp r^2},
\end{align}
for any $a>0$.
We have $Pv=\chi Pu +[P,\chi]u$. The term $[P,\chi]$ is a differential
operator of order 3 and it is supported in $\{ z\in \R^{d+1};\ 7r/2\leq |z-z_0|\le 15r/4
\}$. On this set, we have $\varphi\leq e^{-\csp (7r/2)^2}$. We thus find
\begin{align}
  \label{eq: Commutateur estimate s=0}
  \Norm{ e^{\tau\varphi}  [P,\chi]u}{L^2(Z)}\lesssim  e^{C_1
  \tau}\Norm{u}{H^3(Z)},
  \qquad C_1= e^{-\csp (7r/2)^2}.
\end{align}
In $Z$, we have $\varphi\leq e^{-\csp r^2} < C_3$;
this implies
\begin{align}\label{eq: estimate P, s=0}
\Norm{ e^{\tau\varphi} \chi Pu}{L^2(Z)}\lesssim e^{C_3  \tau}\Norm{Pu}{L^2(Z)}.
\end{align}
In $\{ z \in \R^{d+1};\ |z-z_0|\le 3r \}$, $\chi\equiv 1$ thus $u=v$, and on this set $\varphi\ge e^{-\csp(3r)^2}$ then we have
\begin{align}
  \label{eq: estimate by below v, s=0}
e^{C_2\tau}\Norm{u}{H^3(B(z_0,3r)\cap Z)}
\lesssim
  \sum_{|\mi|\leq 3}\tau^{7/2-|\mi|} \Norm{e^{\tau\varphi}
      D_z^{|\mi|} v }{L^2(B\cap Z)},
  \qquad C_2=e^{-\csp(3r)^2}.
\end{align}
Remark that $C_1<C_2<C_3$,  for $a>0$ chosen \suff small. Following~\eqref{eq: Carleman estimate s=0}--\eqref{eq: estimate by below v, s=0} we obtain
\begin{align*}
\Norm{u}{H^3(B(z_0, 3r)\cap Z)}
  \lesssim e^{(C_3 -C_2) \tau}
\Big( \Norm{Pu}{L^2(Z)}   
+ \sum_{j=0}^3 \norm{D_s^j u\brs}{H^{3-j}(\O)}
\Big) +e^{-(C_2-C_1)  \tau}\Norm{u}{H^3(Z)}.
\end{align*}
Applying Lemma~\ref{lemma: optimization interpolation}, we obtain the result with $V=B(z_0,3r)$.
\end{proof}

We prove Proposition~\ref{prop: interpolation
    interior observation} by means
of two lemmata. For $\alpha' \in (0,\alpha)$ and $a\in (0,1)$,  we set 
\begin{equation}
  \label{eq: def Yeps}
Y_{\alpha', a}=(\alpha',S_0-\alpha')\times \Omega_{a},
\end{equation}
where $\Omega_{a}=\{ x\in\Omega,\ \dist(x,\d\Omega)>a>0\}$.
%%%%%%%%%%%%%%%%%%%%%%%%
% lemma                %
%%%%%%%%%%%%%%%%%%%%%%%%
\begin{lemma}
  \label{lemma: interior interpolation}
  Let $\mathscr Z$ be a nonempty open set in  $Z$.  
  Let $\alpha' \in (0,\alpha)$ and $a\in (0,1)$. 
   There exist $C>0$ and
  $\delta \in (0,1)$ \st for  $u \in H^4(Z)$,
   \begin{align}
    \label{eq: interior interpolation inequality}
    \Norm{u}{H^3(Y_{\alpha',a})} \leq C \Norm{u}{H^3(Z)}^{1-\delta}
    \left(\Norm{P u}{L^2(Z)} + \Norm{u}{L^2(\mathscr Z)} \right)^\delta.
  \end{align}  
\end{lemma}

%%%%%%%%%%%%%%%%%%%%%%%%
% lemma                %
%%%%%%%%%%%%%%%%%%%%%%%%
\begin{lemma}
  \label{lemma: boundary  x interpolation inequality}
  Let $(s_0,x_0) \in (0,S_0) \times \d\Omega$. There exist $\delta \in
  (0,1)$, $C>0$, $V_0$ a \nhd  of $(s_0,x_0)$, $\alpha' \in (0,\alpha)$, and $a\in (0,1)$ \st
  we have
  \begin{align}
    \label{eq: boundary x interpolation inequality}
    \Norm{u}{H^3(V_0\cap Z)} \leq C \Norm{u}{H^3(Z)}^{1-\delta}
    \left(\Norm{P u}{L^2(Z)} + \Norm{u}{H^1(Y_{\alpha',a})} \right)^\delta,
  \end{align}  
   for  $u \in H^4(Z)$ satisfying
   \begin{align*}
    u(s,x)|_{x\in \d \Omega} =0, \ \  \d_\nu u(s,x)|_{x\in \d \Omega}
    =0, \qquad s \in (0,S_0).
  \end{align*}
\end{lemma}

%%%% proof of proposition
\begin{proof}[\bfseries Proof of Proposition~\ref{prop: interpolation
    interior observation}]
   We can assume that $\Norm{P u}{L^2(Z)}\leq \Norm{u}{H^3(Z)}$,
   otherwise inequality~\eqref{eq: interpolation interior inequality} is
   obvious. In particular, if \eqref{eq: interpolation interior inequality} holds
   for a value $\delta=\delta_0>0$ the estimate also holds for all
   $\delta\in [0,\delta_0]$ possibly with a larger constant $C= C_\delta$. The same
   observation can  be made for the estimations \eqref{eq: interior interpolation inequality}
   and \eqref{eq: boundary x interpolation inequality}.

  \medskip
  With a  compactness
  argument we can find a finite number of open sets $V_j$, $j \in J$,
  where estimate~\eqref{eq: boundary x interpolation inequality} holds for
  some values $\delta = \delta_j \in (0,1)$, $\alpha'_j \in
  (0,\alpha)$, and $a_j \in (0,1)$, and such that
  $$(\alpha,S_0-\alpha)\times \d\Omega\subset \cup_{j \in J} V_j.$$ 
  For $a \in (0,1)$ and $\alpha' \in (0,\alpha)$, 
   set $\tilde Y_{\alpha', a}=(\alpha',S_0-\alpha')\times
  \tilde\Omega_a$, where $\tilde\Omega_a=\{ x\in\Omega,\
  \dist (x,\d\Omega)<a\}$. There exists $a_1\in (0,1)$ and
  $\alpha_1 \in (0,\alpha)$ \st $\tilde
  Y_{\alpha_1,a_1}\subset Z\cap (\cup_{j\in J} V_j)$. Applying the local
  interpolation estimate~\eqref{eq: boundary x
    interpolation inequality} for each $V_j$, using now  
  \begin{equation*}
    \delta_1 = \min_{j \in J} \delta_j\in (0,1), \quad \alpha_2 =  \min_{j \in J}
  \alpha'_j \in (0, \alpha), \quad \text{and}\ \ a_2 = \min_{j \in J}
  a_j \in (0,1)
  \end{equation*}
  (note that the set $Y_{\alpha',a}$ increases as
  $\alpha'$ and $a$ decrease) we obtain
  \begin{equation}
    \label{eq: pre global interpolation 1}
    \Norm{u}{H^3(\tilde Y_{\alpha_1, a_1})} \lesssim  \Norm{u}{H^3(Z)}^{1-\delta_1}
    \left(\Norm{P u}{L^2(Z)} + \Norm{u}{H^3(Y_{\alpha_2, a_2})} \right)^{\delta_1}.
  \end{equation}

Let  $\mathscr Z$ be a nonempty open set in  $Z$.  
By Lemma~\ref{lemma: interior interpolation} we obtain, for some
$\delta_2 \in (0,1)$, 
\begin{equation}
    \label{eq: pre global interpolation 2}
\Norm{P u}{L^2(Z)} + \Norm{u}{H^3(Y_{\alpha_2, a_2})} \lesssim \Norm{u}{H^3(Z)}^{1-\delta_2}
    \left(\Norm{P u}{L^2(Z)} + \Norm{u}{L^2(\mathscr Z)} \right)^{\delta_2},
\end{equation}
as the estimate of $\Norm{P u}{L^2(Z)} $ is clear here.
 Then,
estimates~\eqref{eq: pre global interpolation 1} and \eqref{eq: pre global interpolation 2} give
\begin{align}\label{eq: semi-global boundary interpolation inequality}
\Norm{u}{H^3(\tilde Y_{\alpha_1, a_1})} \lesssim  \Norm{u}{H^3(Z)}^{1-\delta_1\delta_2}
    \left(\Norm{P u}{L^2(Z)} + \Norm{u}{L^2(\mathscr Z)} \right)^{\delta_1\delta_2}.
\end{align}
Taking $a\in (0,a_1)$ and $\alpha' \in (0,\alpha)$, we have $Y
\subset Y_{\alpha',a}\cup\tilde Y_{\alpha_1, a_1}$,  and,   by 
\eqref{eq: interior interpolation inequality} in Lemma~\ref{lemma: interior interpolation} and \eqref{eq: semi-global boundary interpolation inequality},
we obtain \eqref{eq: interpolation interior inequality}.
\end{proof}
%%%% proof of lemma
\begin{proof}[\bfseries Proof of Lemma~\ref{lemma: interior interpolation}]
  By a compactness argument, it suffices to prove~\eqref{eq: interior
    interpolation inequality} with $B(z,R)$ in place of $Y_{\alpha',a}$ where
  $z\in\ovl{Y_{\alpha',a}}$ and $0< R \leq  \min(\alpha', a) /2$, implying
  $B(z,R)\subset Z$. Let $z^{(0)}$ be in $\mathscr Z$ and $r_0>0$
  \st $B(z^{(0)},r_0)\Subset \mathscr Z$. As $Y_{\alpha',a}$ is connected, there exists a path
  $\Gamma \subset Y_{\alpha',a} $ from $z^{(0)}=\Gamma(0)$ to $z=\Gamma(1)$. Set
  $r_1=\dist(\Gamma, \d Z)$. We have $r_1 >0$ by
  compactness. 

  \medskip
  Setting now 
  $r=\inf(R, r_0,r_1/4)$, we define a sequence $(z^{(j)})_j$, for $j\ge0$, by
  $z^{(j)}=\Gamma(t_j)$ where $t_0=0$ and 
  \begin{align*}
    t_j= \begin{cases}
      \inf  A_j & \text{if}\  A_j \neq \emptyset,\\
      1 & \text{if}\  A_j =\emptyset,
      \end{cases}\qquad 
      A_j = \{ \sigma\in(t_{j-1},1];\ \Gamma(\sigma)\not\in
      B(z^{j-1},r)\}.
    \end{align*}      
    The sequence $(z^{(j)})_j$ is finite by a compactness argument.
    The construction of the sequence is illustrated in Figure~\ref{fig:
      sequence construction}.
    %%%%%%%%%%%%%%%%%%%%%%%%
% figure               %
%%%%%%%%%%%%%%%%%%%%%%%%
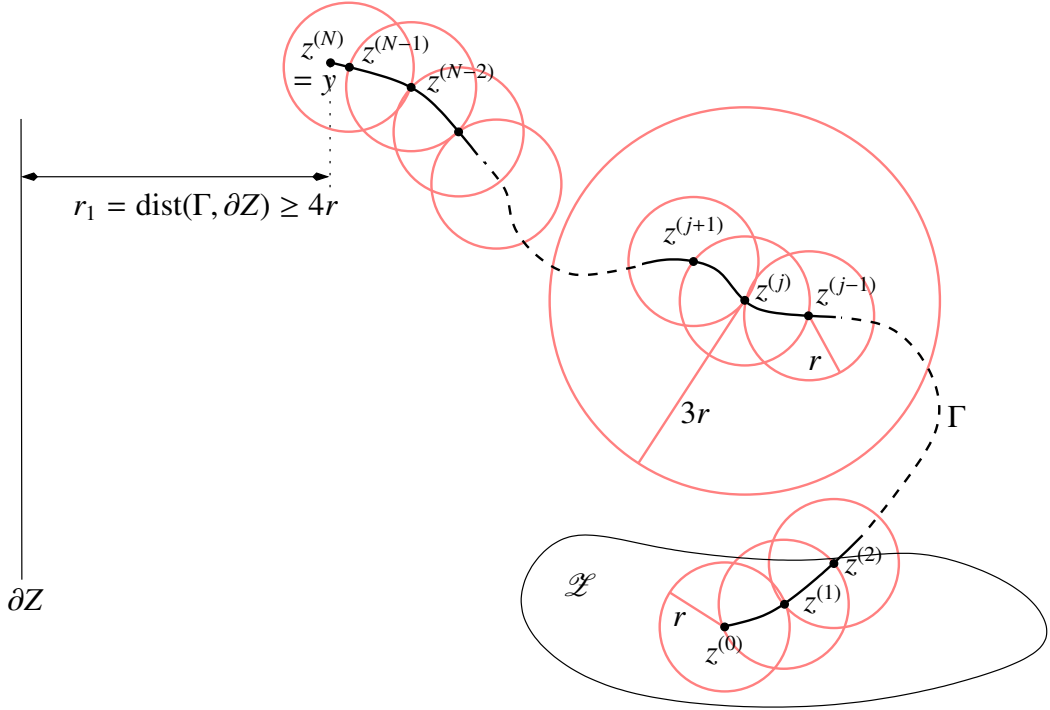
\begin{figure}
  \begin{center}
    \begin{picture}(0,0)%
\includegraphics{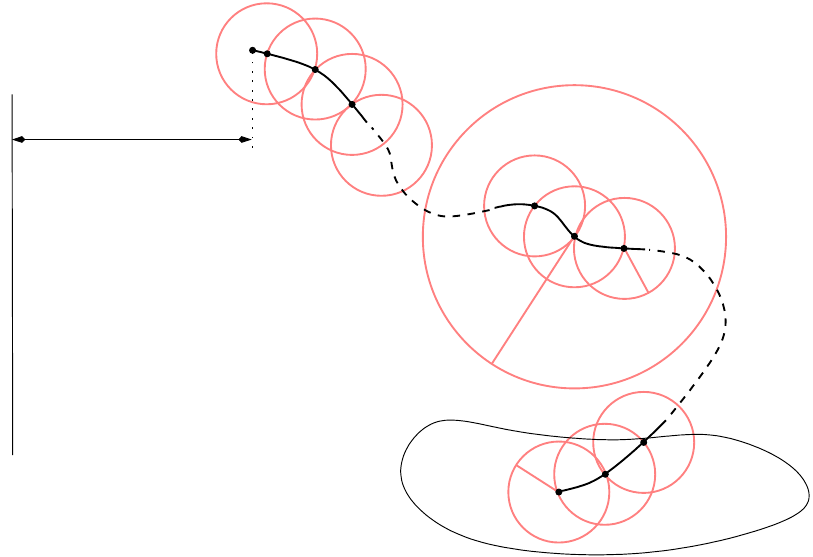}%
\end{picture}%
\setlength{\unitlength}{3947sp}%
\begingroup\makeatletter\ifx\SetFigFont\undefined%
\gdef\SetFigFont#1#2#3#4#5{%
  \reset@font\fontsize{#1}{#2pt}%
  \fontfamily{#3}\fontseries{#4}\fontshape{#5}%
  \selectfont}%
\fi\endgroup%
\begin{picture}(6484,4438)(-1140,-3963)
\put(3055,-2171){\makebox(0,0)[lb]{\smash{{\SetFigFont{12}{14.4}{\rmdefault}{\mddefault}{\updefault}{\color[rgb]{0,0,0}$3 r$}%
}}}}
\put(3823,-3349){\makebox(0,0)[lb]{\smash{{\SetFigFont{12}{14.4}{\rmdefault}{\mddefault}{\updefault}{\color[rgb]{0,0,0}$z^{(1)}$}%
}}}}
\put(3900,-1418){\makebox(0,0)[lb]{\smash{{\SetFigFont{12}{14.4}{\rmdefault}{\mddefault}{\updefault}{\color[rgb]{0,0,0}$z^{(j-1)}$}%
}}}}
\put(3008,-3443){\makebox(0,0)[lb]{\smash{{\SetFigFont{12}{14.4}{\rmdefault}{\mddefault}{\updefault}{\color[rgb]{0,0,0}$r$}%
}}}}
\put(3522,-1396){\makebox(0,0)[lb]{\smash{{\SetFigFont{12}{14.4}{\rmdefault}{\mddefault}{\updefault}{\color[rgb]{0,0,0}$z^{(j)}$}%
}}}}
\put(2950,-1041){\makebox(0,0)[lb]{\smash{{\SetFigFont{12}{14.4}{\rmdefault}{\mddefault}{\updefault}{\color[rgb]{0,0,0}$z^{(j+1)}$}%
}}}}
\put(1481,-84){\makebox(0,0)[lb]{\smash{{\SetFigFont{12}{14.4}{\rmdefault}{\mddefault}{\updefault}{\color[rgb]{0,0,0}$z^{(N-2)}$}%
}}}}
\put(1077, 94){\makebox(0,0)[lb]{\smash{{\SetFigFont{12}{14.4}{\rmdefault}{\mddefault}{\updefault}{\color[rgb]{0,0,0}$z^{(N-1)}$}%
}}}}
\put(2322,-3258){\makebox(0,0)[lb]{\smash{{\SetFigFont{12}{14.4}{\rmdefault}{\mddefault}{\updefault}{\color[rgb]{0,0,0}$\mathscr Z$}%
}}}}
\put(-1125,-3345){\makebox(0,0)[lb]{\smash{{\SetFigFont{12}{14.4}{\rmdefault}{\mddefault}{\updefault}{\color[rgb]{0,0,0}$\d Z$}%
}}}}
\put(4718,-2196){\makebox(0,0)[lb]{\smash{{\SetFigFont{12}{14.4}{\rmdefault}{\mddefault}{\updefault}{\color[rgb]{0,0,0}$\Gamma$}%
}}}}
\put(-720,-871){\makebox(0,0)[lb]{\smash{{\SetFigFont{12}{14.4}{\rmdefault}{\mddefault}{\updefault}{\color[rgb]{0,0,0}$r_1=\dist(\Gamma, \d Z) \geq 4 r$}%
}}}}
\put(4077,-3109){\makebox(0,0)[lb]{\smash{{\SetFigFont{12}{14.4}{\rmdefault}{\mddefault}{\updefault}{\color[rgb]{0,0,0}$z^{(2)}$}%
}}}}
\put(697,123){\makebox(0,0)[lb]{\smash{{\SetFigFont{12}{14.4}{\rmdefault}{\mddefault}{\updefault}{\color[rgb]{0,0,0}$z^{(N)}$}%
}}}}
\put(647,-79){\makebox(0,0)[lb]{\smash{{\SetFigFont{12}{14.4}{\rmdefault}{\mddefault}{\updefault}{\color[rgb]{0,0,0}$=y$}%
}}}}
\put(3211,-3669){\makebox(0,0)[lb]{\smash{{\SetFigFont{12}{14.4}{\rmdefault}{\mddefault}{\updefault}{\color[rgb]{0,0,0}$z^{(0)}$}%
}}}}
\put(3849,-1840){\makebox(0,0)[lb]{\smash{{\SetFigFont{12}{14.4}{\rmdefault}{\mddefault}{\updefault}{\color[rgb]{0,0,0}$r$}%
}}}}
\end{picture}%

    \caption{Construction of the sequence $(z^{(j)})_j$, $j \in J$,  along the path $\Gamma$.}
  \label{fig: sequence construction}
  \end{center}
\end{figure}

    \medskip
    Let $(z^{(0)}, \cdots,z^{(N)})$ be such a sequence with
  $z^{(N)}=z$. Note that we have $B(z^{(j+1)},r)\subset B(z^{(j)}, 3r) \subset Z$, for
  $j=0,\cdots, N-1$, because of the choice we made for $r$ above. Now we claim that there exists $C>0$ and $\delta
  \in (0,1)$ \st
  \begin{align}\label{eq: claim interpolation}
    \Norm{u}{H^3(B(z^{(j+1)},r))}\le \Norm{u}{H^3(B(z^{(j)},3r))} 
    \leq C  \Norm{u}{H^3(Z)}^{1-\delta}
    \Big(\Norm{P u}{L^2(Z)} + \Norm{u}{H^3(B(z^{(j)},r))} \Big)^{\delta},
\end{align}
for $j=0,\dots, N-1$. This claim is proven below.

\medskip We assume that $\Norm{P u}{L^2(Z)}\le \Norm{u}{H^3(Z)}$,
since otherwise the estimate we wish to prove is obvious. We then have
$$
\Norm{P u}{L^2(Z)}+\Norm{u}{H^3(B(z^{(j+1)},r))} \lesssim \Norm{u}{H^3(Z)}^{1-\delta}
    \left(\Norm{P u}{L^2(Z)} + \Norm{u}{H^3(B(z^{(j)},r))} \right)^{\delta}.
$$
By induction on $j$, we find
\begin{align}\label{eq: interpolation H1}
\Norm{P u}{L^2(Z)}+\Norm{u}{H^3(B(z,r))} \lesssim  \Norm{u}{H^3(Z)}^{1-\mu}
    \Big(\Norm{P u}{L^2(Z)} + \Norm{u}{H^3(B(z^{(0)},r))} \Big)^{\mu},
\end{align}
where $\mu=\delta^{N}$.

\medskip
As $P$ is elliptic, and $B(z^{(0)},r) \Subset \mathscr Z$ we have
 $\Norm{u}{H^3(B(z^{(0)},r))}\lesssim \Norm{P
  u}{L^2(Z)} +\Norm{u}{L^2(\mathscr Z)} $. 
This estimate and \eqref{eq: interpolation H1} give \eqref{eq: interior interpolation inequality}.

\bigskip To prove estimation \eqref{eq: claim interpolation} we apply
the local Carleman estimate of Proposition~\ref{prop: Carleman P
  interior}.  We set $\psi(z) = - |z-z^{(j)}|^2$ and
$\varphi(z)=e^{\csp \psi(z)}$ and $\chi \in \Con_c^\infty(B(z^{(j)},4r))$
to be such that
\begin{equation*}
 \chi(z)=
 \begin{cases}
   1 & \text{if}\  3r/4 <|z-z^{(j)}|<7r/2,\\
   0 & \text{if}\ |z-z^{(j)}|<5r/8 \ \text{or} \ 15 r/4 <  |z-z^{(j)}|.
  \end{cases}
\end{equation*}
The function $v=\chi u$ is supported in the open set
$B(z^{(j)},4r)\setminus B(z^{(j)},r/2) \subset Z$ where $d \psi$ does not vanish. 
For $\csp\geq 1$ chosen \suff large, by Proposition~\ref{prop: Carleman P
  interior}, we have
\begin{align} 
  \label{eq: carleman interpolation}
  \sum_{|\mi|\leq 4} \tau^{3 - |\mi|}  \Norm{e^{\tau \varphi} D_z^\mi v}{L^2(Z)}
    \lesssim  \Norm{e^{\tau \varphi} P v}{L^2(Z)}.
\end{align}
We have $Pv=\chi Pu+[P,\chi]u$ and $[P,\chi]$ is a differential
operator of order 3 supported in $A_1 \cup A_2$ with
\begin{equation*}
A_1 = \{ z; \ 5r/8\leq   |z-z^{(j)}|\leq 3r/4\}, \qquad A_2 = \{ z; \ 7r/2 \leq   |z-z^{(j)}|\leq 15r/4\}.
\end{equation*}
We write 
\begin{align*}
  \Norm{e^{\tau \varphi} P v}{L^2(Z)} 
  \leq \Norm{e^{\tau \varphi} P u}{L^2(B(z^{(j)},4r))} 
  + \Norm{e^{\tau \varphi} [P, \chi] u}{L^2(A_1 \cup A_2)}.
\end{align*}
Since $\varphi $ decreases as $ |z-z^{(j)}|$ increases, we find
\begin{align} \label{eq: estimate P interpolation}
 \Norm{e^{\tau \varphi} P v}{L^2(Z)}
 \lesssim e^{\tau C_3} \Norm{ P u}{L^2(Z)}
 + e^{\tau C_3} \Norm{u}{H^3(B(z^{(j)},r))}
 +e^{\tau C_1 } \Norm{u}{H^3( Z)},
\end{align}
where $C_1= e^{-\csp (7r/2)^2} $ and $C_3=e^{-\csp (5r/8)^2}$.

As we have $\chi\equiv 1$ 
on $B(z^{(j)},3r)\setminus B(z^{(j)},r)$ we  have
\begin{align}
  \label{eq: estimate u H1 interpolation}
  e^{\tau C_2}\Norm{u}{H^3(B(z^{(j)},3r)\setminus B(z^{(j)},r))} 
  \leq  \sum_{|\mi|\leq 4} \tau^{3 - |\mi|}  \Norm{e^{\tau \varphi} D_z^\mi v}{L^2(Z)},
\end{align}
where $C_2=e^{-\csp (3r)^2}$. Remark that $C_1<C_2<C_3$.

Inequalities~\eqref{eq: carleman interpolation}, \eqref{eq: estimate P
  interpolation}, and \eqref{eq: estimate u H1 interpolation} give
\begin{align*}
\Norm{u}{H^3(B(z^{(j)},3r))} \lesssim e^{\tau (C_3-C_2)}( \Norm{ P u}{L^2(Z)}+\Norm{u}{H^3(B(z^{(j)},r))})
+e^{-\tau (C_2-C_1) } \Norm{u}{H^3( Z)}.
\end{align*}
as
the estimate on $B(z^{(j)},r)$ is clear with such a \rhs if $\tau\geq
\tauast\geq 1$. 
We can optimize this last  estimate applying  Lemma~\ref{lemma:
  optimization interpolation}, which yields~\eqref{eq: claim
  interpolation}, and concludes the proof of Lemma~\ref{lemma: interior interpolation}.
\end{proof}

 %%%%%%%%%%%%%%%%%%%%%%%% 
  % figure               %
  %%%%%%%%%%%%%%%%%%%%%%%% 
  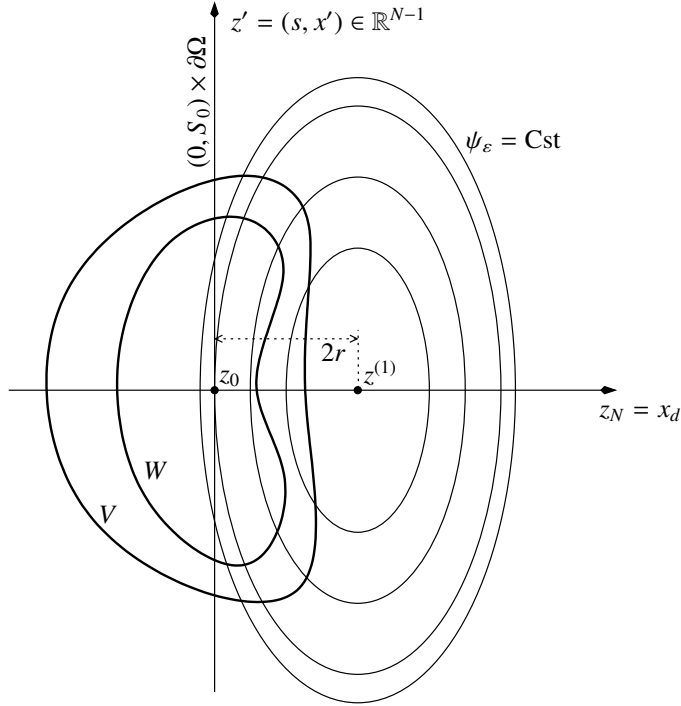
\begin{figure}
  \begin{center}
    \begin{picture}(0,0)%
\includegraphics{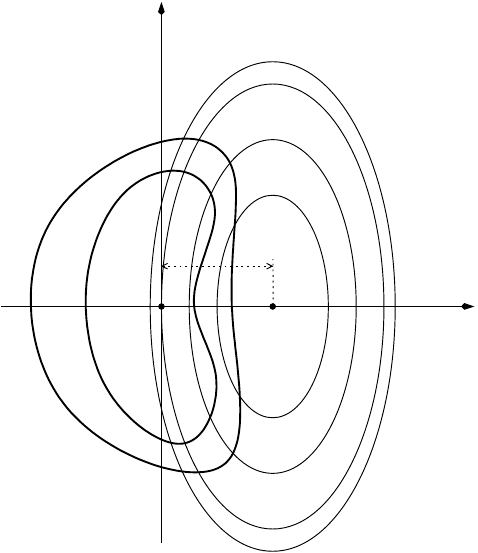}%
\end{picture}%
\setlength{\unitlength}{3947sp}%
\begingroup\makeatletter\ifx\SetFigFont\undefined%
\gdef\SetFigFont#1#2#3#4#5{%
  \reset@font\fontsize{#1}{#2pt}%
  \fontfamily{#3}\fontseries{#4}\fontshape{#5}%
  \selectfont}%
\fi\endgroup%
\begin{picture}(3807,4415)(-986,-3323)
\put(1233,-1311){\makebox(0,0)[lb]{\smash{{\SetFigFont{10}{12.0}{\rmdefault}{\mddefault}{\updefault}{\color[rgb]{0,0,0}$z^{(1)}$}%
}}}}
\put(2698,-1524){\makebox(0,0)[lb]{\smash{{\SetFigFont{10}{12.0}{\rmdefault}{\mddefault}{\updefault}{\color[rgb]{0,0,0}$z_{N}=x_d$}%
}}}}
\put(-417,-2193){\makebox(0,0)[lb]{\smash{{\SetFigFont{10}{12.0}{\rmdefault}{\mddefault}{\updefault}{\color[rgb]{0,0,0}$V$}%
}}}}
\put(-139,-1914){\makebox(0,0)[lb]{\smash{{\SetFigFont{10}{12.0}{\rmdefault}{\mddefault}{\updefault}{\color[rgb]{0,0,0}$W$}%
}}}}
\put(250, 34){\rotatebox{90.0}{\makebox(0,0)[lb]{\smash{{\SetFigFont{10}{12.0}{\rmdefault}{\mddefault}{\updefault}{\color[rgb]{0,0,0}$(0,S_0) \times \d \Omega$}%
}}}}}
\put(1870,157){\makebox(0,0)[lb]{\smash{{\SetFigFont{10}{12.0}{\rmdefault}{\mddefault}{\updefault}{\color[rgb]{0,0,0}$\cpsi= \Cst$}%
}}}}
\put(409,901){\makebox(0,0)[lb]{\smash{{\SetFigFont{10}{12.0}{\rmdefault}{\mddefault}{\updefault}{\color[rgb]{0,0,0}$z'=(s,x')\in \R^{N-1}$}%
}}}}
\put(967,-1176){\makebox(0,0)[lb]{\smash{{\SetFigFont{10}{12.0}{\rmdefault}{\mddefault}{\updefault}{\color[rgb]{0,0,0}$2r$}%
}}}}
\put(336,-1308){\makebox(0,0)[lb]{\smash{{\SetFigFont{10}{12.0}{\rmdefault}{\mddefault}{\updefault}{\color[rgb]{0,0,0}$z_0$}%
}}}}
\end{picture}%

    \caption{Geometry near the boundary for the application of the
      local Carleman estimate of Theorem~\ref{theorem: Carleman boundary x}.}
  \label{fig: local interpolation boundary x -1}
  \end{center}
\end{figure}

%%%% proof of lemma
\begin{proof}[\bfseries Proof of Lemma~\ref{lemma: boundary x
    interpolation inequality}]
 
  The proof follows the same ideas as that of estimate~\eqref{eq: claim
    interpolation} applying the boundary-type local Carleman estimate of Theorem~\ref{theorem: Carleman boundary x}. We use local
  coordinates in a bounded \nhd $V$ in $\R^N$ of the
  point $z_0= (s_0, x_0)$ of $(0,S_0)\times
  \d\Omega$ as introduced in Section~\ref{sec: local setting}, such that this part of the boundary  is locally given by
  $\{ z_{N} = x_d=0\} $ and $Z$ is
  locally given by $\{ z_{N}>0\} $; coordinates can be chosen to have
  moreover  $z_0=
  (z_0',0)$, with $z_0'=0$. We set  $z^{(1)}=(0,2r)$ where $r >0$. 

  We let $\psi \in \Cinf(\R^N)$ be such that 
  \begin{align*}
    \psi(z) =  \begin{cases}
      12 r^2- |z-z^{(1)}|^2 & \text{if}\ |z-z^{(1)}|\leq 3 r, \\
      r^2 & \text{if}\ 4r \leq |z-z^{(1)}|. 
      \end{cases}
  \end{align*}
  We have $\psi(z) \geq r^2>0$, $\Norm{\psi^{(k)}}{L^\infty} < \infty$,  $k
    \in \N$, and 
 \begin{align*}
   \d_\nu  \psi(z)=-\d_{z_{N}}\psi(z)=2 (z_{N}-2r) \leq -C <0,
 \end{align*}
 for $|z-z^{(1)}|\leq 3 r$ and $z_N=0$. 
 Upon reducing the open \nhd $V$, the weight function $\psi$ fulfills
 the requirements listed in \eqref{eq: cond psi} and \eqref{eq: cond psi2}. 
  
  We set $\varphi(z)=e^{\csp \cpsi(z)}$, where $\cpsi (z) = \psi(\ctp
  z '\! ,z_N)$.  
According to
  Theorem~\ref{theorem: Carleman boundary x}, there exist a \nhd
  $W \Subset V$ in $R^N$ of $z_0$, $\tau_0\geq \tauast$, $\csp_0\geq 1$, and
  $\ctp_0 \in (0,1]$ so that the Carleman estimate~\eqref{eq: Carleman
    boundary x} holds for $\tau \geq \tau_0$, $\csp \geq \csp_0$,
  $\ctp \in (0,\ctp_0]$ and smooth functions supported in $W$.  We set
  $\csp=\csp_0$ and $\ctp=\ctp_0$. The geometry of the level sets of
  the weight function is illustrated in Figure~\ref{fig: local
    interpolation boundary x -1}. 
  %%%%%%%%%%%%%%%%%%%%%%%% 
  % figure               %
  %%%%%%%%%%%%%%%%%%%%%%%% 
  \begin{figure}[t]
  \begin{center}
    \begin{picture}(0,0)%
\includegraphics{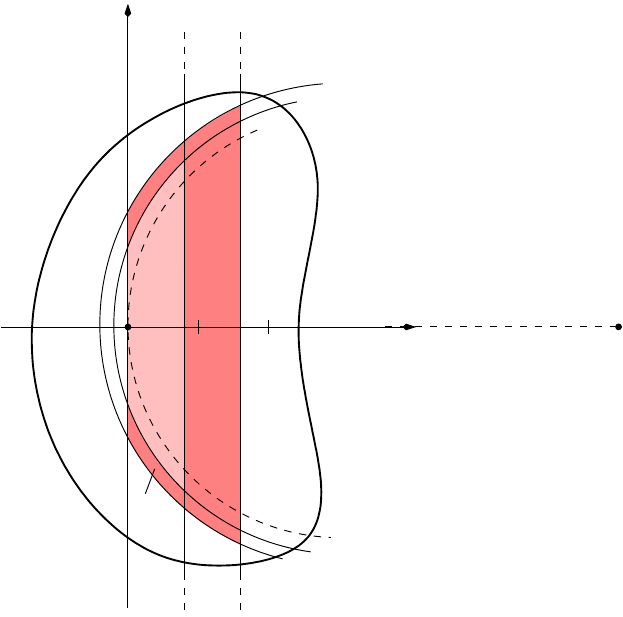}%
\end{picture}%
\setlength{\unitlength}{3947sp}%
\begingroup\makeatletter\ifx\SetFigFont\undefined%
\gdef\SetFigFont#1#2#3#4#5{%
  \reset@font\fontsize{#1}{#2pt}%
  \fontfamily{#3}\fontseries{#4}\fontshape{#5}%
  \selectfont}%
\fi\endgroup%
\begin{picture}(4980,4927)(-611,-3634)
\put(2662,-1542){\makebox(0,0)[lb]{\smash{{\SetFigFont{10}{12.0}{\rmdefault}{\mddefault}{\updefault}{\color[rgb]{0,0,0}$z_N=x_{d}$}%
}}}}
\put( 19,-2553){\makebox(0,0)[lb]{\smash{{\SetFigFont{10}{12.0}{\rmdefault}{\mddefault}{\updefault}{\color[rgb]{0,0,0}$W$}%
}}}}
\put(1026,-1257){\rotatebox{90.0}{\makebox(0,0)[lb]{\smash{{\SetFigFont{10}{12.0}{\rmdefault}{\mddefault}{\updefault}{\color[rgb]{0,0,0}$r/4$}%
}}}}}
\put(1588,-1248){\rotatebox{90.0}{\makebox(0,0)[lb]{\smash{{\SetFigFont{10}{12.0}{\rmdefault}{\mddefault}{\updefault}{\color[rgb]{0,0,0}$r/2$}%
}}}}}
\put(470,1158){\makebox(0,0)[lb]{\smash{{\SetFigFont{10}{12.0}{\rmdefault}{\mddefault}{\updefault}{\color[rgb]{0,0,0}$(s,x') \in \R^{N-1}$}%
}}}}
\put(816,-1264){\rotatebox{90.0}{\makebox(0,0)[lb]{\smash{{\SetFigFont{10}{12.0}{\rmdefault}{\mddefault}{\updefault}{\color[rgb]{0,0,0}$r_0$}%
}}}}}
\put(1274,-1269){\rotatebox{90.0}{\makebox(0,0)[lb]{\smash{{\SetFigFont{10}{12.0}{\rmdefault}{\mddefault}{\updefault}{\color[rgb]{0,0,0}$2r_0$}%
}}}}}
\put(443,-1274){\makebox(0,0)[lb]{\smash{{\SetFigFont{10}{12.0}{\rmdefault}{\mddefault}{\updefault}{\color[rgb]{0,0,0}$z_0$}%
}}}}
\put(2071,-3017){\makebox(0,0)[lb]{\smash{{\SetFigFont{10}{12.0}{\rmdefault}{\mddefault}{\updefault}{\color[rgb]{0,0,0}$\cnorm{z-z^{(1)}} = 2r$}%
}}}}
\put(1777,429){\makebox(0,0)[lb]{\smash{{\SetFigFont{10}{12.0}{\rmdefault}{\mddefault}{\updefault}{\color[rgb]{0,0,0}$\cnorm{z-z^{(1)}} = r_1$}%
}}}}
\put(1988,629){\makebox(0,0)[lb]{\smash{{\SetFigFont{10}{12.0}{\rmdefault}{\mddefault}{\updefault}{\color[rgb]{0,0,0}$\cnorm{z-z^{(1)}} = r_1'$}%
}}}}
\put(1026,-2118){\makebox(0,0)[lb]{\smash{{\SetFigFont{10}{12.0}{\rmdefault}{\mddefault}{\updefault}{\color[rgb]{0,0,0}$A_1$}%
}}}}
\put(475,-2800){\makebox(0,0)[lb]{\smash{{\SetFigFont{10}{12.0}{\rmdefault}{\mddefault}{\updefault}{\color[rgb]{0,0,0}$A_2$}%
}}}}
\put(357,315){\rotatebox{90.0}{\makebox(0,0)[lb]{\smash{{\SetFigFont{10}{12.0}{\rmdefault}{\mddefault}{\updefault}{\color[rgb]{0,0,0}$(0,S_0) \times \d \Omega$}%
}}}}}
\put(4348,-1269){\makebox(0,0)[lb]{\smash{{\SetFigFont{10}{12.0}{\rmdefault}{\mddefault}{\updefault}{\color[rgb]{0,0,0}$z^{(1)}$}%
}}}}
\end{picture}%

    \caption{Geometry near the boundary for the derivation of the
      local interpolation inequality. The light color region shows
      where $\chi\equiv 1$; the dark color region shows where $\chi$
      varies. Note that the relative scale of the two axes has been
      modified, if compared to Figure~\ref{fig: local interpolation
        boundary x -1}, for a better display of
      the regions $A_1$ and $A_2$ near $z_0$.}
  \label{fig: local interpolation boundary x -2}
  \end{center}
\end{figure}
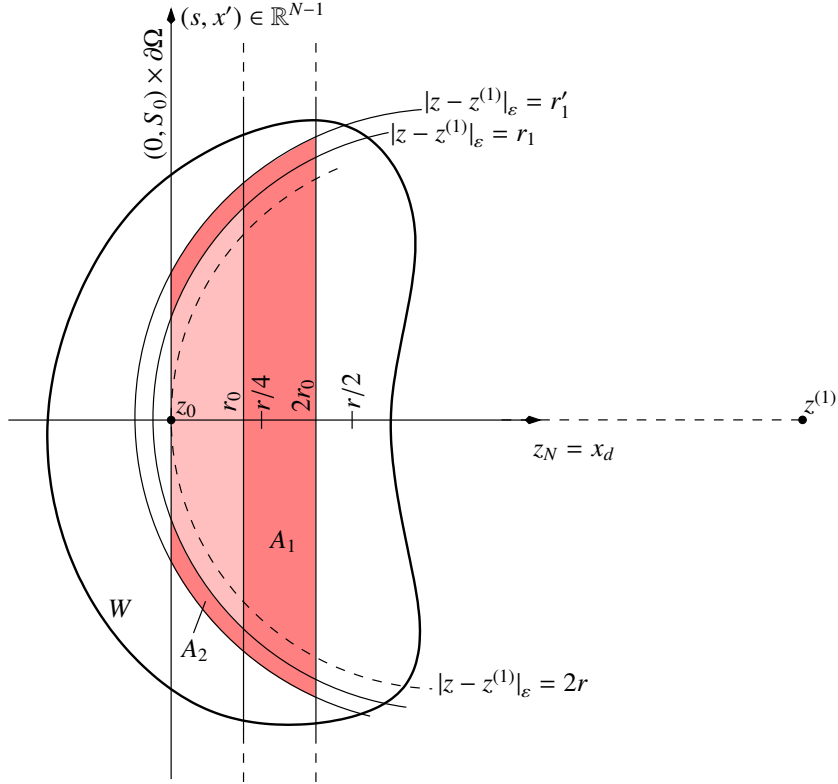

 In connection with the weight
  function $\cpsi$, we introduce the following anisotropic norm in
  $\R^N$, 
  that depends on the (now fixed) parameter $\ctp$, 
  \begin{align*}
    \cnorm{z- y} = \big( \ctp^2|z'-y'|^2 + (z_N-y_N)^2\Big)^{1/2}.
  \end{align*}
  Note that with $\csp$ and $\ctp$ fixed we have $\tau \asymp \ttau$.

  We denote by $\cB(z,r)$ the ball of radius $r$ centered at $z$
  associated with  this norm. We have 
  \begin{align*}
    \cpsi(z) =  \begin{cases}
      12 r^2- \cnorm{z-z^{(1)}}^2 & \text{if}\ \cnorm{z-z^{(1)}}\leq 3 r, \\
      r^2 & \text{if}\ 4r \leq \cnorm{z-z^{(1)}}. 
      \end{cases}
  \end{align*}

  Let $\chi_0\in \Cinfc(\R)$ be \st
  \begin{equation*}
    \chi_0(z_N)= \begin{cases}
      1 & \text{if}\ |z_N|<r_0,\\
      0 & \text{if}\ 2 r_0 < |z_N|,
      \end{cases}
  \end{equation*} 
  where $r_0<r/4$. 
 Let also $\chi_1 \in \Cinfc(\cB(z^{(1)},3r))$  be such that 
  \begin{equation*}
    \chi_1(z)= \begin{cases}
      1 & \cnorm{z-z^{(1)}}<r_1,\\
      0 & \text{if}\ \ r_1' < \cnorm{z-z^{(1)}},
      \end{cases}
  \end{equation*} 
  where $r_1$, $r_1'$ are such that
  $2r<r_1<r_1'<3r$.
Observe that if we choose the values
of $r_1'-2r>0$ and $r_0>0$ \suff small, then the open set
$\{z\in Z;\ z_N\in (0, 2 r_0)\} \cap \{ z\in Z; \ \cnorm{z-z^{(1)}}<r_1'\}$ is
contained in $W$.
We now set $\chi(z) = \chi_1(z) \chi_0(z_N)$. 
  Figure~\ref{fig: local interpolation boundary x -2} shows, near $z_0$,
  the region where $\chi\equiv 1$ and where it varies, that is 
   $ \supp(\chi')\cap Z \subset A_1 \cup A_2$
    with 
    \begin{align*}
    A_1&= \{ z\in Z; z_N\in (r_0, 2 r_0) \    
    \text{and}\ \cnorm{z-z^{(1)}}<r_1'\} , \\
      A_2 &= 
    \{ z\in Z;  z_N\in (0, 2 r_0) \    
    \text{and}\ r_1< \cnorm{z-z^{(1)}}<r_1'\}.
  \end{align*}

The Carleman estimate~\eqref{theorem: Carleman boundary x} applies to
$v = \chi u$, by a density argument.  As $u_{|z_N=0^+} =0$ and $\d_\nu u_{|z_N=0^+} =0$ we
obtain (the values of $\csp$ and $\ctp$ were fixed above)
\begin{align}
  \label{eq: Carleman elliptique bord Dirichlet - interpolation}
   \sum_{|\mi| \leq 3}\tau^{3-|\mi|} \Norm{e^{\tau \varphi}
    D_{s,x}^\mi v}{L^2(W \cap Z)} 
 \lesssim  \Norm{e^{\tau \varphi} P  v}{L^2(W \cap Z)},
  \qquad \tau  \geq \tau_0.
   \end{align}

   We have $Pv=\chi P u+[P,\chi] u$, where
   $[P,\chi ]$ is a differential operator of order 3 that is 
   supported in $A_1 \cup A_2$.
   On $A_1$, we have $\varphi\le e^{\csp ( 12 r^2 - (2r-2r_0)^2)}$. 
   On $A_2$ we have  $\varphi\le e^{\csp ( 12 r^2 - r_1^2) } $. 
   We thus obtain
   \begin{align}
     \label{eq: estimate Q interpolation boundary}
     \Norm{e^{\tau \varphi} P  v}{L^2(W \cap Z)} 
     \lesssim e^{\tau C_3} \big(\Norm{P u}{L^2(Z)}
       +   \Norm{u}{H^3(Y_{\alpha',a})}  \big)   
       +      e^{\tau C_1 }  \Norm{u}{H^3(Z)},
\end{align}
where $C_1= e^{\csp (12 r^2 - r_1^2)  } $, $C_3= e^{\csp (12 r^2 - (2r-2r_0)^2)}$ and
$0< a< r_0$ and some $\alpha' \in (0,\alpha)$ (recalling the definition of the set $Y_{\alpha',a}$ in \eqref{eq: def Yeps}).

We restrict the \lhs of~\eqref{eq: Carleman elliptique bord
  Dirichlet - interpolation} to $V_0=\{z \in Z;\  z_N\in (0,r_0)\}\cap
\{ z\in Z; \ |z-z^{(1)}|  < r_2 \}$, with $r_2 = r+ r_1/2$, 
 whose closure is a  \nhd of $z_0$
in $\ovl{Z}$. Note that $2 r < r_2 < r_1$. As on this set we have $\varphi\geq
e^{\csp( 12 r^2-  r_2^2)}$ and $u\equiv v$,  we obtain
\begin{align}\label{eq: estimate u H1 interpolation boundary}
  e^{\tau C_2}\Norm{u}{H^3(V_0)}\leq  \sum_{|\mi| \leq 3}\tau^{3-|\mi|} \Norm{e^{\tau \varphi}
    D_{s,x}^\mi v}{L^2(W \cap Z)} ,
\end{align}
where $C_2=  e^{\csp(12 r^2 -  (r+r_1/2)^2)}$.
Then \eqref{eq: Carleman elliptique bord Dirichlet - interpolation},  \eqref{eq: estimate Q interpolation boundary} 
and  \eqref{eq: estimate u H1 interpolation boundary} give
\begin{align}
\Norm{u}{H^3(V_0)} \lesssim e^{\tau (C_3 -C_2)}
\big(\Norm{P u}{L^2(Z)}+   \Norm{u}{H^3(Y_{\alpha',a})}  \big)     
+      e^{-\tau (C_2-C_1) }  \Norm{u}{H^3(Z)}.
\end{align}
 Observe that we have $C_1<C_2<C_3$. 
By  Lemma~\ref{lemma: optimization interpolation}, we obtain the sought local
interpolation inequality at the boundary.
\end{proof}

\subsection{Spectral inequality}
\label{sec: from interpolation to spectral inequality}
Let $\phi_j$ and $\mu_j$ be eigenfunctions and associated eigenvalues of the
bi-Laplace operator $B$ with the clamped boundary conditions, that
form a Hilbert basis for $L^2(\Omega)$, \viz, 
\begin{align*}
  B \phi_j = \mu_j \phi_j, \qquad {\phi_j}_{|\d \Omega} = \d_{\nu}  {\phi_j}_{|\d \Omega} 
  = 0,\qquad (\phi_j, \phi_k)_{L^2(\Omega)}=\delta_{jk},
\end{align*}
with $0 < \mu_0 \leq \mu_1 \leq \cdots \leq \mu_j \leq \cdots$.  We
now prove the spectral inequality of Theorem~\ref{theorem: spectral inequality}, namely, for some $C>0$, 
\begin{align}
  \label{eq: spectral inequality bi-laplace 2} 
  \Norm{u}{L^2(\Omega)} \leq C e^{C \mu^{1/4}}
  \Norm{u}{L^2(\O)}, \qquad \mu >0, \quad   u \in \Span \{\phi_j; \ \mu_j \leq \mu\}.
\end{align}

%%%% proof of theorem
\begin{proof}
  We let $\mu>0$ and we pick $\alpha_0, \dots, \alpha_n \in \C$  with $n
  \in \N$ \st
  $\mu_n \leq \mu < \mu_{n+1}$. We set 
  \begin{align*}
    u (x) = \sum_{\mu_j \leq \mu}    \alpha_j \phi_j(x),  
    \qquad 
    w(s,x) = \sum_{\mu_j \leq \mu}  
    \alpha_j\mu_j^{-3/4} f(\mu_j^{1/4} s) \phi_j(x),  
  \end{align*}
  where $f(s) =\gamma \sin(\gamma s)\cosh(\gamma s)  - \gamma \cos(\gamma s)\sinh(\gamma s)$ 
  where here  $\gamma=\sqrt {2}/2$. 
  %= (\cosh(z_0 s) + \cos(z_0 s)) /2$, with 
  %$z_0 = e^{i \pi/4}$. 
  As $D_s^4 f = -f$,  we have $P v=0$, with $P = D_s^4 + B$. 
  We also have 
  \begin{align*}
    f(0) =f'(0) = f^{(2)}(0) =0, \quad  f^{(3)}(0) =1, 
  \end{align*}
  and 
  \begin{align*}
    f(s) =
    g(\gamma s), \quad 
    g(s) = \frac12(e^{-s}\cos ( s-\pi/4)-e^{ s}\cos (s+\pi/4))  .
  \end{align*}
  Since   $ w(s,x)|_{x\in \d \Omega} = \d_\nu w(s,x)|_{x\in \d \Omega}
    =0$, the interpolation inequality of Theorem~\ref{theorem:
    interpolation inequality} yields
  \begin{align*}
    \Norm{w}{H^3(Y)} \leq C \Norm{w}{H^3(Z)}^{1-\delta}
    \Norm{\d_s^3w_{|s=0}}{L^2(\O)}^\delta.
  \end{align*}  
  Observe that 
  we have $\d_s^3 w_{|s=0} = u$ and $\Norm{w}{H^3(Y)} \gtrsim \Norm{w}{L^2(Y)}$ with
  \begin{align*}
    \Norm{w}{L^2(Y)}^2 
    &=  \sum_{\mu_j \leq \mu}   \mu_j^{-3/2}|\alpha_j|^2
      \int_{\alpha}^{S_0-\alpha}
      f(\mu_j^{1/4} s)^2 d s
      = \sum_{\mu_j \leq \mu}   |\alpha_j|^2
      \gamma^{-1} \mu_j^{-7/4} \int_{\alpha \gamma \mu_j^{1/4} }^{(S_0-\alpha)
      \gamma \mu_j^{1/4}}
      g(s)^2 d s\\
    &\gtrsim \mu^{-7/4} \sum_{\mu_j \leq \mu}   |\alpha_j|^2 
      = \mu^{-7/4}  \Norm{u}{L^2(\Omega)}^2,\notag
  \end{align*}  
 using  the  following lemma, whose proof is given below.
  %%%%%%%%%%%%%%%%%%%%%%%%
% lemma                %
%%%%%%%%%%%%%%%%%%%%%%%%
\begin{lemma}
  \label{lemma: minoration integral}
  Let $0 <a < b$ and $t_0>0$.  There exists $C_0$ \st 
  $\int_{a t }^{b t}  g(s)^2 d s \geq C_0$ for $t\ge t_0$. 
\end{lemma}
 We thus obtain
 \begin{align}
   \label{eq: spectral 1}
    \Norm{u}{L^2(\Omega)} \lesssim \mu^{7/8} 
   \Norm{w}{H^3(Z)}^{1-\delta}
    \Norm{u}{L^2(\O)}^\delta.
\end{align}
Next, we estimate $ \Norm{w}{H^3(Z)}$, with the following lemma,
which, from \eqref{eq: spectral 1},  allows one to conclude
the proof of Theorem~\ref{theorem: spectral inequality}.
\end{proof}
 %%%%%%%%%%%%%%%%%%%%%%%%
% lemma                %
%%%%%%%%%%%%%%%%%%%%%%%%
\begin{lemma}
  \label{lemma: }
There exists $C>0$ \st 
$\| w\|_{H^3(Z)}\le C e^{C\mu^{1/4}} \Norm{u}{L^2(\Omega)}$.
\end{lemma}
\begin{proof}
We have 
\begin{align*}
\| w\|_{H^3(Z)}^2 \asymp \sum_{k=0}^{3}\int_0^{S_0} \| \d_s^kw(s,.)\|^2_{H^{3-k}
(\Omega)}ds\lesssim  \sum_{k=0}^{3}\int_0^{S_0} \| \d_s^kw(s,.)\|^2_{H^{4}(\Omega)}ds,
\end{align*}
where $H^s(\Omega)$ denotes the classical Sobolev spaces in $\Omega$. 
Recalling from \eqref{eq: bilaplace elliptic estimate} that,  if $   v_{| \d \Omega}=\d_\nu v_{|\d \Omega}
    =0$, we have $\| v\|_{H^4(\Omega)}\lesssim \| \Delta^2 v\|_{L^2(\Omega)}$, we find
\begin{align*}
 \| \d_s^kw(s,.)\|^2_{H^{4}(\Omega)}&\lesssim \| \Delta ^2 \sum_{\mu_j \leq \mu}  
    \alpha_j\mu_j^{(k-3)/4} f^{(k)}(\mu_j^{1/4} s) \phi_j \|_{L^2(\Omega)}^2= 
     \|  \sum_{\mu_j \leq \mu}  
    \alpha_j\mu_j^{(k+1)/4} f^{(k)}(\mu_j^{1/4} s) \phi_j \|_{L^2(\Omega)}^2
    \\
& =  \sum_{\mu_j \leq \mu}  
    |\alpha_j|^2\mu_j^{(k+1)/2} (f^{(k)}(\mu_j^{1/4} s)  )^2\lesssim\mu ^2 e^{S_0\mu^{1/4}} 
    \sum_{\mu_j \leq \mu}  
    |\alpha_j|^2.
\end{align*}
Integrating this estimate over $(0,S_0)$ and summing over $k$ yields the result.
\end{proof}

%%%% proof of lemma
\begin{proof}[\bfseries Proof of Lemma~\ref{lemma: minoration integral}]
  For $s \in [-\pi/2 + 2k \pi, 2 k \pi]$, $k \in \N^*$, we have 
  $\cos (s+\pi/4) \geq \sqrt{2}/2$. For $t_1$ chosen \suff large, if $t \geq t_1$,
  there exists $k\in\N$ \st $[-\pi/2 + 2k \pi, 2 k \pi]\subset
  [a t, b t]$ and $|g(s)|=\frac12 |e^{-s}\cos (s-\pi/4)-e^s\cos (s+\pi/4)|\ge 1$. Then, 
  $\int_{a t }^{b t}  g(s)^2 d s \geq \pi/2$. Finally, there
  exists $C>0$ \st $\int_{a t }^{b t}  g(s)^2 d s \geq
  C$ for $t \in [t_0,t_1]$, since the function $g(s)^2$ is almost
  everywhere  positive.
\end{proof}

\subsection{A null-controllability result for a higher-order parabolic equation}

Let $T>0$.
We consider here the controlled evolution equation on $(0,T)\times \Omega$
with the clamped boundary conditions ($\nu$ denotes the outer unit normal
to $\d\Omega$):
\begin{align}
  \label{eq: parabolic system}
  \d_t y + \Delta^2 y = \chi_\O f, \qquad y_{|(0,T)\times
  \d\Omega} =0, \quad \d_\nu y_{|(0,T)\times
  \d\Omega} =0,\qquad y_{|t=0} = y_0 \in L^2(\Omega), 
\end{align}
where $\O$ is an open subset of $\Omega$ and
$\chi_\O \in L^\infty(\Omega)$ is \st $\chi_\O >0$ on
$\O$.  The function $f\in L^2((0,T)\times\Omega)$ is the control
function here. Well-posedness for this parabolic system is recalled in
Corollary~\ref{cor: semigroup}. One may wonder if one can choose $f$ to drive the
solution from its initial condition $y_0$ to zero at final time $T$.
Thanks to the spectral inequality of Theorem~\ref{theorem: spectral
  inequality} one can answer positively to this null-controllability
question.
%%%%%%%%%%%%%%%%%%%%%%%%
% theorem              %
%%%%%%%%%%%%%%%%%%%%%%%%
\begin{theorem}[Null-controllability]
  \label{theorem: null controllability}
  There exists $C>0$ \st for any $y_0 \in L^2(\Omega)$, there exists
  $f \in L^2((0,T)\times\Omega)$ \st the solution to \eqref{eq:
    parabolic system} vanishes at $T=0$ and moreover
  $\Norm{f}{L^2((0,T)\times\Omega)}\leq C \Norm{y_0}{L^2(\Omega)}$. 
\end{theorem}
The proof can be adapted in a straight forward manner from the proof
scheme of \cite{LR:95} developed for the heat equation and that is
presented in a fairly synthetic way in the survey article
\cite{LL:12}.

%%%%%%%%%%%%%%%
%% Section             %
%%%%%%%%%%%%%%%
\section{Resolvent estimate and application}
\label{sec: resolvent estimate}

Using one of the interpolation 
inequalities proven in Section~\ref{sec: spectral inequality}
(Proposition~\ref{prop: interpolation interior observation}), we prove
  the resolvent estimate of
Theorem~\ref{theorem: resolvent estimate bilaplace clamped}. Finally, as an
application, we state a stabilization result that follows from
it for the plate equation.

\subsection{Resolvent estimate}
Let $U \in D(\mathcal B) = (H^4(\Omega) \cap  H^2_0(\Omega))
\times  H^2_0(\Omega)$ and $F \in \H = H^2_0(\Omega)
\times L^2(\Omega)$, be such that 
\begin{align}
  \label{eq: resolvent equation}
  (i \sigma \id_\H - \mathcal B) U = F, \quad 
  U = \transp\begin{pmatrix} u_0,  u_1\end{pmatrix}, \ \ 
  F = \transp\begin{pmatrix} f_0,  f_1\end{pmatrix},
\end{align}
for $\sigma \neq 0$. 
Our goal is to find an estimate of the form  $\Norm{U}{\H}
\leq K e^{K |\sigma|^{1/2}} \Norm{F}{\H}$. 
We have
\begin{align}
  \label{eq: resolvent equation 2}
  i \sigma u_0 + u_1 = f_0, \quad 
  ( - \sigma^2 -  i \sigma  \alpha  + B) u_0 = f ,\quad \text{with}\ f=
   (i \sigma -\alpha ) f_0 - f_1.
\end{align}
Multiplying the second equation by $\ovl{u}_0$ and an integration over
$\Omega$ give
\begin{align*}
%	\label{eq: resolvent equation 2bis}
   \inp{(-\sigma^2+B) u_0}{ u_0}_{L^2(\Omega)} 
  - i \sigma \Norm{\alpha^{1/2} u_0}{L^2(\Omega)}^2 
  = \inp{f}{u_0}_{L^2(\Omega)},   
\end{align*}
The first term is real and the second
term is purely imaginary.
We thus have 
\begin{align*} 
  \sigma \Norm{\alpha^{1/2}u_0}{L^2(\Omega)}^2 = - \Im \inp{f}{ u_0}_{L^2(\Omega)}.
\end{align*}
Using that $\alpha \geq \delta>0$ in $\O$ yields
\begin{align}
  \label{eq: resolvent equation 3}
 \delta \sigma_0 \, \Norm{u_0}{L^2(\O)}^2 
 \leq \sigma \Norm{\alpha^{1/2}u_0}{L^2(\Omega)}^2
 \leq  \Norm{f}{L^2(\Omega)} \Norm{u_0}{L^2(\Omega)},
\end{align}
for $\sigma \geq \sigma_0$. 

A key estimate is given by the following lemma. We provide a proof below.
%%%%%%%%%%%%%%%%%%%%%%%%
% sub-lemma            %
%%%%%%%%%%%%%%%%%%%%%%%%
\begin{lemma}
  \label{lemma: pre-resolvent estimate}
  There exists $C>0$ such that 
  \begin{align*}
    \Norm{u_0}{H^3(\Omega)} \leq C e^{C|\sigma|^{1/2}}
    \big( \Norm{f}{L^2(\Omega)} +\Norm{u_0}{L^2(\O)}  \big).
  \end{align*}
\end{lemma}

\bigskip
Then estimate~\eqref{eq: resolvent equation 3} yields
\begin{align*}
   \Norm{u_0}{H^2(\Omega)} \lesssim e^{C|\sigma|^{1/2}}
   \big( \Norm{f}{L^2(\Omega)} +  \Norm{u_0}{L^2(\Omega)}^\hf
   \Norm{f}{L^2(\Omega)}^\hf \big), \quad \sigma \geq \sigma_0,
\end{align*}
and with the Young inequality we obtain
\begin{align*}
   \Norm{u_0}{H^2(\Omega)} 
   \lesssim e^{C|\sigma|^{1/2}}
   \Norm{f}{L^2(\Omega)}.
 \end{align*}
Using the form of $f$ given in \eqref{eq: resolvent equation 2} we then obtain
\begin{align*}
   \Norm{u_0}{H^2(\Omega)} 
   \lesssim  e^{C|\sigma|^{1/2}} \big(\Norm{f_0}{L^2(\Omega)} + \Norm{f_1}{L^2(\Omega)}\big), 
\end{align*}
Finally as $u_1 = f_0 - i \sigma u_0 $ we obtain 
\begin{align}
  \label{eq: resolvent equation 5}
   \Norm{u_0}{H^2(\Omega)}  +   \Norm{u_1}{L^2(\Omega)} 
   \lesssim e^{C|\sigma|^{1/2}}
   \big(\Norm{f_0}{L^2(\Omega)} + \Norm{f_1}{L^2(\Omega)}\big), 
\end{align}
yielding the resolvent estimate of Theorem~\ref{theorem: resolvent estimate bilaplace clamped}.

\subsection{Proof of Lemma~\ref{lemma: pre-resolvent estimate}}

Let $\rho = \exp( i \sgn(\sigma) \pi/4)$, yielding $\rho^2 =
\sgn(\sigma) i$ and $\rho^4 =-1$. We set $u = \exp(s \rho
|\sigma|^{1/2}) u_0$ and have $Q u=e^{s \rho |\sigma|^{1/2}} f$, 
with $Q=D_s^4 + B + \alpha D_s^2$, recalling \eqref{eq: resolvent equation 2}.
Let $S_0>0$ and $\beta \in (0,S_0/2)$. Let also $Z=(0,S_0)\times \Omega$ and
$Y=(\beta,S_0-\beta)\times \Omega$. 
 We then
apply the interpolation inequality of  Proposition~\ref{prop:
  interpolation interior observation}: with $0<
\beta_1 < \beta_2 < S_0$ we have $C>0$ and $\delta_0 >0$ such that
\begin{align}
  \label{eq: resolvent equation 4}
  \Norm{u}{H^3(Y)} \leq C \Norm{u}{H^3(Z)}^{1-\delta}
    \left(\Norm{Q u}{L^2(Z)} + \Norm{u}{L^2((\beta_1, \beta_2)\times \O)} \right)^\delta.
\end{align}
Next, we note that we have 
\begin{align*}
   &\Norm{u}{H^3(Y)}  \geq \Norm{u}{L^2((\beta, S_0 - \beta), H^3(\Omega))}
   \geq e^{-C |\sigma|^{1/2}}   \Norm{u_0}{H^3(\Omega)}, \\
    &\Norm{u}{H^3(Z)} \lesssim e^{C |\sigma|^{1/2}}   \Norm{u_0}{H^3(\Omega)},
    \\
    &\Norm{u}{L^2((\beta_1, \beta_2)\times \O )} \leq 
      e^{C |\sigma|^{1/2}} \Norm{u_0}{L^2(\O)},
\end{align*}
yielding with \eqref{eq: resolvent equation 4}
\begin{align*}
   \Norm{u_0}{H^3(\Omega)} \leq C e^{C|\sigma|^{1/2}}
   \big( \Norm{f}{L^2(Z)} +\Norm{u_0}{L^2(\O)}  \big).
\end{align*}
This concludes the proof of the estimate of
Lemma~\ref{lemma: pre-resolvent estimate}.

\subsection{A stabilization result for the plate equation}
Let now $(y_0,y_1) \in D(\mathcal B^k)$, $k\geq 1$, and $y$ be the solution of the
damped plate equation 
\begin{equation}
  \label{eq: damped plate equation}
  \d_t^2 y  + \Delta^2  y + \alpha \d_t  y =0, \qquad y_{|t=0} = y_0, \ \d_t
  y_{|t=0} = y_1,
  \qquad y_{| [0,+\infty) \times \d \Omega}=\d_\nu y_{| [0,+\infty) \times \d \Omega}=0,
\end{equation}
with $\alpha$ a nonnegative function such that $y\geq \delta >0$ on
$\O$, an open subset of $\Omega$.  If we set $Y = (y,\d_t y)$ we have
$(\d_t + \mathcal B) Y=0$. From the resolvent estimate of
Theorem~\ref{theorem: Carleman boundary x} we obtain the following
energy decay for the damped plate equation, using the results set in an abstract framework in \cite{Ba-Du:08}.
%%%%%%%%%%%%%%%%%%%%%%%%
% theorem              %
%%%%%%%%%%%%%%%%%%%%%%%%
\begin{theorem}
  \label{theorem: stabilisation plate}
  With the energy function introduced in \eqref{eq: energy plate
    equation}
  the solution to the damped plate equation~\eqref{eq: damped plate
    equation} satisfies, for some $C>0$, 
  \begin{equation*}
     E(y)(t)  \le \frac{C}{\big(\log (2+t)\big)^{4k}} \Norm{ \mathcal
       B^k Y_0}{\H}^2, \ \ t>0, \qquad Y_0 = (y_0, y_1) \in D(\mathcal
     B^k).
  \end{equation*}
\end{theorem}
Among the existing results available in the literature for plate type
equations, many of them concern the ``hinged'' boundary conditions.
We first mention these result. An important result obtained in
\cite{Jaffard:90} on the controllability of the plate equation on a
rectangle domain with an arbitrary small control domain. The method
relies on the generalization of Ingham type inequalities in
\cite{Kahane:62}. An exponential stabilization result, in the same
geometry, can be found in \cite{RTT:06}, using similar techniques. In
\cite{RTT:06} the localized damping term involves the time derivative
$\d_ty$ as in \eqref{eq: damped plate equation}. Interior nonlinear feedbacks can be used for exponential
stabilization \cite{Tebou:09}. There, feedbacks are localized in a
neighborhood of part of the boundary that fulfills multiplier-type
conditions.  A general analysis of nonlinear damping that
includes the plate equation is
provided in \cite{AA:11} under multiplier-type conditions.
For ``hinged''
boundary conditions also, with a boundary damping term, we cite
\cite{ATT:07} where, on a square domain, a necessary and sufficient
condition is provided for exponential stabilization. In
\cite{Nouira:09}, a polynomial stabilization rate is obtained if the
condition of \cite{ATT:07} is relaxed. 

For ``clamped'' boundary conditions, few results are available. We
cite \cite{Alabau:06}, where a general analysis of nonlinearly damped
systems that includes the plate equation 
under multiplier-type conditions is provided. In \cite{APT:17}, the
analysis of discretized general nonlinearly damped system is also
carried out, with the plate equation as an application.  In \cite{Tebou:12}, a nonlinear
damping involving the p-$Laplacian$ is used also under multiplier-type
conditions. In \cite{DS:15}, an
exponential decay is obtained in the case of ``clamped'' boundary
conditions, yet with a damping term of the Kelvin-Voigt type, that is
of the form $\d_t \Delta y$, that acts over the whole domain.

Theorem~\ref{theorem: stabilisation plate} provides a
log-type stabilization result. Optimality is a natural
question and one could be interested in seeking geometries that improve this
decay rate, yielding polynomial or exponential rate, in the case of
``clamped'' boundary conditions,  in the spirit of
some of the existing results we cite above.

\def\appendixname{}
% appendices
\appendix
\section*{Appendices}

%%%%%%%%%%%%%%%
%% Section             %
%%%%%%%%%%%%%%%
\section{Proofs of some technical results}

\subsection{Proof of the estimate optimality in the case of  symbol
  flatness}
\label{sec: prop: Carleman P2 - optim}

Here, we provide a proof of Proposition~\ref{prop: Carleman P2 -
  optim}. 

We have $Q(z,D_z) = q(z,D_z) + r_{m-1}(z,D_z) + r_{m-2}(z,D_z)$ with
$r_{m-1}(z,D_z)$ homogeneous of degree $m-1$ and $r_{m-2}(z,D_z)$ of
order $m-2$ (non necessarily homogeneous).
  
  If there exists $(z_0, \zeta_0, \tau_0)$ as in the statement of the
  proposition, then by homogeneity, as $\tau_0 \neq 0$, there exists $\zeta_1
  \in \R^N$, \st 
  \begin{align}
    \label{eq: choice zeta0}
    q(z_0,\theta_1 )= 0, 
    \quad  d_{z,\zeta} q (z_0,\theta_1 )= 0, \quad
    \theta_1^\mi\neq 0, \ \ \text{with} \ \theta_1 = \zeta_1 + i  d \varphi(z_0).
 \end{align}
 Without any loss of generality we may assume that $z_0=0$. Because of the form of
  \eqref{eq: Carleman P2 - optim}, observe also that there is no loss of
  generality  if we assume that $\varphi(0) =0$.

  We then introduce  $w(z) = \inp{z}{\theta_1}$. We note that 
  \begin{align*}
    \varphi(z) - \Im w(z) = G(z) + |z|^3 \O(1), \qquad G(z) = \hf
    d_z^2 \varphi(0) (z,z). 
  \end{align*}
 We then pick $f \in \Cinfc(\R^N)$, $f \not\equiv 0$,  and set
  $u_\tau(z) = e^{i \tau w(z)} f ( \tau^{1/2} z)$.
We have
\begin{align}
  \label{eq: asymp L2 norm}
  \Norm{ e^{\tau \varphi} u_\tau}{L^2(\R^N)}^2
  &= \int_{\R^N}  e^{2 \tau \big(G(z) + |z|^3 \O(1)\big)} 
   |
  f (\tau^{1/2} z)
  |^2 d z 
  = \tau^{-N/2}\int_{\R^N} 
  e^{2 G(y) + \tau^{-{1/2}} |y|^3 \O(1)} 
  |f(y)|^2 d y\\
  &\mathop{\sim}_{\tau \to \infty} 
  \tau^{-N/2} \int_{\R^N} e^{2  G(y)} 
  | f(y)|^2 d y,\notag
\end{align}
with the change of variables $y =\tau^{1/2} z$ and the dominated
convergence theorem.

As we note that 
\begin{align*}
  e^{-i\tau w(z)} D_z^\mi u_\tau 
  = (D_z + \tau \theta_1)^\mi f \big(\tau^{1/2} z\big)
  = \tau^{|\mi|} \theta_1^\mi f \big(\tau^{1/2} z\big)
  + \tau^{|\mi| - {1/2}} \O(1),  
\end{align*}
similarly, we find
\begin{align}
  \label{eq: asymp L2 norm derivatives}
  \Norm{ e^{\tau \varphi} D_z^\mi u_\tau}{L^2(\R^N)}^2
  \mathop{\sim}_{\tau \to \infty} 
  \tau^{2 |\mi|-N/2}  |\theta_1^\mi|^2 \int_{\R^N} e^{2  G(y)} 
  | f(y)|^2 d y,
\end{align}
as we have  $\theta_1^\mi \neq 0$.

\medskip
We have 
\begin{align*}
  e^{-i \tau w(z)} Q  e^{i \tau w(z)}
  = q(z,D_z + \tau \theta_1) + r_{m-1}(z,D_z + \tau \theta_1) + r_{m-2}(z,D_z + \tau
  \theta_1).
\end{align*}
 With the Taylor formula and homogeneity we observe that 
\begin{align*}
  q(z,D_z + \tau \theta_1) &= \tau^m q(z, \theta_1) 
  + \tau^{m-1}  d_\zeta q (z,\theta_1) (D_z)
  +\frac12  \tau^{m-2} d^2_\zeta q(z,\theta_1)(D_z,D_z)
  \\
&\quad  + \frac12 \int_0^1 (1-t)^2 d_\zeta^3  q (z,t D_z + \tau\theta_1)
  (D_z,D_z,D_z)\,  d t.
\end{align*}
Next, we write 
\begin{align*}
  q(z,\theta_1) = \underbrace{q(0,\theta_1)}_{=0}
  + \underbrace{d_z q (0,\theta_1) (z) }_{=0}
  + \hf d_z^2  q(0,\theta_1)(z,z) 
  + \hf \int_0^1 (1-t)^2 
  d_z^3 q(t z,\theta_1)(z,z,z) d t, 
\end{align*}
\begin{align*}
  d_\zeta q (z,\theta_1) (D_z)
  = \underbrace{ d_\zeta q (0,\theta_1) (D_z)}_{=0}
  +  d _{\zeta} d_z  q (0,\theta_1)(D_z,z) 
  + \int_0^1 (1-t) d_\zeta d_z^2 q (t z,\theta_1)(D_z,z,z) d t, 
\end{align*}
and
\begin{align*}
  d_\zeta^2 q (z,\theta_1) (D_z,D_z)
  =d_\zeta^2 q (0,\theta_1) (D_z,D_z)
  + \int_0^1  d_\zeta^2d_z q  (t z,\theta_1)(D_z,D_z,z) d t, 
\end{align*}
which gives 
\begin{align*}
  e^{-i \tau w(z)} Q  e^{i \tau w(z)}
  &= \tau^{m-1} \bigg( \hf d_z^2 q(0,\theta_1)(\tau^{1/2} z,\tau^{1/2}  z) 
  + \hf \tau^{-{1/2}}\int_0^1 (1-t)^2 
  d_z^3 q(t z,\theta_1)(\tau^{1/2} z,\tau^{1/2} z,\tau^{1/2} z)
  dt \\
  &\qquad \quad +
  \tau^{-{1/2}}  d_\zeta d_z q(0,\theta_1)(D_z, \tau^{{1/2}} z) 
  + \tau^{-1} \int_0^1 (1-t) 
  d_\zeta d_z^2 q(tz,\theta_1)(D_z,\tau^{{1/2}}z,\tau^{{1/2}} z) 
  d t\\
   &\qquad \quad
  +  \hf \tau^{-1}  d_\zeta^2 q (0,\theta_1) (D_z,D_z)
  +\hf \tau^{-3/2} \int_0^1  d_\zeta^2d_z q (t z,\theta_1)(D_z,D_z, \tau^{{1/2}} z) d t
   \\
  &\qquad \quad + \tau^{1-m} \big(r_{m-1}(z,D_z + \tau \theta_1) 
  + r_{m-2}(z,D_z + \tau\theta_1)\big)
  \bigg).
\end{align*}
We then find
\begin{align*}
  e^{-i \tau w(z)} Q u_\tau &= \tau^{m-1}
  \Big(\hf d_z^2 q(0,\theta_1)(\tau^{1/2} z,\tau^{1/2}  z) f(\tau^{1/2} z)
  +  d_\zeta d_z q(0,\theta_1)\big(D_z f(\tau^{1/2} z) , \tau^{{1/2}} z\big) 
 \\
 &\qquad\quad  +\hf  \Big( d_\zeta^2 q (0,\theta_1) (D_z,D_z) f\Big)(\tau^{1/2} z)
  + r_{m-1}(z, \theta_1) f(\tau^{1/2} z)
  + \tau^{-{1/2}} \O(1)
  \Big).
\end{align*}
Arguing as for \eqref{eq: asymp L2 norm}, we obtain, as $\tau \to
\infty$, 
\begin{align}
  \label{eq: asymp L2 norm rhs}
  \Norm{e^{\tau \varphi} Q u_\tau}{L^2(\R^N)}^2
  &=
   \tau^{2(m-1)- N/2} \int_{\R^N} e^{2  G(y)} 
  \big|
  \hf d_z^2 q(0,\theta_1)(y,y) f(y)
  + d_\zeta d_z q(0,\theta_1)\big(D_z f(y) , y\big) 
  \\
  &\qquad \qquad \qquad \qquad \qquad +\hf  \big( d_\zeta^2 q (0,\theta_1) (D_z,D_z) f\big)(y)
  + r_{m-1}(0, \theta_1) f(y)
  \big|^2 d y \notag\\
  &\quad + \O (\tau^{2(m-1)- N/2-1/2}).\notag
\end{align}
The assumed estimate \eqref{eq: Carleman P2 - optim} along with \eqref{eq: asymp L2
  norm}--\eqref{eq: asymp L2 norm rhs}  thus implies that $\delta=0$
and moreover that  the integral above does not
vanish.
\hfill \qedsymbol \endproof

\begin{remark}
  Observe that if in addition we assume that $m \geq 3$ then the
  partial Carleman estimate \eqref{eq: Carleman P2 - optim} with the
  loss of a full derivative implies that $d \varphi(z)$ does not
  vanish in $\Omega$.  In fact, if $d \varphi(z_0) =0$ and if we pick
  $\zeta_0 =0$ then $\theta_1 =0$ and since $m\geq 3$ we have the
  properties listed in \eqref{eq: choice zeta0}. The remainder of the
  proof then yields a contradiction as the integral term in
  \eqref{eq: asymp L2 norm rhs} vanishes.

In the case $m=1$,  it is known that a Carleman estimate with the loss of
a half derivative can hold even if the gradient of the weight function
vanishes (see Lemma~8.1.1 in \cite{Hoermander:63}). For instance, for
$\varphi(z) =z_1^2/2$ and for the operator $D_{z_1}$, we have 
\begin{align*}
  \tau^{1/2} \Norm{e^{\tau \varphi} u}{L^2(\R^N)} 
  \lesssim \Norm{e^{\tau \varphi} D_{z_1} u}{L^2(\R^N)}, 
\end{align*}
for $\tau>0$ and $u \in \Cinfc(\R^N)$. 
Then, for the operator $D_{z_1}^2$, we have 
\begin{align*}
  \tau \Norm{e^{\tau \varphi} u}{L^2(\R^N)} 
  \lesssim \Norm{e^{\tau \varphi} D_{z_1}^2 u}{L^2(\R^N)}, 
\end{align*}
for $\tau>0$ and $u \in \Cinfc(\R^N)$. 
We then have the case of an operator of order $m=1$ or $2$ in $\R^N$ such
that an estimate with a loss a full derivative holds and yet $d
\varphi$ may vanish. 
\end{remark}

\begin{remark}
  The reader should observe that the statement of Proposition~\ref{prop: Carleman P2 - optim}
  assumes that the symbol $q(z,\zeta+ i \tau d \varphi(z))$ vanishes at second order at a complex root,
  that is, for $\tau>0$. Flatness at a real root may not yield
  $\delta=0$. In fact, in $\R^N$, $N\ge 2$, consider the
  operator $Q= (D_{z_1} + D_{z_2})^m$ with $m\geq 2$ and $\varphi(z) =z_1$. 
  Then $q (\zeta + i \tau d \varphi) = (\zeta_1 + \zeta_2 + i \tau)^m$ which
  vanishes (at order $m$) for $\tau=0$ and $\zeta_1 + \zeta_2=0$. Yet, we have the
  following estimate
  \begin{align}
    \label{eq: hyperbolic estimate}
    \tau^m \Norm{e^{\tau \varphi} u}{L^2(\R^N)}  \leq\Norm{e^{\tau
        \varphi}P u}{L^2(\R^N)}, 
  \end{align}
  for $v \in \Cinfc(\R^2)$. 
  This means $\delta =1$ here.
  
  The proof of \eqref{eq: hyperbolic estimate} is as follows. We write
  $e^{\tau \varphi} (D_{z_1} + D_{z_2}) u  = (D_{z_1}  + i \tau +
  D_{z_2}) v$ with $v = e^{\tau \varphi} u$.
  We then have 
  \begin{align*}
    \Norm{e^{\tau \varphi} (D_{z_1} + D_{z_2}) u }{L^2(\R^N)}^2
    &= \Norm{ (D_{z_1} + D_{z_2}) v}{L^2(\R^N)}^2 + \Norm{\tau v}{L^2(\R^N)}^2
   - 2 i \tau \Re \int  \ovl{v} (D_{z_1} + D_{z_2}) v\ dz\\
   &= \Norm{ (D_{z_1} + D_{z_2}) v}{L^2(\R^N)}^2 + \Norm{\tau v}{L^2(\R^N)}^2
   - \tau  \underbrace{\int  (\d_{z_1} + \d_{z_2} )|v|^2  dz}_{=0 \
     \text{as} \ \supp(v) \ \text{compact}}\\
   &\geq \tau^2 \Norm{ v}{L^2(\R^N)}^2 =  \tau^2 \Norm{ e^{\tau \varphi} u}{L^2(\R^N)}^2.
  \end{align*}
  Multiple applications of this estimate yield \eqref{eq: hyperbolic
    estimate}. 

  \medskip Note however that we do not claim to have
  $\Norm{e^{\tau \varphi} D u}{L^2(\R^N)} \lesssim \Norm{e^{\tau
      \varphi} (D_{z_1} + D_{z_2}) u}{L^2(\R^N)}$,
  as $D_{z_1} + D_{z_2}$ is not elliptic.
\end{remark}

\subsection{Proofs associated with the semi-classical calculus}
\label{sec: proofs semi-classical calculus}
\subsubsection{Proof of Proposition~\ref{prop: admissible metric and order function}}
\label{proof prop: admissible metric and order function}
The dual quadratic form of $g$ on $\W$ is given by
  \begin{align*}
    g^\sigma = \lambsc^2 |d z|^2 
    + \frac{|d \zeta'|^2}{(1+ \csp \ctp)^{2} }
    + \frac{|d \zeta_N|^2}{\csp^{2} }.
  \end{align*}
  We then have, for $X = (z_X, \zeta_X)$, as $\gamma \geq 1$, 
  \begin{align*}
    \big(h_g)^{-1}(X) 
    &= \inf_{T \in \W \atop T \neq 0} 
    \Big( g_X^\sigma(T) /g_X(T)\Big)^{1/2}
    = \min \big( \csp^{-1}, (1+ \csp \ctp)^{-1} \big) \lambsc(X)\\
    &\geq (2\csp)^{-1} \lambsc(X) \geq \tau \varphi(z_X) /2 \geq 1,  
    \end{align*}
    as $\tau \geq \tauast\geq 2$. 
    The uncertainty principle is thus fulfilled.

    For $X=(z_X, \zeta_X) \in \W$, we write $z_X= (z_X', (z_X)_N)$,
    with $z_X'\in \R^{N-1}$.
    Similarly, we also write $\zeta_X = (\zeta_X', (\zeta_X)_N)$, with 
    $\zeta_X'\in \R^{N-1}$.

    We now prove the slow variations of $g$ and $\varphi$, $\lambsc$, namely, there exist
    $K>0$, $r>0$, \st 
    \begin{align*}
      \forall X,Y,T \in \W, \quad 
      g_X(Y-X) \leq r^2 \ \ \imp \ \ 
      \begin{cases}
        g_Y(T) \leq K g_X(T),\\   
        K^{-1} \leq \frac{\varphi(z_X)}{\varphi(z_Y)} \leq K, 
        \quad 
        K^{-1} \leq \frac{\lambsc(X)}{\lambsc(Y)} \leq K,
      \end{cases}
    \end{align*}
    where $X = (z_X,\zeta_X)$ and $Y= (z_Y,\zeta_Y)$.
    We thus assume that $g_X(Y-X) \leq r^2$, with $0<r<1$ to be chosen
    below. This gives
    \begin{align}
      \label{eq: slow variation 1}
      (1+\csp \ctp) (|z_X' - z_Y'|) + \csp |(z_X)_N - (z_Y)_N| + \lambsc(X)^{-1} (
      |\zeta_X-\zeta_Y|)\leq C r.
    \end{align}
    We observe that we have
    \begin{align*}
      \varphi(z_X) = e^{\csp \cpsi(z_X)} 
      = \varphi(z_Y) e^{\csp \big(\cpsi(z_X)- \cpsi(z_Y)\big)},
    \end{align*}
    where $\cpsi (z_X) = \psi(\ctp z_X', (z_X)_N)$.
   Note  that
    \begin{align*}
      |\cpsi(z_X)- \cpsi(z_Y)|
      \leq \big(\ctp |z_X' - z_Y'| + |(z_X)_N - (z_Y)_N|\big) 
      \Norm{\psi'}{L^\infty}.
    \end{align*}
    With~\eqref{eq: slow variation 1}, 
    we thus obtain 
     \begin{align}
         \label{eq: varphi x -> y}
      \varphi(z_X) 
      \leq \varphi(z_Y) e^{C r\Norm{\psi'}{L^\infty}}\lesssim \varphi(z_Y).
    \end{align}
    Similarly, we have 
    \begin{align}
      \label{eq: varphi y -> x}
    \varphi(z_Y) 
    \lesssim \varphi(z_X).
    \end{align}

    We also have 
    \begin{align}
      \label{eq: eta < mu (X)}
      |\zeta_Y| \leq |\zeta_Y - \zeta_X| + |\zeta_X|
      \leq C r \lambsc(X) + |\zeta_X| \lesssim \lambsc(X).
    \end{align}
    Next, we write 
    \begin{align*}
      |\zeta_X| \leq |\zeta_Y - \zeta_X| + |\zeta_Y|
      \leq  C r \lambsc(X) +  |\zeta_Y|
      \leq C r \big( \tau \csp \varphi(z_X) + |\zeta_X| \big) +  |\zeta_Y|.
    \end{align*}
    Hence, for $r$ \suff small, with \eqref{eq: varphi x -> y}, we have
    \begin{align}
      \label{eq: xi < mu (Y)}
      |\zeta_X| \lesssim   \tau \csp \varphi(z_X) + |\zeta_Y|
      \lesssim \lambsc(Y).
    \end{align}
    With \eqref{eq: varphi x -> y} and \eqref{eq: xi < mu (Y)}, \resp
    \eqref{eq: varphi y -> x} and \eqref{eq: eta < mu (X)},  we find
    \begin{align*}
      \lambsc(X) \lesssim \lambsc(Y), \quad \text{\resp}\ \  \lambsc(Y) \lesssim \lambsc(X).
    \end{align*}
    Then, if $T= (z_T, \zeta_T) \in \W$ we find
    \begin{align*}
      \frac{|\zeta_T|^2}{\lambsc(Y)^2} \lesssim \frac{|\zeta_T|^2}{\lambsc(X)^2} 
     \lesssim  \frac{|\zeta_T|^2}{\lambsc(Y)^2},
    \end{align*}
    and this gives $g_Y(T) \lesssim g_X(T) \lesssim g_Y(T)$,
    concluding the proof of the slow variations of $\lambsc$ and $g$.

    \bigskip
    We now prove the temperance of $g$, $\varphi$ and $\lambsc$, namely, there exist 
    $K>0$, $N>0$, \st 
    \begin{align*}
      &\forall X,Y,T \in \W, \quad  
      \frac{g_X(T)}{g_Y(T)} \leq C \big(1 + g_X^\sigma(X-Y)\big)^N, \\
      &\forall X,Y \in \W, \quad  
        \frac{\varphi(z_X)}{\varphi(z_Y)} \leq C \big(1 +
        g_X^\sigma(X-Y)\big)^N,
        \quad         
      \frac{\lambsc(X)}{\lambsc(Y)} \leq C \big(1 + g_X^\sigma(X-Y)\big)^N,
    \end{align*}
     where $X = (z_X,\zeta_X)$ and $Y= (z_Y,\zeta_Y)$.
     We have
    \begin{align*}
      g_X^\sigma(X-Y)=\lambsc(X)^2 |z_X-z_Y|^2
      + \frac{|\zeta_X'- \zeta_Y'|^2}{(1+ \csp \ctp)^2} 
      + \frac{|(\zeta_X)_N- (\zeta_Y)_N|^2}{\csp^2}.
    \end{align*}
    
    We note that 
    \begin{align}
      \label{eq: est xi}
      |\zeta_X| 
      &\leq |\zeta_Y| + |\zeta_X- \zeta_Y| 
      \leq |\zeta_Y| 
      + \frac{|\zeta_X - \zeta_Y|}{\csp}  \tau \csp \varphi(z_Y)\\
      &\leq |\zeta_Y| 
      + \bigg(2\frac{|\zeta_X' - \zeta_Y'|}{1+ \csp \ctp} 
      + \frac{|(\zeta_X)_N - (\zeta_Y)_N|}{\csp} \bigg) \tau \csp
      \varphi(z_Y)
      \notag\\
      &\lesssim  \big(1 + g_X^\sigma(X-Y)^{1/2}  \big) \lambsc(Y).
      \notag
    \end{align}
    First, if $(1+ \ctp \csp )|z_X'-z_Y'| + \csp|(z_X)_N - (z_Y)_N | \leq 1$,  
    then, 
    arguing as in \eqref{eq: varphi x -> y}, we find
    $$
    \varphi(z_X) \lesssim \varphi(z_Y), \qquad \tau \csp \varphi(z_X) \lesssim \lambsc(Y).
    $$
    Second, if $(1+ \ctp \csp)|z_X'-z_Y'| + \csp|(z_X)_N - (z_Y)_N | \geq
    1$, we then have $2|z_X - z_Y| \geq 1/\csp$. We write, as
    $\tau\geq \tauast\geq 1$, 
    $$
    \varphi(z_X) = \frac{\ttau(z_X)}{\csp \tau} \leq
    \frac{\lambsc(X)}{\csp}
    \lesssim |z_X-z_Y|\lambsc(X) \lesssim \big(1 + g_X^\sigma(X-Y)^{1/2}
    \big)
    \lesssim \big(1 + g_X^\sigma(X-Y)^{1/2} \varphi(z_Y), 
    $$
    using that $\varphi\geq 1$. 
    We also write 
    $$
    \tau \csp \varphi(z_X) \lesssim \lambsc(X) \leq \lambsc(X)\frac{\lambsc(Y)}{\csp}
    \lesssim  |z_X-z_Y|\lambsc(X)\lambsc(Y) 
    \lesssim \big(1 + g_X^\sigma(X-Y)^{1/2}  \big) \lambsc(Y).
    $$
    In any case, we have 
    $$\varphi(z_X) 
    \leq \big(1 + g_X^\sigma(X-Y)^{1/2} \big) \varphi(z_Y) 
    \lesssim \big(1 + g_X^\sigma(X-Y) \big) \varphi(z_Y), 
    $$
    that is, the temperance of $\varphi$ and we have 
$\tau \csp \varphi(z_X) \lesssim \big(1 +
    g_X^\sigma(X-Y)^{1/2} \big) \lambsc(Y)$, which, along with \eqref{eq: est xi},
    yields the temperance of $\lambsc$:
    $$
    \lambsc(X) \lesssim \big(1 +
    g_X^\sigma(X-Y)^{1/2} \big) \lambsc(Y) 
    \lesssim \big(1 +
    g_X^\sigma(X-Y) \big) \lambsc(Y).
    $$
    
    For the temperance of $g$ we need to prove
    \begin{align*}
      (1+\ctp \csp) |z_T'| +  \csp |(z_T)_N| + \frac{|\zeta_T|}{\lambsc(X)}
      \lesssim \big(1 + g_X^\sigma(X-Y)\big)^N 
      \Big( (1+\ctp \csp) |z_T'| +  \csp |(z_T)_N| + \frac{|\zeta_T|}{\lambsc(Y)}\Big),
    \end{align*}
    for $T=(z_T, \zeta_T) \in \W$. 
    To conclude it suffices to prove
    \begin{align*}
      \lambsc(Y) \lesssim \big(1 + g_X^\sigma(X-Y)\big)^N  \lambsc(X).
    \end{align*} 
    We have 
    \begin{align}
      \label{eq: est eta}
      |\zeta_Y| \leq |\zeta_X| + |\zeta_X - \zeta_Y| 
      &\leq |\zeta_X| 
      + \frac{|\zeta_X - \zeta_Y|}{\csp}  \tau \csp \varphi(z_X)
      \\
      &\leq |\zeta_X| 
      + \bigg(2\frac{|\zeta_X' - \zeta_Y'|}{1+ \csp\ctp}
      + \frac{|(\zeta_X)_N- (\zeta_Y)_N|}{\csp}\bigg)
      \tau \csp \varphi(z_X) \notag\\
      &\lesssim  \big(1 + g_X^\sigma(X-Y)^{1/2}  \big) \lambsc(X) . 
        \notag
    \end{align}
    It thus remains to prove
    \begin{align}
      \label{eq: final estimate}
      \tau \csp \varphi(z_Y) \lesssim \big(1 + g_X^\sigma(X-Y)\big)^N  \lambsc(X).
    \end{align} 
    First, if $(1+\csp \ctp) |z_X'-z_Y'|   + \csp|(z_X)_N-(z_Y)_N| \leq 1$, then $\varphi(z_Y) \lesssim
    \varphi(z_X)$, arguing as in \eqref{eq: varphi y ->
      x}. Estimate~\eqref{eq: final estimate} is then clear.  Second,
    if $(1+\csp \ctp) |z_X'-z_Y'|   + \csp|(z_X)_N-(z_Y)_N| \geq 1$, which
    implies $2|z_X-z_Y|   \geq 1/\csp$,  with
    \eqref{eq: cond inf-sup} we write
    \begin{align*}
      \tau \csp \varphi(z_Y) 
      &\leq \tau \csp \varphi(z_X)^{k+1} \lesssim \frac{\lambsc(X)^{k+1}}{(\tau \csp)^k}
      \lesssim \Big( \frac{\lambsc(X)}{\tau \csp}\Big)^k \lambsc(X)
      \lesssim \Big( |z_X-z_Y| \frac{\lambsc(X)}{\tau}\Big)^k \lambsc(X)\\
      &\lesssim \big(1 + g_X^\sigma(X-Y)^{1/2} \big)^k   \lambsc(X),
    \end{align*} 
    since $\tau\geq \tauast\geq 1$. In any case, we thus have
\begin{align*}
      \tau \csp \varphi(z_Y) 
      \lesssim \big(1 + g_X^\sigma(X-Y)^{1/2} \big)^k   \lambsc(X),
    \end{align*}
    which concludes the proof.
\hfill \qedsymbol \endproof

\subsubsection{Proof of Lemma~\ref{lemma: ttau in the calculus}}
\label{proof: lemma: ttau in the calculus}
  We have $\ttau \lesssim \lambsc$ (\resp
  $\ttau \lesssim \lambsct$) and $d_\zeta \ttau =0$. Only
  differentiations of $\ttau$ with respect to $z$ thus need to be
  considered.  Recalling that
  $\ttau = \tau \csp \varphi_{\csp,\ctp}$ we find that, for
  $\mi= (\mi', \mi_N) \in \N^N$, we can write
  $\d_{z}^{\mi} \ttau(\y')$ as a linear
combination of terms of the form
  \begin{align*}
    &\tau \csp^{1+ k} \varphi_{\csp,\ctp}(z)\  
   \prod_{j=1}^k \d_{z}^{\mi^{(j)}} \cpsi (z)
      = \tau \csp^{1+ k} \ctp^{|\mi'|}\  \varphi_{\csp,\ctp}(z)\ 
      \prod_{j=1}^k \d_{z}^{\mi^{(j)}} \psi (\ctp z', z_N) ,
  \end{align*}
with $\mi^{(1)} + \cdots +
\mi^{(k)} = \mi$, $|\mi^{(j)}|\geq 1$, $j=1,\dots, k$, and $k \leq |\mi|$,
 implying, as $\csp \geq 1$,  
    $|\d_{z}^{\mi}\ttau(\y')| 
    \lesssim \ttau (\y')
    \csp^{|\mi|} \ctp^{|\mi'|}
    \lesssim \ttau (\y')
    \csp^{\mi_N} (\ctp \csp)^{|\mi'|}$,
as $\Norm{\psi^{(\ell)}}{L^\infty}\leq C$ for any $\ell\in \N$,  
which yields the results.\hfill \qedsymbol \endproof

\subsubsection{Proof of Lemma~\ref{lemma: from on calculus to the other}{}}
\label{proof lemma: from on calculus to the other}
  For $\mi = (\mi', \mi_N) \in \N^N$ and $\smi' \in \N^{N-1}$,
  we may write  $\d_{z}^{\mi} \d_{\zeta'}^{\smi'} a (\y')$ as a linear
  combination of terms of the form, 
  \begin{align*}
   b (\y') =  \big(  \prod_{j=1}^{k} \d_{z}^{\mi^{(j)}}  \htau_{p_j}(\y')\big) 
    \
    \d_{z}^{\mi^{(a)}} \d_{\zeta'}^{\smi'} \d_{\hatt}^{\mi^{(b)} }
    \hat{a} \big(\kappa (\y')\big),
  \end{align*}
  for some $\mi^{(b)} \in \N^N$, with $k = |\mi^{(b)}|$, with $\mi = \mi^{(a)} +
  \mi^{(1)} + \cdots + \mi^{(k)}$,   $|\mi^{(j)}|\geq 1$, and
  where $p_j \in \{1, \dots, N\}$, $j=1,\dots,k$.
   Using Lemma~\ref{lemma: htau in the calculus} and Definition~\ref{def: semi-classical
    symbol}, and we obtain
  \begin{align*}
    | b (\y') | 
    &\lesssim 
      \prod_{j=1}^{k}  
      \big(\lambsct (1+\ctp \csp)^{|\mi^{(j)\prime}|} \csp^{|\mi^{(j)}_N|}\big)\
      (|\htau(\y')| +|\zeta'|)^{m-|\smi'| - |\mi^{(b)}|}   \\
      &\lesssim
      (1+ \ctp  \csp)^{|\mi^{(1)\prime} | + \cdots + |\mi^{(k) \prime}|}
      \csp^{\mi^{(1)}_N  + \cdots + \mi^{(k)}_N  }\lambsct^k 
         (|\htau(\y')| +|\zeta'|)^{m-|\smi'| - |\mi^{(b)}|}  
    \\
     &\lesssim (1+ \ctp  \csp)^{|\mi'|} \csp^{\mi_N}
       \lambsct^k 
       (|\htau(\y')| +|\zeta'|)^{m-k - |\smi'| },
  \end{align*}
  as $\csp \geq 1$. 
  If $\hat{a}$ is polynomial then the term $b (\y')$ vanishes if
  $m-|\smi'| - |\mi^{(b)}| <0$. Thus if $m-|\smi'| - |\mi^{(b)}|
  \geq 0$ and, as $|\htau|\lesssim \ttau$ in $\U$,  we obtain
  \begin{align*}
    | b (\y') |\leq (1+ \ctp  \csp)^{|\mi'|} \csp^{\mi_N}
       \lambsct^{m-|\smi'|},
    \end{align*}
  which yields the result.
   If $\hat{a}$ is not polynomial and if  we have $\ttau \asymp |\htau|$,
   we obtain the same estimation, even if $m-|\smi'| - |\mi^{(b)}| <0$.
\hfill \qedsymbol \endproof

\subsubsection{Proof of Lemma~\ref{lemma: equivalence norms}}
\label{proof lemma: equivalence norms}
By applying~\eqref{eq: boundedness operators}, we have 
  \begin{equation*}
    \Norm{\Lambsct^{m} \ttau^{r} u}{+}
    \lesssim
   \Norm{\Opt(\ttau^{r}\lambsct^{m}) u}{+} .
  \end{equation*}
  Next, we write $\Opt(\ttau^{r}\lambsct^{m})  = \Op(\lambsct^{m})
  \ttau^{r}  + \csp R$, with $R \in \Psi( \ttau^{r}
  \lambsct^{m-1}, \gt)$ by the tangential calculus we have introduced.
  This yields, as $\ttau^r \in S(\lambsct^r, \gt)$, 
  \begin{equation*}
    \Norm{\Opt(\ttau^{r}\lambsct^{m}) u}{+} 
    \lesssim
    \Norm{\Op(\lambsct^{m}) \ttau^{r} u}{+} 
    + \csp  \Norm{\Opt(\ttau^{r}\lambsct^{m-1}) u}{+},
  \end{equation*}
  which yields \eqref{eq: equivalence norms} by choosing $\tau$ \suff
  large. 
  Estimation~\eqref{eq: equivalence norms 2} follows the same.
\hfill \qedsymbol \endproof 
\subsubsection{Proof of Lemma~\ref{lemma: different form norm}}
\label{proof lemma: different form norm}
By definition of the Sobolev norms introduced in Section~\ref{sec: operator sobolev} we 
  have 
  \begin{align*}
    \Norm{\ttau^r w}{m,m',\ttau} 
    \asymp \sum_{j=0}^m \Norm{D_{x_d}^j ( \ttau^r w)}
    {0, m + m' - j,\ttau} 
    = \sum_{j=0}^m \Norm{ \Lambsct^{m+m'-j} D_{x_d}^j ( \ttau^r w)}
    {+}.
  \end{align*}
  Let $m''_j \in \R$. We have
  $[\Lambsct^{m''_j}, D_{x_d}^j] \in \sum_{i=1}^j \csp^i
  \Psi(\lambsct^{m''_j}, \gt) D_{x_d}^{j-i}$, 
  from the tangential calculus we have introduced.  With
  Lemma~\ref{lemma: ttau in the calculus} we have
  $[\ttau^r, \Lambsct^{m''_j}] \in (1+ \ctp\csp) \Psi(\ttau^r
  \Lambsct^{m''_j-1}, \gt)$. With the same lemma, 
  for $r'_j
  \in \R$ we also have $[\ttau^{r'_j} , D_{x_d}^j] \in \sum_{i=1}^j
  \csp^i\Psi( \ttau^{r'_j}, \gt) D_{x_d}^{j-i}$.  For $r= r'_j+
  r''_j$, and $m+ m' -j = m_j''+ m_j'''$, with $r_j', r_j'' \in
  \R$ and $m_j'', m_j''' \in \R$, we thus obtain, by
  Proposition~\ref{prop: Sobolev regularity pseudo}, 
  \begin{align*}
    \Norm{\ttau^r w}{m,m',\ttau} 
    &\geq 
    \sum_{j=0}^m \Norm{ \ttau^{r'_j}\Lambsct^{m''_j} D_{x_d}^j
    ( \ttau^{r''_j} \Lambsct^{m'''_j} w)} {+} 
    - C'\sum_{j=1}^m \sum_{i=1}^j \csp^i\Norm{  \ttau^{r}
      D_{x_d}^{j-i} w}{0, m + m' - j,\ttau}\\
    &\qquad - C''\sum_{j=0}^m \csp \Norm{  \ttau^{r}
      D_{x_d}^{j} w}{0, m + m' - j-1,\ttau}\\
    &\geq 
    \sum_{j=0}^m \Norm{ \ttau^{r'_j}\Lambsct^{m''_j} D_{x_d}^j
    ( \ttau^{r''_j} \Lambsct^{m'''_j} w)} {+} 
    - C'\sum_{j=0}^{m-1} \sum_{i=1}^m \csp^i\Norm{  \ttau^{r}
      D_{x_d}^{j} w}{0, m + m' - j-i,\ttau}\\
    &\qquad - C''\sum_{j=0}^m \csp \Norm{  \ttau^{r}
      D_{x_d}^{j} w}{0, m + m' - j-1,\ttau}.
\end{align*}
With the argument given in \eqref{eq: usual semi-classical argument},
we have 
\begin{align*}
  \sum_{j=0}^{m-1} \sum_{i=1}^m \csp^i\Norm{  \ttau^{r}
  D_{x_d}^{j} w}{0, m + m' - j-i,\ttau}
  + \sum_{j=0}^m \csp \Norm{  \ttau^{r}
      D_{x_d}^{j} w}{0, m + m' - j-1,\ttau}
 \ll \Norm{\ttau^r w}{m,m',\ttau}, 
\end{align*}
for $\tau$ chosen \suff large, 
and we thus find
\begin{align*}
    \Norm{\ttau^r w}{m,m',\ttau} 
    &\gtrsim 
    \sum_{j=0}^m \Norm{ \ttau^{r'_j}\Lambsct^{m''_j} D_{x_d}^j
    ( \ttau^{r''_j} \Lambsct^{m'''_j} w)} {+} ,
\end{align*}
for $\tau$ chosen \suff large.
Similarly, we find that 
  \begin{align*}
    \Norm{\ttau^r w}{m,m',\ttau} 
    &\lesssim 
    \sum_{j=0}^m \Norm{ \ttau^{r'_j}\Lambsct^{m''_j} D_{x_d}^j
    ( \ttau^{r''_j} \Lambsct^{m'''_j} w)} {+} ,
\end{align*}
for $\tau$ chosen \suff large. The result for the trace norms is
obtained arguing the same.  \hfill \qedsymbol \endproof

%%%%%%%%%%%%%%%
%% Section             %
%%%%%%%%%%%%%%%
\section{Elliptic and sub-elliptic estimates at the boundary
  $(0,S_0)\times \d\Omega$}
\label{sec: elliptic sub-elliptic estimates boundary x}

\subsection{Roots with negative imaginary part: a perfect elliptic
  estimate}
\label{sec: perfect elliptic estimate}
For $z_0 \in \d Z$, $V$ denotes the \nhd introduced in
Section~\ref{sec: local setting}.  We recall that
$\MtV = V \times \R^{N-1} \times [\tauast,+\infty) \times
[1,+\infty)\times [0,1]$.

 Let $\ell(\y)\in \Symbsc^{m,0}$, with $\y = (z,\zeta,\tau, \csp,\ctp)$ and
  $m\geq 1$, be polynomial in $\zeta_N$ with homogeneous coefficients
  in $(\zeta',\htau)$ and $L = \ell(z,D_z,\tau,\csp,\ctp)$.

%%%%%%%%%%%%%%
% lemma             %
%%%%%%%%%%%%%%
\begin{lemma}
   \label{lemma: microlocal elliptic estimate}
   Let $\U$ be a conic open subset of
  $\MtV$. 
   We assume that, for $\ell(\y',\zeta_N)$
   viewed as a polynomial in $\zeta_N$, for $\y' \in \U$, 
   \begin{itemize}
     \item the leading coefficient is $1$;
     \item all roots of $\ell(\y',\zeta_N)=0$ have negative imaginary
       part. 
   \end{itemize}
   Let $\chi(\y') \in S(1,\gt)$, be homogeneous of degree zero and \st
   $\supp(\chi) \subset \U$. Then, for any $M \in \N$,
   there exist $C>0$, $\tau_0\geq \tauast$, $\csp_0\geq 1$ \st
  \begin{equation*}
    \Norm{\Opt(\chi)w}{m,0,\ttau}
    +\norm{\trace(\Opt(\chi) w)}{m-1,1/2,\ttau}
    \leq C\Big(
    \Norm{L \Opt(\chi)w}{+}
    +\Norm{w}{m,-M,\ttau}\Big),
\end{equation*}
for $w\in \S(\Rpb)$
and $\tau\geq \tau_0$, $\csp\geq \csp_0$, $\ctp \in [0,1]$.
\end{lemma}
This lemma can be proven by adapting the proof of \cite[Lemma
6.5]{BLR:13} to the semi-classical calculus we use here. For the
notion of homogeneity for symbols and conic sets in the present calculus, we refer to
Section~\ref{sec: conic homog}. 

\subsection{Sub-ellipticity quantification}

For $z_0 \in \d Z$, $V$ denotes the \nhd introduced in
Section~\ref{sec: local setting}.  We let the function $\psi$ be as
introduced in Section~\ref{sec: Carleman boundary x}, satisfying
\eqref{eq: cond psi} and \eqref{eq: cond psi2}, and we recall that  $\cpsi (z) =
\psi(\ctp z '\! ,z_N)$
and $\varphi(z) = \exp(\csp \cpsi(z))$.
We also recall that $\lambsc^2 =
  \ttau^2 + |\zeta|^2$ with $\ttau(\y') = \tau \csp \varphi(z)$. 
%%%%%%%%%%%%%%%%%%%%%%%%
% proposition          %
%%%%%%%%%%%%%%%%%%%%%%%%
\begin{proposition}
  \label{prop: simple characteristic estimate}
  Let $\ell(z,\zeta)$ be polynomial of degree $m$ in $\zeta$, with
  smooth coefficient in $z$. We assume that for any $M\in \R^N
  \setminus \{0\}$, the
  symbol $\ell$ satisfies the simple-characteristic property in
  direction $M$ in a \nhd of $V$ (see Definition~\ref{def: simple characteristics}).  There exist $C>0$ and $\csp_0\geq 1$
  \st,
  \begin{align*}
  &|\ell(z,\zeta + i \htau(\y'))|^2 
  + \tau \varphi(z) |\cpsi'(z)|^2
  \big\{ \Re \ell(z,\zeta + i \htau(\y')), \Im  \ell(z,\zeta + i \htau(\y'))\big\}
    \geq C \lambsc^{2m},
\end{align*}
for $z \in \ovl{V}$, $\zeta\in \R^N$, $\tau\geq \tauast$,
$\csp\geq\csp_0$ and $\ctp\in [0,1]$.
\end{proposition}
%%%% proof of proposition
\begin{proof}
We have $0<  C_0 \leq |\cpsi'(z)|\leq C_1$ for $z \in \ovl{V}$ and we
set $K = \{ M \in \R^N; C_0 \leq |M| \leq C_1\}$. 
As $V$ is assumed bounded (see section~\ref{sec: local setting}), we
consider the compact set 
\begin{align*}
  \Con = \big\{ (z,\zeta, \theta,M); \ \theta^2 + |\zeta|^2 =1, \ \ z \in
  \ovl{V}, \ \zeta \in \R^N, \ \theta \in \R_+, \ M \in K\big\}.
\end{align*}
We define
\begin{align}
  \label{eq: def f appendix}
  f(z, \zeta, \theta,M) = |\ell(z,\zeta + i \theta M)|^2 + 
   |\theta M|^2
  \big| \inp{\ell'_\zeta (z,\zeta + i \theta M)}{M}  \big|^2.
 \end{align}
As the simple-characteristic property holds in direction $M$ for all $M \in
K$ and $z \in \ovl{V}$, we have
\begin{align*}
  f(z, \zeta, \theta,M) \geq C>0, \qquad (z, \zeta, \theta,M) \in \Con.
\end{align*}
By homogeneity, we obtain
\begin{align}
  \label{eq: positivity f appendix}
  f(z, \zeta, \theta,M) \geq C(\theta^2 + |\zeta|^2)^{m}, \qquad z \in
  \ovl{V},\ \zeta \in \R^N,\ \theta \in \R_+,\ M  \in K.
\end{align}

We compute the following Poisson bracket, with $\htau (\y') = \tau d
\varphi(z)$, 
\begin{align*}
  \{ \Re \ell(z,\zeta + i \htau(\y')), \Im  \ell(z,\zeta + i \htau(\y'))\} 
  &=
  \frac{1}{2i}
  \{ \ovl{\ell(z,\zeta + i \htau(\y'))},  \ell(z,\zeta + i \htau(\y'))\} 
 = \Theta_{\ell,\varphi}(z,\zeta,\tau),
\end{align*}
with 
\begin{align*}
\Theta_{\ell, \phi}(z,\zeta,t) 
  &:= t \sum_{j,k} \d^2_{z_jz_k} \phi(z)\, 
    \d_{\zeta_j} \ell(z,\zeta + i t d \phi(z))\, 
    \d_{\zeta_k} \ovl{\ell(z,\zeta + i t d \phi (z))} \\
    &\quad + \Im \sum_j \d_{z_j} \ell(z,\zeta+ i t d \phi (z)) \, 
     \d_{\zeta_j}\ovl{\ell (z,\zeta + i t d \phi (z))}. \nonumber
\end{align*}
Note that $\Theta_{\ell, \phi}(z,\zeta,t)$ is homogeneous of degree
$2m-1$ in $(\zeta,t)$.
With $\varphi(z) = \exp(\csp \cpsi(z))$ we obtain 
\begin{align*}
\Theta_{\ell,\varphi}(z,\zeta,\tau) = \Theta_{\ell,\cpsi}(z,\zeta,\ttau(\y')) 
  + \csp \ttau(\y') |\inp{\ell'_\zeta(z,\zeta+ i \htau(\y'))}{\cpsi'(z)}|^2.
\end{align*}

We thus find, with $f$ defined in \eqref{eq: def f appendix}, 
\begin{align}
  \label{eq: minoration subellipticity}
  &|\ell(z,\zeta + i \htau(\y'))|^2 
  + \tau \varphi(z) |\cpsi'(z)|^2
  \big\{ \Re \ell(z,\zeta + i \htau(\y')), \Im  \ell(z,\zeta + i
    \htau(\y')) \big\} \\
  &\qquad =|\ell(z,\zeta + i \htau(\y'))|^2 
    + \tau \varphi(z) |\cpsi'(z)|^2 \Theta_{\ell,\varphi}(z,\zeta,\tau) \notag\\
     &\qquad =|\ell(z,\zeta + i \htau(\y'))|^2 
       + |\ttau(\y') \cpsi'(z)|^2 |\inp{\ell'_\zeta(z,\zeta+ i
       \htau(\y'))}{\cpsi'(z)}|^2
       + \tau \varphi(z) |\cpsi'(z)|^2
        \Theta_{\ell,\cpsi}(z,\zeta,\ttau(\y')) \notag\\
  &\qquad = f(z, \zeta, \ttau(\y'),\cpsi'(z)) + \tau \varphi(z) |\cpsi'(z)|^2
        \Theta_{\ell,\cpsi}(z,\zeta,\ttau(\y')).\notag
\end{align}
Now, as  $\cpsi'(z)$ remains in the compact set $K$, we find, by \eqref{eq: positivity f
  appendix},
\begin{align}
  \label{eq: minoration f appendix}
  f(z, \zeta, \ttau(\y'),\cpsi'(z)) \gtrsim
(\htau(\y')^2 + |\zeta|^2)^m \gtrsim\lambsc^{2m}, 
 \end{align}
since $|\htau(\y')| = |\cpsi'| \ttau(\y') \geq C_0 \ttau(\y')$. 
The homogeneity of $\Theta_{\ell,\cpsi}(z,\zeta,\ttau(\y'))$ gives
\begin{align*}
  |\tau \varphi(z) |\cpsi'(z)|^2
        \Theta_{\ell,\cpsi}(z,\zeta,\ttau(\y'))| 
  \lesssim \gamma^{-1} \ttau(\y') \lambsc^{2m-1} \lesssim \gamma^{-1} \lambsc^{2m}.
\end{align*}
With \eqref{eq: minoration subellipticity} and \eqref{eq: minoration f
  appendix}, we obtain the result for $\csp$ chosen \suff large.
\end{proof}

We recall the definition of $q_k(\y)$ given in \eqref{eq: def symbol
  qk}, we have $q_k (y) = p_k(z,\zeta + i \htau(\y'))$ with
$p_k(z,\zeta) = (-1)^k i \sigma^2 + \xi_d^2 + r(x,\xi')$.
From Proposition~\ref{prop: simple-char} and Proposition~\ref{prop:
  simple characteristic estimate}, we have the following result, in any dimension
$N\geq 2$, that is, $d\geq 1$.
%%%%%%%%%%%%%%%%%%%%%%%%
% corollary              %
%%%%%%%%%%%%%%%%%%%%%%%%
\begin{corollary}
  \label{cor: subellipticity qk}
  Let $k=1$ or $2$. There exist $C>0$ and $\csp_0\geq 1$ \st
  \begin{align*}
  &|q_k(\y)|^2 
  + \tau \varphi(z) |\cpsi '(z)|^2
  \{ \Re q_k(\y'), \Im  q_k(\y)\}
    \geq C \lambsc^{4},\qquad \y = (z,\zeta, \tau, \csp, \ctp),
\end{align*}
for $z \in \ovl{V}$, $\zeta\in \R^N$, $\tau\geq \tauast$,
$\csp\geq\csp_0$, and $\ctp \in [0,1]$, and where $\htau(\y') = \tau \csp \varphi(z) d \cpsi(z)$.
\end{corollary}

\subsection{Estimates for first-order factors}

In this section, we shall assume that $\U_0 \subset \Mt$ is a conic open
set where the symbol $q_k(\y) = p_k(z,\zeta + i \htau(\y'))$ can be factorized
into two {\em smooth} first-order terms,
\begin{align*}
  q_k (\y) = q_{k,-} (\y)  q_{k,+}(\y) , \quad q_{k,\pm}(\y)  = \xi_d - \rho_{k,\pm}(\y').
\end{align*}
By Lemma~\ref{lemma: root behaviors1} we see that $q_{k,-} $ is
elliptic, and $q_{k,+}$ may vanish.

\subsubsection{A root with a positive imaginary part: an elliptic
  estimate with a trace term}
\label{sec: A root with a positive imaginary part}

Here, we further assume that there exists a second conic open set $\U_1 \subset \U_0$
\st  $\Im \rho_{k,+}(\y') \gtrsim \lambsct$, for $\y' \in
\U_1$.  We let $\chi, \uchi \in S(1, \gt)$ be homogeneous of degree zero
and \st 
\begin{align*}
  \uchi \equiv 1\ \text{on a \cnhd of}\ \supp(\chi), \qquad \supp(\uchi) \subset
  \U_1.
\end{align*} 
With $Q_{k,+}= D_{x_d} - \Opt^w(\uchi^2 \rho_{k,+})$ we have the
following estimation.
%%%%%%%%%%%%%%%%%%%%%%%% 
  % lemma                %
  %%%%%%%%%%%%%%%%%%%%%%%% 
  \begin{lemma}
    \label{lemma: elliptic estimate Qk+}
    Let $\ell \in \R$ and $M \in \N$. There exist $\tau_1\geq \tauast$,
    $\csp_1\geq 1$, and $C>0$ , \st
    \begin{align}
      \label{eq: ellitptic estimate Qk+}
      \Norm{ \Opt(\chi) w}{1,\ell, \ttau} 
      \leq C \Big( \Norm{Q_{k,+} \Opt(\chi) w}{0,\ell,\ttau}
  +  \norm{\trace(\Opt(\chi) w)}{0,\ell+ 1/2,\ttau} 
    + \Norm{w}{0,-M, \ttau} 
    \Big), 
  \end{align}
  for $\tau \geq \tau_1$, $\csp \geq \csp_1$, $\ctp \in [0,1]$,
  and for $w \in \S(\Rpb)$.
\end{lemma}
\begin{proof}
  We write $Q = A - i B$ with 
    \begin{align*}
      A =  D_{x_d} - \Opt^w(\uchi^2\Re \rho_{k,+}), \quad B = 
      \Opt^w(\uchi^2\Im \rho_{k,+}),
    \end{align*}
 both formally selfadjoint.

  We use a pseudo-differential multiplier technique, following for instance \cite{LR-L:13}
  and compute, with $s = 2 \ell +1$,
  \begin{align*}
    &2 \Re \scp{Q \Opt(\chi) w}{- i \Lambsct^s \Opt(\chi) w}_+\\
    &\qquad = 
      -2 \Re \scp{A\Opt(\chi) w}{ i \Lambsct^s \Opt(\chi)w}_+
                + 2 \Re \scp{B\Opt(\chi) w}{ \Lambsct^s \Opt(\chi)w}_+\\
    &\qquad =
      - \scp{i [A ,  \Lambsct^s
      ]\Opt(\chi) w}{\Opt(\chi) w}_+
      + 2 \Re \scp{B \Opt(\chi) w}{\Lambsct^s \Opt(\chi) w}_+\\
      &\qquad \quad - \scp{\Lambsct^s \Opt(\chi) w\brx}{\Opt(\chi)
        w\brx}_{L^2(\R^{N-1})}\\
      & \qquad \geq  2 \Re \scp{B \Opt(\chi) w}{\Lambsct^s \Opt(\chi) w}_+
        - C  \csp^2 \Norm{\Opt(\chi) w}{0,\ell + 1/2, \ttau}^2\\
        &\qquad \quad - \norm{\Opt(\chi) w \brx}{0,\ell + 1/2,\ttau}^2,
  \end{align*}
which by the (microlocal) G{\aa}rding inequality of
Proposition~\ref{prop: microlocal tangential Garding}  yields, for any $M \in \N$, 
\begin{align*}
  \Re \scp{\Lambsct^{\ell} Q \Opt(\chi) w}{- i \Lambsct^{\ell+1} \Opt(\chi) w}_+
      + \norm{\Opt(\chi) w\brx}{0,\ell + 1/2,\ttau}^2
      + \Norm{w}{0,-M, \ttau}^2
      \gtrsim 
      \Norm{\Opt(\chi) w}{0,\ell + 1, \ttau}^2,
\end{align*}
for $\tau$ and $\csp$ chosen \suff large. 
Then, with the Young inequality, we obtain
\begin{align*}
  \Norm{Q \Opt(\chi) w}{0,\ell, \ttau}
      + \norm{\Opt(\chi) w\brx}{0,\ell + 1/2,\ttau}
      + \Norm{w}{0,-M, \ttau}
      \gtrsim 
      \Norm{\Opt(\chi) w}{0,\ell + 1, \ttau}.
\end{align*}
Finally, 
observing that we have 
\begin{align*}
  \Norm{D_{x_d} \Opt(\chi) w}{0,\ell, \ttau} 
  \lesssim \Norm{Q \Opt(\chi) w}{0,\ell, \ttau} + \Norm{\Opt(\chi) w}{0,\ell + 1, \ttau}
\end{align*}
allows one to conclude the proof.
\end{proof}

\subsubsection{Transmitted sub-ellipticity}

In $\U_0$ where $q_k(\y)$ is smoothly factorized, $q_k (\y) = q_{k,-} (\y)  q_{k,+}(\y)$,
we now describe how the sub-ellipticity property of
Corollary~\ref{cor: subellipticity qk} is ``transmitted'' to the
nonelliptic factor $q_{k,+}$. 
%%%%%%%%%%%%%%%%%%%%%%%%
% proposition          %
%%%%%%%%%%%%%%%%%%%%%%%%
\begin{proposition}
  \label{prop: sub-ellipticity q k+}
  Let $k=1$ or $2$. There exist $\csp_0 \geq 1$, $\alpha_0>0$,  and $C>0$ \st
  \begin{align}
  \label{eq: E- estim Qk+ subellipticity}
  \alpha \csp \ttau^{-1} |\Im \rho_{k,+}|^2  +\big\{\xi_d - \Re \rho_{k,+}, -\Im
  \rho_{k,+}\big\} \geq C \csp \ttau^{-1} \lambsct^2,\qquad \y' \in \U_0,
 \end{align}
 for $\csp \geq \csp_0$ and $\alpha\geq \alpha_0$.
\end{proposition}
\begin{proof}
We write 
\begin{align*}
   2i \{ \Re q_k, \Im  q_k\} = \{ \ovl{q_k}, q_k\}
  = |q_{k,-}|^2 \{ \ovl{q_{k,+}}, q_{k,+}\}
  + |q_{k,+}|^2 \{ \ovl{q_{k,-}}, q_{k,-}\}
  + 2 i \Im \big(\{\ovl{q_{k,-}} , q_{k,+}\} \ovl{q_{k,+}} q_{k,-} \big),
\end{align*}
yielding 
 \begin{align*}
   \{ \Re q_k, \Im  q_k\} 
  = |q_{k,-}|^2 \{ \Re q_{k,+},\Im q_{k,+}\}
  + |q_{k,+}|^2 \{ \Re q_{k,-}, \Im q_{k,-}\}
  +\Im \big(\{\ovl{q_{k,-}} , q_{k,+}\} \ovl{q_{k,+}} q_{k,-} \big).
\end{align*}
We write, for $M>0$, 
\begin{align*}
  |q_{k,+}|^2 \big| \{ \Re q_{k,-}, \Im q_{k,-}\}\big|
  +\big|\Im \big(\{\ovl{q_{k,-}} , q_{k,+}\} \ovl{q_{k,+}} q_{k,-}
  \big)\big|
  &\leq C\Big( \gamma \lambsc |q_{k,+}|^2 + \gamma \lambsc^2 |q_{k,+}|\big)\\
  &\leq C'(1 +M) \gamma \lambsc |q_{k,+}|^2 + C'M^{-1} \gamma \lambsc^3.
\end{align*}
For $M>0$ and $\csp_0 \geq 1$ chosen \suff large we obtain, with Corollary~\ref{cor: subellipticity qk},
\begin{align*}
  &|q_{k,-}(\y)|^2 \Big(
    |q_{k,+}(\y)|^2 
  + \tau \varphi(z) |\cpsi'(z)|^2
  \{ \Re q_{k,+},\Im q_{k,+}\}(\y) \Big)
    \geq C \lambsc^{4} - C' (1+M) \ttau\lambsc |q_{k,+}|^2,
\end{align*}
In $\U_0$ we have $|q_{k,-}(\y)| \asymp \lambsc$, as $q_{k,-}$ is
elliptic which gives
\begin{align*}
  \alpha  |q_{k,+}(\y)|^2 
  + \tau \varphi(z) |\cpsi'(z)|^2
  \{ \Re q_{k,+},\Im q_{k,+}\}(\y) 
    \geq C \lambsc^{2},\qquad \y' \in \U_0, \xi_d \in \R,
\end{align*}
for $\alpha>0$ chosen \suff large. If we now choose $\xi_d = \Re
\rho_{k,+}(\y')$ we then obtain the result.
\end{proof}

\subsubsection{A root with a vanishing imaginary part: a sub-elliptic estimate with a trace term.}
\label{sec: A root with a vanishing imaginary part}

Here, we consider as above a conic open set $\U_0 \subset \Mt$, \st
the symbol $q_k(\y) = p_k(z,\zeta + i \htau(\y'))$ can be factorized
into two {\em smooth} first-order terms,
$q_k (\y) = q_{k,-} (\y) q_{k,+}(\y)$.  We let
$\chi, \uchi \in S(1, \gt)$ be as above and we recall that
$Q_{k,+} := D_{x_d} - \Opt^w(\uchi^2 \rho_{k,+})$.
  We have the following lemma.
  %%%%%%%%%%%%%%%%%%%%%%%% 
  % lemma                %
  %%%%%%%%%%%%%%%%%%%%%%%% 
  \begin{lemma}
    \label{lemma: sub-ellitptic estimate Qk+}
    Let $\ell, m \in \R$ and $M\in \N$. There exist $\tau_1\geq \tauast$,
    $\csp_1 \geq 1$, and $C>0$ , \st
    \begin{align}
      \label{eq: sub-ellitptic estimate Qk+}
      \csp^{1/2}\Norm{\ttau^{m-1/2} \Opt(\chi) w}{1,\ell, \ttau} 
      \leq C \Big( \Norm{\ttau^{m} Q_{k,+} \Opt(\chi) w}{0,\ell,\ttau}
  +  \norm{\trace(\ttau^{m}\Opt(\chi) w)}{0,\ell+ 1/2,\ttau} 
  + \Norm{w}{0,-M, \ttau}
    \Big), 
  \end{align}
   for $\tau \geq \tau_1$, $\csp \geq \csp_1$, $\ctp \in [0,1]$,
  and for $w \in \S(\Rpb)$.
\end{lemma}
\begin{proof}
 For concision, we write $Q$ in place of $Q_{k,+}$. We decompose
 $Q$ according to $Q = A + i B $
with
 \begin{align}
   \label{eq:  decomposition Qk+, A B}
   A = D_{x_d} 
     - \Opt^w (\uchi^2\Re \rho_{k,+}) \in
     \PsiOpsc^{1,0}, 
     \qquad B = 
    -  \Opt^w( \uchi^2\Im \rho_{k,+}) \in
     \PsiOpsc^{0,1} = \Psi(\lambsct, \gt). 
  \end{align}
  Observe that both $A$ and $B$ are formally selfadjoint.

  We set $\wlm =  \ttau^{m} \Lambsct^\ell \Opt(\chi)  w$ and compute
\begin{align}
  \label{eq: Carleman Qk+}
  \Norm{Q  \wlm}{+}^2 
  = \Norm{(A  + i B)\wlm}{+}^2
    = \Norm{A\wlm}{+}^2 
    + \Norm{B \wlm}{+}^2 
  + 2 \Re \scp{A \wlm}{i B\wlm}_{+}
\end{align}
From the form of $A$ and $B$  given in \eqref{eq:  decomposition Qk+,
  A B} we find
\begin{align*}
  2 \Re \scp{A \wlm}{ i B \wlm}_{+}
  &= i \scp{[A,B] \wlm }{\wlm}_{+}
    - \scp{\Opt^w( \uchi^2\Im\rho_{k,+}){\wlm}\brx }
  { {\wlm}\brx}_{L^2(\R^{N-1})}.
\end{align*}
yielding, with \eqref{eq: Carleman Qk+}, 
\begin{align*}
  &\Norm{Q  \wlm}{+}^2 
  + \norm{\trace( \ttau^{m}\Opt(\chi)w)}{0,\ell+ 1/2,\ttau}^2\\
  &\qquad \gtrsim \Norm{A \wlm}{+}^2 
    + \Norm{B \wlm}{+}^2 + i \scp{[A,B] \wlm }{\wlm}_{+}\\
  &\qquad\gtrsim 
    \Norm{A \wlm}{+}^2 + 
  \scp{ \big(\alpha \csp \ttau^{-1} B^2 +i [A,B] \big) \wlm}
    {\wlm}_{+}\\
  &\qquad\gtrsim 
    \Norm{A \wlm}{+}^2 + 
  \scp{ \Lambsct^{\ell}\ttau^{m}\big(\alpha \csp \ttau^{-1} B^2 +i [A,B] \big)
    \ttau^{m}\Lambsct^{\ell} \Opt(\chi)w}
    {\Opt(\chi)w}_{+}, 
\end{align*}
for $\alpha=\alpha_0$ with $\alpha_0$ given by Proposition~\ref{prop:
  sub-ellipticity q k+}, and for $\tau$  \st 
$\alpha \csp \ttau^{-1}\leq 1$.
As the principal symbol of $\Lambsct^{\ell}\ttau^{m}\big(\alpha \csp \ttau^{-1} B^2 +i [A,B] \big)
    \ttau^{m}\Lambsct^{\ell}$ is given, in a \cnhd of $\supp(\chi)$, where
$\uchi \equiv 1$,  by
\begin{align*}
  \ttau^{2m} \lambsct^{2\ell} 
  \big(\alpha \csp \ttau^{-1} (\Im \rho_{k,+})^2 
  + \{\xi_d - \Re \rho_{k,+}, -\Im\rho_{k,+}\}\big) 
  \in S(\csp \ttau^{2m-1}\lambsct^{2+ 2 \ell},\gt),
\end{align*}
then Proposition~\ref{prop: sub-ellipticity q k+} and the (microlocal) G{\aa}rding
inequality of Proposition~\ref{prop: microlocal tangential Garding} yield, for any $M \in N$, by choosing $\tau$ and $\csp$ \suff large,  
\begin{align*}
  \Norm{Q  \wlm}{+} 
  + \norm{\trace(\ttau^{m}\Opt(\chi)w)}{0,\ell+ 1/2,\ttau}
  +  \Norm{w}{0,-M}
  &\gtrsim \Norm{A \wlm}{+}   
    + \csp^{1/2} \Norm{\ttau^{m-1/2}\Opt(\chi)w}{0,1+\ell,\ttau}.
\end{align*}
From the form of $A$ in \eqref{eq:  decomposition Qk+, A B} we have 
\begin{align*}
  \csp^{1/2}\Norm{\ttau^{-1/2} D_{x_d} \wlm }{+} 
  &\lesssim
  \Norm{A \wlm}{+}    
    + \csp^{1/2}
    \Norm{\ttau^{-1/2}\wlm}{0,1,\ttau}\\
  &\lesssim
   \Norm{A \wlm}{+}    
    + \csp^{1/2}
    \Norm{\ttau^{m-1/2}\Opt(\chi)w}{0,1+\ell,\ttau}.
\end{align*}
We thus obtain
\begin{align*}
  &\Norm{Q  \wlm}{+} 
  + \norm{\trace(\ttau^{m}\Opt(\chi)w)}{0,\ell+ 1/2,\ttau}
  +  \Norm{w}{0,-M}\\
  &\qquad \gtrsim 
     \csp^{1/2}  \big(\Norm{\ttau^{m-1/2}\Opt(\chi)w}{0,1+\ell,\ttau}
    +  \Norm{\ttau^{-1/2} D_{x_d} \wlm}{+}\big)\\
    &\qquad 
    \gtrsim 
     \csp^{1/2}  \Norm{\ttau^{m-1/2}\Opt(\chi)w}{1,\ell,\ttau},
\end{align*}
by choosing $\tau$ \suff large and using Lemma~\ref{lemma: different
  form norm}. This concludes the proof.
\end{proof}

\subsection{Estimate concatenations}
\label{sec: Estimate concatenations}

Let $\U_0$ be on conic open set of $\Mt$. 
Let $\uchi(\y') \in S(1,\gt)$ be homogeneous of degree zero \st
$\supp(\uchi) \subset \U_0$.
Let $\rho^{(k)}(\y') \in S(\lambsct,\gt)$, $k=1,2$, be homogeneous of
degree one in $\U_0$ and define $Q^{(k)} = D_{x_d} -
\Opt^w(\uchi^2 \rho^{(k)})$. The operators $Q_{k,\pm}$, $k=1,2$,
defined in what precedes
and in Section~\ref{sec: Carleman boundary x} are of this form.
Above, for such operators, we proved some microlocal
estimates of the form
\begin{multline}
  \label{eq: generic 1st order estimate}
    \csp^{\alpha_k/2}\Norm{\ttau^{\alpha_k(m-1/2)} \Opt(\chi) w}{1,\ell,
    \ttau}
    + \delta_k \norm{\trace(\ttau^{m \alpha_k} \Opt(\chi) w)}{0,\ell+ 1/2,\ttau} \\
      \leq C \Big( \Norm{\ttau^{m \alpha_k}  Q^{(k)}\Opt(\chi) w}{0,\ell,\ttau}
  +  (1 -\delta_k)\norm{\trace(\ttau^{m \alpha_k} \Opt(\chi) w)}{0,\ell+ 1/2,\ttau} 
  + \Norm{w}{0,-M, \ttau}
    \Big),
  \end{multline}
  with $\delta_k = (1-\alpha_k) (1-\beta_k)$ and
  $\alpha_k,\beta_k \in \{0,1\}$, $\ell, m \in \R$, and  where $\chi \in S(1,\gt)$,
  homogeneous of degree zero and \st $\uchi\equiv 1$ on a \cnhd of $\supp(\chi)$.

If $\alpha_k=0$ and $\beta_k=0$ the estimate reads
\begin{align*}
  \Norm{\Opt(\chi) w}{1,\ell,\ttau}
  + \norm{\trace(\Opt(\chi) w)}{0,\ell+ 1/2,\ttau} 
  \leq C \Big( \Norm{Q^{(k)} \Opt(\chi) w}{0,\ell,\ttau}
  + \Norm{w}{0,-M, \ttau}
  \Big).
  \end{align*}
This is a
perfect elliptic estimate that holds  if $\rho^{(k)}$ is in the lower half complex plane --see
Lemma~\ref{lemma: microlocal elliptic estimate}.

If $\alpha_k=0$ and $\beta_k=1$ the estimate reads
\begin{align*}
  \Norm{\Opt(\chi) w}{1,\ell,\ttau}
  \leq C \Big( \Norm{Q^{(k)} \Opt(\chi) w}{0,\ell,\ttau}
  + \norm{\trace(\Opt(\chi) w)}{0,\ell+ 1/2,\ttau} 
  + \Norm{w}{0,-M, \ttau}
  \Big).
  \end{align*}
This is an elliptic estimate, yet with a trace observation term in the \rhs,  that holds  if $\rho^{(k)}$
is in the upper half complex plane --see
Lemma~\ref{lemma: elliptic estimate Qk+}. 

Finally, if $\alpha_k=1$, independently of the value of $\beta_k$ we
have 
\begin{align*}
    \csp^{1/2}\Norm{\ttau^{m-1/2} \Opt(\chi) w}{1,\ell,
    \ttau}
      \leq C \Big( \Norm{\ttau^{m} Q^{(k)} \Opt(\chi) w}{0,\ell,\ttau}
  + \norm{\trace(\ttau^{m}\Opt(\chi) w)}{0,\ell+ 1/2,\ttau} 
  + \Norm{w}{0,-M, \ttau}
    \Big).
  \end{align*}
  This estimate is characterized by the loss of a half derivative and
  a boundary observation term in the \rhs; such an estimate is proven
  in Lemma~\ref{lemma: sub-ellitptic estimate Qk+} when the root
  $\rho^{(k)}$ may cross the real axis.

\bigskip
We shall now describe how such estimates can be concatenated, as this is
often done in the course of the proof of Theorem~\ref{theorem: Carleman boundary x}. 
%%%%%%%%%%%%%%%%%%%%%%%%
% proposition          %
%%%%%%%%%%%%%%%%%%%%%%%%
\begin{proposition}
  \label{prop: estimate concatenation}
  Let $\ell \in \R$ and $M \in \N$.
   Let $Q^{(k)}$ be defined as above, for $k=1,2$. Let $\tau_0\geq
   \tauast$, $\csp_0\geq 1$
  and $C>0$ \st estimate \eqref{eq: generic 1st order estimate} holds,
  with $\ell, m \in \R$, 
  with $\alpha_k,
  \beta_k \in\{0,1\}$, for
  both $k=1$ and $2$, for $\tau \geq \tau_0$, $\csp \geq \csp_0$, $\ctp \in [0,1]$,
  and for $w \in \S(\Rpb)$.  We assume that $\alpha_1 \leq \alpha_2$ and
  $1 -\delta_1\leq 1 -\delta_2$. 

  Let $\chi \in S(1,\gt)$,
  be homogeneous of degree zero and \st $\uchi\equiv 1$ on $\supp(\chi)$.
  There exist $\tau_1\geq \tauast$, $\csp_1\geq 1$
  and $C>0$ \st the following estimate for the second-order operator
  $Q^{(1)} Q^{(2)}$ holds, 
  \begin{align*}
    &\csp^{(\alpha_1+\alpha_2)/2}\Norm{\ttau^{-(\alpha_1+\alpha_2)/2}
       \Opt(\chi) w}{2,\ell,\ttau}
    + \norm{\trace(\Opt(\chi) w)}{1,\ell+ 1/2,\ttau} 
    \\
    &\qquad\leq C \Big(
      \Norm{Q^{(1)} Q^{(2)} \Opt(\chi)  w}{0,\ell,\ttau}
    + (1 -\delta_1)\norm{\trace(\Opt(\chi) w)}{1,\ell+ 1/2,\ttau} \notag\\
    &\quad\qquad + (1-\delta_2) \csp^{\alpha_1/2}
    \norm{\trace(\ttau^{-\alpha_1/2} \Opt(\chi) w)}{0,\ell+ 3/2,\ttau}
     + \Norm{w}{2,-M, \ttau} \Big),
      \notag
  \end{align*}
   for $\tau \geq \tau_1$, $\csp \geq \csp_1$, $\ctp \in [0,1]$,
  and for $w \in \S(\Rpb)$. 
\end{proposition}
Note that the assumptions made on $\alpha_k$ and $1- \delta_k$, $k=1,2$, imply
that $Q^{(1)}$ yields an estimate of better quality than that
associated with $Q^{(2)}$.

%%%% proof of proposition
\begin{proof}
  We introduce $\chi_1 \in S(1,\gt)$ that is \st
  $\chi_1 \equiv 1$ on $\supp(\chi)$ and $\uchi\equiv 1$ on
  $\supp(\chi_1)$. For concision, we write $\Xi= \Opt(\chi)$ and $\Xi_1 =
  \Opt(\chi_1)$. 
  Here, $M$ will denote an arbitrary large integer whose value may
  change from one line to the other.

 Using $Q^{(2)}\Xi w$ as the unknown function in
  the estimate \eqref{eq: generic 1st order estimate} for $Q^{(1)}$, with $m=0$ gives, 
\begin{align}
  \label{eq: concatenate 1}
  &\csp^{\alpha_1/2}\Norm{\ttau^{-\alpha_1/2} Q^{(2)}\Xi w}{1,\ell,
  \ttau} 
  + \delta_1  \norm{\trace( Q^{(2)} \Xi w)}{0,\ell+ 1/2,\ttau} \\
  &\qquad \lesssim 
    \csp^{\alpha_1/2}\Norm{\ttau^{-\alpha_1/2} \Xi_1 Q^{(2)}\Xi w}{1,\ell,
  \ttau} 
  + \delta_1  \norm{\trace(  \Xi_1 Q^{(2)} \Xi w)}{0,\ell+ 1/2,\ttau} 
  \notag\\
   &\qquad \quad  + \Norm{w}{1,-M, \ttau} + \norm{\trace(w)}{1,-M, \ttau}  \notag\\
    &\qquad \lesssim  \Norm{Q^{(1)} \Xi_1 Q^{(2)}\Xi w}{0,\ell,\ttau}
      + (1 -\delta_1)\norm{\trace(\Xi_1 Q^{(2)}\Xi w)}{0,\ell+ 1/2,\ttau} 
      + \Norm{w}{1,-M, \ttau} + \norm{\trace(w)}{1,-M, \ttau}  
      \notag \\
  &\qquad \lesssim  \Norm{Q^{(1)} Q^{(2)} \Xi  w}{0,\ell,\ttau}
    + (1 -\delta_1)\norm{\trace(\Xi w)}{1,\ell+ 1/2,\ttau} 
    + \Norm{w}{2,-M, \ttau}.
    \notag
  \end{align}
  
  Observe now that we can write, using that $D_{x_d} - Q^{(2)}
  \in\Psi(\lambsct, \gt)$, 
  \begin{align*}
   & (1-\delta_2) \csp^{\alpha_1/2}
   \norm{\trace(\ttau^{-\alpha_1/2}  \Xi w)}{1,\ell + 1/2,\ttau} \\
    &\qquad \lesssim  
      \delta_1 \csp^{\alpha_1/2}\norm{\trace( \ttau^{-\alpha_1/2}  Q^{(2)} \Xi w)}{0,\ell+ 1/2,\ttau}
      +  (1-\delta_1) \csp^{\alpha_1/2} \norm{\trace(\ttau^{-\alpha_1/2} \Xi w)}{1,\ell+  1/2,\ttau}\\
      &\qquad\quad + (1-\delta_2) 
    \csp^{\alpha_1/2}\norm{\trace(\ttau^{-\alpha_1/2} \Xi w)}{0,\ell+ 3/2,\ttau}\\
    &\qquad \lesssim  
      \delta_1 \norm{\trace( Q^{(2)} \Xi w)}{0,\ell+ 1/2,\ttau}
      +  (1-\delta_1)  \norm{\trace(\Xi w)}{1,\ell+  1/2,\ttau}
      + (1-\delta_2) 
    \csp^{\alpha_1/2}\norm{\trace(\ttau^{-\alpha_1/2} \Xi w)}{0,\ell+ 3/2,\ttau}.
  \end{align*}
With this estimate and \eqref{eq: concatenate 1}, we thus obtain 
\begin{align}
  \label{eq: concatenate 1 A}
  &\csp^{\alpha_1/2} \Big(
    \Norm{\ttau^{-\alpha_1/2} Q^{(2)}\Xi w}{1,\ell,\ttau} 
    + (1-\delta_2) \norm{\trace(\ttau^{-\alpha_1/2}  \Xi w)}{1,\ell + 1/2,\ttau} \Big)\\
  &\qquad \lesssim  \Norm{Q^{(1)} Q^{(2)} \Xi  w}{0,\ell,\ttau}
    + (1 -\delta_1)\norm{\trace(\Xi w)}{1,\ell+ 1/2,\ttau} 
    + (1-\delta_2) 
    \csp^{\alpha_1/2}\norm{\trace(\ttau^{-\alpha_1/2} \Xi w)}{0,\ell+ 3/2,\ttau}
  \notag\\
  &\qquad \quad 
    + \Norm{w}{2,-M, \ttau}.
    \notag
  \end{align}
  Up to creating error terms, we shall now modify this inequality to
  be able to apply the estimate \eqref{eq: generic 1st order estimate}
  associated with $Q^{(2)}$. We write 
  \begin{align*}
    &\Norm{\ttau^{-\alpha_1/2} Q^{(2)} \Xi_1 D_{x_d}\Xi w}{0,\ell,\ttau} 
    + (1-\delta_2) \norm{\trace(\ttau^{-\alpha_1/2}  \Xi_1 D_{x_d}\Xi
      w)}{0,\ell + 1/2,\ttau}\\
    &\qquad \lesssim
      \Norm{\ttau^{-\alpha_1/2} Q^{(2)}D_{x_d}\Xi w}{0,\ell,\ttau} 
    + (1-\delta_2) \norm{\trace(\ttau^{-\alpha_1/2}  D_{x_d}\Xi
      w)}{0,\ell + 1/2,\ttau}
    + \Norm{w}{1,-M, \ttau} + \norm{\trace(w)}{1,-M, \ttau}\\
    &\qquad \lesssim
      \Norm{\ttau^{-\alpha_1/2} Q^{(2)}\Xi w}{1,\ell,\ttau} 
    + (1-\delta_2) \norm{\trace(\ttau^{-\alpha_1/2}  \Xi
      w)}{1,\ell + 1/2,\ttau}
      + \Norm{w}{1,-M, \ttau} + \norm{\trace(w)}{1,-M, \ttau}\\
    &\qquad\quad + \csp\Norm{\ttau^{-\alpha_1/2} \Xi w}{1,\ell,\ttau},
  \end{align*}
  using that $[D_{x_d}, Q^{(2)}] \in \csp \PsiOpsc^{1,0}$ and using  Lemma~\ref{lemma: different
  form norm}.
  Hence with~\eqref{eq: concatenate 1 A} we have
  \begin{align}
     \label{eq: concatenate 1 B}
    &\csp^{\alpha_1/2} \Big(\Norm{\ttau^{-\alpha_1/2} Q^{(2)} \Xi_1 D_{x_d}\Xi w}{0,\ell,\ttau} 
    + (1-\delta_2) \norm{\trace(\ttau^{-\alpha_1/2}  \Xi_1 D_{x_d}\Xi
      w)}{0,\ell + 1/2,\ttau}\\
    &\qquad + \Norm{\ttau^{-\alpha_1/2} Q^{(2)}\Xi w}{0,\ell+1,\ttau} 
    + (1-\delta_2) \norm{\trace(\ttau^{-\alpha_1/2}  \Xi w)}{0,\ell + 3/2,\ttau}\Big) \notag\\
     &\quad\lesssim \csp^{\alpha_1/2} \Big(
    \Norm{\ttau^{-\alpha_1/2} Q^{(2)}\Xi w}{1,\ell,\ttau} 
    + (1-\delta_2) \norm{\trace(\ttau^{-\alpha_1/2}  \Xi w)}{1,\ell
       + 1/2,\ttau} \notag\\
        &\quad\qquad \quad + \csp \Norm{\ttau^{-\alpha_1/2} \Xi w}{1,\ell,\ttau} 
          \Big)+ \Norm{w}{1,-M, \ttau} + \norm{\trace(w)}{1,-M, \ttau}\notag\\
      &\quad\lesssim 
       \Norm{Q^{(1)} Q^{(2)} \Xi  w}{0,\ell,\ttau}
    + (1 -\delta_1)\norm{\trace(\Xi w)}{1,\ell+ 1/2,\ttau} 
    + (1-\delta_2) 
        \csp^{\alpha_1/2}
    \norm{\trace(\ttau^{-\alpha_1/2} \Xi w)}{0,\ell+ 3/2,\ttau}\notag\\
    &\quad\qquad \quad
      + \csp^{1+ \alpha_1/2}\Norm{\ttau^{-\alpha_1/2} \Xi w}{1,\ell,\ttau} 
    + \Norm{w}{2,-M, \ttau} \notag
  \end{align}
  We write, with Lemma~\ref{lemma: different form norm}, for $\tau$ chosen \suff large,
  \begin{align}
    \label{eq: concatenate 1 C}
    \Norm{\ttau^{-(\alpha_1+\alpha_2)/2} \Xi w}{2,\ell,\ttau}
   &\asymp  \Norm{\ttau^{-(\alpha_1+\alpha_2)/2}
       D_{x_d} \Xi w}{1,\ell,  \ttau}
      + \Norm{\ttau^{-(\alpha_1+\alpha_2)/2}
      \Xi w}{1,\ell+1, \ttau}\\
   &\lesssim \Norm{\ttau^{-(\alpha_1+\alpha_2)/2}
      \Xi_1 D_{x_d}\Xi w}{1,\ell, \ttau}
      + \Norm{\ttau^{-(\alpha_1+\alpha_2)/2} \Xi w}{1,\ell+1, \ttau}
      + \Norm{w}{2,-M, \ttau},\notag
  \end{align}
  and 
  \begin{align}
    \label{eq: concatenate 1 D}
    \norm{\trace(\ttau^{-\alpha_1/2} 
       \Xi w)}{1,\ell+ 1/2,\ttau} 
    &\asymp \norm{\trace(\ttau^{-\alpha_1/2} 
       D_{x_d}\Xi w)}{0,\ell+ 1/2,\ttau} 
    + \norm{\trace(\ttau^{-\alpha_1/2} 
      \Xi w)}{0,\ell+ 3/2,\ttau}\\
    &\lesssim \norm{\trace(\ttau^{-\alpha_1/2} \Xi_1 D_{x_d}
       \Xi w)}{0,\ell+ 1/2,\ttau} 
    + \norm{\trace(\ttau^{-\alpha_1/2} 
      \Xi w)}{0,\ell+ 3/2,\ttau}
    + \norm{\trace(w)}{1,-M,\ttau}.\notag
  \end{align}
  Applying now  estimate \eqref{eq: generic 1st order estimate}
  associated with $Q^{(2)}$ to  $D_{x_d}\Xi w$ and $w$,
  with $m=-\alpha_1 /2$, using that $\alpha_1  = \alpha_1 \alpha_2$, 
  we obtain
  \begin{align}
    \label{eq: concatenate 1 E}
    &\csp^{(\alpha_1+\alpha_2)/2}\Norm{\ttau^{-(\alpha_1+\alpha_2)/2}
      \Xi_1 D_{x_d}\Xi w}{1,\ell, \ttau}
    + \delta_2 \csp^{\alpha_1/2}  \norm{\trace(\ttau^{-\alpha_1/2} \Xi_1 D_{x_d}
       \Xi w)}{0,\ell+ 1/2,\ttau} \\
    &\qquad \lesssim 
       \csp^{\alpha_1/2} \Big(\Norm{\ttau^{-\alpha_1/2} Q^{(2)} \Xi_1 D_{x_d}\Xi w}{0,\ell,\ttau} 
    + (1-\delta_2) \norm{\trace(\ttau^{-\alpha_1/2}  \Xi_1 D_{x_d}\Xi
      w)}{0,\ell + 1/2,\ttau}\Big) +  \Norm{w}{1,-M, \ttau} \notag
  \end{align}
 and 
 \begin{align}
   \label{eq: concatenate 1 F}
    &\csp^{(\alpha_1+\alpha_2)/2}
      \Norm{\ttau^{-(\alpha_1+\alpha_2)/2} \Xi w}{1,\ell+1, \ttau}
      + \delta_2 \csp^{\alpha_1/2}  \norm{\trace(\ttau^{-\alpha_1/2} 
      \Xi w)}{0,\ell+ 3/2,\ttau}
      \\
   &\qquad \lesssim \csp^{\alpha_1/2} \Big( 
     \Norm{\ttau^{-\alpha_1/2} Q^{(2)}\Xi w}{0,\ell+1,\ttau} 
    + (1-\delta_2) 
     \norm{\trace(\ttau^{-\alpha_1/2}  \Xi w)}{0,\ell + 3/2,\ttau}
     \Big)  +  \Norm{w}{0,-M, \ttau}.\notag
     \end{align}
With \eqref{eq: concatenate 1 C}--\eqref{eq: concatenate 1 F}, we achieve
\begin{align*}
    &\csp^{(\alpha_1+\alpha_2)/2}\Norm{\ttau^{-(\alpha_1+\alpha_2)/2}
       \Xi w}{2,\ell,\ttau}
    + \delta_2 \csp^{\alpha_1/2} \norm{\trace(\ttau^{-\alpha_1/2} 
       \Xi w)}{1,\ell+ 1/2,\ttau} \\
    &\lesssim
    \csp^{\alpha_1/2} \Big(\Norm{\ttau^{-\alpha_1/2} Q^{(2)} \Xi_1 D_{x_d}\Xi w}{0,\ell,\ttau} 
    + (1-\delta_2) \norm{\trace(\ttau^{-\alpha_1/2}  \Xi_1 D_{x_d}\Xi
      w)}{0,\ell + 1/2,\ttau}\\
    &\qquad \qquad \qquad + \Norm{\ttau^{-\alpha_1/2} Q^{(2)}\Xi w}{0,\ell+1,\ttau} 
    + (1-\delta_2) \norm{\trace(\ttau^{-\alpha_1/2}  \Xi w)}{0,\ell
      + 3/2,\ttau}\Big) + \Norm{w}{2,-M, \ttau}.
  \end{align*}
  Combining this latter estimate with \eqref{eq: concatenate 1 B} we obtain
  \begin{align*}
    &\csp^{(\alpha_1+\alpha_2)/2}\Norm{\ttau^{-(\alpha_1+\alpha_2)/2}
       \Xi w}{2,\ell,\ttau}
    + \delta_2 \csp^{\alpha_1/2} \norm{\trace(\ttau^{-\alpha_1/2} 
       \Xi w)}{1,\ell+ 1/2,\ttau} \\
    &\qquad\lesssim
      \Norm{Q^{(1)} Q^{(2)} \Xi  w}{0,\ell,\ttau}
    + (1 -\delta_1)\norm{\trace(\Xi w)}{1,\ell+ 1/2,\ttau} 
    + (1-\delta_2) \csp^{\alpha_1/2}
    \norm{\trace(\ttau^{-\alpha_1/2} \Xi w)}{0,\ell+ 3/2,\ttau}\notag\\
    &\quad\qquad 
      + \csp^{1+ \alpha_1/2} \Norm{\ttau^{-\alpha_1/2} \Xi w}{1,\ell,\ttau} 
    + \Norm{w}{2,-M, \ttau} \notag,
  \end{align*}
  which,  with  the usual
  semi-classical inequality \eqref{eq: usual semi-classical argument}
  \begin{align*}
    &\csp^{(\alpha_1+\alpha_2)/2}\Norm{\ttau^{-(\alpha_1+\alpha_2)/2}
       \Xi w}{2,\ell,\ttau}
    + \delta_2 \csp^{\alpha_1/2} \norm{\trace(\ttau^{-\alpha_1/2} 
       \Xi w)}{1,\ell+ 1/2,\ttau} \\
    &\qquad\lesssim
      \Norm{Q^{(1)} Q^{(2)} \Xi  w}{0,\ell,\ttau}
    + (1 -\delta_1)\norm{\trace(\Xi w)}{1,\ell+ 1/2,\ttau} 
    + (1-\delta_2) \csp^{\alpha_1/2}
    \norm{\trace(\ttau^{-\alpha_1/2} \Xi w)}{0,\ell+ 3/2,\ttau}\notag\\
    &\quad\qquad 
     + \Norm{w}{2,-M, \ttau}.
      \notag
  \end{align*}
  Let us now consider two cases:
  \begin{description}
    \item[\bfseries Case $\bld{\alpha_1 =1}$] Then $\delta_1 =0$ and $\alpha_2=1$. 
      We thus have the term $\norm{\trace(\Xi w)}{1,\ell+ 1/2,\ttau}$
      in the \rhs of the estimation and the sought result then holds.
    \item [\bfseries Case $\bld{\alpha_1 =0}$]
      Then we write 
      \begin{align*}
    \norm{\trace( \Xi w)}{1,\ell+ 1/2,\ttau} 
    &\lesssim \norm{\trace( Q^{(2)} \Xi w)}{0,\ell+ 1/2,\ttau} 
    + \norm{\trace( \Xi w)}{0,\ell+ 3/2,\ttau} \\
    &\lesssim \norm{\trace( Q^{(2)} \Xi w)}{0,\ell+ 1/2,\ttau} 
    + \delta_2 \norm{\trace(\Xi w)}{1,\ell+ 1/2,\ttau}
      + (1-\delta_2) \norm{\trace(\Xi w)}{0,\ell+ 3/2,\ttau}. 
  \end{align*}
  which leads to 
  \begin{align*}
    \delta_1 \norm{\trace(\Xi w)}{1,\ell+ 1/2,\ttau} 
    &\lesssim \delta_1 \norm{\trace( Q^{(2)} \Xi w)}{0,\ell+ 1/2,\ttau} 
    + \delta_2 \norm{\trace(\Xi w)}{1,\ell+ 1/2,\ttau}
      + (1-\delta_2) \norm{\trace(\Xi w)}{0,\ell+ 3/2,\ttau}. 
  \end{align*}
  Recalling that the term $\delta_1 \norm{\trace( Q^{(2)} \Xi w)}{0,\ell+
    1/2,\ttau} $ can be found in the \lhs of \eqref{eq: concatenate 1}, 
  We thus obtain 
  \begin{align*}
    &\csp^{\alpha_2/2}\Norm{\ttau^{-\alpha_2/2}
       \Xi w}{2,\ell,\ttau}
    + (\delta_1 + \delta_2) \norm{\trace(\Xi w)}{1,\ell+ 1/2,\ttau} \\
    &\qquad\lesssim
      \Norm{Q^{(1)} Q^{(2)} \Xi  w}{0,\ell,\ttau}
    + (1 -\delta_1)\norm{\trace(\Xi w)}{1,\ell+ 1/2,\ttau} 
    + (1-\delta_2)  \norm{\trace(\Xi w)}{0,\ell+ 3/2,\ttau}\notag\\
    &\quad\qquad 
     + \Norm{w}{2,-M, \ttau}.
      \notag
  \end{align*}
  If $\delta_1 + \delta_2 >0$ we then have the sought estimate in the
  case $\alpha_1 =0$. If $\delta_1 + \delta_2=0$ then the term
  $\norm{\trace(\Xi w)}{1,\ell+ 1/2,\ttau}$
  can be found in the \rhs of the estimation and can thus be
  ``artificially'' added in the \lhs.
    \end{description}
    This concludes the proof of Proposition~\ref{prop: estimate concatenation}.
\end{proof}

\bigskip
We now show how to obtain microlocal estimates for some products of two 
factors of order two. 
%%%%%%%%%%%%%%%%%%%%%%%%
% proposition          %
%%%%%%%%%%%%%%%%%%%%%%%%
\begin{proposition}
  \label{prop:  estimate order 4 operator}
  Let assume that $Q^{-}(z,D_z,\tau,\csp, \ctp) \in \PsiOpsc^{2,0}$ fulfills the requirement of
  Lemma~\ref{lemma: microlocal elliptic estimate} in  some conic open
  subset $\U$.  Let $Q^{+}(z,D_z,\tau,\csp, \ctp)
  \in \PsiOpsc^{2,0}$ be \st, there exist $\tau_0\geq \tauast$,
  $\csp_0 \geq 1$
  and $C>0$ \st, for $\ell \in \{0,1,2\}$ and all $\chi \in S(1,\gt)$, homogeneous of
  degree zero, with $\supp(\chi) \subset \U$, for $\Xi = \Opt(\chi)$, 
  \begin{align}
    \label{eq: estimate Q+}
    &\csp^{(\alpha_1+\alpha_2)/2}\Norm{\ttau^{-(\alpha_1+\alpha_2)/2}
       \Xi w}{2,\ell,\ttau}
    + \norm{\trace(\Xi w)}{1,\ell+ 1/2,\ttau} 
    \\
    &\qquad\leq C \Big(
      \Norm{Q^{+} \Xi  w}{0,\ell,\ttau}
    + (1 -\delta_1)\norm{\trace(\Xi w)}{1,\ell+ 1/2,\ttau} \notag\\
    &\quad\qquad + (1-\delta_2) \csp^{\alpha_1/2}
    \norm{\trace(\ttau^{-\alpha_1/2} \Xi w)}{0,\ell+ 3/2,\ttau}
     + \Norm{w}{2,-M, \ttau} \Big),
      \notag
  \end{align}
   for $\tau \geq \tau_0$, $\csp \geq \csp_0$, $\ctp \in [0,1]$,
  and for $w \in \S(\Rpb)$, where $\alpha_1, \alpha_2 \in \{0,1\}$ and
  $\delta_1, \delta_2 \in \{0,1\}$ with $\alpha_1 \leq
  \alpha_2$, $1-\delta_1 \leq 1-\delta_2$ and  moreover $\delta_k=0$ if
  $\alpha_k=1$, $k=1,2$. We also assume that $Q^+ \Xi =D_{x_d}^2
  \Xi + T_{1,1} \Xi$ with
  $T_{1,1}\in \PsiOpsc^{1,1}$.

  \medskip Let $M \in \N$ and let $\chi \in S(1,\gt)$ be as above.  In
  the case $\alpha_1 + \alpha_2 =2$, we furthermore assume that, for
  any $M\in \N$,
  $[D_{x_d} + i \htau_{\xi_d}, Q^{+}] \Opt(\chi_1)= (1+ \ctp\csp)
  R_{2,0} \Opt(\chi_1) + R_{2,-M}$,
  with $R_{2,0}\in \PsiOpsc^{2,0}$ and $R_{2,-M}\in \PsiOpsc^{2,-M}$,
  if $\chi_1 \in S(1,\gt)$ is homogeneous of degree 0 and \st
  $\chi_1 \equiv 1$ in a \cnhd of $\supp(\chi)$ and
  $\supp(\chi_1)\subset \U$.

There exist
  $\tau_1\geq \tauast$, $\csp_1 \geq 1$, $\ctp_1\in (0,1]$, and $C>0$ \st
  \begin{align}
    \label{eq: estimate Q-Q+}
    &\csp^{(\alpha_1+\alpha_2)/2}\Norm{\ttau^{-(\alpha_1+\alpha_2)/2}
       \Xi w}{4,0,\ttau}
    + \norm{\trace(\Xi w)}{3, 1/2,\ttau} 
    \\
    &\qquad\leq C \Big(
      \Norm{Q^- Q^{+} \Xi  w}{+}
    + (1 -\delta_1)\norm{\trace(\Xi w)}{1,5/2,\ttau} \notag\\
    &\quad\qquad + (1-\delta_2) \csp^{\alpha_1/2}
    \norm{\trace(\ttau^{-\alpha_1/2} \Xi w)}{0,7/2,\ttau}
     + \Norm{w}{4,-M, \ttau} \Big),
      \notag
  \end{align}
   for $\tau \geq \tau_1$, $\csp \geq \csp_1$, $\ctp \in [0,\ctp_1]$,
  and for $w \in \S(\Rpb)$. In the case $\alpha_1 + \alpha_2\leq 1$,
  we can take $\ctp_1 =1$. 
\end{proposition}
In Section~\ref{sec: Carleman boundary x}, for example, this proposition will be applied to $Q^{+} =
Q_{1,+}Q_{2,+}$ for which an estimation of the form of \eqref{eq: estimate Q+} will
hold by Proposition~\ref{prop: estimate concatenation}.
Note that this proposition, in the case  $\alpha_1 + \alpha_2 =2$,  is one instance where it is important to
take $\ctp>0$ \suff small. 

%%%% proof of proposition
\begin{proof}
  We introduce $\chi_1 \in S(1,\gt)$ that is \st
  $\chi_1 \equiv 1$ on $\supp(\chi)$ and 
  $\supp(\chi_1)\subset \U$. For concision, we write $\Xi= \Opt(\chi)$ and $\Xi_1 =
  \Opt(\chi_1)$. 
  Here, $M$ will denote an arbitrary large integer whose value may
  change from one line to the other.

  Using $Q^{+}\Xi w$ as the unknown function in
  the estimate of Lemma~\ref{lemma: microlocal elliptic estimate} for
  the operator $Q^{-}$: 
  \begin{align}
    \label{eq: concatenate 2 A}
     &\Norm{ Q^{+}\Xi w}{2,0,\ttau}
    +\norm{\trace(Q^{+}\Xi w)}{1,1/2,\ttau}\\
    &\qquad \lesssim
    \Norm{\Xi_1 Q^{+}\Xi w}{2,0,\ttau}
    +\norm{\trace(\Xi_1 Q^{+}\Xi w)}{1,1/2,\ttau}
    + \Norm{w}{4,-M,\ttau}\notag\\
    &\qquad \lesssim
    \Norm{Q^{-} \Xi_1 Q^{+}\Xi w}{+}
    + \Norm{w}{4,-M,\ttau}\notag\\
    &\qquad \lesssim
    \Norm{Q^{-} Q^{+}\Xi w}{+}
    + \Norm{w}{4,-M,\ttau}.\notag
\end{align}

Combining \eqref{eq: estimate Q+}, for $\ell=2$, with \eqref{eq: concatenate 2 A} we find
\begin{align}
    \label{eq: concatenate 2 B}
     &\Norm{ Q^{+}\Xi w}{2,0,\ttau}
    +\norm{\trace(Q^{+}\Xi w)}{1,1/2,\ttau} 
       + \norm{\trace(\Xi w)}{1,5/2,\ttau} \\
    &\qquad \lesssim
    \Norm{Q^{-} Q^{+}\Xi w}{+}
      + (1 -\delta_1)\norm{\trace(\Xi w)}{1,5/2,\ttau} 
      + (1-\delta_2) \csp^{\alpha_1/2}
    \norm{\trace(\ttau^{-\alpha_1/2} \Xi w)}{0,7/2,\ttau}\notag\\
    &\qquad\quad + \Norm{w}{4,-M,\ttau}.\notag
\end{align}
We now make the following claim whose proof is given below.
%%%%%%%%%%%%%%%%%%%%%%%%
% sub-lemma                %
%%%%%%%%%%%%%%%%%%%%%%%%
\begin{lemma}
  \label{sublemma: trace Q+}
There exists $C>0$ \st 
\begin{align*} 
  \norm{\trace(\Xi  v)}{3,1/2,\ttau} 
  & \leq C\big(
    \norm{\trace( Q^{+}\Xi w)}{1,1/2,\ttau}
    + \norm{\trace(\Xi v)}{1,5/2,\ttau} 
      \big).
\end{align*}
\end{lemma}
This gives
\begin{multline}
    \label{eq: concatenate 2 C}
     \Norm{ Q^{+}\Xi w}{2,0,\ttau}
    + \norm{\trace(\Xi  v)}{3,1/2,\ttau}  
    \lesssim
    \Norm{Q^{-} Q^{+}\Xi w}{+}
      + (1 -\delta_1)\norm{\trace(\Xi w)}{1,5/2,\ttau} \\
      + (1-\delta_2) \csp^{\alpha_1/2}
    \norm{\trace(\ttau^{-\alpha_1/2} \Xi w)}{0,7/2,\ttau}
    + \Norm{w}{4,-M,\ttau}.
\end{multline}

\bigskip
First, we treat the case $\alpha_1 + \alpha_2 \leq 1$. As $\alpha_1
\leq \alpha_2$ then $\alpha_1=0$. 
We write 
\begin{align*}
 &\sum_{j=0}^2 \Big( \Norm{ Q^{+}\Xi_1 D_{x_d}^j \Xi w}{0,2-j,\ttau}
  +\norm{\trace(\Xi_1 D_{x_d}^j \Xi  v)}{1,5/2-j,\ttau}  \Big)\\
  &\qquad \lesssim \sum_{j=0}^2 \Big( \Norm{ Q^{+} D_{x_d}^j \Xi w}{0,2-j,\ttau}
  +\norm{\trace(D_{x_d}^j \Xi  v)}{1,5/2-j,\ttau}  \Big)
  + \Norm{w}{4,-M,\ttau}\\
  &\qquad \lesssim \sum_{j=0}^2 \Big( \Norm{ D_{x_d}^j Q^{+}\Xi w}{0,2-j,\ttau}
  +\norm{\trace(D_{x_d}^j \Xi  v)}{1,5/2-j,\ttau}  \Big)
    + \csp \Norm{\Xi w}{3,0}
  + \Norm{w}{4,-M,\ttau}\\
   &\qquad \lesssim \Norm{ Q^{+}\Xi w}{2,0,\ttau}
    + \norm{\trace(\Xi  v)}{3,1/2,\ttau}  
    + \csp \Norm{\Xi w}{3,0}
  + \Norm{w}{4,-M,\ttau}.
\end{align*}
With \eqref{eq: concatenate 2 C} we then find
\begin{align*}
    &\sum_{j=0}^2 \Big( \Norm{ Q^{+}\Xi_1 D_{x_d}^j \Xi w}{0,2-j,\ttau}
  +\norm{\trace(\Xi_1 D_{x_d}^j \Xi  v)}{1,5/2-j,\ttau}  \Big)\\
  &\qquad \lesssim
    \Norm{Q^{-} Q^{+}\Xi w}{+}
      + (1 -\delta_1)\norm{\trace(\Xi w)}{1,5/2,\ttau} 
      + (1-\delta_2) 
    \norm{\trace(\Xi w)}{0,7/2,\ttau}\notag\\
    &\qquad \quad + \csp \Norm{\Xi w}{3,0}
    + \Norm{w}{4,-M,\ttau}.\notag
\end{align*}
Now, applying \eqref{eq: estimate Q+} with $\ell=2-j$,  we obtain
\begin{align}
  \label{eq: concatenate 2 D}
    &\sum_{j=0}^2\Big(  \csp^{\alpha_2/2}\Norm{\ttau^{-\alpha_2/2}
       \Xi_1 D_{x_d}^j \Xi w}{2,2-j,\ttau}
    + \norm{\trace(\Xi_1 D_{x_d}^j \Xi w)}{1,5/2-j,\ttau}\Big) 
    \\
    &\qquad \lesssim
    \Norm{Q^{-} Q^{+}\Xi w}{+}
      + (1 -\delta_1)\norm{\trace(\Xi w)}{1,5/2,\ttau} 
      + (1-\delta_2) 
    \norm{\trace(\Xi w)}{0,7/2,\ttau}\notag\\
    &\qquad \quad + \csp \Norm{\Xi w}{3,0}
    + \Norm{w}{4,-M,\ttau}.\notag
  \end{align}
With Lemma~\ref{lemma: different
  form norm}, we write, for $\tau$ chosen \suff large, 
\begin{align*}
  &\csp^{\alpha_2/2}\Norm{ \ttau^{-\alpha_2/2}
  \Xi w}{4,0,\ttau}
  + \norm{\trace(\Xi w)}{3,1/2,\ttau}
    \\
   &\qquad \asymp\sum_{j=0}^2 \Big( \csp^{\alpha_2/2}\Norm{\ttau^{-\alpha_2/2}D_{x_d}^j 
  \Xi w}{2,2-j,\ttau}
  + \norm{\trace(D_{x_d}^j \Xi w)}{1,5/2-j,\ttau}\Big)\\
    &\qquad \lesssim \sum_{j=0}^2  \Big(\csp^{\alpha_2/2}\Norm{\ttau^{-\alpha_2/2}
       \Xi_1 D_{x_d}^j \Xi w}{2,2-j,\ttau}
    + \norm{\trace(\Xi_1 D_{x_d}^j \Xi w)}{1,5/2-j,\ttau} \Big)
    + \Norm{w}{4,-M,\ttau}.
\end{align*}
Finally, using \eqref{eq: concatenate 2 D} we obtain
\begin{align*}
  &\csp^{\alpha_2/2}\Norm{ \ttau^{-\alpha_2/2}
  \Xi w}{4,0,\ttau}
  + \norm{\trace(\Xi w)}{3,1/2,\ttau}\\
    &\qquad \lesssim
    \Norm{Q^{-} Q^{+}\Xi w}{+}
      + (1 -\delta_1)\norm{\trace(\Xi w)}{1,5/2,\ttau} 
      + (1-\delta_2) 
    \norm{\trace(\Xi w)}{0,7/2,\ttau}\notag\\
    &\qquad \quad + \csp \Norm{\Xi w}{3,0}
    + \Norm{w}{4,-M,\ttau},\notag
  \end{align*}
and taking $\tau$ \suff large, as $0\leq \alpha_2 \leq 1$,
 we achieve the sought estimate.

\bigskip 
Second, we treat the case $\alpha_1 + \alpha_2 =2$, that is, $\alpha_1
= \alpha_2 =1$. 
We set $\tD_{x_d} = D_{x_d} + i \htau_{\xi_d} \in
\PsiOpsc^{1,0}$. We use the further assumption made in this case,
namely, for any $M \in \N$, 
$[\tD_{x_d}, Q^+]\Xi_1 =  (1 + \ctp\csp) R_{2,0}\Xi_1 + R_{2,-M}$ with
$R_{2,0}\in\PsiOpsc^{2,0}$ and $R_{2,-M} \in \PsiOpsc^{2,-M}$.  
We write 
\begin{align*}
 &\sum_{j=0}^2 \Big( \Norm{ Q^{+}\Xi_1 \tD_{x_d}^j \Xi w}{0,2-j,\ttau}
  +\norm{\trace(\Xi_1 \tD_{x_d}^j \Xi  v)}{1,5/2-j,\ttau}  \Big)\\
  &\qquad \lesssim \sum_{j=0}^2 \Big( \Norm{ Q^{+} \tD_{x_d}^j  \Xi_1 \Xi w}{0,2-j,\ttau}
  +\norm{\trace(\tD_{x_d}^j \Xi  v)}{1,5/2-j,\ttau}  \Big)
  + \Norm{w}{4,-M,\ttau}\\
  &\qquad \lesssim \sum_{j=0}^2 \Big( \Norm{ \tD_{x_d}^j Q^{+}\Xi_1\Xi w}{0,2-j,\ttau}
  +\norm{\trace(\tD_{x_d}^j \Xi  v)}{1,5/2-j,\ttau}  \Big)
    +(1 +\ctp \csp) \Norm{\Xi w}{3,0}
  + \Norm{w}{4,-M,\ttau}\\
   &\qquad \lesssim \Norm{ Q^{+}\Xi w}{2,0,\ttau}
    + \norm{\trace(\Xi  v)}{3,1/2,\ttau}  
    + (1 +\ctp \csp) \Norm{\Xi w}{3,0}
  + \Norm{w}{4,-M,\ttau}.
\end{align*}
With \eqref{eq: concatenate 2 C} we then find
\begin{align*}
    &\sum_{j=0}^2 \Big( \Norm{ Q^{+}\Xi_1 \tD_{x_d}^j \Xi w}{0,2-j,\ttau}
  +\norm{\trace(\Xi_1 \tD_{x_d}^j \Xi  v)}{1,5/2-j,\ttau}  \Big)\\
  &\qquad \lesssim
    \Norm{Q^{-} Q^{+}\Xi w}{+}
      + (1 -\delta_1)\norm{\trace(\Xi w)}{1,5/2,\ttau} 
      + (1-\delta_2) \csp^{1/2}
    \norm{\ttau^{-1/2} \trace(\Xi w)}{0,7/2,\ttau}\notag\\
    &\qquad \quad + (1 +\ctp \csp)\Norm{\Xi w}{3,0}
    + \Norm{w}{4,-M,\ttau}.\notag
\end{align*}
Now, applying \eqref{eq: estimate Q+} with $\ell=2-j$,  we obtain
\begin{align}  
  \label{eq: concatenate bis}
    &\sum_{j=0}^2\Big(  \csp\Norm{\ttau^{-1}
       \Xi_1 \tD_{x_d}^j \Xi w}{2,2-j,\ttau}
    + \norm{\trace(\Xi_1 \tD_{x_d}^j \Xi w)}{1,5/2-j,\ttau}\Big) 
    \\
    &\qquad \lesssim
    \Norm{Q^{-} Q^{+}\Xi w}{+}
      + (1 -\delta_1)\norm{\trace(\Xi w)}{1,5/2,\ttau} 
      + (1-\delta_2) \csp^{1/2}
    \norm{\ttau^{-1/2}\trace(\Xi w)}{0,7/2,\ttau}\notag\\
    &\qquad \quad +  (1 +\ctp \csp)\Norm{\Xi w}{3,0}
    + \Norm{w}{4,-M,\ttau}.\notag
  \end{align}
  Now, as $[\tD_{x_d}, \ttau^{-1}] \in \gamma \PsiOpsc^{0,-1}$, we have
\begin{align*}
  &\sum_{j=0}^2 \Big( \csp\Norm{\tD_{x_d}^j \ttau^{-1}
    \Xi w}{2,2-j,\ttau}
    + \norm{\trace(\tD_{x_d}^j \Xi w)}{1,5/2-j,\ttau}\Big)\\
  &\qquad \lesssim
    \sum_{j=0}^2 \Big( \csp\Norm{ \ttau^{-1}\tD_{x_d}^j
    \Xi w}{2,2-j,\ttau}
    + \norm{\trace(\tD_{x_d}^j \Xi w)}{1,5/2-j,\ttau}\Big)
    +\csp^2 \Norm{\ttau^{-1} \Xi w}{3,0}
    \\
    &\qquad \lesssim \sum_{j=0}^2  \Big(\csp \Norm{\ttau^{-1}
       \Xi_1 \tD_{x_d}^j \Xi w}{2,2-j,\ttau}
    + \norm{\trace(\Xi_1 \tD_{x_d}^j \Xi w)}{1,5/2-j,\ttau} \Big)\\
     &\qquad \qquad +  \csp^2 \Norm{\ttau^{-1} \Xi w}{3,0}
    + \Norm{w}{4,-M,\ttau}, 
\end{align*}
yielding with \eqref{eq: concatenate bis}, as $\csp^2 \ttau^{-1}\lesssim 1$, 
\begin{align}
  \label{eq: concatenate 2 E}
  &\sum_{j=0}^2 \Big( \csp\Norm{\tD_{x_d}^j \ttau^{-1}
  \Xi w}{2,2-j,\ttau}
  + \norm{\trace(\tD_{x_d}^j \Xi w)}{1,5/2-j,\ttau}\Big)\\
 &\qquad \lesssim
    \Norm{Q^{-} Q^{+}\Xi w}{+}
      + (1 -\delta_1)\norm{\trace(\Xi w)}{1,5/2,\ttau} 
      + (1-\delta_2) \csp^{1/2}
    \norm{\ttau^{-1/2} \trace(\Xi w)}{0,7/2,\ttau}\notag\\
    &\qquad \quad  +  (1 +\ctp \csp)\Norm{\Xi w}{3,0}
    + \Norm{w}{4,-M,\ttau}.\notag
  \end{align}
As $ D_{x_d} -\tD_{x_d}  = T\in\PsiOpsc^{0,1}$,  observe that we have 
\begin{align*}
  \Norm{D_{x_d} \ttau^{-1} \Xi w}{2,1,\ttau}
  \lesssim \Norm{\tD_{x_d} \ttau^{-1} \Xi w}{2,1,\ttau}
  + \Norm{\ttau^{-1} \Xi w}{2,2,\ttau}, 
\end{align*}
meaning that we have
\begin{align*}
  \Norm{\ttau^{-1} \Xi w}{3,1,\ttau}
  \lesssim \Norm{\tD_{x_d} \ttau^{-1} \Xi w}{2,1,\ttau}
  + \Norm{\ttau^{-1} \Xi w}{2,2,\ttau}.
\end{align*}
Next, we write
\begin{align*}
\Norm{D_{x_d}^2\ttau^{-1}\Xi w}{2,0,\ttau}
  &\lesssim \Norm{D_{x_d} \tD_{x_d} \ttau^{-1}\Xi w}{2,0,\ttau}
  + \Norm{D_{x_d} T \ttau^{-1}\Xi w}{2,0,\ttau}\\
  &\lesssim \Norm{\tD_{x_d}^2 \ttau^{-1}\Xi w}{2,0,\ttau}
    + \Norm{\tD_{x_d} \ttau^{-1}\Xi w}{2,1,\ttau}
  + \Norm{ \ttau^{-1}\Xi
  w}{3,1,\ttau},
\end{align*}
and thus
\begin{align*}
  \Norm{\ttau^{-1} \Xi w}{4,0,\ttau}
  \lesssim \sum_{j=0}^2 \Norm{\tD_{x_d}^j \ttau^{-1}
  \Xi w}{2,2-j,\ttau}.
\end{align*}
Similarly, we find
\begin{align*}
  \norm{\trace(\Xi w)}{3,1/2,\ttau}
  \lesssim
  \sum_{j=0}^2\norm{\trace(\tD_{x_d}^j \Xi w)}{1,5/2-j,\ttau}.
\end{align*}
With \eqref{eq: concatenate 2 E} we thus obtain 
\begin{align*}
  &\csp \Norm{\ttau^{-1} \Xi w}{4,0,\ttau} + \norm{\trace(\Xi w)}{3,1/2,\ttau}\\
 &\qquad \lesssim
    \Norm{Q^{-} Q^{+}\Xi w}{+}
      + (1 -\delta_1)\norm{\trace(\Xi w)}{1,5/2,\ttau} 
      + (1-\delta_2) \csp^{1/2}
    \norm{\ttau^{-1/2} \trace(\Xi w)}{0,7/2,\ttau}\notag\\
    &\qquad \quad +  (1 +\ctp \csp)\Norm{\Xi w}{3,0}
    + \Norm{w}{4,-M,\ttau}.\notag
  \end{align*}
Then, taking $\csp$ \suff large and $\ctp>0$ \suff small we obtain the
sought estimate.
\end{proof}

\begin{proof}[\bfseries Proof of Lemma~\ref{sublemma: trace Q+}]
  Recalling that $Q^+ \Xi =D_{x_d}^2 \Xi + T_{1,1} \Xi$, where
  $T_{1,1}\in \PsiOpsc^{1,1}$, we have
\begin{align*}
  \norm{\trace(\Xi v)}{2, 3/2, \ttau} &\asymp
  \norm{\trace(D_{x_d}^2\Xi v)}{0, 3/2, \ttau}
  + \norm{\trace(\Xi v)}{1, 5/2, \ttau}
  =  \norm{\trace((Q^+-T_{1,1})\Xi v)}{0, 3/2, \ttau}
  + \norm{\trace(\Xi v)}{1, 5/2, \ttau}\\
  &\lesssim 
    \norm{\trace(Q^+\Xi v)}{0, 3/2, \ttau}
  + \norm{\trace(\Xi v)}{1, 5/2, \ttau}.
\end{align*}
We then write 
\begin{align*}
  \norm{\trace(\Xi v)}{3, 1/2, \ttau} &\asymp
  \norm{\trace(D_{x_d}^3\Xi v)}{0, 1/2, \ttau}
  + \norm{\trace(\Xi v)}{2, 3/2, \ttau}
  =  \norm{\trace(D_{x_d}(Q^+-T_{1,1})\Xi v)}{0, 1/2, \ttau}
  + \norm{\trace(\Xi v)}{2, 3/2, \ttau}\\
  &\lesssim 
    \norm{\trace(Q^+\Xi v)}{1, 1/2, \ttau}
  + \norm{\trace(\Xi v)}{2, 3/2, \ttau}.
\end{align*}
Combining the two estimates yields the result.
\end{proof}
\subsection{An Estimate for $Q_k$}
\label{sec: An Estimate for Qk}

We recall that 
\begin{align*}Q_k = \big(D_{x_d} + i \htau_{\xi_d}(\y')\big)^2 
+  (-1)^k i \big(D_{s} + i \htau_{\sigma}(\y')\big)^2
+ r\big(x,D_{x'} + i \htau_{\xi'}(\y')\big),
\end{align*} 
with $k=1,2$.
For this operator we have the following estimation.
%%%%%%%%%%%%%%%%%%%%%%%%
% proposition          %
%%%%%%%%%%%%%%%%%%%%%%%%
\begin{proposition}
  \label{prop: estimate Qk}
  Let $V' \Subset V$. 
  Let $\ell \in \R$. There exist  $\tau_0\geq \tauast$, $\csp_0\geq 1$ and $C>0$ \st 
  \begin{align*}
    &\csp^{1/2}\Norm{\ttau^{-1/2} v}{2,\ell,\ttau}
    + \norm{\trace(v)}{1,\ell+ 1/2,\ttau} 
      \leq C\Big(
      \Norm{Q_k v}{0,\ell,\ttau}
    + \norm{\trace(v)}{0,\ell+ 3/2,\ttau}\Big),
      \qquad k=1,2,
   \end{align*}
   for $\tau \geq \tau_0$, $\csp \geq \csp_0$, $\ctp \in [0,1]$,
  and for $v = w|_{\Rpb}$, with
  $ w \in\Cinfc(\R^N)$ and $\supp(w) \subset V'$. 
\end{proposition}
The open \nhd $V$ is that introduced in Section~\ref{sec: local setting}. 
%%%% proof of proposition
\begin{proof}
  Let $k$ be equal to $1$ or $2$. We write $Q$ in place of $Q_k$ for
  concision. We also write $\mu$ in place of $\mu_k$. 

  We need to define microlocalization symbols and operators as in
  Section~\ref{sec: def microlocal regions} and use some of the
  symbols introduced therein. Let $\chi_{V'} \in \Cinf(\R^N)$ be \st
  $\supp(\chi_{V'}) \subset V$ and $\chi_{V'}\equiv 1$ on an open \nhd
  of $V'$.

 For $\delta \in (0,1]$, we set 
\begin{align*}
  &\chi_{\delta,-} (\y') = \chi_{V'}(z)\ 
    \chi_-(\mu(\y')/\delta) \in S(1,\gt)
    \qquad \tchi_{\delta,0} (\y') = \chi_{V'}(z)\ 
  \big(1-\chi_-(\mu(\y')/\delta)\big)  \in S(1,\gt), 
\end{align*} for $\chi_-$ defined in Section~\ref{sec: def microlocal regions}, 
and observe that $\chi_{\delta,-}  + \chi_{\delta,0} = 1$ on
$\MtVp$. We set $\Xi_{\delta,-} = \Opt(\chi_{\delta,-})$ and 
$\Xi_{\delta,0}  = \Opt(\tchi_{\delta,0})$.

 In a \cnhd of $\supp(\chi_{\delta,-}) \subset \MtV$  we have $\mu \leq -C
  \delta$. As \eqref{eq: cond psi2} holds in $V$ we have 
  $\htau_{\xi_d} \geq C \ttau$ and thus $|\htau_\xi| \asymp  \ttau$.
  Thus, by Lemma~\ref{lemma: root behaviors1}, both roots of the
  symbol $q$ of the operator $Q$ are in the lower half complex
  plane. Then, with Lemma~\ref{lemma: microlocal elliptic estimate}
  we have the following perfect elliptic estimate, for any $M >0$, 
  \begin{equation}
    \label{eq: proof estimate Qk}
    \Norm{\Xi_{\delta,-} v}{2,0,\ttau}
    +\norm{\trace(\Xi_{\delta,-} v)}{1,1/2,\ttau}
    \lesssim
    \Norm{Q \Xi_{\delta,-} v}{+}
    + \Norm{v}{2,-M,\ttau},
\end{equation}
for $v\in \S(\Rpb)$, for $\tau\geq \tauast$, $\csp\geq 1$ chosen \suff large, and
$\ctp \in [0,1]$.

We now let $\uchi_{\delta}, \chi_{\delta, 1} \in S(1, \gt)$ supported
in $\MtV$, homogeneous of degree zero, be \st
$\mu \geq -C \delta$ on their supports and
$\chi_{\delta,1} \equiv 1$ in a \cnhd of $\supp(\tchi_{\delta,0})$ and
$\uchi_{\delta} \equiv 1$ in a \cnhd of $\supp(\chi_{\delta,1})$.

\medskip We choose $\delta>0$ \suff small so that the result of
Lemma~\ref{lemma: mu > -C} applies, that is, on
$\supp(\uchi_\delta)$ the roots of $q$ are simple. 
We have
\begin{align*}
  q (\y) = q_{-} (\y)  q_{+}(\y) , \quad q_{\pm}(\y)  = \xi_d - \rho_{\pm}(\y').
\end{align*}
We set $Q_{\pm} := D_{x_d} - \Opt^w(\uchi^2 \rho_{\pm})$.

We shall denote by $R_{j,k}$ as a generic operator in
  $\PsiOpsc^{j,k}$, $j \in \mathbb N$, $k \in \R$,  whose expression may
  change from one line to the other. We denote by $M$ an arbitrary
  large integer whose value may change from one line to the other. 
  We have with a proof similar to that of
Lemma~\ref{lemma: E0 factorization Qk},
\begin{align}
  \label{eq: proof estimate Qk 0}
  Q \Xi_{\delta,0} = Q_- Q_+ \Xi_{\delta,0} + \csp R_{1,0}
\Xi_{\delta,0} + R_{2,-M}.
\end{align}

In a \cnhd of $\supp(\tchi_{\delta, 0})$, the root of the symbol of $Q_-$ is 
 in the lower half complex
  plane. Then, with Lemma~\ref{lemma: microlocal elliptic estimate},
  we have the following perfect elliptic estimate, for any $M >0$, 
  \begin{equation}
    \label{eq: proof estimate Qk 1}
    \Norm{\Xi_{\delta,0} v}{1,0,\ttau}
    +\norm{\trace(\Xi_{\delta,0} v)}{0,1/2,\ttau}
    \lesssim
    \Norm{Q_- \Xi_{\delta,0} v}{+}
    + \Norm{v}{1,-M,\ttau},
\end{equation}
for $v\in \S(\Rpb)$, for $\tau\geq \tauast$, $\csp\geq 1$ chosen \suff large, and
$\ctp \in [0,1]$.

For $Q_+$ we have the following estimate,
characterized by the loss of a half derivative and a trace
observation, as given
by Lemma~\ref{lemma: sub-ellitptic estimate Qk+},
\begin{align*}
      \csp^{1/2}\Norm{\ttau^{m-1/2} \Xi_{\delta,0} v}{1,\ell, \ttau} 
      \lesssim \Norm{\ttau^{m} Q_{+} \Xi_{\delta,0} v}{0,\ell,\ttau}
  +  \norm{\trace(\ttau^{m}\Xi_{\delta,0} v)}{0,\ell+ 1/2,\ttau} 
  + \Norm{v}{0,-M, \ttau}
  , 
  \end{align*}
  for $v \in \S(\Rpb)$ and $\ell \in \R$, and for $\tau$ and $\csp$
  chosen \suff large, and $\ctp \in [0,1]$.  Then, according to
  Proposition~\ref{prop: estimate concatenation}, applied with
  $\alpha_1 =0$, $\alpha_2=1$, $\delta_1=1$, and $\delta_2=0$, we have
  the following estimates for the operator $Q_-Q_+$, for $M >0$ and
  $\ell \in \R$,
  \begin{align*}
    \csp^{1/2}\Norm{\ttau^{-1/2}\Xi_{\delta,0}  v}{2,0,\ttau}
    + \norm{\trace(\Xi_{\delta,0}  v)}{1, 1/2,\ttau} 
      \lesssim
      \Norm{Q_{-}Q_{+} \Xi_{\delta,0}  v}{+}
    +\norm{\trace(\Xi_{\delta,0}  v)}{0, 3/2,\ttau} 
      + \Norm{v}{2,-M, \ttau},
  \end{align*}
for $v \in \S(\Rpb)$, and for $\tau$ and $\csp$ chosen \suff large.
With \eqref{eq: proof estimate Qk 0} we thus obtain
\begin{align}
    \label{eq: proof estimate Qk 2}
    \csp^{1/2}\Norm{\ttau^{-1/2}\Xi_{\delta,0}  v}{2,0,\ttau}
    + \norm{\trace(\Xi_{\delta,0}  v)}{1, 1/2,\ttau} 
      \lesssim
      \Norm{Q \Xi_{\delta,0}  v}{+}
    +\norm{\trace(\Xi_{\delta,0}  v)}{0, 3/2,\ttau} 
      + \Norm{v}{2,-M, \ttau},
  \end{align}
  for $\tau$ chosen \suff large  with  the usual
  semi-classical inequality \eqref{eq: usual semi-classical argument}.

  Using that $\chi_{\delta,-}  + \chi_{\delta,0} = 1$ on
$\MtVp$ we obtain, with \eqref{eq: proof estimate Qk} and \eqref{eq: proof estimate Qk 2}
\begin{align*}
  &\csp^{1/2}\Norm{\ttau^{-1/2} v}{2,0,\ttau}
    + \norm{\trace( v)}{1, 1/2,\ttau} \\
  &\qquad  \lesssim
    \csp^{1/2}\Norm{\ttau^{-1/2} \Xi_{\delta,-} v}{2,0,\ttau}
    + \csp^{1/2}\Norm{\ttau^{-1/2} \Xi_{\delta,0}  v}{2,0,\ttau}
    + \norm{\trace( \Xi_{\delta,-} v)}{1, 1/2,\ttau}
    + \norm{\trace( \Xi_{\delta,0}  v)}{1, 1/2,\ttau}\\
  &\qquad  \lesssim
    \Norm{Q \Xi_{\delta,-}  v}{+}
    +  \Norm{Q \Xi_{\delta,0}  v}{+}
    +\norm{\trace(\Xi_{\delta,0}  v)}{0, 3/2,\ttau} 
    + \Norm{v}{2,-M, \ttau},
  \end{align*}
  for $v = w|_{\Rpb}$, with
  $ w \in\Cinfc(\R^N)$ and $\supp(w) \subset V'$. 
  Observing that $[Q, \Xi_{\delta,-} ]$ and $[Q, \Xi_{\delta,0} ]$ are
  both in $\csp \PsiOpsc^{1,0}$ we conclude the proof  with  the usual
  semi-classical inequality \eqref{eq: usual semi-classical argument}
  for $\tau$ chosen \suff large.
\end{proof}

\bigskip

\subsection*{Conflict of interest:}  The authors declare that they have no conflict of interest.

% references
\bibliographystyle{amsalpha}
\bibliography{ref-jerome}

\end{document}